\documentclass[a4paper,11pt,reqno]{amsart}
\usepackage{amsthm}
\usepackage{amsmath}
\usepackage{enumerate}
\usepackage{amsfonts}
\usepackage{amssymb}
\usepackage{fullpage}
\usepackage{amsmath,amscd}
\usepackage{graphicx}
\usepackage{array}
\usepackage{amsmath,amssymb,amsfonts}
\usepackage[utf8]{inputenc} 
\usepackage[T1]{fontenc}
\usepackage[english,french]{babel}

   \usepackage{gfsartemisia}

 \usepackage{xspace}
\usepackage{xparse}
\usepackage{graphicx}
\usepackage{float}
\usepackage{pdfpages}
\usepackage[a4paper, margin = 3cm, bottom = 3cm]{geometry}
\usepackage{ifpdf}
\usepackage{marginnote}
\usepackage{hyperref}
\usepackage{tikz}
\usepackage{csquotes}
\usetikzlibrary{calc,matrix,arrows,shapes,decorations.pathmorphing,decorations.markings,decorations.pathreplacing}
\hypersetup{pdfborder=0 0 0, 
	    colorlinks=true,
	    citecolor=black,
	    linkcolor=blue,
	    urlcolor=black,
	    pdfauthor={Guillaume Tahar,Quentin Gendron}
	   }

\usepackage[
	    style=alphabetic,
            isbn=false,
            doi=false,
            url=false,
            backend=bibtex, 
            maxnames = 5,
            sorting=nty,
           ]{biblatex}
\defbibheading{bibliography}[\bibname]{%
  \section*{Références}%
}

\DeclareFieldFormat[article]{title}{{\it #1}}
\DeclareFieldFormat{journaltitle}{{\rm #1}}
\renewbibmacro{in:}{%
  \ifentrytype{article}{}{\printtext{\bibstring{in}\intitlepunct}}}
  \bibliography{res_biblio}

\newcommand{\RR}{\mathbf{R}}

\newcommand{\CC}{\mathbf{C}}

\newcommand{\ZZ}{\mathbf{Z}}
\newcommand{\PP}{{\mathbb{P}}}

\newcommand{\pgcd}{{\rm pgcd}}

\newcommand{\ord}{\operatorname{ord}\nolimits}

\renewcommand{\tilde}{\widetilde}

\newcommand{\rot}{{\rm rot}} 
\newcommand{\ind}{{\rm Ind}}

\newcommand\Res{\operatorname{Res}}
\newcommand{\Resk}[1][k]{\Res^{#1}}
\DeclareDocumentCommand{\rec}{ O{a} O{n}}{\lbrace #1 \rbrace^{#2}}

\DeclareDocumentCommand{\appresk}{ O{g} O{k}}{\mathfrak{R}_{#1}^{#2}}

\DeclareDocumentCommand{\espresk}{O{g} O{k}}{\mathcal{R}_{#1}^{#2}}
\DeclareDocumentCommand{\appreskrho}{ O{g} O{\rho}}{\mathfrak{R}_{#1}^{k,#2}}

\newcommand{\moduli}[1][g]{{\mathcal M}_{#1}}
\newcommand{\komoduli}[1][g]{{\Omega^k\mathcal M}_{#1}}
\newcommand{\omoduli}[1][g]{{\Omega\mathcal M}_{#1}}

\newcommand{\whz}{\widehat{z}}

\newcommand{\whX}{\widehat{X}}
\newcommand{\whomega}{\widehat{\omega}}

\newcommand{\hk}{k}

\def\={\;=\;}

\newtheorem{thm}{Théorème}[section]
\newtheorem{cor}[thm]{Corollaire}
\newtheorem{prop}[thm]{Proposition}
\newtheorem{lem}[thm]{Lemme}

\theoremstyle{definition}
\newtheorem{defn}[thm]{Définition}
\theoremstyle{remark}

\theoremstyle{definition}

\theoremstyle{definition}

\theoremstyle{definition}

\theoremstyle{definition}

\numberwithin{equation}{section} 
  \usepackage{gfsartemisia}

\title{$\protect\hk$-différentielles à singularités prescrites}
\author{Quentin Gendron}
\address[Quentin Gendron]{Instituto de Matem\'{a}ticas de la UNAM
Ciudad Universitaria, CDMX, 04510,
M\'{e}xico}
\email{quentin.gendron@im.unam.mx}

\author{Guillaume Tahar}
\address[Guillaume Tahar]{Beijing Institute of Mathematical Sciences and Applications, Huairou District, Beijing, China}
\email{guillaume.tahar@bimsa.cn}

\date{\today}
\keywords{$k$-differential, Flat surface, Strata, Residue}
\begin{document}

\maketitle

\selectlanguage{english}

\section*{Abstract}
We study the local invariants that a meromorphic $k$-differential on a Riemann surface of genus $g \geq 0$ can have for $k \geq 3$. These local invariants include the orders of zeros and poles, as well as the $k$-residues at the poles. We show that for a given pattern of orders of zeros, there exists, with a few exceptions, a primitive holomorphic $k$-differential having zeros of these orders. In the meromorphic case, for genus $g \geq 1$, every expected tuple appears as a configuration of $k$-residues. On the other hand, for certain strata in genus zero, finitely many tuples (up to simultaneous scaling) do not occur as configurations of $k$-residues for a $k$-differential.

\selectlanguage{french}

\section*{Résumé}
Nous étudions les invariants locaux qu'une $k$-différentielle méromorphe sur une surface de Riemann de genre $g \geq 0$ peut posséder pour $k \geq 3$. Ces invariants locaux sont les ordres des zéros et des pôles, ainsi que les $k$-résidus aux pôles. Nous montrons que pour une distribution donnée d'ordres de zéros, il existe, à quelques exceptions près, une $k$-différentielle primitive holomorphe ayant des zéros de ces ordres. Dans le cas méromorphe, en genre $g \geq 1$, chaque $k$-uplet attendu apparaît sous la forme d'une configuration de $k$-résidus. En revanche, dans certaines strates en genre zéro, un nombre fini de $k$-uplets (à homothétie près) ne sont pas des configurations de $k$-résidus d'une $k$-différentielle.

\setcounter{tocdepth}{1}
\tableofcontents

\section{Introduction}

Soient $X$ une surface de Riemann de genre $g$ et $K_{X}$ son fibré en droites canonique. Les sections méromorphes de $K_{X}$ sont les {\em différentielles abéliennes} de $X$ et les sections du produit tensoriel $K_{X}^{\otimes k}$ sont les {\em $k$-différentielles} de $X$. Localement, une $k$-différentielle s'écrit $f(z)(dz)^{k}$, où $f$ est une fonction méromorphe. Une $k$-différentielle est {\em primitive} si elle n'est pas la puissance d'une  $d$-différentielle avec $d<k$.

Il est bien connu (voir par exemple \cite[Encadré~III.2]{dSG}) que les invariants en un point~$P$ d'une différentielle abélienne $\omega$ sont l'{\em ordre} de la différentielle en~$P$ et le {\em résidu $\Res_{P}(\omega)$}  de celle-ci dans le cas où $P$ est un pôle de $\omega$. Ce résultat a été étendu au cas des $k$-différentielles dans \cite{BCGGM3}.
Plus précisément, les invariants en $P$ d'une $k$-différentielle $\xi$ sont  l'{\em ordre} de la différentielle en $P$ et le {\em $k$-résidu}  $\Resk_{P}(\xi)$ si $P$ est un pôle de $\xi$ d'ordre divisible par $k$.

Ces invariants ne peuvent pas être fixés arbitrairement mais vérifient certaines relations. Tout d'abord le $k$-résidu d'un pôle d'ordre $-k$ est toujours non nul alors que le $k$-résidu d'un pôle dont l'ordre n'est pas divisible par $k$ est toujours nul par convention. Ensuite, la somme des ordres des zéros et des pôles d'une $k$-différentielle est égale à $k(2g-2)$. Enfin, dans le cas des différentielles abéliennes, la somme des résidus s'annule.

Comme on peut associer une surface plate à toute $k$-différentielle, les constructions géométriques de telles surfaces donnent des résultats d'existence pour certaines classes de $k$-différentielles (voir \cite{troyanov}). Cependant, aucune construction systématique n'existe à ce jour pour répondre à la question suivante.

\begin{center}
{\em \'Etant donnés les ordres des zéros et pôles ainsi que les résidus aux pôles, existe-t-il une $k$-différentielle primitive ayant ces invariants locaux?}
\end{center} 

Le cas des différentielles abéliennes ($k=1$) a été traité préalablement dans \cite{getaab} et celui des différentielles quadratiques ($k=2$) dans \cite{getaquad}. Dans cet article, nous donnons la caractérisation complète pour les $k$-différentielles avec $k \geq 3$. 

Ce problème apparaît dans différents contextes. Par exemple, il est un ingrédient essentiel dans la résolution du problème de la réalisation de diviseurs tropicaux $k$-canoniques à partir de courbes définies sur un corps non-archimédien résolu dans \cite{RS}.
Il apparaît aussi dans le problème de dégénérescence des $k$-différentielles comme étudié dans \cite{BCGGM3}. Cela peut servir à comprendre certains lieux de l'espace des modules, comme cela est fait par exemple dans \cite{HoSch}.

\smallskip
\par
\subsection{Définitions.}
Afin de préciser la question centrale, nous introduisons un certain nombre de notions. Nous dénotons par
$$\mu:=(a_{1},\dots,a_{n};-b_{1},\dots,-b_{p};-c_{1},\dots,-c_{r};\underbrace{-k,\dots,-k}_{s}) \,,$$
une partition de  $k(2g-2)$
où les $a_{i}$ sont supérieurs ou égaux à $-k+1$, les $b_{i}:= k\ell_{i}$ sont supérieurs ou égaux à $2k$ et divisibles par~$k$, les~$c_{i}$ sont supérieurs ou égaux à $k$ et non divisibles par $k$  et qui contient~$s$ fois~$-k$.
La {\em strate primitive} $\Omega^{k}\mathcal{M}_{g}(\mu)$ paramètre les $k$-différentielles {\em primitive} de type~$\mu$. Les strates non vides de $k$-différentielles primitives sont des orbifoldes de dimension $2g-2+n+p+r+s$.
 
Dans le cas des $k$-différentielles pour $k\geq2$, rappelons tout d'abord la notion de $k$-résidu. Plus de détails sont donnés dans \cite[Section~3]{BCGGM3}. Pour une $k$-différentielle $\xi$, au voisinage de chaque point $P$ de $X$, il existe une coordonnée $z$ telle que $\xi$ est de la forme
\begin{equation}\label{eq:standard_coordinates}
    \begin{cases}
      z^m\, (dz)^{k} &\text{si $m> -k$ ou  $k\nmid m$,}\\
      \left(\frac{\alpha}{z}\right)^{k}(dz)^{k} &\text{si $m = -k$,}\\
        \left(z^{m/k} + \frac{t}{z}\right)^{k}(dz)^{k} &\text{si $m < -k$ et $k\mid m$,}
    \end{cases}
\end{equation}
où $t \in \CC$ et $\alpha \in \CC^{\ast}$. Les nombres $\alpha$ et $t$ sont définis à une racine de l’unité près. Le {\em $k$-résidu} $\Resk_{P}(\xi)$ de $\xi$ en $P$ est la puissance $k$-ième de $r$ dans le second cas, la puissance $k$-ième de~$t$ dans le troisième cas et zéro sinon. Ainsi, le $k$-résidu est non nul dans le cas des pôles d'ordre $-k$ et  peut ne pas être nul uniquement dans le cas des pôles d'ordre divisible par~$k$. Notons qu'il n'existe pas de théorème des résidus pour les $k$-différentielles. Ainsi étant donnée une strate $\Omega^{k}\mathcal{M}_{g}(\mu)$, nous définissons l'{\em espace $k$-résiduel de type $\mu$} par
\begin{equation}
\espresk(\mu) := \CC^{p}\times(\CC^{\ast})^{s}.
\end{equation}
Cet espace paramètre les configurations de $k$-résidus qu'une $k$-différentielle de $\Omega^{k}\mathcal{M}_{g}(\mu)$ peut a priori posséder.

L'{\em application $k$-résiduelle} est donnée par
\begin{equation}
\appresk(\mu) \colon \Omega^{k}\mathcal{M}_{g}(\mu) \to \espresk(\mu):\ (X,\xi) \mapsto (\Resk_{P_{i}}(\xi)),
\end{equation}
où les $P_{i}$ sont les pôles de $\xi$ d'ordre divisible par $k$. Insistons sur le fait que par définition, les $k$-différentielles de $\Omega^{k}\mathcal{M}_{g}(\mu)$ sont {\em primitives}.
Le but de cet article est de déterminer l'image de cette application pour chaque strate.
\par
\smallskip
\subsection{Genre supérieur ou égal à un.}

Nous sommes maintenant en mesure d'énoncer les résultats centraux de cet article. Rappelons que $\Omega^{k}\mathcal{M}_{g}(\mu)$ paramètre les $k$-différentielles {\em primitives} de type $\mu$.

Rappelons qu'en genre $g\geq1$, les strates de $k$-différentielles ne sont en général pas connexes. En genre un, les composantes connexes sont caractérisées par le nombre de rotation (voir \cite{chge} et la section~\ref{sec:rapcomp}). En particulier, une strate de nombre de rotation $\rho$ est primitive si et seulement si $\rho$ et $k$ sont premiers entre eux. On a alors le résultat suivant.
\begin{thm}\label{thm:CC1}
Soit $S$  une composante connexe d'une strate $\Omega^{k}\mathcal{M}_{1}(\mu)$ de $k$-différentielles de genre un. La restriction à $S$ de l'application résiduelle de $\Omega^{k}\mathcal{M}_{1}(\mu)$ est surjective sauf dans deux cas.
\par
L'image de l'application résiduelle de la composante de nombre de rotation $\rho=1$ de la strate $\Omega^{3}\mathcal{M}_{1}(6;-6)$ est $\mathbb{C}^{\ast}$.
\par
Pour la composante d'invariant $\rho=2$ dans la strate $\Omega^{3}\moduli[1](12;-6,-6)$, l'image est $\mathbb{C}^{2}\setminus (0,0)$.
\end{thm}

Nous considérons maintenant le cas des strates de genre supérieur ou égal à~$2$. Dans ce cas la classification des composantes connexes n'est pas connue. Toutefois, \cite{chge} décrit différents invariants qui permettent de décrire quelques composantes (éventuellement non connexes). Ces invariants sont l'hyperellipticité et la parité des composantes. Nous montrons dans le lemme~\ref{lem:nocomphyp} qu'il n'y a pas de composante hyperelliptique de $k$-différentielles primitives dont les pôles sont d'ordre divisible par $k$ pour $k\geq3$. On définit la composante paire (resp. impaire) comme l'union des composantes connexes paires (resp. impaires) d'une strate donnée. On a alors le résultat suivant.

\begin{thm}\label{thm:ggeq2}
Si $g\geq2$ et $\mu$ contient un élément inférieur ou égal à $-k$, alors l'application résiduelle de la restriction aux composantes paires et impaires de $\appresk(\mu) \colon \Omega^{k}\mathcal{M}_{g}(\mu) \to \espresk(\mu)$ est surjective. 
% La restriction à la composante hyperelliptique est surjective si il n'y a qu'un pôle. Son image est la diagonale $(R,(-1)^{k}R)$ avec $R\in\CC^{\ast}$ (resp. $\CC$), si il y a deux pôles d'ordre $-k$ (resp. $-\ell k$ avec $\ell\geq2$).
\end{thm}
Il est utile de noter que notre méthode permettrait d'obtenir la surjectivité de l'application résiduelle restreinte à chaque composante de genre $g\geq2$ si celles-ci pouvaient être obtenues en cousant une anse et en éclatant des zéros d'une $k$-différentielle primitive.

\smallskip
\par
Enfin, nous traitons le cas des strates de différentielles n'ayant que des singularités d'ordre strictement supérieur à $-k$. Ces différentielles correspondent à des surfaces plates d'aire finie. 

\begin{thm}\label{thm:strateshol}
Soit $\mu=(a_{1},\dots,a_{n})$ une partition de $k(2g-2)$ avec $k\geq 3$ telle que les $a_{i}$ soient strictement supérieurs à $-k$. La strate $\Omega^{k}\mathcal{M}_{g}(\mu)$ paramétrant les différentielles primitives de profil $\mu$ est vide si et seulement si $g=1$ et soit $\mu=(1,-1)$, soit $\mu=\emptyset$.
\end{thm}
Ce résultat précise \cite{troyanov} qui ne se souciait pas de la primitivité des $k$-différentielles correspondant aux métriques plates obtenues. 

\smallskip
\par
\subsection{$k$-différentielles en genre zéro.}\label{sec:introg0}

Rappelons qu'étant donnée une partition
$$\mu:=(a_{1},\dots,a_{n};-b_{1},\dots,-b_{p};-c_{1},\dots,-c_{r};\rec[-k][s])$$
de $-2k$, l'espace $\Omega^{k}\mathcal{M}_{0}(\mu)$ paramètre les $k$-différentielles primitives de type $\mu$. Nous rappelons la convention selon laquelle les $a_{1},\dots,a_{n}$ sont des nombres entiers strictement supérieurs à~$-k$, tandis que $b_{1},\dots,b_{p}$ sont des multiples entiers de $k$ strictement supérieurs à $k$. Enfin, les $c_{1},\dots,c_{r}$ sont strictement supérieurs à $k$ et non divisibles par $k$. La notation $\rec[-k][s]$ désigne l’occurrence de $s$ fois le nombre $-k$ comme précisé dans la section de conventions à la fin de cette introduction.

On commence par remarquer (cf lemme~\ref{lem:puissk}) que ces strates sont non vides si et seulement si $\pgcd(\mu,k)=1$. Dans la suite, cette condition sera toujours implicitement satisfaite.

Dès lors qu'au moins trois des singularités ont un ordre qui n'est pas divisible par $k$, il n'y aucune obstruction à l'existence d'une configuration de résidus.

\begin{thm}\label{thm:g0sing3}
Soit $\Omega^{k}\mathcal{M}_{0}(\mu)$ une strate de genre zéro avec au moins trois singularités d'ordre non divisible par $k$. L'application résiduelle de $\Omega^{k}\mathcal{M}_{0}(\mu)$ est surjective.
\end{thm}

Quand il y a au plus deux singularités dont l'ordre n'est pas divisible par $k$, cela signifie en fait qu'il y en a exactement deux. Dans ce cas, il existe des obstructions non triviales. Un premier théorème décrit les obstructions lorsque les pôles ne sont pas tous d'ordre~$-k$.

\begin{thm}\label{thm:g0sing2}
Soit $\Omega^{k}\mathcal{M}_{0}(a_{1},\dots,a_{n};-b_{1},\dots,-b_{p};-c_{1},\dots,-c_{r};\rec[-k][s])$ une strate de genre zéro avec exactement deux singularités d'ordre non divisible par $k$ et telle que $p+r \neq 0$. L'image de l'application résiduelle est
 \begin{itemize}
 \item[i)] $\espresk[0](\mu)\setminus\left\{0\right\}$ si $s=0$ et que la somme des ordres des zéros d'ordre divisible par $k$ est strictement inférieure à~$kp$;
 \item[ii)] $\espresk[0](\mu)$ sinon.
\end{itemize}  
\end{thm}

Lorsque tous les pôles sont d'ordre $-k$, il n'existe qu'un nombre fini d'obstructions sporadiques. Le théorème suivant en donne la liste.

\begin{thm}\label{thm:geq0kspe}
L'application résiduelle des strates $\Omega^{k}\mathcal{M}_{0}(a_{1},\dots,a_{n};\rec[-k][s])$ est surjective pour $k\geq3$ sauf dans les cas suivants.
\begin{enumerate}
\item L'image de $\appresk[0](-1,1;-k,-k)$ est $(\CC^{\ast})^{2}\setminus \CC^{\ast}\cdot(1,(-1)^{k})$.
\item L'image de $\appresk[0][3](-1,4;\rec[-3][3])$ est $(\CC^{\ast})^{3}\setminus \CC^{\ast}\cdot\rec[1][3]$.
\item L'image de $\appresk[0][3](1,2;\rec[-3][3])$ est $(\CC^{\ast})^{3}\setminus \CC^{\ast}\cdot\rec[1][3]$.
\item L'image de $\appresk[0][3](2,4;\rec[-3][4])$ est $(\CC^{\ast})^{4}\setminus \CC^{\ast}\cdot(1,1,-1,-1)$.
\item L'image de $\appresk[0][3](2,7;\rec[-3][5])$ est $(\CC^{\ast})^{5}\setminus \CC^{\ast}\cdot(\rec[1][4],-1)$.
\item L'image de $\appresk[0][3](2,10;\rec[-3][6])$ est $(\CC^{\ast})^{6}\setminus \CC^{\ast}\cdot\rec[1][6]$.
\item L'image de $\appresk[0][3](5,7;\rec[-3][6])$ est $(\CC^{\ast})^{6}\setminus \CC^{\ast}\cdot\rec[1][6]$.
\item L'image de $\appresk[0][4](-1,5;\rec[-4][3])$ est $(\CC^{\ast})^{3}\setminus \CC^{\ast}\cdot(1,1,-4)$.
\item L'image de $\appresk[0][4](3,5;\rec[-4][4])$ est $(\CC^{\ast})^{4}\setminus \CC^{\ast}\cdot\rec[1][4]$.
\item L'image de $\appresk[0][4](-1,9;\rec[-4][4])$ est $(\CC^{\ast})^{4}\setminus \CC^{\ast}\cdot\rec[1][4]$.
\item L'image de $\appresk[0][4](3,13;\rec[-4][6])$ est $(\CC^{\ast})^{6}\setminus \CC^{\ast}\cdot\rec[1][6]$.
\item L'image de $\appresk[0][6](-1,7;\rec[-6][3])$ est $(\CC^{\ast})^{3}\setminus \CC^{\ast}\cdot\rec[1][3]$.
\item L'image de $\appresk[0][6](-1,13;\rec[-6][4])$ est $(\CC^{\ast})^{3}\setminus \CC^{\ast}\cdot\rec[1][4]$.
\end{enumerate}
\end{thm}
Il est intéressant de noter que mis à part les strates de la forme $\Omega^{k}\mathcal{M}_{0}(-1,1;-k,-k)$, toutes les exceptions proviennent du fait que les racines $k$-ièmes de l'unité engendrent un réseau de $\CC$ si $k\in\left\{3,4,6\right\}$.

\smallskip
\par
\subsection{Organisation de l'article.}

Le schéma de la preuve de ces théorèmes est le suivant. Dans un premier temps nous utilisons la correspondance entre les $k$-différentielles méromorphes et certaines classes de surfaces plates introduites par \cite{BCGGM3} afin de construire explicitement des $k$-différentielles ayant les propriétés souhaitées lorsque le genre et le nombre de zéros sont petits.

Dans un second temps,  nous déduisons le résultat général grâce à deux opérations : l'{\em éclatement d'un zéro} et la {\em couture d'anse}. La première de ces opérations permet d'augmenter le nombre de zéros sans changer le genre d'une $k$-différentielle. La seconde préserve le nombre de singularités mais augmente le genre de la surface sous-jacente.

Enfin, dans les cas où l'application résiduelle n'est pas surjective, nous développons des méthodes ad hoc afin de montrer la non-existence de $k$-différentielles possédant certains invariants. 
\smallskip
\par
L'article s'organise comme suit. Pour terminer cette introduction, nous posons quelques conventions. Dans la section~\ref{sec:bao} nous faisons les rappels nécessaires sur les représentations plates des $k$-différentielles méromorphes et sur les deux opérations précédemment citées. Nous y introduisons de plus les briques élémentaires qui nous permettent de construire les $k$-différentielles avec les propriétés souhaitées. La section~\ref{sec:avecnondiv} est consacrée au cas des $k$-différentielles de genre zéro avec au moins un pôle d'ordre non divisible par $k$. La section~\ref{sec:4DIVIS} est dédiée au cas des $k$-différentielles de genre zéro dont les pôles ne sont pas tous d'ordre~$-k$ mais sont néanmoins tous divisibles par $k$. La section~\ref{sec:juste-k} traite les $k$-différentielles en genre zéro dont tous les pôles sont égaux à $-k$. La section~\ref{sec:ggeq1} est consacrée aux cas des $k$-différentielles de genre supérieur ou égal à~$1$.

\smallskip
\par
\subsection{Conventions.}
Dans cet article le $k$-résidu sera $(2i\pi)^{k}$ fois le $k$-résidu défini par l'équation~\eqref{eq:standard_coordinates}. Remarquons que cette convention n'a aucune incidence sur l'énoncé des résultats, mais rend les preuves plus agréables.

Pour une $k$-différentielle $\xi$, nous appelons {\em zéro} une singularité de $\xi$ d'ordre strictement supérieur à $-k$ et {\em pôle} une singularité d'ordre inférieur ou égal à $-k$. Cette convention de langage est justifiée par le fait que grâce à elle les singularités coniques de la surface plate correspondent aux zéro de la $k$-différentielle. Le $k$-résidu d'une $k$-différentielle sera noté avec une lettre majuscule $R$ tandis qu'une racine $k$-ième de $R$ sera notée par une lettre minuscule~$r$. 

Si une strate paramètre des $k$-différentielles avec $m$ singularités égales à $a$, alors nous noterons cela $(\rec[a][m])$. Par exemple $\Omega^{3}\mathcal{M}_{3}(-1,3,3,3,3,4;-3)$ pourra être notée $\Omega^{3}\mathcal{M}_{3}(-1,(\rec[3][4]),4;-3)$. Plus généralement, si nous considérons une suite $(a,\dots,a)$ de $m$ nombres complexes tous identiques, nous noterons cette suite $(\rec[a][m])$. Nous espérons que ces notations seront claires dans le contexte.

\smallskip
\par
\subsection{Remerciements.} Le premier auteur remercie l'{\em Institut für algebraische Geometrie} de la {\em Leibniz Universität Hannover} et le {\em Centro de Ciencias Matemáticas} de la {\em Universidad Nacional Autonoma de México} où il a élaboré une grande partie de ce texte. Le second auteur bénéficie du financement de la Beijing Natural Science Foundation (IS23005). Nous remercions chaleureusement le rapporteur anonyme d'une version antérieure de ce texte pour sa relecture attentive et ses remarques précieuses.

\section{Boîte à outils}
\label{sec:bao}

Dans cette section, nous introduisons les objets et les opérations de base pour nos constructions. Nous commençons par quelques rappels sur les $k$-différentielles dans la section~\ref{sec:pluridiffbao}. Ensuite nous introduisons dans la section~\ref{sec:briques} les briques élémentaires de nos surfaces plates. Nous poursuivons par un rappel sur les différentielles entrelacées et les opérations de scindage de zéro et de couture d'anse dans la section~\ref{sec:pluridiffentre}. Nous introduisons dans la section~\ref{sec:arrhyp} une stratification de l'espace résiduel via la géométrie plate. Enfin la section~\ref{sec:estimation} est dédiée à la preuve d'un certain nombre d'estimations sur la somme de racines.

\subsection{$k$-différentielles méromorphes et surfaces plates}
\label{sec:pluridiffbao}

Dans ce paragraphe, nous rappelons des résultats élémentaires sur les $k$-différentielles méromorphes et leur relation avec les surfaces plates. Plus de détails peuvent être trouvés dans \cite{BCGGM3}.
\par
 Soit $X$ une surface de Riemann de genre $g$ et $\xi$ une $k$-différentielle, ie une section méromorphe de la puissance tensorielle $k$-ième du fibré canonique $K_{X}$. On notera $Z$ les zéros  et $P$ les pôles de $\xi$. L'intégration d'une racine $k$-ième de $\xi$ sur $X\setminus P$ induit une structure plate sur $X\setminus \left\{Z,P\right\}$ donnée par un atlas dont les changements de cartes sont compositions de translations et rotations d'angles multiples de $\tfrac{2\pi}{k}$. Dans la complétion métrique, les pôles d'ordre $-k<a<0$ et les zéros d'ordre $a \geq 0$ correspondent aux singularités coniques d'angle $(a+k)\frac{2\pi}{k}$. Cette similitude explique la convention de langage que les singularités d'ordre $a\geq-k+1$ sont des {\em zéros} d'ordre $a$ de $\xi$. Enfin, les pôles d'ordre~$-k$ correspondent à des demi-cylindres infinis et les pôles d'ordre $-b<-k$ à un revêtement d'ordre $b-k$ d'un domaine angulaire d'angle $\frac{2\pi}{k}$. 

Notons que le $k$-résidu $\Resk$ d'une $k$-différentielle non primitive s'obtient facilement. En effet, soit~$\xi$ une $k$-différentielle puissance $d$-ième d'une $(k/d)$-différentielle~$\eta$. Pour tout pôle~$P$ nous avons
\begin{equation}\label{eq:multiplires}
 \Resk_{P}(\xi)=\left( \Resk[k/d]_{P}(\eta) \right)^{d}\,.
\end{equation}

Nous énonçons un résultat élémentaire sur les $k$-différentielles en genre zéro.
\begin{lem}\label{lem:puissk}
Soient $\mu=(m_{1},\dots,m_{t})$ un $t$-uplet tel que $\sum m_{i}=-2k$ et $d=\pgcd(\mu,k)$. Toutes les $k$-différentielles de type $\mu$ sont la puissance $d$-ième d'une $k/d$-différentielle primitive de $\Omega^{k/d}\mathcal{M}_{0}(\mu/d)$.
\end{lem}

\begin{proof}
Une $k$-différentielle $\xi$ sur $\PP^{1}$ de type $\mu$ est donnée par la formule 
$$\xi=\prod_{i=1}^{t}(z-z_{i})^{m_{i}}(dz)^{k}=\left( \prod_{i=1}^{t}(z-z_{i})^{m_{i}/d}(dz)^{k/d}\right)^{d}. $$
\end{proof}

Finalement rappelons que le {\em cœur} d'une $k$-différentielle est l'enveloppe convexe des singularités coniques pour la métrique définie par la différentielle. 
Le complémentaire du cœur d'une surface plate admet autant de composantes connexes que de pôles. On appelle \textit{domaine polaire} la composante à laquelle un pôle appartient. Le bord d'un domaine polaire est toujours formé par un nombre fini de liens selles.

\subsection{Briques élémentaires}
\label{sec:briques}

Dans ce paragraphe, nous introduisons des surfaces plates à bord qui nous serviront de briques pour construire les $k$-différentielles ayant les propriétés locales souhaitées. Ces briques sont appelées \textit{$k$-parties polaires}.
\smallskip
\par
Étant donnés des vecteurs $(v_{1},\dots,v_{l})$ dans $(\CC^{\ast})^{l}$, nous considérons la ligne brisée $L$ dans~$\CC$ donnée par la concaténation d'une demi-droite correspondant à $\RR_{-}$, des $v_{i}$ pour $i$ croissant et d'une demi-droite correspondant à $\RR_{+}$.
Nous supposerons que les $v_{i}$ sont tels que $L$ ne possède pas de points d'auto-intersection. Nous donnerons une condition suffisante pour que cela soit possible dans le lemme~\ref{lem:noninter}.

Toutefois, il ne sera pas toujours possible de réaliser cette condition. Dans certains cas nous devrons donc autoriser des demi-droites plus générales. Le reste de la construction s'adaptera facilement à ce cas. Les objets que nous obtiendrons par cette construction seront nommés {\em $k$-parties polaires généralisées}. 

À partir du complémentaire $\CC\setminus L$ de la ligne brisée $L$, nous définissons :
\begin{itemize}
    \item le {\em domaine basique positif} $D^{+}(v_{1},\dots,v_{l})$ qui est l'adhérence de la composante connexe de $\CC\setminus L$ contenant les nombres complexes au dessus de~$L$;
    \item {\em domaine basique négatif} $D^{-}(v_{1},\dots,v_{l})$ qui est l'adhérence de l'autre composante connexe de $\CC\setminus L$.
\end{itemize}

Étant donné un domaine positif $D^{+}(v_{1},\dots,v_{l})$ et un négatif $D^{-}(w_{1},\dots,w_{l'})$, on construit le {\em domaine basique ouvert à gauche} $D_{g}(v_{1},\dots,v_{l};w_{1},\dots,w_{l'})$ en collant par translation les deux demi-droites correspondant à~$\RR_{+}$.

Si on recolle plutôt les deux demi-droites correspondant à~$\RR_{-}$, nous obtenons le {\em domaine basique ouvert à droite} $D_{d}(v_{1},\dots,v_{l};w_{1},\dots,w_{l'})$.

\par
On se donne $b:=k\ell$ avec $\ell\geq 2$ et $\tau\in\left\{1,\dots,\ell-1\right\}$.
Soient $(v_{1},\dots,v_{l};w_{1},\dots,w_{l'})$ des vecteurs de~$\CC^{\ast}$ tels que les parties réelles des sommes $\sum v_{i}$ et $\sum w_{i}$ sont positives et les domaines basiques $D^{+}(v_{1},\dots,v_{l})$ et $D^{-}(w_{1},\dots,w_{l'})$ existent.
La partie polaire d'ordre $b$ et de type~$\tau$ associée à $(v_{1},\dots,v_{l};w_{1},\dots,w_{l'})$ est la surface plate à bord obtenue de la façon suivante. Nous partons de l'union disjointe de :
\begin{itemize}
    \item[(i)] $\tau-1$ domaines basiques ouverts à gauche associé à la suite vide;
    \item[(ii)] $\ell-\tau-1$ domaines basiques ouverts à droite associé à la suite vide;
    \item[(iii)] le domaine positif associé aux $v_{i}$;
    \item[(iv)] le domaine négatif associé aux $w_{j}$.
\end{itemize}
On colle alors par translation la demi-droite inférieure du $i$-ième domaine polaire ouvert à gauche à la demi-droite supérieure du $(i+1)$-ième. La demi-droite inférieure du domaine $\tau-1$ est identifiée à la demi-droite de gauche du domaine positif. La demi-droite de gauche du domaine négatif est identifiée à la positive du premier domaine ouvert à gauche. On procède de même à droite. La figure~\ref{fig:ordreplusmoins} illustre cette construction.

\begin{figure}[htb]
\center
\begin{tikzpicture}[scale=1]
%Figure haut gauche

\begin{scope}[xshift=-6cm]
\fill[fill=black!10] (0,0)  circle (1cm);

\draw[] (0,0) coordinate (Q) -- (-1,0) coordinate[pos=.5](a);

\node[above] at (a) {$\tau$};
\node[below] at (a) {$1$};

\fill (Q)  circle (2pt);

\node at (1.5,0) {$\dots$};
\end{scope}

\begin{scope}[xshift=-3cm]
\fill[fill=black!10] (0,0)  circle (1cm);

\draw[] (0,0) coordinate (Q) -- (-1,0) coordinate[pos=.5](a);

\node[above] at (a) {$2$};
\node[below] at (a) {$3$};

\fill (Q)  circle (2pt);
\end{scope}
%deuxieme figure
\begin{scope}[xshift=0cm]
\fill[fill=black!10] (0,0)  circle (1.5cm);
      \fill[color=white]
      (-.4,0.02) -- (.4,0.02) -- (.4,-0.02) -- (-.4,-0.02) --cycle;

\draw[] (-.4,0.02)  -- (.4,0.02) coordinate[pos=.5](a);
\draw[] (-.4,-0.02)  -- (.4,-0.02) coordinate[pos=.5](b);

\node[above] at (a) {$v$};
\node[below] at (b) {$v$};

\draw[] (-.4,0) coordinate (q1) -- (-1.5,0) coordinate[pos=.5](d);
\draw[] (.4,0) coordinate (q2) -- (1.5,0) coordinate[pos=.5](e);

\node[above] at (d) {$1$};
\node[below] at (d) {$2$};
\node[above] at (e) {$\tau+1$};
\node[below] at (e) {$\tau+2$};

\fill (q1) circle (2pt);
\fill[white] (q2) circle (2pt);
\draw (q2) circle (2pt);

\node at (4.5,0) {$\dots$};
\end{scope}

\begin{scope}[xshift=3cm]
\fill[fill=black!10] (0,0) coordinate (Q) circle (1cm);

\draw[] (0,0) coordinate (Q) -- (1,0) coordinate[pos=.5](a);

\node[above] at (a) {$\tau +2$};
\node[below] at (a) {$\tau+3$};

\fill[white] (Q)  circle (2pt);
\draw (Q)  circle (2pt);
\end{scope}
%troisieme dessin
\begin{scope}[xshift=6cm]
\fill[fill=black!10] (0,0) coordinate (Q) circle (1cm);

\draw[] (0,0) coordinate (Q) -- (1,0) coordinate[pos=.5](a);

\node[above] at (a) {$\ell$};
\node[below] at (a) {$\tau+1$};

\fill[white] (Q)  circle (2pt);
\draw (Q)  circle (2pt);
\end{scope}

\end{tikzpicture}
\caption{Une $k$-partie polaire d'ordre $k\ell$ de type $\tau$ associée à $(v;v)$. Les demi-droites dont les labels coïncident sont identifiés par translation} \label{fig:ordreplusmoins}
\end{figure}

Si $\sum v_{i} =\sum w_{j}$ nous dirons que cette partie polaire est {\em triviale}. Dans le cas contraire, nous dirons que la partie polaire est {\em non triviale}. Sur la figure~\ref{fig:partiesnontrivial}, le dessin de gauche illustre une partie polaire non triviale. Le $k$-résidu du pôle d'ordre $-b=-kl$ correspondant est donné par la puissance $k$-ième de la somme $\sum v_{i}-\sum w_{j}$.

On se donne maintenant des vecteurs $(v_{1},\dots,v_{t})$ avec $t\geq1$ tels que la concaténation $V$ de ces vecteurs dans cet ordre n'a pas de points d'auto-intersection. De plus, on suppose qu'il existe deux demi-droites parallèles $L_{D}$ et $L_{F}$ de vecteur directeur $\overrightarrow{w}$, issues respectivement du point de départ $D$ et final $F$ de $V$, ne rencontrant pas $V$ et telles que $(\overrightarrow{DF},\overrightarrow{w})$ est une base positive de $\RR^{2}$. On définit la partie polaire $C(v_{1},\dots,v_{t})$ d'ordre $k$ associé aux $v_{i}$ comme le quotient du sous-ensemble de $\CC$ entre $V$ et les demi-droites $L_{D}$ et~$L_{F}$ par l'identification de $L_{D}$ à~$L_{F}$ par translation. Le $k$-résidu du pôle d'ordre $-k$ correspondant est donné par la puissance $k$ième de la somme $F-D$ des~$v_{i}$. Une partie polaire d'ordre $k$ est donnée à droite de la figure~\ref{fig:partiesnontrivial}.

\begin{figure}[htb]
\center
\begin{tikzpicture}[scale=1.2]

%premier dessin
\begin{scope}[xshift=-7cm]
\fill[fill=black!10] (0,0) coordinate (Q) circle (1.5cm);

\coordinate (a) at (-.5,0);
\coordinate (b) at (.5,0);
\coordinate (c) at (0,.2);

\fill (a)  circle (2pt);
\fill[] (b) circle (2pt);
    \fill[white] (a) -- (c)coordinate[pos=.5](f) -- (b)coordinate[pos=.5](g) -- ++(0,-2) --++(-1,0) -- cycle;
 \draw  (a) -- (c) coordinate () -- (b);
 \draw (a) -- ++(0,-1.1) coordinate (d)coordinate[pos=.5] (h);
 \draw (b) -- ++(0,-1.1) coordinate (e)coordinate[pos=.5] (i);
 \draw[dotted] (d) -- ++(0,-.3);
 \draw[dotted] (e) -- ++(0,-.3);
\node[below] at (f) {$v_{1}$};
\node[below] at (g) {$v_{2}$};
\node[left] at (h) {$1$};
\node[right] at (i) {$1$};

\draw (b)-- ++ (1,0)coordinate[pos=.6] (j);
\node[below] at (j) {$2$};
\node[above] at (j) {$3$};
    \end{scope}

%deuxieme figure
\begin{scope}[xshift=-3.5cm]
\fill[fill=black!10] (0,0) coordinate (Q) circle (1.5cm);

\draw[] (0,0) -- (1.5,0) coordinate[pos=.5](a);

\node[above] at (a) {$2$};
\node[below] at (a) {$3$};
\fill[] (Q) circle (2pt);
\end{scope}

%troisieme figure
\begin{scope}[xshift=2.5cm,yshift=-1cm]
\coordinate (a) at (-1,0);
\coordinate (b) at (1,0);
\coordinate (c) at (0,.2);

    \fill[fill=black!10] (a) -- (c)coordinate[pos=.5](f) -- (b)coordinate[pos=.5](g) -- ++(0,1.5) --++(-2,0) -- cycle;
    \fill (a)  circle (2pt);
\fill[] (b) circle (2pt);
 \draw  (a) -- (c) coordinate () -- (b);
 \draw (a) -- ++(0,1.3) coordinate (d)coordinate[pos=.5](h);
 \draw (b) -- ++(0,1.3) coordinate (e)coordinate[pos=.5](i);
 \draw[dotted] (d) -- ++(0,.3);
 \draw[dotted] (e) -- ++(0,.3);
\node[below] at (f) {$v_{1}$};
\node[below] at (g) {$v_{2}$};
\node[left] at (h) {$3$};
\node[right] at (i) {$3$};

    \end{scope}
\end{tikzpicture}
\caption{Une partie polaire non triviale associée à $(v_{1},v_{2};\emptyset)$ d'ordre $3k$ (de type $1$) à gauche et d'ordre $k$ à droite} \label{fig:partiesnontrivial}
\end{figure}

Nous traitons maintenant le cas des pôles d'ordre non divisible par $k$. Soient $c=\ell k+s$ avec $0<s<k$ et $\ell\geq 1$. On se donne des vecteurs $v_{i}$ de $\CC^{\ast}$ de partie réelle positive. La {\em $k$-partie polaire d'ordre $c$ associée aux $(v_{1},\dots,v_{l};\emptyset)$} est donnée par la construction suivante. Nous concaténons les $v_{i}$ dans le plan et traçons deux demi-droites $L_{1}$ et $L_{2}$ issues respectivement du point final et  initial de la concaténation telles que l'angle entre $L_{1}$ et $L_{2}$ est $s\frac{2\pi}{k}$.
Comme dans le cas des parties polaires d'ordre $b$, la ligne brisée ainsi formée doit être sans points d'intersection. Nous donnerons une condition suffisante pour que cela soit possible dans le lemme~\ref{lem:noninter}.
Nous considérons la surface au dessus de cette courbe brisée. Ensuite nous prenons $\ell-1$ domaines basiques ouvert dans la direction de $L_{1}$ associés à la suite vide. Puis nous identifions les demi-droites cycliquement par translation, à l'exception du dernier qui est identifié par translation et rotation d'angle $s\frac{2\pi}{k}$ à la demi droite $L_{2}$. Cette construction est illustrée à gauche de la figure~\ref{fig:partiepolairekdiff}. 
\par
Pour simplifier certaines constructions, il est utile de définir la $k$-partie polaire d'ordre~$c$ associée à $(\emptyset;v_{1},\dots,v_{l})$. Elle est définie de manière similaire à la partie polaire précédente en considérant la surface sous la courbe brisée suivante. Nous concaténons les $v_{i}$ dans le plan et traçons deux demi-droites $L_{1}$ et $L_{2}$ issues respectivement du point final et  initial de la concaténation telles que l'angle entre $L_{1}$ et $L_{2}$ est $-s\frac{2\pi}{k}$. La fin de cette construction est similaire à la précédente.
\par
Nous donnons maintenant une condition suffisante pour que la ligne brisée décrite aux paragraphes précédents soit sans points d'intersection.
\begin{lem}\label{lem:noninter}
Dans la construction des $k$-parties polaires d'ordre $c$, si les $v_{i}$ sont soit de partie réelle strictement positive, soit de partie réelle nulle et de partie imaginaire strictement positive, alors, quitte à permuter l'ordre des $v_{i}$, il existe des demi-droites $L_{1}$ et $L_{2}$ telles que la ligne brisée formée des $v_{i}$ et des $L_{j}$ soit sans point d'intersection.
\end{lem}
\begin{proof}
 Quitte à permuter les $v_{i}$, on peut supposer que les arguments des vecteurs~$v_{i}$, appartenant à l'ensemble $\left] -\tfrac{\pi}{2};\tfrac{\pi}{2}\right]$ sont décroissants. Notons que pour tout $\theta\in \left] \tfrac{\pi}{2}; \tfrac{3\pi}{2}\right]$ la demi-droite de pente $\theta$ partant du point initial de la concaténation n'a pas d'autres points d'intersection avec celle-ci. On a le même résultat pour les demi-droites d'angle $\phi$ partant du point final pour tout  $\phi\in \left] \tfrac{-\pi}{2}; \tfrac{\pi}{2}\right]$. Cela implique que l'on peut trouver des droites $L_{1}$ et~$L_{2}$ sans point d'intersection avec le reste de la construction qui forment n'importe quel angle strictement  compris entre $0$ et $2\pi$.
\end{proof}

Nous aurons besoin dans un cas d'une construction un peu plus générale que dans le cas précédent. On note comme précédemment $c=k\ell+s$ avec $0<s<k$.  \'Etant donné $(v_{1},v_{2})\in\CC^{\ast}$ la {\em partie polaire d'ordre $c$ associée à $(\emptyset;v_{1},v_{2})$ de type $t$} avec $1\leq t \leq \ell$ est donnée par la construction suivante. Nous concaténons $v_{1}$ avec $v_{2}$ dans le plan et traçons deux demi-droites $L_{1}$ et $L_{2}$ issues respectivement du point final et initial de la concaténation telles que l'angle entre $L_{1}$ et $L_{2}$ est $s\frac{2\pi}{k}$. Nous considérons la surface au dessus de cette courbe brisée. Ensuite nous prenons $t-1$ domaines basiques ouvert à droite associés à la suite vide. Puis nous identifions les demi-droites cycliquement par translation, à l'exception du dernier qui est identifié par translation et rotation d'angle $s\frac{2\pi}{k}$ à la demi droite $L_{2}$. Puis nous coupons la surface en partant du point d'intersection entre $v_{1}$ et $v_{2}$ le long d'une demi-droite $L_{3}$. Puis nous collons $\ell-t$ domaines basiques de manière cyclique à cette demi-droite. Cette construction est illustrée à droite de la figure~\ref{fig:partiepolairekdiff}.
\begin{figure}[htb]
\center
 \begin{tikzpicture}

%premier dessin
\begin{scope}[xshift=-2cm]
 \fill[black!10] (-1,0)coordinate (a) -- (1.5,0)-- (a)+(1.5,0) arc (0:120:1.5)--(a)+(120:1.5) -- cycle;

   \draw (a)  -- node [below] {$v_{1}$} (0,0) coordinate (b);
 \draw (0,0) -- (1,0) coordinate[pos=.5] (c);
 \draw[dotted] (1,0) --coordinate (p1) (1.5,0);
 \fill (a)  circle (2pt);
\fill[] (b) circle (2pt);
\node[below] at (c) {$1$};

 \draw (a) -- node [above,rotate=120] {$2$} +(120:1) coordinate (d);
 \draw[dotted] (d) -- coordinate (p2) +(120:.5);
 
    \end{scope}

%deuxieme figure
\begin{scope}[xshift=1cm,yshift=.5cm]
\fill[fill=black!10] (0,0) coordinate (Q) circle (1.2cm);

\draw[] (0,00) -- (1.2,0) coordinate[pos=.5](a);

\node[below] at (a) {$2$};
\node[above] at (a) {$1$};
\fill[] (Q) circle (2pt);
\end{scope}

%Linkes Figur!

\begin{scope}[xshift=6.5cm,yshift=1cm]
 \fill[black!10] (-1,0)coordinate (a) -- (-.7,.4)coordinate (b) -- (0,0) coordinate (c)-- (a)+(1.8,0) arc (0:-60:1.8)--(a)+(-60:1.8) -- cycle;

   \draw (a)  -- node [above left] {$v_{2}$} (b);
      \draw (b)  -- node [above] {$v_{3}$} (c);
 \draw (0,0) -- (1,0) coordinate[pos=.5] (d);
 \draw[dotted] (1,0) --coordinate (p1) (1.5,0);
 \fill (a)  circle (2pt);

\fill[] (c) circle (2pt);
\node[below] at (d) {$1$};
 \draw (a) -- node [below,rotate=-60] {$1$} +(-60:1.3) coordinate (d);
 \draw[dotted] (d) -- coordinate (p2) +(-60:.5);
 
 \draw (b)   -- ++(-60:1.9) coordinate[pos=.5](f);
\node[left] at (f) {$4$};
\node[right] at (f) {$3$};
 
 \fill[white] (b) circle (2pt);
  \draw[] (b) circle (2pt);

    \end{scope}
\begin{scope}[xshift=9cm,yshift=.5cm]
\fill[fill=black!10] (0,0) coordinate (Q) circle (1.2cm);
\draw[] (0,00) --++ (-60:1.2) coordinate[pos=.5](a);

\node[left] at (a) {$3$};
\node[right] at (a) {$4$};
 \fill[white] (0,0) circle (2pt);
  \draw[] (0,0) circle (2pt);
\end{scope}

\end{tikzpicture}
\caption{La $3$-partie polaire associée à $(v_{1};\emptyset)$ d'ordre~$7$ à gauche et la $6$-partie polaire associée à $(\emptyset;v_{2},v_{3})$ d'ordre~$13$ et de type~$1$ à droite.} \label{fig:partiepolairekdiff}
\end{figure}

Nous résumons maintenant les propriétés des constructions du paragraphe précédent.
\begin{lem}\label{lm:kresidu}
Soient $(v_{i};w_{j})$ des nombres complexes, le pôle obtenu à partir de la $k$-partie polaire d'ordre $b=k\ell$ et de type $\tau$ associée à $(v_{i};w_{j})$ est d'ordre $-b$ et possède un $k$-résidu égal à $\left(\sum v_{i}-\sum w_{j}\right)^{k}$. 

Soit $(v_{1},\dots,v_{l})$ avec $l\geq1$, le pôle associé au domaine basique d'ordre $k$ associé à $v_{i}$ est d'ordre $-k$ et possède un $k$-résidu égal à $\left(\sum v_{i}\right)^{k}$.
\end{lem}

\subsection{$k$-différentielles entrelacées, éclatement de zéros et couture d'anses.}
\label{sec:pluridiffentre}

Dans ce paragraphe, nous rappelons certains cas particuliers des résultats obtenus dans \cite{BCGGM3} au sujet des différentielles entrelacées. Cela nous permet de rappeler les constructions de l'{\em éclatement des zéros} et de la {\em couture d'anse}. 

Tout d'abord, nous rappelons la définition d'une différentielle entrelacée.
\'Etant donnée une partition~$\mu:=(m_{1},\dots,m_{t})$ telle que $\sum_{i=1}^t m_i = k(2g-2)$, une {\em $k$-différentielle entrelacée $\eta$ de type~$\mu$}
sur une courbe stable $n$-marquée $(X,z_1,\ldots,z_t)$
est une collection de $k$-différentielles non nulles~$\eta_v$ sur les composantes irréductibles~$X_v$ de~$X$ satisfaisant aux conditions suivantes.
\begin{itemize}
\item[(0)] {\bf (Annulation comme prescrit)} Chaque $k$-différentielle $\eta_v$ est méromorphe et le support de son diviseur est inclus dans l'ensemble des nœuds et des points marqués de $X_v$. De plus, si un point marqué $z_i$ se trouve sur~$X_v$, alors $\ord_{z_i} \eta_v=m_i$.
\item[(1)] {\bf (Ordres assortis)} Pour chaque nœud de $X$ qui identifie $q_1 \in X_{v_1}$ à $q_2 \in X_{v_2}$,
$$\ord_{q_1} \eta_{v_1}+\ord_{q_2} \eta_{v_2} = -2k . $$
\item[(2)] {\bf (Résidus assortis aux pôles d'ordre $-k$)} Si a un nœud de $X$
qui identifie $q_1 \in X_{v_1}$ avec $q_2 \in X_{v_2}$ on a $\ord_{q_1}\eta_{v_1}=
\ord_{q_2} \eta_{v_2}=-k$, alors
$$\Resk_{q_1}\eta_{v_1} = (-1)^k\Resk_{q_2}\eta_{v_2}.$$
\end{itemize}

Ce n'est que pour des cas très particuliers que nous aurons besoin de savoir quand une $k$-différentielle entrelacée est lissable. Nous rappelons ici uniquement les cas qui nous intéressent.  Le premier cas est celui où l'ordre des $k$-différentielles à tous les nœuds est égal à $-k$.
\begin{lem}\label{lem:lisspolessimples}
Soit $\eta=\left\{\eta_{v}\right\}$ une $k$-différentielle entrelacée. Si l'ordre des $k$-différentielles~$\eta_v$ aux nœuds est $-k$, alors $\eta$ est lissable localement.
\end{lem}
Notons que dans ce cas, la notion de différentielle entrelacée correspond à la notion classique de différentielle stable. 

Maintenant nous regardons le cas des $k$-différentielles entrelacées à deux composantes.
\begin{lem}\label{lem:lissdeuxcomp}
 Supposons que $X$ possède exactement deux composantes $X_{1}$ et $X_{2}$ reliées par un unique nœud  qui identifie $q_1 \in X_{1}$ à $q_2 \in X_{2}$. Si $\ord_{q_1} \eta_{1}>-k>\ord_{q_2} \eta_{2}$, alors la $k$-différentielle entrelacée est lissable si et seulement si l'une des deux conditions suivantes est vérifiée.
 \begin{enumerate}[i)]
  \item $\Resk_{q_2}\eta_{2}=0$
  \item $\eta_{1}$ n'est pas la puissance $k$-ième d'une $1$-forme holomorphe.
 \end{enumerate}
 De plus, le lissage peut se faire sans modifier les $k$-résidus de $\eta_{1}$ si et seulement si l'une des  deux conditions suivantes est vérifiée.
  \begin{enumerate}[i)]
  \item $\Resk_{q_2}\eta_{2}=0$
  \item $\eta_{1}$ n'est pas la puissance $k$-ième d'une différentielle abélienne méromorphe.
 \end{enumerate}
\end{lem}
Remarquons que la deuxième partie du lemme n'est pas explicitement prouvée dans \cite{BCGGM3}. Toutefois, cela peut se montrer sans problèmes en combinant la preuve du théorème~1.5 et le lemme~4.4 de \cite{BCGGM3}.

Maintenant, nous donnons deux applications cruciales du lemme~\ref{lem:lissdeuxcomp}.

\begin{prop}[Éclatement d'un zéro]\label{prop:eclatZero}
Soient $(X,\xi)$ une $k$-différentielle de type $\mu$ et $z_{0}\in X$ un zéro d'ordre $a_{0}>-k$ de $\xi$. Soit $(\alpha_{1},\dots,\alpha_{t})$ un $t$-uplet d'entiers strictement supérieurs à $-k$ tel que $\sum_{i}\alpha_{i}=a_{0}$. 

Il existe une opération sur $(X,\xi)$ en $z_{0}$ qui fournit une  $k$-différentielle $(X',\xi')$ de type $(\alpha_{0},\dots,\alpha_{t},\mu\setminus\lbrace a_{0}\rbrace)$ qui ne modifie pas les $k$-résidus de $\xi$ si et seulement si  l'une des  deux conditions suivantes est vérifiée.
  \begin{enumerate}[i)]
  \item $\xi$ n'est pas la puissance $k$-ième d'une différentielle abélienne méromorphe.
  \item Il existe une $k$-différentielle de genre zéro (éventuellement non primitive) et de type $(\alpha_{1},\dots,\alpha_{t};-a_{0}-2k)$ dont le $k$-résidu au pôle d'ordre $-a_{0}-2k$ est nul. 
 \end{enumerate} 
De plus, si $\xi=\omega^{k}$ avec $\omega$ une différentielle abélienne méromorphe, alors la $k$-différentielle~$\xi'$ est primitive si et seulement si  $\pgcd(\alpha_{i},d)=1$. 
\end{prop}
\begin{proof}
 Partons de $(X,\xi)$. On forme une différentielle entrelacée en attachant au point~$z_{0}$ une droite projective avec une différentielle ayant les ordres souhaités. Le lemme~\ref{lem:lissdeuxcomp} implique facilement la proposition~\ref{prop:eclatZero}.
\end{proof}

La seconde construction nous permettra en particulier de faire une récurrence sur le genre des surfaces de Riemann.

\begin{prop}[Couture d'anse]\label{prop:attachanse}
 Soient $(X,\xi)$ une $k$-différentielle (primitive) dans la strate $\Omega^{k}\mathcal{M}_{g}(\mu)$ et $z_{0}\in X$ un zéro d'ordre $a_{0}$ de $\xi$. 
 Il existe une opération locale à $z_{0}$ qui produit une $k$-différentielle $(X',\xi')$ dans la strate $\Omega^{k}\mathcal{M}_{g+1}(a_{0}+2k,\mu\setminus\left\{a_{0}\right\})$. 
\end{prop}
\begin{proof}
 Partons de $(X,\xi)$. On forme une différentielle entrelacée en attachant au point $z_{0}$ une courbe elliptique avec une différentielle de type $(a_{0}+2k;-a_{0}-2k)$. Le lemme~\ref{lem:lissdeuxcomp} permet de conclure.
\end{proof}

\smallskip
\par
Il est à noter que dans le cas des différentielles quadratiques, l'éclatement d'un zéro peut ne pas être réalisable de manière locale (voir \cite{MZ}).
Nous pouvons caractériser tous les cas où l'éclatement d'une singularité conique n'est pas possible localement.
\begin{prop}\label{prop:eclatintro}
Soient $\xi$ une $k$-différentielle avec $k\geq1$ et $z_{0}$ une singularité d'ordre $a_{0}>-k$ de~$\xi$.
Le scindage de $z_{0}$ en $t$ singularités d'ordres  $(\alpha_{1},\dots,\alpha_{t})$ avec $a_{0}=\sum \alpha_{i}$ et $\alpha_{i}>-k$ n'est pas {\em local} précisément lorsque $k\geq2$, $t=2$, $k\mid a_{0}$ et  $k\nmid\pgcd(\alpha_{i})$.
\end{prop}

\begin{proof}
L'éclatement d'un zéro d'ordre $a>-k$ d'une $k$-différentielle~$\xi$ en $n$ zéros d'ordres $(a_{1},\dots,a_{n})$ correspond au lissage d'une $k$-différentielle entrelacée. Cette $k$-différentielle entrelacée est constituée d'une $k$-différentielle $\xi_{0}$ sur $\PP^{1}$ avec  des zéros d'ordres $(a_{1},\dots,a_{n})$ et un pôle d'ordre $-a-2k$ attachée au zéro d'ordre $a$ de~$\xi$ (voir la proposition~\ref{prop:eclatZero}). Cette $k$-différentielle entrelacée est lissable localement si et seulement tous les $k$-résidus de $\xi_{0}$ sont nuls. L'éclatement est possible localement si et seulement s'il existe une $k$-différentielle de genre zéro de type $(a_{1},\dots,a_{n};-2k-a)$ dont tous les $k$-résidus sont nuls. La proposition~\ref{prop:eclatintro} est alors une conséquence directe des théorèmes~\ref{thm:g0sing3} et \ref{thm:g0sing2}, ainsi que de l'équation~\eqref{eq:multiplires} dans le cas $k\geq3$. Le cas $k=1$ se déduit du théorème des résidus. Enfin le cas $k=2$ peut se déduire des théorèmes~1.3 et~1.7 de \cite{getaquad}. Toutefois ce cas a été considéré dans les sections~5 et~6 de \cite{MZ} dans le cas de deux zéros et le cas avec trois zéros (dont deux sont impairs) dans \cite{Boissy09} (voir en particulier la figure~8). Cette dernière figure se généralise facilement pour obtenir le résultat pour un nombre quelconque de zéros.
\end{proof}

Nous terminons cette section par un résultat qui sera utile pour éclater les zéros d'une $k$-différentielle en genre zéro. 

\begin{lem}\label{lem:combiscindage}
Pour $k \geq 3$ et $n \geq 2$, on considère une famille de $n$ entiers $a_{1},\dots,a_{n}$ telle que :
\begin{itemize}
    \item[i)] $a_{i}>-k$ pour tout $1 \leq i \leq n$;
    \item[ii)] $\pgcd(a_{1},\dots,a_{n})=1$;
    \item[iii)] $\sum a_{i} =kL$ où $L \geq -1$.
\end{itemize}
\smallskip
Pour tout $d$ tel que $2 \leq d \leq n$, on peut alors partitionner les $(a_{1},\dots,a_{n})$ en $d$ sous-ensembles $E_{1},\dots,E_{d}$ non vides de sommes $A_{1},\dots,A_{d}$ tels que :
\begin{itemize}
    \item[i)] $A_{1},\dots,A_{d} >-k$;    
    \item[ii)] $\pgcd(A_{1},\dots,A_{d}) \notin \lbrace{ \frac{k}{2};k \rbrace}$.
\end{itemize}
On appelle \textit{admissible} une telle partition.
\end{lem}

\begin{proof}
Il est suffisant de démontrer la propriété pour $d=2$. En effet, supposons une partition admissible en $d$ sous-ensembles. Si $d<n$, alors l'un des sous-ensembles (disons~$E_{1}$) n'est pas un singleton. Considérons l'élément le plus petit de $E_{1}$ (disons $a_{1}$). Si l'on découpe~$E_{1}$ entre un singleton contenant $a_{1}$ et un sous-ensemble les autres éléments, nous obtenons une nouvelle partition de $(a_{1},\dots,a_{n})$ qui sera également admissible. Ainsi, si la propriété est valide pour $d=2$, elle le sera pour chaque $d$ entre $2$ et $n$.

Dans ce qui suit, nous démontrons la propriété pour $d=2$. Tout d'abord, nous traitons le cas où il existe un élément $a_{i}<0$ tel que $a_{i} \neq -\tfrac{k}{2}$. Posons $\lbrace a_{i} \rbrace$ comme premier sous-ensemble. Nous avons donc $A_{1}=a_{i}$ et $A_{2}=kL-a_{i}$. Il apparaît que $A_{1},A_{2} \notin \frac{k}{2}\mathbb{Z}$ et $A_{1},A_{2}>-k$. 

Nous supposerons donc à partir de maintenant que les $a_{j}$ appartiennent à $\lbrace -\frac{k}{2} \rbrace \cup \mathbb{N}^{\ast}$. Supposons que  $a_{j}\geq0$ pour tout $j$. Comme le pgcd des $a_{j}$ est $1$, il existe au moins un  élément $a_{i} \notin \frac{k}{2}\mathbb{Z}$. Il suffit de choisir comme premier sous-ensemble $\lbrace a_{i} \rbrace $ et comme second sous-ensemble les $n-1$ autres éléments pour obtenir la propriété. 

Supposons finalement qu'il existe un élément $a_{i}>0$ non divisible par $\frac{k}{2}$ et un élément $a_{j}=-\frac{k}{2}$. Notons que $a_{i}+a_{j} \notin \frac{k}{2}\mathbb{Z}$ et $a_{i}+a_{j}>-k$. Il suffit donc de répéter cette opération jusqu'à se ramener à l'un des cas précédent.
\end{proof}

Le lemme~\ref{lem:combiscindage} montre que l'on peut toujours obtenir une $k$-différentielle avec $n\geq3$ zéros d'ordres prescrits en éclatant un zéro d'une $k$-différentielle avec $2$ zéros. La $k$-différentielle éclatée n'est pas toujours primitive, mais le lemme assure qu'elle n'est pas la puissance d'une $1$-forme ou d'une différentielle quadratique. Par exemple, les $30$-différentielles de la strate $\Omega^{30}\mathcal{M}_{0}(-3,-5,-22;-30)$ ne peuvent jamais s'obtenir à partir de $30$-différentielles primitives car les trois strates $\Omega^{30}\mathcal{M}_{0}(-3,-27;-30)$, $\Omega^{30}\mathcal{M}_{0}(-5,-25;-30)$ et $\Omega^{30}\mathcal{M}_{0}(-8,-22;-30)$ ne sont pas primitives. 

\subsection{Stratification de résonance et prolongement plat}\label{sec:arrhyp}

Les longueurs des liens-selles des $k$-différentielles des strates $\komoduli[0](a_{1},a_{2};-b_{1},\dots,-b_{p})$  sont proportionnelles à des normes de sommes partielles de racines de $k$-résidus. De fait, la longueur d'un lien-selle fermé est égale à la norme d'une somme partielle. En effet, l'une des deux composantes connexes du complémentaire d'un tel lien-selle est une surface de translation car son intérieur ne contient aucun des deux zéros. L'invariance par déformation de l'intégrale d'une $1$-forme le long d'un chemin fermé dans cette composante implique une relation entre la période de l'unique lien-selle de bord et la somme des résidus entourés. De même, le complémentaire d'un lien-selle reliant les deux zéros est une surface de translation et le théorème des résidus montre que la longueur d'un lien-selle entre les deux zéros est proportionnelle à une somme totale avec un coefficient de proportionnalité égal à la norme d'une différence de deux racines $k$-ièmes de l'unité. La dégénérescence de ces liens-selles correspond donc à l'annulation de certaines de ces sommes. Nous encodons ces conditions dans un arrangement d'hyperplans complexes appelé \textit{arrangement de résonance} dont la projection sur l'espace résiduel $\espresk[0](a_{1},a_{2};-b_{1},\dots,-b_{p})$ définit la \textit{stratification de résonance}.

Notons $W_{k}= \lbrace{ 0\rbrace} \cup \lbrace{ e^{\frac{2ai\pi}{k}}~|~a \in \mathbb{Z}/k\mathbb{Z} \rbrace}$.
Dans $\CC^{p}$ un \textit{hyperplan de résonance} est l'ensemble des points $(r_{1},\dots,r_{p})$ satisfaisant une équation de la forme $\sum\limits_{j=1}^{p} w_{j}r_{j}=0$, avec $w_{j}\in W_{k}$ non tous nuls.
L'ensemble $H_{k,p}$ des hyperplans de résonance définit un arrangement d'hyperplans complexes dans $\CC^{p}$. Un cas particulier de cette structure a été étudié dans \cite{IFRA} pour calculer le nombre d'éléments des fibres isorésiduelles discrètes dans le cas $k=1$.

\begin{defn}\label{defn:stratres}
Soit $R=(R_{1},\dots,R_{p}) \in \espresk[0](a_{1},a_{2};-b_{1},\dots,-b_{p})$ et $r=(r_{1},\dots,r_{p})$ des racines $k$-ièmes des $R_{i}$. L'ensemble des hyperplans de résonance contenant $r$ est $H_{k,p}(R) \subset H_{k,p}$.

La \textit{stratification de résonance de $\espresk[0](a_{1},a_{2};-b_{1},\dots,-b_{p})$} est la projection de l'union des hyperplans de résonance par l'application $(r_{1},\dots,r_{p}) \mapsto (r_{1}^{k},\dots,r_{p}^{k})$. Deux configurations $R,R'$ appartiennent à la même strate de résonance si et seulement si $H_{k,p}(R)=H_{k,p}(R')$.
\par
Notons que l'ensemble $H_{k,p}(R)$ ne dépend pas du choix de $r$.
\end{defn}

Dans chaque strate de résonance, nous définissons une systole résiduelle qui va nous permettre de contrôler les déformations de $k$-différentielles.

\begin{defn}\label{defn:systole}
Soit une configuration $R=(R_{1},\dots,R_{p})$ de $k$-résidus, la {\em systole résiduelle} de $R$ est
 \[ \sigma = \min \left\{ \left| \sum_{i\in I} r_{i} \right| : I\subset \lbrace1,\dots,p \rbrace, r_{i}^{k}=R_{i} \text{ et } \sum_{i\in I} r_{i} \neq 0 \right\}\,.\]
\end{defn}

\begin{prop}\label{prop:systole}
La systole résiduelle varie continûment dans chaque strate de résonance.
\end{prop}

\begin{proof}
Il y a un nombre fini de sommes pondérées de racines et chacune d'elles varie continûment en fonction des $k$-résidus. De plus, par définition, pour tout $I\subset \lbrace1,\dots,p \rbrace$ tel que $\sum_{i\in I} r_{i} \neq 0$, cette somme demeure non nulle tant que l'on reste dans la même strate (par définition). 
\end{proof}

Nous allons maintenant utiliser un prolongement plat qui nous permettra d'obtenir l'existence de $k$-différentielles par déformation.

\begin{cor}\label{cor:defplate}
S'il existe une $k$-différentielle de $\komoduli[0](a_{1},a_{2};-b_{1},\dots,-b_{p})$ dont les résidus appartiennent à une strate de résonance, alors toutes les configurations de cette strate de résonance sont réalisables par une $k$-différentielle de $\komoduli[0](a_{1},a_{2};-b_{1},\dots,-b_{p})$.
\end{cor}

\begin{proof}
Soit $\xi$ une $k$-différentielle dont les résidus sont $R_{1},\dots,R_{p}$. Dans chaque carte de $\komoduli[0](a_{1},a_{2};-b_{1},\dots,-b_{p})$ chaque lien-selle est de longueur proportionnelle à une somme partielle de racines $r_{i}$. La longueur de ces liens-selles est minorée par un multiple de la systole résiduelle (voir définition~\ref{defn:systole}). Comme celle-ci varie continûment et que la dégénérescence d'une différentielle implique que la longueur d'un lien-selle tende vers $0$, on peut déformer $\xi$ dans un voisinage dans la strate de résonance. Comme les strates sont connexes (puisque l'intersection d'une strate avec d'autres strates est un lieu de codimension complexe au moins $1$) on peut déformer $\xi$ en une différentielle ayant n'importe quelle configuration de résidus de la même strate de résonance.
\end{proof}

\subsection{Estimation des sommes de racines}
\label{sec:estimation}

Donnons nous une configuration de $k$-résidus non nuls $R_{1},\dots,R_{s}$. Selon les nécessités des construction géométriques, on souhaite que le choix $r_{1},\dots,r_{s}$ d'une racine $k$-ième pour chacun des résidus permette que la somme de ces racines soit de module grand ou petit.
\par

Nous donnons d'abord une borne inférieure à ces sommes de racines de grand module.

\begin{lem}\label{lem:estimatecos}
Pour $k \geq 3$ et $s \geq 2$, nous considérons des nombres $R_{1},\dots,R_{s} \in \CC^{\ast}$ rangés par ordre décroissant de module et les racines $k$-ièmes $r_{1},\dots,r_{s}$ des $R_{1},\dots,R_{s}$ pour lequel $|\sum r_{i}|$ prend la valeur maximale. Nous avons :
$$ |\sum r_{i}| \geq |r_{1}| + \cos(\pi/k)\sum\limits_{i=2}^{s} |r_{i}|
.
$$
\end{lem}

\begin{proof}
La démonstration se fait par récurrence. L'étape d'initialisation $s=1$ est triviale. Pour l'étape d'hérédité, il suffit de constater qu'il existe une racine $r_{s+1}$ de $R_{s+1}$ dont l'argument diffère de celui $\sum\limits_{i=1}^{s} r_{i}$ d'au plus $\frac{\pi}{k}$. Il s'ensuit que la projection de $r_{s+1}$ sur l'axe de $\sum\limits_{i=1}^{s} r_{i}$ est pondérée d'un coefficient valant au moins $\cos(\pi/k)$.
\end{proof}

Trouver une borne supérieure pour une somme non nulle de telles racines est plus délicat et fait intervenir les entiers de Gauss et d'Eisenstein, i.e. des éléments des réseaux $\ZZ \oplus i\ZZ$ et $\ZZ \oplus \omega\ZZ$ où $\omega=e^{\frac{2i\pi}{3}}$. La preuve est différente selon les valeurs de $k$.

\begin{lem}\label{lem:estimate}
Pour $R_{1},\dots,R_{s} \in \CC^{\ast}$ avec $s \geq 2$, nous notons $r_{\max}= \max(|R_{i}|^{1/k})$.
\par
Il existe un choix de racines $k$-ièmes $r_{1},\dots,r_{s}$ tel que $0 < |\sum r_{i}| < r_{max}$ sauf dans les {\em configurations exceptionnelles} qui suivent pour lesquelles seule une borne plus faible est réalisable.
\par
Si $k=3$ et les configurations sont $\CC^{\ast}\cdot \left(\rec[1][s_{1}],\rec[-1][s_{2}],(\rec[3i\sqrt{3}][t_{1}]),(\rec[-3i\sqrt{3}][t_{2}])\right)$ avec $t_{1}+t_{2} \geq 1$ et $s_{1}-s_{2} \in 3 \ZZ$, il existe des racines $r_{i}$ telles que :
\begin{itemize}
    \item $ 0 < |\sum r_{i}| \leq \sqrt{3} r_{\max}$, 
    si $s_{1}=s_{2}=0$ et $t_{2}-t_{1} \in 3\ZZ$;
    \item $ 0 < |\sum r_{i}| \leq r_{\max}$, sinon.
\end{itemize}
\par
Si $k=4$ et les configurations sont $\CC^{\ast}\cdot(\rec[1][s_{1}],\rec[-4][s_{2}])$ avec $s_{1}$ pair et $s_{2} \geq 1$, il existe des racines $r_{i}$ telles que :
\begin{itemize}
    \item $ 0 < |\sum r_{i}| \leq \sqrt{2} r_{\max}$, 
    si $s_{1}+2s_{2} \in 4\ZZ$;
    \item $ 0 < |\sum r_{i}| \leq r_{\max}$ 
    si $s_{1}+2s_{2} \in 4\ZZ+2$.
\end{itemize}
\par
Si $k=5$ et les configurations sont $\CC^{\ast} \cdot (1,-1)$, alors il existe des racines $r_{1},r_{2}$ telles que $|r_{1}+r_{2}|=2\sin(\pi/5)|r_{1}|=\sqrt{\frac{5}{2}-\frac{\sqrt{5}}{2}}|r_{1}|$.
\par
Si $k=6$ et les configurations sont $\CC^{\ast}\cdot(\rec[1][s])$, on a des racines $r_{i}$ telles que $0 < |\sum r_{i}| \leq r_{\max}$.
\par
Quand $s \geq 3$, ces inégalités peuvent toujours être satisfaites par un choix de racines $r_{1},\dots,r_{s}$ qui ne sont pas $\mathbb{R}$-colinéaires.
\end{lem}

Par convention, nous supposerons que les nombres $R_{1},\dots,R_{s}$ sont classés par ordre croissant de module. De plus, quand bien même la racine $r_{i}$ de $R_{i}$ n'a pas été encore choisie, son module $|r_{i}|=|R_{i}|^{1/k}$ est toujours bien défini.
\par
De plus, dans une configuration exceptionnelle la somme de racines appartient (à un facteur près) au réseau $\ZZ+\omega\ZZ$, $3\ZZ+(-1+\omega)\ZZ$, $\ZZ+i\ZZ$ ou $(1+i)\ZZ+(1-i)\omega\ZZ$ (selon le cas). Ceci explique les constantes $\sqrt{2}$ et $\sqrt{3}$.

\subsubsection{\bf Preuve de la borne supérieure dans le cas $k \geq 6$}

Si $s = 2$, en fixant $r_{2} = 1$, en dehors de la configuration $(1,1)$ pour $k=6$, il y a toujours deux choix de racine $k$-ième de $R_{1}$ telle que $0<|r_{1}+ r_{2} | < 1$. Il suffit de choisir celle qui n'est pas réelle. Dans le cas de la configuration exceptionnelle $(1,1)$ avec $k=6$, $r_{1}=e^{\frac{2i\pi}{3}}$, $r_{1},r_{2}$ ne sont pas $\mathbb{R}$-colinéaires et leur somme satisfait l'inégalité affaiblie.
\par
Pour $s=3$, supposons pour commencer que $(R_{1},R_{2})$ n'est pas une configuration exceptionnelle du cas $s = 2$, autrement dit $R_{1} \neq R_{2}$ ou $k \neq 6$. Nous trouvons des racines $r_{1},r_{2}$ telles que $0<|r_{1}+r_{2}|< |r_{2}|\leq |r_{3}|$. Ensuite, comme $k \geq 6$, il y a au moins deux choix possibles pour une racine $r_{3}$ de $R_{3}$ telle que $|r_{1}+r_{2}+r_{3}|<|r_{3}|$. Il suffit d'en choisir une telle que les trois racines ne sont pas $\mathbb{R}$-colinéaires.
\par
Si $R_{1}=R_{2}$ et $k=6$, on est dans le cas $(1,1,R_{3})$ avec $|R_{3}| \geq 1$ et $k=6$. Si $|R_{3}|>1$, alors $r_{1}=1$ et $r_{2}=\omega$ satisfont $0<|r_{1}+r_{2}| \leq |r_{2}|<|r_{3}|$. Le même raisonnement que précédemment permet alors de trouver nos trois racines. Si $|R_{3}|=1$ mais $R_{3} \neq 1$, alors nous pouvons permuter les indices et nous nous ramenons au cas dans lequel $(R_{1},R_{2})$ n'est pas une configuration exceptionnelle du cas $s=2$. Enfin, le cas $R_{3}=1$ est l'unique configuration exceptionnelle pour $s=3$.
\par
Pour $s \geq 4$, la preuve se termine par récurrence. En effet, supposons que la proposition soit valide pour un certain rang $s \geq 3$. Dans le cas d'une configuration exceptionnelle, on peut supposer par récurrence que la somme des $s$ premières racines est égale à $1$. Il suffit donc de choisir $\omega$ pour la dernière racine. Considérons une configuration non exceptionnelle $R_{1},\dots,R_{s+1}$. Nous pouvons remplacer une paire $R_{i}\neq R_{j}$ par un seul nombre $(r_{i}+r_{j})^{k}$ en choisissant $r_{i},r_{j}$ de telle façon que $0<|r_{i}+r_{j}|< \max(|R_{i}|^{1/k},|R_{j}|^{1/k}|)$. L'hypothèse de récurrence permet de choisir les autres racines non toutes $\mathbb{R}$-colinéaires de façon à vérifier l'encadrement souhaité.

\subsubsection{\bf Preuve de la borne supérieure dans le cas $k=5$}

Si $s=2$, en fixant $r_{2}=1$, il y a toujours une racine $5$-ième de $R_{1}$ telle que $|r_{1}+ r_{2} | < 1$. Si $R_{1} \neq -1$, il suffit de choisir cette racine et nous aurons automatiquement $|r_{1}+r_{2}|>0$ également. Cependant, si $R_{1}=-1$, il faut choisir $r_{1}=-e^{\frac{2i\pi}{5}}$ et un calcul trigonométrique montre que $|r_{1}+r_{2}|=2\sin(\pi/5)$.
\par
Pour $s=3$, nous supposerons d'abord que la configuration $(R_{1},R_{2})$ n'est pas proportionnelle à $(1,-1)$. Dans ce cas, on trouve des racines $r_{1},r_{2}$ telles que $0<|r_{1}+r_{2}|<|r_{2}| \leq |r_{3}|$. Ensuite, comme $k=5$, on peut toujours trouver une racine $r_{3}$ telle que la somme $r_{1}+r_{2}+r_{3}$ vérifie $0 < |r_{1}+r_{2}+r_{3}| < |r_{3}|$.
\par
Nous supposerons maintenant que $R_{2}=-R_{1}$. Si de surcroît nous avons $R_{3}=-R_{2}$, la configuration $(R_{1},R_{2},R_{3})$ est proportionnelle à $(1,1,-1)$ et il suffit de choisir les racines $e^{\frac{2i\pi}{5}}$ et $e^{-\frac{2i\pi}{5}}$ et $-1$ dont la somme vaut $-1+2\cos(2\pi/5)=\frac{\sqrt{5}-3}{2}$.
\par
Nous pouvons désormais supposer que $R_{3} \neq -R_{2}$. On peut alors trouver des racines $r_{2}$ et~$r_{3}$ telles que $0<|r_{2}+r_{3}|<|r_{3}|$. En dehors du cas dans lequel $(r_{2}+r_{3})^{5}=-R_{1}$, on peut alors choisir la racine $r_{1}$ pour avoir $0<|r_{1}+r_{2}+r_{3}|<|r_{3}|$. Si nous avons $(r_{2}+r_{3})^{5}=-R_{1}$, alors la configuration $(R_{1},R_{2},R_{3})$ est proportionnelle à l'une des configurations $(1,-1,(1-e^{\frac{2mi\pi}{5}})^{5})$ avec $1 \leq m \leq 4$.
\par
Les cas $m=3$ et $m=4$ se déduisent de $m=1$ et $m=2$ par conjugaison. Pour $m=2$, il suffit de choisir $r_{3}=1-e^{\frac{4i\pi}{5}}$ puis $r_{1}=e^{\frac{2i\pi}{5}}$ et $r_{2}=-1$ pour obtenir $r_{1}+r_{2}+r_{3}=e^{\frac{2i\pi}{5}}-e^{\frac{4i\pi}{5}}$.  Pour $m=1$, nous choisissons $r_{3}=1-e^{\frac{2i\pi}{5}}$, puis $r_{1}=e^{\frac{4i\pi}{5}}$ et $r_{2}=-e^{\frac{6i\pi}{5}}$ pour obtenir $r_{1}+r_{2}+r_{3}=(1+e^{\frac{4i\pi}{5}})r_{3}$. Nous obtenons bien $|r_{1}+r_{2}+r_{3}|<|r_{3}|$ dans les deux cas.
\par
Si cette inégalité est réalisée par des racines qui ne sont pas $\mathbb{R}$-colinéaires, la proposition est démontrée. Si elles le sont, nous allons montrer qu'un autre choix de racines, cette fois-ci non $\mathbb{R}$-colinéaires, permet aussi de satisfaire l'inégalité. Sans perte de généralité, nous supposerons donc que $R_{1},R_{2},R_{3}$ sont réels. Si ce sont tous des réels de même signe, un choix de racines  $\mathbb{R}$-colinéaires n'a pas pu satisfaire l'inégalité car nous aurions $|r_{1}+r_{2}+r_{3}|=|r_{1}|+|r_{2}|+|r_{3}|$. Nous supposerons donc que deux résidus $R_{i},R_{j}$ sont positifs tandis que le troisième $R_{k}$ est négatif.
\par
Nous choisissons les racines $r_{i},r_{j}$ avec $\arg(r_{i})=0$ et $\arg(r_{j})=\frac{4\pi}{5}$, nous avons $0<|r_{i}+r_{j}|< \max(|R_{i}|^{1/5},|R_{j}|^{1/5}|)$. Comme dans le cas $s=2$, sauf si $R_{k}=-(r_{i}+r_{j})^{5}$, il existe toujours une racine $r_{k}$ telle que $0<|r_{i}+r_{j}+r_{k}|<|r_{3}|$ (et nos trois racines ne sont pas toutes $\mathbb{R}$-colinéaires).
\par
Dans le dernier cas, on a $R_{k}=-(r_{i}+r_{j})^{5}$, ce qui n'est possible que si $\arg(r_{i}+r_{j})=\frac{2\pi}{5}$ et donc $|r_{i}|=|r_{j}|$ (puisque $R_{k}$ est un réel négatif). Ce cas, dans lequel les résidus sont proportionnels à $(1,1,-1)$, a déjà une solution satisfaisante vue précédemment, ce qui clôt le cas $s=3$. 
\par
Pour $s \geq 4$, la preuve se termine par récurrence. Nous supposons la propriété valide pour un certain rang $s \geq 3$. Pour une configuration de $s+1$ résidus, nous pouvons trouver des racines $r_{1},\dots,r_{s}$ des $s$ premiers résidus telles que $0<|r_{1}+\dots+r_{s}|<r_{s}$ (et qui ne sont pas $\mathbb{R}$-colinéaires). Sachant que $|R_{s+1}| \geq |R_{s}|$, il existe alors toujours un choix de racine $r_{s+1}$ de~$R_{s+1}$ telle que $|r_{1}+\dots+r_{s}+r_{s+1}|<|r_{s+1}|$.

\subsubsection{\bf Preuve de la borne supérieure dans le cas $k=4$}

Si $s = 2$, en fixant $r_{2} = 1$, on peut toujours trouver une racine $4$-ième de $R_{1}$ telle que $0<|r_{1}+ r_{2} | < 1$ sauf si $R_{1}=1$. Nous avons alors une configuration exceptionnelle.
\par
Pour $s=3$, nous normalisons par $r_{3}=1$. Nous traitons d'emblée le cas $R_{1}=R_{2}$ en choisissant des racines $r_{1},r_{2}$ telles que $r_{2}=ir_{1}$. Comme $|(1+i)r_{1}|\leq \sqrt{2}$, il va exister un choix pour $r_{1}$ tel que $0<|1+r_{1}+r_{2}|<1$ sauf si $R_{1}=R_{2}=1$ ou $R_{1}=R_{2}=-\frac{1}{4}$. Il s'agit des deux familles de configurations exceptionnelles.
\par
Si au contraire $R_{1} \neq R_{2}$, nous trouvons des racines $r_{1},r_{2}$ telles que $0<|r_{1}+r_{2}|< |r_{2}| \leq |r_{3}|$. La résolution du cas $s = 2$ nous donne les racines adéquates pour $(r_{1}+r_{2})^{4},R_{3}$, dont nous déduisons des racines de $R_{1},R_{2},R_{3}$ vérifiant les deux inégalités. Seulement, il est possible que $r_{1}$ et $r_{2}$ soient deux racines réelles. Ceci n'est possible que si $R_{1},R_{2} \in \mathbb{R}^{+}$. Dans ce dernier cas, choisissons $r_{2}$ une racine réelle négative tandis que $r_{1}$ est une racine imaginaire pure. Les trois racines ne sont pas $\mathbb{R}$-colinéaires et il est clair que $1+r_{1}+r_{2} \neq 0$. Enfin, comme $0<R_{1}<R_{2} \leq 1$, $r_{1}+r_{2}$ est de module strictement plus petit que $\sqrt{2}$ tandis que son argument se situe dans $]\frac{3\pi}{4},\frac{5\pi}{4}[$. Il s'ensuit que $|1+r_{1}+r_{2}|<1$.
\par
Pour $s \geq 4$, la preuve se termine par récurrence comme dans le cas $k \geq 5$ sauf si chaque choix de paire d'indices conduit à une configuration exceptionnelle. Dans un tel cas, pour tout choix de paire d'indices, nous savons que les rapports entre les $s-2$ nombres restants valent $1$ ou $-4$. En choisissant successivement chacune des paires, nous en déduisons que $R_{1},\dots,R_{s+1}$ est équivalent à $(\rec[1][s_{1}],\rec[-4][s_{2}])$. Quitte à normaliser, nous pouvons supposer que $s_{2} \geq 1$. Si $s_{1}$ est pair, nous obtenons bien une configuration exceptionnelle. Si $s_{1}$ est impair, alors nous pouvons choisir $s_{1}-1$ racines de $1$ et $s_{2}$ racines de $-4$ de façon à ce que leur somme soit $\pm (1 + i)$, $\pm (1-i)$, $\pm 2$ ou $\pm 2i$. Le choix de la dernière racine de $1$ permet d'obtenir  $ 0 < |\sum r_{i}|< \sqrt{2}$. Dans tous les cas, il possible de garantir que les racines ne soient pas toutes $\mathbb{R}$-colinéaires.

\subsubsection{\bf Preuve de la borne supérieure dans le cas $k =3$}

Si $s = 2$, en fixant $r_{2} = 1$, on peut toujours trouver une racine troisième de $R_{1}$ telle que $0<|r_{1}+ r_{2} | < 1$ sauf si $R_{1}=\pm 1$. Ce sont deux configurations exceptionnelles.
\par
Pour $s=3$, nous normalisons par $r_{3}=1$. Nous traitons d'emblée le cas $R_{1}=R_{2}$. Nous choisissons des racines $r_{1},r_{2}$ telles que $r_{2}=\omega r_{1}$. Comme $|(1+\omega)r_{1}|\leq 1$, il existe un choix pour~$r_{1}$ tel que $0<|1+r_{1}+r_{2}|<1$ sauf si $R_{1}=R_{2}=1$. Il s'agit d'une configurations exceptionnelle. Nous traitons ensuite le cas $R_{1}=-R_{2}$. Nous choisissons des racines $r_{1},r_{2}$ telles que $r_{2}=-\omega r_{1}$. Comme $|(1+\omega)r_{1}|\leq \sqrt{3}$, il  existe un choix pour $r_{1}$ tel que $0<|1+r_{1}+r_{2}|<1$ sauf si $R_{1}=1$ (et donc $R_{2}=-1$) ou $R_{1}=\frac{i}{3\sqrt{3}}$ (et donc $R_{2}=-\frac{i}{3\sqrt{3}}$). Ce sont aussi des configurations exceptionnelles.
\par
Si au contraire $R_{1} \neq \pm R_{2}$, il existe des racines $r_{1},r_{2}$ telles que $0<|r_{1}+r_{2}|< |r_{2}|\leq |r_{3}|$. La résolution du cas $s = 2$ nous donne les racines adéquates pour $(r_{1}+r_{2})^{3},R_{3}$, dont nous déduisons des racines de $R_{1},R_{2},R_{3}$ vérifiant les deux inégalités. Seulement, il est possible que $r_{1}$ et $r_{2}$ soient deux racines réelles. Ceci n'est possible que si $R_{1},R_{2} \in \mathbb{R}$. Désignons par~$x_{1}$ et~$x_{2}$ leur unique racine réelle. En particulier, $|x_{1}|<|x_{2}|$ (car $R_{1} \neq \pm R_{2}$ et que ces racines sont réelles). Supposons que l'une de ces racines $x_{i}$ est positive, tandis que l'autre (que nous notons $x_{j}$) est négative. Nous posons $r_{i}=\omega x_{i}$ tandis que $r_{j}=x_{j}$. Les trois racines ne sont pas $\mathbb{R}$-colinéaires et il est clair que $0<|1+r_{1}+r_{2}|<1$. Si $x_{1},x_{2}<0$, alors nous posons $r_{2}=x_{2}$ tandis que $r_{1}=\omega x_{1}$. Puisque $|x_{1}|<|x_{2}|$, nous obtenons encore $0<|1+r_{1}+r_{2}|<1$. Enfin, si $x_{1},x_{2}>0$, nous choisissons $r_{1}=\omega x_{1}$ et $r_{2}=\omega^{2}x_{2}$. Nous obtenons $|1+r_{1}+r_{2}|<1$ et il est impossible d'avoir $1+r_{1}+r_{2}=0$ sans quoi nous serions dans le cas $R_{1}= \pm R_{2}$. Ceci termine les constructions pour le cas $s=3$.
\par
Pour $s \geq 4$, la preuve se termine par récurrence comme dans le cas $k \geq 5$ sauf si chaque choix de paire d'indices conduit à une configuration exceptionnelle. Dans un tel cas, pour tout choix de paire d'indices, nous savons que les rapports entre les $s-2$ nombres restants valent $\pm 1$ ou $\pm 3i\sqrt{3}$. En choisissant successivement chacune des paires, nous en déduisons que $R_{1},\dots,R_{s+1}$ est équivalent à $(\rec[1][s_{1}],\rec[-1][s_{2}],\rec[3i\sqrt{3}][t_{1}],\rec[-3i\sqrt{3}][t_{2}])$. Quitte à normaliser, nous pouvons supposer que $t_{1} \geq 1$. Si $s_{1}-s_{2} \in 3\ZZ$, nous obtenons bien une configuration exceptionnelle. Si $s_{1}-s_{2} \notin 3\ZZ$, alors nous pouvons choisir les $t_{1}+t_{2}$ racines de $\pm 3i\sqrt{3}$ de façon à ce que leur somme soit sur le cercle de rayon $\sqrt{3}$. Nous pouvons ensuite ajouter un multiple de~$3$ de racines de $1$, ainsi qu'un multiple de~$3$ de racines de $-1$ tout en maintenant la somme sur ce même cercle. Le choix des dernières racines de $\pm 1$ permet d'obtenir  $ 0 < |\sum r_{i}|< \sqrt{3}$. Dans tous les cas, il possible de garantir que les racines ne soient pas toutes $\mathbb{R}$-colinéaires.

\section{Au moins un pôle non divisible par~$k$ en genre zéro}\label{sec:avecnondiv}

Rappelons que la strate $\Omega^{k}\mathcal{M}_{0}(\mu)$ paramètre les $k$-différentielles primitives de type~$\mu$. 
Dans cette section, les strates contiennent des pôles d'ordres non divisibles par $k$. Nous commençons par montrer qu'il existe des strates dont l'image de l'application résiduelle  de contient pas l'origine. Puis nous montrons que tous les autres résidus sont obtenus sont dans l'image de l'application résiduelle.

\subsection{Obstruction}

Nous montrons que l'image de l'application $k$-résiduelle peut ne pas contenir l'origine.
\begin{lem}\label{lem:g=0gen1obs}
 L'image de l'application $k$-résiduelle de $\Omega^{k}\mathcal{M}_{0}(a_{1},kl_{2},\dots,kl_{n};-b_{1},\dots,-b_{p};-c)$  avec $\sum_{i=2}^{n} l_{i}<p$ ne contient pas l'origine.
\end{lem}

\begin{proof}
\'Etant donné une $k$-différentielle de la strate $\Omega^{k}\mathcal{M}_{0}(\mu)$ dont tous les $k$-résidus sont nuls, le revêtement canonique définit une différentielle méromorphe $(\widehat X, \widehat\omega)$.  Par les résultats de la section~2.1 de \cite{BCGGM3}, le revêtement est uniquement ramifié au zéro d'ordre~$a_{1}$ et au pôle d'ordre~$-c$. Donc la surface $\widehat X$ est une sphère de Riemann. De plus la différentielle~$\widehat \omega$ possède  $k$ zéros d'ordres $l_{i}$ pour $i\geq 2$, un zéro d'ordre $a_{1}+k-1$, de plus $k$ pôles d'ordres $-b_{j}/k$ pour tout $j\geq1$ et un pôle d'ordre $-c+k-1$. Notons de plus que le résidu de chaque pôle est nul. D'après le  point i) du théorème~1.2 de \cite{getaab}, cela est possible si et seulement si l'ordre de tous les zéros est inférieur ou égal à $$(c+1-k)+k\sum_{j=1}^{p} \frac{b_{j}}{k} - (kp+2)\,.$$ 
Cette condition implique que 
$$ a_{1} \leq c + \sum b_{i} -2k -kp \,.$$
Le fait que la différence entre la somme des ordres des zéros et la somme des ordres des pôles est égale à~$-2k$ implique le résultat. 
\end{proof}

\subsection{Constructions}
Dans cette section, nous allons montrer qu'à l'exception des obstructions de la section précédente, tous les résidus peuvent être obtenus.

Nous commençons par les strates ayant un unique zéro.
\begin{lem}\label{lem:g=0gen1const}
 Soit $\Omega^{k}\mathcal{M}_{0}(a;-b_{1},\dots,-b_{p};-c_{1},\dots,-c_{r};\rec[-k][s])$ une strate de genre zéro telle que $r\neq0$. L'image de l'application résiduelle est
 \begin{itemize}
 \item[i)] $\espresk[0](\mu)$ si $r\geq2$ ou $s\geq1$,
 \item[ii)]  $\espresk[0](\mu)\setminus\left\{(0,\dots,0)\right\}$ si $r=1$ et $s=0$.
\end{itemize}  
\end{lem}

\begin{proof}
Nous commençons par donner la construction d'une $k$-différentielle dans la strate $\Omega^{k}\mathcal{M}_{0}(\mu)$ avec $\mu=(a;-b_{1},\dots,-b_{p};-c_{1},\dots,-c_{r};\rec[-k][s])$ dont les $k$-résidus sont donnés par $(R_{1},\dots,R_{p+s})$ dans $\espresk[0](\mu)\setminus\left\{(0,\dots,0)\right\}$.
\par 
Pour les pôles $P_{p+i}$ d'ordres $-k$, nous prenons une $k$-partie polaire d'ordre~$k$ associée à une racine $k$-ième $r_{p+i}$ de $R_{p+i}$. Pour chaque pôle $P_{i}$ d'ordre $-b_{i}=-k\ell_{i}$ tel que $R_{i}\neq0$, nous prenons une $k$-partie polaire non triviale d'ordre $b_{i}$ associée à $(r_{i};\emptyset)$. Nous choisissons une racine $r_{i}$ avec une partie réelle positive. Pour les pôles d'ordre $-b_{i}$ tels que $r_{i}=0$ nous prenons une $k$-partie polaire triviale d'ordre $b_{i}$ associée à $(r_{j_{i}};r_{j_{i}})$ où $r_{j_{i}}$ est l'une des racines choisie précédemment.
\par
Maintenant, pour tous les pôles d'ordre $-c_{i}$ sauf un, disons $P_{1}$ d'ordre $-c_{1}$, nous prenons une $k$-partie polaire de type $c_{i}$ associée à $(1;\emptyset)$. Pour le dernier pôle $P_{1}$ d'ordre $-c_{1}$, nous prenons la $k$-partie polaire de type $c_{1}$, sans points d'intersection (voir le lemme~\ref{lem:noninter}), associée à $(\emptyset;\rec[1][r-1],r_{1},\dots,r_{l})$ où $l$ est le nombre de résidus non nuls.
\par
La surface est obtenue par les recollements suivant. Nous collons le segment inférieur~$r_{j_{i}}$ de chaque $k$-partie polaire triviale au bord de la partie non triviale du pôle $P_{j_{i}}$. Nous faisons les collages similaires, pour chacun des pôles d'ordre divisible par $k$ dont le résidu est nul. Ensuite nous collons par translation les bords des $k$-parties polaires différentes de $P_{1}$ aux segments correspondants du bord de la $k$-partie polaire de $P_{1}$. Cette construction est illustrée par la figure~\ref{ex:8,8,12res}.
La $k$-différentielle $\xi$ associée à cette surface plate possède les ordres de pôles et les $k$-résidus souhaités.

\begin{figure}[htb]
\begin{tikzpicture}

%Figure haut gauche

\begin{scope}[xshift=-1.5cm]
      
 \fill[black!10] (-1,0)coordinate (a) -- (1.5,0)-- (a)+(1.5,0) arc (0:120:1.5)--(a)+(120:1.5) -- cycle;

   \draw (a)  -- node [below] {$2$} (0,0) coordinate (b);
 \draw (0,0) -- (1,0) coordinate[pos=.5] (c);
 \draw[dotted] (1,0) --coordinate (p1) (1.5,0);
 \fill (a)  circle (2pt);
\fill[] (b) circle (2pt);
\node[below] at (c) {$a$};

 \draw (a) -- node [above,rotate=120] {$a$} +(120:1) coordinate (d);
 \draw[dotted] (d) -- coordinate (p2) +(120:.5);

     \end{scope}
     
%deuxieme dessin
\begin{scope}[xshift=1.5cm,yshift=.5cm]
\fill[fill=black!10] (0.5,0)coordinate (Q)  circle (1.1cm);
    \coordinate (a) at (0,0);
    \coordinate (b) at (1,0);

     \fill (a)  circle (2pt);
\fill[] (b) circle (2pt);
    \fill[white] (a) -- (b) -- ++(0,-1.1) --++(-1,0) -- cycle;
 \draw  (a) -- (b);
 \draw (a) -- ++(0,-1);
 \draw (b) -- ++(0,-1);

\node[above] at (Q) {$1$};
    \end{scope}

%troisieme dessin
\begin{scope}[xshift=6.5cm,yshift=1.5cm,rotate=180]

 \fill[black!10] (-1,0)coordinate (a) -- (1,0)-- (1,0) arc (0:120:2) -- cycle;
\draw (-1,0) coordinate (a) -- node [below,xshift=-1] {$2$} (0,0) coordinate (b);
\draw (b) -- node [below] {$1$} +(1,0) coordinate (c);
\draw (c) -- node [below] {$b$} +(1,0) coordinate (d);
\draw[dotted] (d) -- +(.5,0);
\fill (a)  circle (2pt);
\fill[] (b) circle (2pt);
\fill[] (c) circle (2pt);

\draw (a) -- node [above,rotate=120] {$b$} +(120:1) coordinate (e);
\draw[dotted] (e) -- +(120:.5);

\end{scope}

\end{tikzpicture}
\caption{Une $3$-différentielle de $\Omega^{3}\mathcal{M}_{0}(8;-4,-4;-6)$ avec un résidu non nul au pôle d'ordre $-6$.} \label{ex:8,8,12res}
\end{figure}

Il reste donc à montrer que le genre de la surface est zéro et que la différentielle $\xi$ possède un unique zéro.
Pour cela, il suffit de vérifier que si l'on coupe la surface le long d'un lien selle, alors on sépare cette surface en deux parties. C'est une conséquence du fait que les liens-selles correspondent aux bords des domaines polaires. 
\smallskip
\par
Il reste à considérer le cas où les $k$-résidus sont égaux à $(0,\dots,0)$. Remarquons que s'il existe un pôle d'ordre $-k$, alors $(0,\dots,0)$ n'est pas dans l'espace résiduel de la strate. Nous supposerons donc que $s=0$ dans la suite de la preuve. Par le lemme~\ref{lem:g=0gen1obs}, il suffit de montrer que  si $r\geq2$ alors l'origine appartient à l'image de l'application $k$-résiduelle.
\par
Soient  $P_{1}$ et $P_{2}$ deux pôles d'ordres respectifs $-c_{1}$ et $-c_{2}$ non divisible par $k$. Pour tous les autres pôles nous associons la même $k$-partie polaire que précédemment. Plus précisément, pour chaque pôle d'ordre $-b_{i}$ divisible par $k$ nous prenons une $k$-partie polaire d'ordre $b_{i}$ associée à $(1;1)$. Pour les pôles d'ordres non divisibles par $k$ distincts de $P_{1}$, nous prenons la $k$-partie polaire associée à $(1;\emptyset)$. Pour $P_{1}$ nous prenons la $k$-partie polaire d'ordre $c_{1}$ associée à $(\emptyset;(\rec[1][r-1])$.
\par
Les collages sont les suivants. Nous collons le bord inférieur de la $k$-partie polaire du $i$-ième pôle d'ordre divisible par $k$ au bord d'en haut du $(i+1)$-ième pôle d'ordre divisible par $k$. Le bord inférieur de la $k$-partie polaire associée au pôle $P_{p}$ est collé au bord du segment du pôle~$P_{2}$. Enfin, tous les segments restant sont collés au bord de la $k$-partie polaire associée à~$P_{1}$. Cette construction est illustrée par la figure~\ref{ex:8,8,12}. On vérifie facilement que cette surface possède les propriétés souhaitées.
\begin{figure}[htb]
\begin{tikzpicture}
%Figure haut gauche

\begin{scope}[xshift=-1.5cm,yshift=.15cm]
      
 \fill[black!10] (-1,0)coordinate (a) -- (1.5,0)-- (a)+(1.5,0) arc (0:120:1.5)--(a)+(120:1.5) -- cycle;

   \draw (a)  -- node [below] {$1$} (0,0) coordinate (b);
 \draw (0,0) -- (1,0) coordinate[pos=.5] (c);
 \draw[dotted] (1,0) --coordinate (p1) (1.5,0);
 \fill (a)  circle (2pt);
\fill[] (b) circle (2pt);
\node[below] at (c) {$a$};

 \draw (a) -- node [above,rotate=120] {$a$} +(120:1) coordinate (d);
 \draw[dotted] (d) -- coordinate (p2) +(120:.5);

     \end{scope}
%deuxieme dessin

\begin{scope}[xshift=1.5cm,yshift=.8cm]
\fill[fill=black!10] (0.5,0)  circle (1.1cm);

 \draw (0,0) coordinate (a) -- coordinate (c) (1,0) coordinate (b);

 \fill (a)  circle (2pt);
\fill[] (b) circle (2pt);
\node[below] at (c) {$1$};
\node[above] at (c) {$2$};

    \end{scope}

%troisieme dessin
\begin{scope}[xshift=6cm,rotate=180,yshift=-1.3cm]
      
 \fill[black!10] (-1.5,0)coordinate (a) -- (0,0)-- (60:1.5) arc (60:180:1.5) -- cycle;

   \draw (-1,0) coordinate (a) -- node [above] {$2$} (0,0) coordinate (b);
 \draw (a) -- +(-1,0) coordinate[pos=.5] (c) coordinate (e);
 \draw[dotted] (e) -- +(-.5,0);
 \fill (a)  circle (2pt);
\fill[] (b) circle (2pt);
\node[below] at (c) {$b$};

 \draw (b) -- node [above,rotate=-120] {$b$} +(60:1) coordinate (d);
 \draw[dotted] (d) -- +(60:.5);
\end{scope}

\end{tikzpicture}
\caption{Une $3$-différentielle de $\Omega^{3}\mathcal{M}_{0}(8;-4,-4;-6)$ avec un résidu nul.} \label{ex:8,8,12}
\end{figure}
\end{proof}

Nous réalisons une construction spécifique pour des strates ayant deux zéros.

\begin{lem}\label{lem:g=0gen2construction}
Considérons la partition $\mu=(a_{1},a_{2};-b_{1},\dots,-b_{p};-c)$ où $a_{1},a_{2}$ ne sont pas de la forme $kl_{i}$ avec $l_{i}<p$. L'image de l'application $k$-résiduelle $\appresk[0](\mu)$ contient l'origine.
\end{lem}

\begin{proof}
Nous écrivons $a_{i}=kl_{i}+\bar{a_{i}}$ avec $-k<\bar{a_{1}},\bar{a_{2}}\leq 0$. 
Nous considérons deux cas selon que $l_{1},l_{2}\geq p$ ou s'il existe $l_{i}<p$.
\smallskip
\par
Dans le cas où $l_{1},l_{2}\geq p$, on associe aux pôles d'ordres $-b_{i}$ les $k$-parties polaires d'ordres~$b_{i}$ et de types~$\tau_{i}$ associées à $(1;1)$. De plus, on choisit les types $\tau_{i}$ de telle sorte que la somme $\sum_{i}\tau_{i}$ soit inférieure ou égale à $ l_{1}$ et maximale pour cette propriété. On note $\bar{l_{1}}=l_{1}-\sum_{i}\tau_{i}$. Pour le pôle d'ordre $-c$ on procède de la façon suivante.  On prend la $k$-partie polaire d'ordre~$c$ associée à $\left(\emptyset;1,\exp\left(i\pi+\bar{a_{2}}\tfrac{2i\pi}{k}\right)\right)$ et de type $\bar{l_{1}}+1$. Notons que cette construction est possible même si $\bar{a_{2}} = 0$.
On obtient la différentielle souhaitée en identifiant le bord inférieur de la $k$-partie polaire associée à $P_{i}$ au bord supérieur de celle de $P_{i+1}$. Le bord supérieur de $P_{1}$ est identifié au segment $1$ de la $k$-partie polaire d'ordre~$c$. Le bord inférieur de la $k$-partie polaire d'ordre $b_{p}$ est identifié par rotation au segment $\exp\left(\bar{a_{2}}\tfrac{2i\pi}{k}\right)$. Cette construction est illustrée à gauche de la figure~\ref{fig:rgeq1ngeq2}.
\par
Considérons le cas où il existe $l_{i}<p$ et $a_{i}\neq kl_{i}$. Nous supposerons sans perte de généralité qu'il s'agit de $l_{1}$. On associe au pôle d'ordre $-c$ la $k$-partie polaire d'ordre $c$ associée à $(\emptyset;1)$. Pour l'un des pôles d'ordre $-b_{i}$, disons $-b_{1}$, on associe la $k$-partie polaire d'ordre $b_{1}$ associée à $(1;v_{1},v_{2})$, avec $v_{i}$ de même longueur, $v_{1}+v_{2}=1$ et l'angle (dans la partie polaire) entre ces deux étant $2\pi + \tfrac{2\bar{a_{1}}\pi}{k}$. Comme $\bar{a_{1}} \neq 0$, cette partie polaire est non dégénérée.  
Pour $l_{1}$ pôles d'ordre $-b_{i}$, on associe des $k$-parties triviales associées à $(v_{1};v_{1})$ et de type $\tau_{i}=b_{i}-1$. On colle ces $k$-parties polaires de manière cyclique à $v_{1}$ et~$v_{2}$. Le point correspondant à l'intersection entre $v_{1}$ et $v_{2}$ est la singularité d'ordre~$a_{1}$.  On associe aux autres pôles d'ordre $-b_{i}$ la $k$-partie polaire triviale associée à $(1;1)$. On colle ces parties polaires de manière cyclique aux parties polaires d'ordres $b_{1}$ et $c$  pour obtenir la surface plate souhaitée. Cette construction est illustrée à droite de la figure~\ref{fig:rgeq1ngeq2}.
\begin{figure}[htb]
\center
 \begin{tikzpicture}

%premier dessin gauche
\begin{scope}[xshift=-6cm]
    \fill[fill=black!10] (0,0) ellipse (1.5cm and .7cm);
\draw[] (0,0) coordinate (Q) -- (-1.5,0) coordinate[pos=.5](a);

\node[above] at (a) {$1$};
\node[below] at (a) {$2$};

\draw[] (Q) -- (.7,0) coordinate (P) coordinate[pos=.5](c);
\draw[] (Q) -- (P);

\fill (Q)  circle (2pt);

\fill[color=white!50!] (P) circle (2pt);
\draw[] (P) circle (2pt);
\node[above] at (c) {$v_{2}$};
\node[below] at (c) {$v_{1}$};

    \fill[fill=black!10] (0,-1.5) ellipse (1.1cm and .5cm);
\draw[] (.5,-1.5) coordinate (Q) --++ (-1.5,0) coordinate[pos=.5](b);

\fill (Q)  circle (2pt);

\node[above] at (b) {$2$};
\node[below] at (b) {$1$};

\end{scope}
%premier dessin droite
\begin{scope}[xshift=-3cm]

 \fill[black!10] (-.5,0)coordinate (a) -- (a) + (80:1.5) coordinate (b)-- (b) arc (80:-40:1.5)--(a)+(-40:1.5) coordinate (c) -- cycle;
\draw[] (a) -- node [left,rotate=-10] {$5$}  (b);
\draw[] (a) -- node [right,rotate=-130] {$5$} (c);
\draw (a) -- ++(.7,0)coordinate[pos=.5](e) coordinate (d);
\draw (d) -- ++(.8,0)coordinate[pos=.5](f);
\fill (a) circle (2pt);
\fill[color=white] (d) circle (2pt);
\draw[] (d) circle (2pt);
\node[above] at (f) {$4$};
\node[below] at (f) {$3$};
\node[above] at (e) {$v_{1}$};
\node[below] at (e) {$v_{2}$};

\fill[fill=black!10] (0,-1.5) ellipse (1.1cm and .5cm);
\draw[] (-.5,-1.5) coordinate (Q) --++ (1.5,0) coordinate[pos=.5](b);
\filldraw[fill=white] (Q)  circle (2pt);

\node[above] at (b) {$4$};
\node[below] at (b) {$3$};
\end{scope}

%euxieme dessin gauche
\begin{scope}[xshift=2cm]
\begin{scope}[yshift=1cm,xshift=-.5cm]
    \fill[fill=black!10] (.5,-.3) ellipse (1.1cm and .9cm);
   
 \coordinate (A) at (0,0);
  \coordinate (B) at (1,0);
    \coordinate (C)  at (-30:.6);
 
 \filldraw[fill=white] (A)  circle (2pt);
\filldraw[fill=white] (B) circle (2pt);
    \fill (C)  circle (2pt);

   \fill[color=white]     (A) -- (B) -- (C) --cycle;
 \draw[]     (A) -- (B) coordinate[pos=.5](a);
 \draw (B) -- (C) coordinate[pos=.3](b);
 \draw (C) -- (A)  coordinate[pos=.6](c);

\node[above] at (a) {$1$};
\node[below] at (c) {$v_{1}$};
\node[below] at (b) {$v_{2}$};
\end{scope}

    \fill[fill=black!10] (0,-1.5) ellipse (1.1cm and .5cm);
\draw[] (-.5,-1.5) coordinate (Q) -- (.5,-1.5) coordinate[pos=.5](b) coordinate (P);

\filldraw[fill=white] (Q) circle (2pt);
\filldraw[fill=white] (P) circle (2pt);

\node[above] at (b) {$2$};
\node[below] at (b) {$1$};

\end{scope}
%second dessin droite
\begin{scope}[xshift=5cm]
\begin{scope}[rotate=180,yshift=-1.3cm]
      
 \fill[black!10] (-1.5,0)coordinate (a) -- (0,0)-- (60:1.5) arc (60:180:1.5) -- cycle;

   \draw (-1,0) coordinate (a) -- node [above] {$2$} (0,0) coordinate (b);
 \draw (a) -- +(-1,0) coordinate[pos=.5] (c) coordinate (e);
 \draw[dotted] (e) -- +(-.5,0);
  \draw (b) -- node [above,rotate=-120] {$b$} +(60:1) coordinate (d);
 \draw[dotted] (d) -- +(60:.5);
\filldraw[fill=white] (a) circle (2pt);
\filldraw[fill=white] (b) circle (2pt);
\node[below] at (c) {$b$};
\end{scope}

\begin{scope}[yshift=-1.3cm]
\fill[fill=black!10] (0,0) coordinate (O) ellipse (1.1cm and .6cm);
\coordinate (P) at (-30:.3);
\coordinate (Q) at (150:.3);
\draw[] (Q) -- (P)coordinate[pos=.5](b);
\filldraw[fill=white] (Q) circle (2pt);
\fill (P)  circle (2pt);

\node[below] at (b) {$v_{2}$};
\node[above] at (b) {$v_{1}$};
\end{scope}
\end{scope}
\end{tikzpicture}
\caption{Une $3$-différentielle de $\Omega^{3}\mathcal{M}_{0}(4,6;-9;-7)$ (à gauche) et de $\Omega^{3}\mathcal{M}_{0}(2,8;\rec[-6][3];-4)$ (à droite) dont tous les $3$-résidus sont nuls.} \label{fig:rgeq1ngeq2}
\end{figure}
\end{proof}

Nous traitons maintenant le cas général des strates ayant au moins deux zéros.
\begin{lem}\label{lem:g=0gen2}
Considérons la partition $\mu=(a_{1},\dots,a_{n};-b_{1},\dots,-b_{p};-c_{1},\dots,-c_{r},\rec[-k][s])$ avec $n\geq2$. Si $r\geq2$ ou $s\geq1$, alors l'application $k$-résiduelle $\appresk[0](\mu)$ est surjective. Dans le cas où $s=0$ et $r=1$ alors l'application contient le complémentaire de l'origine. De plus si au moins deux $a_{i}$ ne sont pas divisibles par $k$ ou si la somme des $a_{i}$ divisibles par $k$ est supérieure ou égale à~$kp$ alors l'application $k$-résiduelle contient l'origine.
\end{lem}

\begin{proof}
Si $r\geq2$ ou $r=1$ et $s\geq1$, on peut simplement éclater le zéro des différentielles données par le lemme~\ref{lem:g=0gen1const}. On suppose maintenant $r=1$ et $s=0$. Par éclatement de zéros, il suffit de montrer que l'origine appartient à l'image de l'application résiduelle des strates $\Omega^{k}\mathcal{M}_{0}(a_{1},a_{2};-b_{1},\dots,-b_{p};-c)$ dans le cas où $a_{i}$ n'est pas de la forme $kl_{i}$ avec $l_{i}<p$. Le lemme \ref{lem:g=0gen2construction} fournit la construction adéquate.
\end{proof}

\section{Les pôles sont divisibles par $k$ en genre zéro}\label{sec:4DIVIS}

Le but de cette section est de décrire l'image de l'application $k$-résiduelle des strates  $\Omega^{k}\mathcal{M}_{0}(a_{1},\dots,a_{n};-b_{1},\dots,-b_{p};\rec[-k][s])$ avec $p \neq 0$. Rappelons que nous ne considérons que les strates primitives, c'est-à-dire telles que $\pgcd(a_{1},\dots,a_{n},k)=1$.
Nous notons $a_{i}:=kl_{i}+\bar{a_{i}}$ avec $-k<\bar{a_{i}}\leq 0$ et $b_{i}:=k\ell_{i}$ que nous étendons aux pôles d'ordre $-k$ en posant $\ell_{i}=1$ pour $p+1 \leq i \leq p+s$.

\subsection{Tous les résidus sont nuls}\label{sub:41CUN}

Dans le cas des strates satisfaisant $s=0$, il existe une unique obstruction concernant la configuration uniformément nulle.

\begin{lem}\label{lem:41CUN}
Soit $\mu=(a_{1},\dots,a_{n};-b_{1},\dots,-b_{p})$ une partition telle que $a_{1},a_{2}$ sont premiers avec~$k$ et $a_{i}$ est divisible par $k$ pour  $i \geq 3$. Si l'application résiduelle $\appresk[0](\mu)$ contient $(0,\dots,0)$, alors $\sum_{i=3}^{n} a_{i} \geq kp$.
\end{lem}

La preuve repose sur l'étude du revêtement canonique d'une telle $k$-différentielle. Celui-ci est introduit et étudié dans la section~2.1 de \cite{BCGGM3}. Rappelons qu'il s'agit d'un revêtement $\pi\colon \tilde X \to X$ tel que $\pi^{\ast} \xi$ est la puissance $k$-ième d'une différentielle abélienne.

\begin{proof}
Soit $(\PP^{1},\xi)$ une $k$-différentielle dans la strate $\Omega^{k}\mathcal{M}_{0}(\mu)$ dont les résidus sont nuls. Son revêtement canonique $\pi\colon (\whX,\whomega) \to (\PP^{1},\xi)$ est ramifié exactement aux singularités~$a_{1}$ et~$a_{2}$. La surface de Riemann $\whX$ est donc de genre zéro. Les singularités de la racine $\whomega$ de $\pi^{\ast}\xi$ sont :
\begin{itemize}
\item deux zéros $\whz_{1},\whz_{2}$ d'ordres $a_{1}+k-1$ et $a_{2}+k-1$ respectivement (car l'ordre de ramification du revêtement en ces points est $k$ puisque $a_{1}$ et $a_{2}$ sont premiers avec $k$);
\item pour chaque i tel que $3 \leq i \leq n$, il y a $k$ zéros d'ordre $\frac{a_{i}}{k}$;
\item pour chaque j tel que $1 \leq j \leq p$, il y a $k$ pôles d'ordre $-\frac{b_{j}}{k}$ dont le résidu est nul.
\end{itemize} 
Il s'ensuit que la différentielle abélienne $\whomega$ est exacte.
Pour toute fonction méromorphe $f\colon \PP^{1} \to \CC$ telle que $\whomega=df$, la symétrie du revêtement canonique implique que $f(\whz_{1})=f(\whz_{2})$. On supposera donc sans perte de généralité que $f(\whz_{1})=f(\whz_{2})=0$. En tant que primitive de~$\whomega$, la fonction $f$ a $kp$ pôles d'ordres $1-\frac{b_{j}}{k}$ (chacun de ces ordres étant présent $k$ fois). La somme des ordres des pôles est donc $kp-\sum_{j=1}^{p} b_{j}$.
Aux points $\whz_{1}$ et $\whz_{2}$, la fonction $f$ possède des zéros d'ordres $k+a_{1}$ et $k+a_{2}$. Cela implique que $2k+a_{1}+a_{2} \leq \sum_{j=1}^{p} b_{j} -kp$. De façon équivalente, $a_{1}+a_{2} \leq \sum_{i=1}^{n} a_{i} -kp$ et donc $\sum_{i=3}^{n} a_{i} \geq kp$.
\end{proof}

Nous montrons maintenant que dans le cas où uniquement deux zéros sont d'ordres non divisibles par $k$, la condition du lemme~\ref{lem:41CUN} est suffisante pour que l'origine soit dans l'image de l'application $k$-résiduelle.

\begin{lem}\label{lem:42CUN}
Soit $\mu=(a_{1}\dots,a_{n};-b_{1},\dots,-b_{p})$ une partition telle que $a_{1}$ et $a_{2}$ ne sont pas divisibles par $k$ et $a_{i}$ est divisible par $k$ pour $i \geq 3$. Si $\sum_{i=3}^{n} a_{i} \geq kp$, alors l'application résiduelle $\appresk[0](\mu)$ contient l'origine.
\end{lem}

\begin{proof}
Nous allons démontrer le cas $n=3$. Le cas $n>3$ s'obtient en éclatant $a_{3}$. En effet, les seuls ordres de singularités qui ne sont pas divisible par $k$ sont $a_{1}$ et $a_{2}$. Il s'ensuit que $\pgcd(a_{1},a_{2},k)=1$. Ainsi, la strate satisfaisant $n=3$ à partir de laquelle nous procédons a l'éclatement de singularité est nécessairement primitive. De plus, comme tous les zéros créés sont d'ordres divisibles par $k$, il s'ensuit que l'éclatement ne change pas les résidus aux pôles (voir la proposition~\ref{prop:eclatZero}).

Nous considérerons deux cas, selon que l'équation 
\begin{equation}\label{eq:lemme42}
 l_{1} + l_{2} \geq p-1 
\end{equation}
soit satisfaite ou non.
\smallskip
\par
\paragraph{\bf L'équation~\eqref{eq:lemme42} est satisfaite.} Commençons par le cas où tous les pôles sont d'ordre $2k$. Dans ce cas l'équation~\eqref{eq:lemme42} et l'hypothèse $a_{3} \geq kp$
impliquent que $l_{1}+l_{2}=p-1$. On prend une $k$-partie polaire d'ordre $2k$ associée à $(v_{1}, \zeta v_{1};v_{1},\zeta v_{1})$ avec $\zeta$ la racine $k$-ième de l'unité telle que l'angle au dessus du point d'intersection entre $v_{1}$ et $\zeta v_{1}$ est égal à $(k+\bar{a}_{1})\tfrac{2\pi}{k}$.
Pour tous les autres pôles on prend la  $k$-partie polaires d'ordre $2k$ associées à $(v_{1};v_{1})$. On colle cycliquement $l_{1}$ parties polaires aux vecteurs supérieurs $v_{1}$ et $\alpha v_{1}$ et $l_{2}$ parties polaires aux vecteurs inférieurs. Cette construction est représentée dans la figure~\ref{fig:polesdivkaveczero}.  
 \begin{figure}[hbt]
\begin{tikzpicture}
\begin{scope}[xshift=-6cm,yshift=.5cm]

     \foreach \i in {1,2,...,4}
  \coordinate (a\i) at (\i,0); 
  \coordinate (b) at (2.5,.25);
   \fill[black!10] (a1) -- (a4) -- ++(0,0) arc (0:180:1.5) -- cycle;
      \fill[white] (b)  circle (2pt);
      \draw (b) circle (2pt);
 \fill[white] (a2) -- (a3) -- (b) -- cycle;
       \foreach \i in {2,3}
   \fill (a\i)  circle (2pt);
   \draw (a1) -- (a2) coordinate[pos=.5](e1) -- (b) coordinate[pos=.5](e2) -- (a3) coordinate[pos=.5](e3) -- (a4) coordinate[pos=.5](e4);

     \foreach \i in {1,2,...,4}
  \coordinate (c\i) at (\i,-1); 
  \coordinate (d) at (2.5,-.75);
   \fill[black!10] (c1) -- (c2) -- (d) -- (c3) -- (c4) -- ++(0,0) arc (0:-180:1.5) -- cycle;
     \foreach \i in {2,3}
   \fill (c\i)  circle (2pt);
   \draw (c1) -- (c2) coordinate[pos=.5](f1) -- (d) coordinate[pos=.5](f2) -- (c3) coordinate[pos=.5](f3) -- (c4) coordinate[pos=.5](f4);
     \fill[red] (d)  circle (2pt);
      \draw (d) circle (2pt);

\node[above] at (e1) {$1$};
\node[below] at (f1) {$1$};
\node[above] at (e4) {$2$};
\node[below] at (f4) {$2$};

\node[above,rotate=30] at (e2) {$3$};
\node[below,rotate=150] at (e3) {$4$};
\node[above,rotate=30] at (f2) {$5$};
\node[below,rotate=150] at (f3) {$5$};
\end{scope}

%dessin de droite
\begin{scope}[xshift=1cm]
\fill[fill=black!10] (0,0) coordinate (Q) circle (1.5cm);

\coordinate (b1) at (-.25,-1/8);
\coordinate (b2) at (.25,1/8);
\draw (b1) -- (b2) coordinate[pos=.5](c);

\fill (b1)  circle (2pt);
\fill[white] (b2)  circle (2pt);
\draw (b2) circle (2pt);
\node[below,rotate=30] at (c) {$3$};
\node[above,rotate=30] at (c) {$4$};
\end{scope}

\end{tikzpicture}
\caption{Différentielle cubique de $\Omega^{3}\mathcal{M}_{0}(6,2,-2;-6,-6)$ dont les $3$-résidus sont nuls.} \label{fig:polesdivkaveczero}
\end{figure}

On suppose maintenant que les pôles $b_{i}$ sont arbitraires. La construction est similaire à celle que nous venons de faire. La partie polaire associée à $(v_{1}, \zeta v_{1};v_{1},\zeta v_{1})$ est le pôle d'ordre minimal $P_{1}$. On considère alors des pôles $P_{i}$ distincts de $P_{1}$ tels que $\sum (\ell_{i}-1) \leq l_{1} < \ell_{1} + \sum (\ell_{i}-1) $. On colle alors de manière cyclique les parties polaires associées aux $P_{i}$ au $v_{1}$ supérieur, de telle sorte que toutes ces parties polaires contribuent de l'angle  $(\ell_{i}-1)2\pi$  à l'angle de $a_{1}$. Pour finir nous coupons la $k$-partie polaire de $P_{1}$ à partir de $a_{1}$ par une demi-droite et collons le nombre de plan nécessaire pour obtenir l'ordre $a_{1}$. On fait de même pour $a_{2}$ affin d'obtenir les invariants souhaités. 
\smallskip
\par
\paragraph{\bf L'équation~\eqref{eq:lemme42} n'est pas satisfaite.} 
Nous considérons deux $k$-parties polaires d'ordre $2k$ associées respectivement à $(v_{1},\zeta v_{1}; v_{2})$ et  à $(v_{2};v_{1},\zeta v_{1})$ satisfaisant les conditions suivantes. L'angle au dessus du point d'intersection entre $v_{1}$ et $\zeta v_{1}$ est égal à $(k+\bar{a}_{1})\tfrac{2\pi}{k}$. On a l'égalité $v_{2}= v_{1}+\zeta v_{1}$. On considère $l_{1} + l_{2}$ $k$-parties polaires associées à $(v_{1};v_{1})$ et les autres sont associées à $(v_{2};v_{2})$. On colle cycliquement $l_{1}$ des premières au $v_{1}$ supérieur, les $l_{2}$ autres au~$v_{1}$ inférieur et enfin les autres entre les $v_{2}$. Chacune des $l_{1}+l_{2}$ parties polaires associées à $(v_{1};v_{1})$ contribue d'un angle $2\pi$ au zéro correspondant.
\end{proof}

Enfin nous montrons que s'il existe au moins trois zéros d'ordres non divisibles par $k$, alors l'origine est toujours dans l'image de l'application $k$-résiduelle.

\begin{lem}\label{lem:43CUN}
Considérons la partition $\mu=(a_{1},\dots,a_{n};-b_{1},\dots,-b_{p})$. Si au moins trois des ordres $a_{1},\dots,a_{n}$ ne sont pas divisibles par $k$, alors l'application résiduelle $\appresk[0](\mu)$ contient l'origine.
\end{lem}

\begin{proof}
Nous démontrons le lemme dans le cas $n=3$. L'éclatement des zéros permet d'en déduire le cas général (voir le lemme \ref{lem:combiscindage} et la proposition~\ref{prop:eclatZero}).

Si $l_{3}> p$, alors les constructions données dans la preuve du lemme~\ref{lem:42CUN} permettent alors d'obtenir les $k$-différentielles avec les invariants souhaités.
 \par
Nous supposons maintenant que nous $l_{i}\leq p$ pour $i=1,2,3$. Dans ce cas, nous donnons deux constructions selon que $\bar{a}_{1}+\bar{a}_{2}+\bar{a}_{3}=-2k$ ou  $\bar{a}_{1}+\bar{a}_{2}+\bar{a}_{3}=-k$.
\par
Dans le cas où $\bar{a}_{1}+\bar{a}_{2}+\bar{a}_{3}=-2k$, ces $k$-différentielles sont formées de deux triangles plats, dont les côtés sont reliés deux à deux par des chaînes de domaines polaires de valence~$2$. Dans l'une de ces chaînes les identifications sont des translations et dans les deux autres chaînes, ce sont des translations et une rotation. Chaque zéro correspond à un sommet de chaque triangle et à l'un des deux sommets de chaque domaine polaire dans deux des trois chaînes. Comme les sommets des triangles contribuent à chaque singularité d'un angle strictement inférieure à $2\pi$ cet angle est égal à $(k+\bar{a}_{i})\frac{2\pi}{k}$. Comme l'angle total des deux triangles est $2\pi$, nous en déduisons que $\bar{a}_{1}+\bar{a}_{2}+\bar{a}_{3}=-2k$. Chaque domaine polaire contribue aux angles de deux des trois singularités en ajoutant $\ell_{i}$ fois $2\pi$ à ceux-ci. Maintenant, on met tous les domaines polaires sur les deux chaînes qui bordent~$z_{3}$, où l'on suppose $l_{1} \leq l_{2} \leq l_{3}$. Dans ce cas, chacun des $p$ domaines polaires contribuera d'un angle d'au moins $2\pi$ à chaque zéro. Donc la condition implique que $l_{1}+l_{2} \geq p$. Cette condition est vérifiée car comme $kl_{1}+kl_{2}+kl_{3} = \sum_{j=1}^{p} b_{j}$ implique que $l_{1}+l_{2}+l_{3} \geq 2p$ et donc si nous avions $l_{1}+l_{2} < p$ alors nous aurions $l_{3}>p$.
\par
Pour finir, il reste à traiter le cas où $\bar{a}_{1}+\bar{a}_{2}+\bar{a}_{3}=-k$. Ce cas se traite de manière similaire au cas précédent en remplaçant chaque triangle par une partie polaire associée au complémentaire du triangle. 
\end{proof}

\subsection{Deux zéros, l'un d'ordre négatif}\label{sec:expneg}

Une construction systématique peut être réalisée lorsque l'une des singularités coniques a un angle strictement inférieur à $2\pi$.

\begin{lem}\label{lem:44zeronegatif}
Toute configuration de $k$-résidus $(R_{1},\dots,R_{p+s})$ non tous nuls est réalisable dans la strate $\Omega^{k}\mathcal{M}_{0}(a_{1},a_{2};-kl_{1},\dots,-kl_{p};\rec[-k][s])$ avec $p \geq 1$ et $-k<a_{1}<0$.
\end{lem}

\begin{proof}
Sans perte de généralité, nous supposerons que ce sont les $t \leq p$ premiers résidus qui sont nuls. On construit une $k$-différentielle dont les $k$-résidus sont $(R_{1},\dots,R_{p+s})$ de la manière suivante. Pour chaque $k$-résidu $R_{i}$ non nul tel que $i \geq 2$, nous choisissons une racine $k$-ième~$r_{i}$ de~$R_{i}$ telle que $\Re(r_{i})>0$. Si $R_{1} \neq 0$, nous choisissons une racine $r_{1}$ de $R_{1}$ telle que la somme $S=\sum_{i\geq  2}r_{i} -r_{1}$ est non nulle. En revanche, si $R_{1}=0$, nous avons $r_{1}=0$ et la somme $S=\sum_{i\geq  2}r_{i} -r_{1}$ satisfait $\Re(S)>0$.
\par
Nous définissons~$v_{1}$ et $v_{2}$ deux vecteurs de même longueur, tels que l'angle  entre $v_{1}$ et $v_{2}$ est  $2\pi +\frac{2a_{1}\pi}{k}$, et de somme égale à $S=\sum_{i\neq i_{0}}r_{i} -r_{i_{0}}$. Notons que ces vecteurs sont non nuls.
\par
Prenons le pôle $P_{1}$ d'ordre~$-k\ell_{1}$ et de $k$-résidu~$R_{1}$. On associe alors à ce pôle la $k$-partie polaire triviale d'ordre $k\ell_{1}$ associée à $(v_{1},v_{2};r_{t+1},\dots,r_{p})$ si $R_{1}=0$ et la $k$-partie polaire non triviale associée à $(v_{1},v_{2};r_{2},\dots,r_{p})$ si $R_{1}\neq 0$. Notons que l'angle (calculé dans~$D^{+}$) au point d'intersection des $v_{i}$ est $2\pi +\frac{2a_{1}\pi}{k}$. Pour les autres pôles $P_{i}$, on prend une $k$-partie polaire associée à $(r_{i};\emptyset)$ si $R_{i}\neq0$ (ou simplement un cylindre bordé par un unique lien-selle si $P_{i}$ est un pôle d'ordre $-k$) et $(r_{j_{i}};r_{j_{i}})$ avec $j_{i}>t$ si $R_{i}=0$. Il reste à identifier tous les segments par translation à l'exception des $v_{i}$ que nous identifions par rotation. On vérifie que cette surface plate possède les invariants locaux souhaités.
\end{proof}

\subsection{Constructions de domaines polaires}\label{sub:s4constructionsDP}

Avant d'aborder les autres cas, nous introduisons des constructions générales de domaines polaires. Les différents choix de signes sont faits pour que les lemmes soient utilisés directement par la suite afin d'obtenir les $k$-résidus souhaités. 

\begin{lem}\label{lem:existenceracines}
Soient $R,R_{1},R_{2}\in \CC$ avec $R_{1},R_{2}$ tous les deux non nuls et $a$ un entier satisfaisant $-\frac{k}{2}<a<0$. Étant donné $\alpha=\exp\left(\frac{2ia\pi}{k}\right)$, il existe des racines $k$-ièmes $r,r_{1},r_{2}$ des $R,R_{1},R_{2}$ et $v\in\CC^{\ast}$ tels que :
\begin{enumerate}
\item $-r_{1}-r_{2}+v-\alpha v = r$;
\item la partie polaire d'ordre $2k$ associée à $(v,-r_{1};r_{2},\alpha v)$ existe;
\item la somme des angles au bord de $-r_{1}$ est égale à $3\pi - a \frac{2\pi}{k} \in ]3\pi,4\pi[$;
\item la somme des angles au bord de $r_{2}$ est égale à $3\pi +a \frac{2\pi}{k}\in ]2\pi,3\pi[$.
\end{enumerate}
\end{lem}

\begin{proof}
Nous traitons tout d'abord le cas où $(R,R_{1},R_{2})=(0;(-1)^{k},1)$. On considère la $k$-partie polaire d'ordre~$2k$ associée à $(v,-r_{1};r_{2},\alpha v) = (1,\exp(2i a\pi/k);1,\exp(2i a\pi/k))$. Le résidu du pôle d'ordre $-2k$ associé à cette partie polaire est égal à~$0$. L'angle à gauche du segment $-r_{1}$ est  $\pi -2\bar a_{2}\pi/k $ et celui à droite est $2\pi$. La somme est bien égale à $3\pi -2 a\pi/k$. La somme des angles au bord de $r_{2}$ se calcule de manière similaire.
\smallskip
\par
Nous traitons maintenant le cas où $(R,R_{1},R_{2})=(R;(-1)^{k},1)$ avec $R\neq 0$. Dans ce cas, posons $r_{i} = (-1)^{i}$. On a alors $v = \frac{r}{1-\alpha} \neq 0$. Considérons la  $k$-partie polaire d'ordre~$2k$ associée à $(v,-r_{1};r_{2},\alpha v)$ en considérant que $v$, resp $\alpha v$, est en dessous de $-r_{1}$, resp $r_{2}$, si $v$, resp $\alpha v$ sont réels négatifs. Dénotons $\theta\in ]-\pi,\pi]$ l'argument de $v$ et calculons la somme des angles au bord de  $-r_{1}$ dans le cas où $\theta \in I=]-\pi  - \frac{2a\pi}{k} , \pi]$. Ce cas est représenté dans la figure~\ref{fig:partiepol2k1}. Notons que la somme des angles est indépendante de $\theta$ tant que $\theta$ reste dans $I$. En revanche, lorsque $\theta$ sort de $I$, le vecteur $\alpha v$ traverse $\tilde r_{2}$ et les sommes changent par $\pm 2\pi$.  Dans le domaine basique positif, ils sont respectivement de $\pi + \theta$ et~$\pi$. Dans le domaine basique négatif, l'angle est de $\pi - \theta - \frac{2a\pi}{k}$. Cela donne une somme égale à $3\pi - \frac{2a\pi}{k}$ comme souhaité.
Finalement en prenant l'argument de $r$ dans $]\pi-\frac{2\pi}{k},\pi]$, comme $(1-\alpha)v=r$ et que l'argument de $1-\alpha$ se trouve dans $]0,\frac{\pi}{2}[$, l'argument $\theta$ de $v$ sera dans $]0,\pi[$ pour $k \geq 4$ et $]-\frac{\pi}{6},\pi[$. Dans les deux cas, nous obtenons bien $\theta \in I$.

\begin{figure}[hbt]
\begin{tikzpicture}
\begin{scope}
\clip (-.5,-1) rectangle (5,1.5);
  \coordinate (a0) at (0,0); \coordinate (a1) at (1.5,0); \coordinate (a2) at (2,.5);  \coordinate (a3) at (3,.5);
\coordinate (a4) at (4.5,.5);
   \fill[black!10] (a0) -- (a1) -- (a2) -- (a3) -- (a4) -- ++(0,0) arc (0:193:2.27) -- cycle;
   \draw (a0) -- (a1) coordinate[pos=.5](e1)  -- (a2) coordinate[pos=.5](e2) -- (a3) coordinate[pos=.5](e3) -- (a4) coordinate[pos=.5](e4);
       \foreach \i in {2,3}
   \fill (a\i)  circle (2pt);
     \filldraw[fill=white] (a1)  circle (2pt);
\node[above,rotate=45] at (e2) {$v$};
\node[above] at (e3) {$-\tilde r_{1}$};
\end{scope}

\begin{scope}[yshift=-.5cm]
\clip (-.5,1) rectangle (5,-1.5);
  \coordinate (a0) at (0,0); \coordinate (a1) at (1.5,0); \coordinate (a2) at (2.5,0);  \coordinate (a3) at (3,-.5);
\coordinate (a4) at (4.5,-.5);
   \fill[black!10] (a0) -- (a1) -- (a2) -- (a3) -- (a4) -- ++(0,0) arc (0:-193:2.27) -- cycle;
     \draw (a0) -- (a1) coordinate[pos=.5](e1)  -- (a2) coordinate[pos=.5](e2) -- (a3) coordinate[pos=.5](e3) -- (a4) coordinate[pos=.5](e4);
         \foreach \i in {1,2}
     \filldraw[fill=white] (a\i)  circle (2pt);
          \fill (a3)  circle (2pt);
   
   \node[above,rotate=-45] at (e3) {$\alpha v$};
\node[above] at (e2) {$\tilde r_{2}$};
\end{scope}

\end{tikzpicture}
\caption{La $k$-partie polaire d'ordre~$2k$ associée à $(v,-r_{1};r_{2},\alpha v)$ dans le cas où $(R,R_{1},R_{2})=(R;(-1)^{k},1)$ avec $R\neq 0$.} \label{fig:partiepol2k1}
\end{figure}
\smallskip
\par
Nous considérons maintenant le cas où $R=0$ et $(R_{1},R_{2}) \neq ((-1)^{k},1)$. Nous normalisons à $r_{1}=-1$ et nous choisissons $r_{2}$ tel que son argument soit dans $[0, \frac{2\pi}{k}[$. Notons $\theta\in ]-\pi,\pi]$ l'argument de $v=\frac{r_{1}+r_{2}}{1-\alpha}$. La somme des angles au bord de  $-r_{1}$ dans le cas où $\theta \in I=]-\pi  +\arg(r_{2})- \frac{2a\pi}{k} , \pi]$ se calcule comme précédemment et donne l'angle souhaité. Commençons par le cas où $r_{2}$ est un nombre réel strictement supérieur à $1$, illustré à gauche de la figure~\ref{fig:partiepol2k2}. Dans ce cas, la concaténation de $(1-r_{2})$, $v$ et $-\alpha v$ forme un triangle isocèle d'angle $-\frac{2a\pi}{k}$ en un point qui est dans le demi-plan inférieur. On en déduit que $\theta = -\frac{\pi}{2} - \frac{a\pi}{k}$. et que $\theta$ appartient à l'intervalle $I$.
\par
Maintenant supposons que $r_{2}$ est d'argument $[0, \frac{2\pi}{k}[$ (et $r_{2}<1$ si l'argument est nul). La construction suivante est illustrée à droite de la figure~\ref{fig:partiepol2k2}. Considérons le cercle de centre~$0$ qui contient $1-r_{2}$. Ce cercle coupe en un point $-\rho \in \RR_{-}$. Pour un point $\rho \exp(i\phi)$ de ce cercle, le triangle formé par la concaténation de $(1-r_{2})$, $v$ et $-\alpha v$ est une rotation d'angle $\pi + \phi$ du triangle du paragraphe précédent. On en déduit que $\theta = -\frac{\pi}{2} - \frac{a\pi}{k} + (\pi + \phi)$. Comme $\pi + \phi > \arg(r_{2})$ et que $-\frac{k}{2}<
a<0$, on a bien que $v \in I$, ce qui permet de conclure ce cas.
\begin{figure}[hbt]
\begin{tikzpicture}[scale=2.2,decoration={
    markings,
    mark=at position 0.4 with {\arrow[very thick]{>}}}]
\begin{scope}[xshift=-3.9cm]
 \coordinate (a0) at (0,0);\node[above] at (a0) {$0$};
  \coordinate (a1) at (1,0);\node[above] at (a1) {$1$};
   \coordinate (a2) at (-1.2,0);\node[above] at (a2) {$1-r_{2}$};

   \draw (a0) -- (a1);\draw (a0) -- ++(-1.8,0);\draw (a1) -- ++(0,-1.5);
\draw[postaction={decorate}] (a2) -- ++ (-45:.85) coordinate[pos=.4] (b1) coordinate (a3); \draw[postaction={decorate}] (a3) -- coordinate[pos=.4] (b2) (a0);\draw (a1) --  (a2);
\node[below left] at (b1) {$v$};\node[below right] at (b2) {$-\alpha v$};

            \foreach \i in {0,1,2,3}
     \fill (a\i)  circle (1pt);
     
     \draw[->] (a3)--++(135:.2) node[above right] {$\alpha$} arc  (135:45:.2);
      \draw[->] (-1,0) arc  (0:-45:.2); \node at (-.9,-.1) {$\theta$};

    \draw[->] (.8,0) arc  (-180:-90:.2); \node at (.8,-.24) {$\frac{2\pi}{k}$};
\end{scope}

\begin{scope}
 \coordinate (a0) at (0,0);\node[above] at (a0) {$0$};
  \coordinate (a1) at (1,0);\node[above] at (a1) {$1$};
   \coordinate (a2) at (-130:1.4);\node[below left] at (a2) {$\rho\exp(i\phi) = 1-r_{2}$};

   \draw (a0) -- (a1);\draw (a0) -- ++(-1.8,0);\draw (a1) -- ++(0,-1.5);
\draw (-1.4,0)node [above] {$-\rho$} arc (-180:-44.5:1.4);
\draw (a0) -- (a2);\draw[postaction={decorate}] (a2) -- ++ (5:1.0) coordinate[pos=.4] (b1) coordinate (a3); \draw[postaction={decorate}] (a3)-- coordinate[pos=.3] (b2) (a0);\draw (a1) --  (a2);
\node[below] at (b1) {$v$};\node[right] at (b2) {$-\alpha v$};

            \foreach \i in {0,1,2,3}
     \fill (a\i)  circle (1pt);
     
           \draw[->] (-.2,0) arc  (-180:-130:.2); \node at (-.5,-.1) {$\pi + \phi$};
           \draw[->] (.8,0) arc  (-180:-150:.2); \node at (.6,.1) {$\arg(r_{2})$};

\end{scope}
\end{tikzpicture}
\caption{La configuration des différents vecteurs dans le cas  où $R=0$ et $(R_{1},R_{2}) \neq ((-1)^{k},1)$ à gauche et $R\neq0$ à droite.} \label{fig:partiepol2k2}
\end{figure}
\smallskip
\par
Dans les autres cas, i.e. si $(R;R_{1},R_{2}) \neq (0;(-1)^{k},1)$, nous  normalisons $r_{1}=-1$ et nous choisissons $r_{2}$ tel que son argument est dans $[0, \frac{2\pi}{k}]$ comme précédemment. Étant donnée une racine $r$ de $R$, le vecteur $v$ satisfait
$$v = \frac{r}{1-\alpha}+\frac{r_{2}-1}{1-\alpha} \,.$$
Il suffit donc de choisir une racine $r$ telle que l'argument de $\frac{r}{1-\alpha}$ est dans $I$ pour obtenir de manière similaire à précédemment que l'argument de $v$ appartient à $I$. On obtient alors les sommes d'angles souhaitées.
\end{proof}

Nous donnons maintenant le lemme correspondant dans le cas où $R_{1}=0$.
\begin{lem}\label{lem:existenceracinesbis}
Soient $R,R_{1},R_{2}\in \CC$ avec $R_{1}=0$ et $R_{2}\neq0$ et $a\in\left\{-\left[ \frac{k}{2}\right],\dots,-1\right\}$. Étant donné $\alpha=\exp(\frac{2ia\pi}{k})$, il existe des racines $k$-ièmes $r,r_{2}$ des $R,R_{2}$ et $v\in\CC^{\ast}$ tels que :
\begin{enumerate}
\item $-r_{2}+v-\alpha v = r$;
\item la partie polaire d'ordre $2k$ associée à $(v;r_{2},\alpha v)$ existe;
\item la somme des angles au bord de $-r_{1}$ est égale à $2\pi - a \frac{2\pi}{k} \in ]2\pi;3\pi[$;
\item la somme des angles au bord de $r_{2}$ est égale à $3\pi +a \frac{2\pi}{k} \in ]2\pi;3\pi[$.
\end{enumerate}
\end{lem}

\begin{proof}
La preuve est similaire à celle du lemme~\ref{lem:existenceracines} en laissant $R_{1}$ tendre vers~$0$. Lorsque $R_{1}=0$, l'angle total au bord de $r_{2}$ ne change pas, mais celui au bord de $r_{1}$ diminue de $\pi$. La partie polaire obtenue est donnée dans la figure~\ref{fig:surjr=0sneq01bis}.
 \begin{figure}[hbt]
\center
\begin{tikzpicture}[scale=.8]
\begin{scope}
  \clip (-1,-1) rectangle (6,.7);
\draw (0,0) coordinate (a1) --  (0:1.5) coordinate[pos=.5] (b1) coordinate (a2) --  ++(-60:1) coordinate[pos=.5] (b2) coordinate (a3) -- ++(0:2)coordinate[pos=.5] (b3) coordinate (a4);
% \clip (-1.5,-3) rectangle (9,.7);
 \fill[black!10] (a1) -- (a2) -- (a3) -- (a4)  --++ (0:.4) arc (0:157:2.2) -- cycle;  
\draw (0,0) coordinate (a1) --  (a2) --   (a3) --  (a4); 
 
\filldraw[fill=white] (a2) circle (3pt);
\fill (a3) circle (3pt);

\draw[dotted] (a1)-- ++(180:.4);
\draw[dotted] (a4)-- ++(0:.4);

\node[below] at (b1) {$a$};
 \node[rotate=-45] at (b2) {$v$};
 \node[above] at (b3) {$b$};
\end{scope}

\begin{scope}[xshift=0cm,yshift=-2cm]
  \clip (-2,-1.5) rectangle (6,2);
\draw (0,0) coordinate (a1) --  (0:1.5) coordinate[pos=.5] (b1)  coordinate (a2)  -- ++(0:1) coordinate[pos=.5] (b3) coordinate (a3)--  ++(-120:1) coordinate[pos=.5] (b2) coordinate (a4) -- ++(0:1.5) coordinate[pos=.5] (b4) coordinate (a5); 
% \clip (-1.5,-3) rectangle (9,.7);
 \fill[black!10] (a1) -- (a2) -- (a3) -- (a4) -- (a5) --++ (0:.4) arc (0:-200:2.2) -- cycle;  
\draw (0,0) coordinate (a1) --  (a2) --   (a3) --  (a4) --  (a5); 
 
\filldraw[fill=white] (a2) circle (3pt);
\filldraw[fill=white] (a3) circle (3pt);
\fill (a4) circle (3pt);

\draw[dotted] (a1)-- ++(180:.4);
\draw[dotted] (a5)-- ++(0:.4);

\node[above] at (b1) {$a$};
\node[above] at (b4) {$b$};
 \node[rotate=60] at (b2) {$\alpha v$};
 \node[above] at (b3) {$1$};
\end{scope}

\end{tikzpicture}
\caption{Partie polaire associée à $(0,1,0)$.} \label{fig:surjr=0sneq01bis}
\end{figure}
\end{proof}

La démonstration du lemme~\ref{lem:48EXCEPT2} va nécessiter des constructions spécifiques que nous donnons dans les deux lemmes suivants. 

\begin{lem}\label{lem:nonintersecp3cas0}
 On se donne un entier $k\geq3$, un entier $0<a<\frac{k}{2}$, et $R\in\CC^{\ast}$ distinct de $(-1)^{k}$ si $a=1$. Il existe $v_{1},v_{2}\in\CC^{\ast}$ et une racine $k$-ième $r$ de $R$ satisfaisant
 \begin{eqnarray*}
  v_{1}+v_{2}&=&1 \\
  v_{1} + \exp\left(\frac{2ai\pi}{k}\right) v_{2} &=& -r
 \end{eqnarray*}
tels que les $k$-parties polaires associées à $(v_{1},v_{2};\emptyset)$ et à $\left(\emptyset;\exp\left(\frac{2ai\pi}{k}\right) v_{2},v_{1}\right)$ existent. De plus, en collant les $v_{1}$ par translation et les $v_{2}$ par rotation, on a que :
\begin{enumerate}
 \item la somme des angles à la singularité correspondant au point initial de $v_{1}$ est égale à $4\pi - \frac{2a\pi}{k} \in ]3\pi,4\pi[$;
 \item la somme des angles à la singularité correspondant au point final de $v_{1}$ est  égale à  $4\pi + \frac{2a\pi}{k}\in ]4\pi,5\pi[$.
\end{enumerate}
\end{lem}

\begin{proof}
On note $\zeta=\exp\left(\frac{2ai\pi}{k}\right)$ et $\arg \zeta = \frac{2a\pi}{k}$.
Pour obtenir les angles du lemme il suffit que $(v_{1},v_{2})$ vérifie la condition~(*) ci-dessous. 
\begin{itemize}
 \item[(*)] Si $(v_{1},v_{2})$ forme une base positive de $\RR^{2}$, alors $(v_{1},\zeta v_{2})$ forme aussi une base positive de $\RR^{2}$. Si $(v_{1},v_{2})$ ne forme pas une base positive, il n'y a pas de contrainte supplémentaire (ce second cas est représenté à droite de la figure~\ref{fig:partiepol2k2bis}).
\end{itemize}

Calculons alors la somme des angles au point initial de $v_{1}$. Supposons que l'argument de~$v_{1}$ est égal à $\pi -\epsilon$ avec $0<\epsilon<\frac{2\pi}{k}$ auquel cas la condition~(*) est satisfaite. Le fait que $v_{1}+v_{2}=1$ implique que l'argument $\theta$ de~$v_{2}$ est dans l'intervalle $]-\epsilon,0[$. On en déduit que l'argument de $\zeta v_{2}$ est dans l'intervalle $]\arg(\zeta)-\epsilon,\arg(\zeta)[$. L'angle au point initial de $v_{1}$ est donc $\epsilon$, plus $\pi+\theta$ pour le point final de~$v_{2}$, plus $\pi$ pour la moitié inférieure de la première partie polaire, plus $2\pi-\epsilon-\theta -\arg(\zeta)$ pour le point d'intersection des vecteurs de la seconde partie polaire. La somme est $4\pi  -\arg(\zeta)$ comme souhaité.
Comme la somme des angles est invariante pas tant que $\zeta v_{2}$ ne devient pas proportionnel à $-v_{1}$ lorsque l'on fait une rotation de $v_{1}$ dans le sens direct on obtient la condition~(*).
\begin{figure}[hbt]
\begin{tikzpicture}[decoration={
    markings,
    mark=at position 0.4 with {\arrow[very thick]{>}}},scale=2]

\begin{scope}[xshift=-4.8cm,yshift=-1cm,scale=2]
 \coordinate (a0) at (0,0);\node[below left] at (a0) {$0$};
  \coordinate (a1) at (1,0);\node[above] at (a1) {$1$};
   \coordinate (a2) at (40:1.2); \node[above] at (a2) {$-r$};
\draw[dotted] (a1) -- (a2);
\draw[thin] (-.1,0) -- (1.1,0);\draw[thin] (0,-.1) -- (0,.9);
\draw[dashed] (-.1,.28) -- (1,.4) coordinate[pos=.6] (a3);
\draw[postaction={decorate}] (a0) -- node[above] {$v_{1}$} (a3);\draw[postaction={decorate}] (a3) -- node[above] {$v_{2}$} (a1);\draw[postaction={decorate}] (a3) -- node[above left] {$\zeta v_{2}$} (a2);

            \foreach \i in {0,1,...,3}
     \fill (a\i)  circle (.5pt);
\end{scope}

\begin{scope}[scale=1.3]
\clip (-1.5,-1) rectangle (1.9,.5);
   \coordinate (a0) at (-1,0); \coordinate (a1) at (0,0); \coordinate (a2) at  (160:.3);  \coordinate (a3) at (1,0);
\coordinate (a4) at (1.5,0);
   \fill[black!10] (a0) -- (a1) -- (a2) -- (a3) -- (a4) -- ++(0,0) arc (0:180:1.25) -- cycle;
   \draw (a0) -- (a1) coordinate[pos=.5](e1)  -- (a2) node[above] {$v_{1}$} --node[above] {$v_{2}$} (a3) -- (a4) coordinate[pos=.5](e4);
       \foreach \i in {1,3}
   \fill (a\i)  circle (.5pt);
     \filldraw[fill=white] (a2)  circle (.5pt);

       \draw[->] (160:.2) arc  (160:180:.2); \node at (-.3,.04) {$\epsilon$};
       \draw[->] (1.1,0) arc  (0:175:.1); \node at (1,.16) {$\pi+\theta$};
       
\begin{scope}[yshift=-.5cm]
\clip (-1.5,2) rectangle (1.9,-.3);
   \coordinate (a0) at (-1,0); \coordinate (a1) at (0,0); \coordinate (a2) at  (.97,.22) ;
   \fill[black!10] (a0) -- (a1) -- (a2) -- ++ (160:.3) coordinate (a3) --++(1,0) coordinate (a4) -- ++(0,0) arc (0:-145:1.2) -- cycle;
     \draw (a0) -- (a1) coordinate[pos=.5](e1)  --node[below]  {$\zeta v_{2}$}  (a2)  -- (a3) node[left]  {$v_{1}$}  -- (a4) coordinate[pos=.5](e4);
       \draw[dotted] (a2) --++(.3,0);
    
       \draw[->] (a3) --++(0:.2) arc  (0:-20:.2); \node at (.99,.28) {$-\epsilon$};
       \draw[->] (a2) --++(-167:.1) arc  (-167:0:.1); \node at (1.2,.05) {$\pi-\theta-\arg(\zeta)$};

                \foreach \i in {1,3}
     \filldraw[fill=white] (a\i)  circle (.5pt);
          \fill (a2)  circle (.5pt);
\end{scope}
\end{scope}
\end{tikzpicture}
\caption{La configuration des différents vecteurs à gauche et les moitiés intéressantes des parties polaires $(v_{1},v_{2};\emptyset)$  et à $(\emptyset;\zeta v_{2},v_{1})$ à droite (pour d'autres vecteurs).} \label{fig:partiepol2k2bis}
\end{figure}

Il reste donc à montrer qu'il existe un choix de la racine $r$ de $R$ tel que $(v_{1},v_{2})$ satisfait la condition~(*). Nous choisissons $r$ telle que $\arg(-r) \in [0,\frac{2\pi}{k}[$ sauf dans le cas où  $R=(-1)^{k}$ avec $a>1$ auquel cas $-r = \exp\left(\frac{2i\pi}{k}\right)$. Les différents vecteurs sont représentés à gauche de la figure~\ref{fig:partiepol2k2bis}.
Remarquons que la condition~(*) est équivalente au fait que $(0,0)$ n'appartient pas au cône $\mathcal{C}$ délimité par les demi-droites $[v_{1},1)$ et $[v_{1},-r)$ de vecteurs directeurs respectifs~$v_{2}$ et $\zeta v_{2}$. En effet, si $(v_{1},v_{2})$ forme une base positive, le vecteur $v_{1}$ appartient au demi-plan inférieur ouvert. Le fait que l'origine n'appartienne pas à $\mathcal{C}$ est équivalent au fait que $v_{2}$ et~$\zeta v_{2}$ appartiennent au même demi-plan délimité par la droite contenant l'origine et de vecteur directeur $v_{1}$. Les deux bases $(v_{1},v_{2})$ et $(v_{1},\zeta v_{2})$ ont donc le même signe et sont donc positives.

Nous vérifions maintenant que l'origine n'appartient pas au cône $\mathcal{C}$. Supposons par l'absurde que $v_{1} = 0$, on a alors $-r = \exp\left(\frac{2ai\pi}{k}\right)$ et donc ce cas n'apparaît que si $R=(-1)^{k}$ et $a=1$ qui est exclu. Supposons que $(0,0)$ n'est pas le sommet de $\mathcal{C}$ mais est contenu dans celui-ci. Il suffit de vérifier le cas où $\zeta$ est d'argument minimal $\frac{2i\pi}{k}$. Dans ce cas, le cône $\mathcal{C}$ est une transformation rigide du cône délimité par les demi-droites $[0,1)$ et $[0,\zeta)$. Comme les demi-droites $[0,1)$ et $[v_{1},1)$ se coupent en $1$, cela implique que la demi-droite $[v_{1},\zeta v_{2})$ ne coupe pas $[0,\zeta)$ sauf si elles sont confondues. Ce dernier cas correspond à $R=(-1)^{k}$ et $a=1$ qui est exclu. On en déduit donc que l'argument de $-r$ n'est pas dans $[0,\frac{2\pi}{k}]$, ce que nous avions supposé.
% Si $-r\in \RR_{>0}\setminus\lbrace 1\rbrace$, la médiatrice est verticale et on a donc que les bases $(v_{1},v_{2})$ et $(v_{1},\zeta v_{2})$ sont positives si $-r \in ]0,1[$ et sont négatives si $-r >1$.  Pour conclure, si $\arg(-r) \in ]0,\frac{2\pi}{k}]$, on déforme $-r$ le long d'une demi-droite de la forme $-r_{0} + \exp(\frac{2i\pi}{k}) \RR_{>0}$ avec $-r_{0} \in \RR_{>0}\setminus\lbrace 1\rbrace$. L'argument du vecteur $v_{1}$ est croissant tandis que celui de $v_{2}$ et $\zeta v_{2}$ est décroissant. Comme de plus les segments ne se coupent pas lors de cette déformation, on en déduit que la condition sur les bases reste satisfaite. 
% 
% Finalement il reste à vérifier le cas où $-r$ appartient à la demi-droite $1 + \exp(\frac{2i\pi}{k}) \RR_{>0}$ et où $-r = \exp(\frac{2i\pi}{k})$ dans le cas $a>1$. Ces deux cas se vérifient sans problèmes.
\end{proof}

On donne maintenant une variation du lemme précédent.
\begin{lem}\label{lem:nonintersecp3}
 On se donne un entier $k\geq3$, un entier $0<a<\frac{k}{2}$, un nombre complexe $R\in\CC^{\ast}$ et $S_{1},S_{2}\in\CC$ non tous les deux nuls. Il existe $v_{1},v_{2}\in\CC^{\ast}$, des racines $k$-ièmes $s_{i}$ de $S_{i}$ et une racine $k$-ième $r$ de $R$ satisfaisant
 \begin{eqnarray*}
  v_{1}+v_{2}&=&1 + s_{1} \\
  v_{1} + \exp\left(\frac{2ai\pi}{k}\right) v_{2} &=& -r +s_{2}
 \end{eqnarray*}
tels que les $k$-parties polaires associées à $(v_{1},v_{2};s_{1})$  et à $\left(s_{2};\exp\left(\frac{2ai\pi}{k}\right)  v_{2},v_{1}\right)$ existent. De plus, en identifiant les $v_{1}$ par translation et les $v_{2}$ par rotation on a que :
\begin{enumerate}
 \item la somme des angles au bord de $s_{1}$ est $(5- \delta_{1})\pi - \frac{2a\pi}{k} \in \left](4- \delta_{1})\pi,(5- \delta_{1})\pi\right[$,
 \item la somme des angles au bord de $s_{2}$ est $(5- \delta_{2})\pi + \frac{2a\pi}{k} \in \left](5- \delta_{2})\pi,(6- \delta_{2})\pi\right[$,
\end{enumerate}
où $\delta_{i} = 0$ si $S_{i} \neq 0$ et $\delta_{i} = 1$ si $S_{i} = 0$.
\end{lem}

\begin{proof}
Nous supposons tout d'abord que $S_{1}$ et $S_{2}$ sont non nuls. Choisissons une racine~$s_{1}$ de $S_{1}$ dont l'argument est dans $[-\frac{\pi}{k},\frac{\pi}{k}]$. On en déduit que l'argument $\gamma$ de $1+s_{1}$ est dans $]-\frac{\pi}{k},\frac{\pi}{k}[$.  Choisissons maintenant $s_{2}$ d'argument dans $[\gamma,\gamma +\frac{2\pi}{k}[$ et $-r$ d'argument dans $]\gamma,\gamma +\frac{2\pi}{k}]$.
\begin{figure}[hbt]
\begin{tikzpicture}[decoration={
    markings,
    mark=at position 0.4 with {\arrow[very thick]{>}}},scale=2]
\begin{scope}
\clip (-1.5,-.7) rectangle (2.5,.5);
   \coordinate (a0) at (-.7,0); \coordinate (a1) at (0,0); \coordinate (a2) at  (160:.3);  \coordinate (a3) at (1,0);
   \draw[dotted] (a1) --  (1.27,.22) coordinate (a5);
   \fill[black!10] (a0) -- (a1) -- (a2) -- (a5) --++(.4,0) coordinate(a4)  -- ++(0,0) arc (0:190:1.2) -- cycle;
   \draw (a0) -- (a1) coordinate[pos=.5](e1)  -- (a2) node[above] {$v_{1}$} --node[above] {$v_{2}$} (a5) -- (a4);
       \foreach \i in {1,5}
   \fill (a\i)  circle (.5pt);
     \filldraw[fill=white] (a2)  circle (.5pt);
     \node[below] at (a1) {$0$};   \node[above] at (a5) {$1+s_{1}$};
       
     \begin{scope}[xshift=.5cm,yshift=-.5cm]  
    \coordinate (a0) at (-.8,0); \coordinate (a1) at (-.2,0); \coordinate (a2) at  (.27,.22) ;
   \fill[black!10] (a0) -- (a1) -- (a2) --++(.6,0) coordinate (a3) -- ++(0,0) arc (0:-145:1) -- cycle;
     \draw (a0) -- (a1) coordinate[pos=.5](e1)  --node[below]  {$s_{1}$}  (a2)  -- (a3);

                \foreach \i in {1,2}
     \fill (a\i)  circle (.5pt);
   \end{scope}
     \end{scope}

\begin{scope}[xshift=3.5cm,yshift=-.5cm]
\clip (-1.5,1) rectangle (2.5,-.35);
   \coordinate (a0) at(-1,-.2); \coordinate (a1) at (-.3,-.2); \coordinate (a2) at  (.97,.22) ;
   \fill[black!10] (a0) -- (a1) -- (a2) -- ++ (160:.3) coordinate (a3) --++(1,0) coordinate (a4) -- ++(0,0) arc (0:-145:1.3) -- cycle;
     \draw (a0) -- (a1) coordinate[pos=.5](e1)  --node[below]  {$\zeta v_{2}$}  (a2)  -- (a3) node[left]  {$-r+s_{2}$}node[below right]  {$v_{1}$}  -- (a4) coordinate[pos=.5](e4); 

                \foreach \i in {1,3}
     \filldraw[fill=white] (a\i)  circle (.5pt);
          \fill (a2)  circle (.5pt);
          \draw[dotted] (a1) -- (a3);
          
\begin{scope}[xshift=.3,yshift=.6cm]
 \coordinate (a0) at (-.8,0); \coordinate (a1) at (-.2,0); \coordinate (a2) at  (.37,.27);\coordinate (a4) at  (.63,.11);
   \fill[black!10] (a0) -- (a1) -- (a2) --++(.6,0) coordinate (a3) -- ++(0,0) arc (0:145:1) -- cycle;
     \draw (a0) -- (a1) coordinate[pos=.5](e1)  --node[above left]  {$s_{2}$}  (a2)  -- (a3);
\draw[dotted] (a1) -- (a4);\draw[dotted] (a1) --++ (.85,0) coordinate (a5);  \draw[->] (a5) arc  (0:40:.2); \node at (.79,.08) {$\gamma$};
                \foreach \i in {1,2}
     \filldraw[fill=white] (a\i)  circle (.5pt);
\end{scope}          
\end{scope}
\end{tikzpicture}
\caption{Les parties polaires $(v_{1},v_{2};s_{1})$  et à $(s_{2};\zeta v_{2},v_{1})$.} \label{fig:partiepol2k3}
\end{figure}

Par ce choix, les puissances $k$-ièmes des nombres complexes $1+v_{1}$ et $v_{2}-r$ ne sont pas de la forme $(1,(-1)^{k})$. Le lemme~\ref{lem:nonintersecp3cas0} donne une solution $(v_{1},v_{2})$ telle que les parties polaires $(v_{1},v_{2};\emptyset)$ et $\left(\emptyset;\exp\left(\frac{2ai\pi}{k}\right) v_{2},v_{1}\right)$ existent et les angles sont donnés donnés dans ce lemme. Il suffit alors de remplacer les $\emptyset$ par les $s_{i}$ dans les parties polaires. Ajouter ces vecteurs ajoute~$\pi$ aux angles bordant un $s_{i}$ différent de zéro.

Le cas où l'un des $S_{1},S_{2}$ est vide se traite de façon similaire.  
\end{proof}

\subsection{Le cas de $2$ zéros, d'ordres positifs, tous les résidus sont non nuls}\label{sub:s4constructions}

Nous considérons  les strates $\Omega^{k}\mathcal{M}_{0}(a_{1},a_{2};-k\ell_{1},\dots,-k\ell_{p};\rec[-k][s])$ avec $a_{1},a_{2} \geq 1$ (et donc $l_{1},l_{2} \geq 1$).  Dans cette section, nous supposerons  sans perte de généralité que $-k<\bar a_{2}<-\frac{k}{2}<\bar a_{1}<0$ et $\ell_{1} \geq \dots \geq \ell_{p+s}$. Nous traitons dans cete section le cas dans lequel tous les résidus sont non nuls.
\par
Le cas $p=1$ et $s=0$ est pratiquement trivial. Toutes les $k$-différentielles de telles strates sont de la forme $\lambda z^{a_{1}}(z-1)^{a_{2}}dz^{k}$. Il suffit de choisir convenablement la constante $\lambda$ pour obtenir n'importe quel $k$-résidu non nul. Nous supposerons à partir de maintenant que $p+s \geq 2$.

\begin{lem}\label{lem:47NONNUL}
Toute configuration de $k$-résidus $(R_{1},\dots,R_{p+s})\in (\CC^{\ast})^{p+s}$ est réalisable dans les strates $\Omega^{k}\mathcal{M}_{0}(a_{1},a_{2};-k\ell_{1},\dots,-k\ell_{p};\rec[-k][s])$,
% avec $p \geq 1$, 
à l'exception possible des strates de la forme $\Omega^{k}\mathcal{M}_{0}(a_{1},a_{2};\rec[-k\ell][p])$ avec $l_{1}+1 \in \ell \mathbb{Z}$ et $l_{2} \in \ell \mathbb{Z}$.
\end{lem}

\begin{proof}
 Pour une partition $S_{1} \cup S_{2}$ de l'ensemble $\lbrace{ 2,\dots,p+s \rbrace}$ nous définissons :
 \begin{equation*}
  m_{1} = \sum\limits_{j \in S_{1}} \ell_{j} \text{ et }m_{2} = \sum\limits_{j \in S_{2}} \ell_{j}\,.
 \end{equation*}
Nous prouvons que nous pouvons toujours choisir une partition telle que $l_{1} \geq m_{1}$ et $l_{2} \geq m_{2}+1$. Nous choisissons l'entier $t$ le plus grand possible tel que $\sum_{2 \leq i \leq t} \ell_{i} \leq l_{1}$ et posons $S_{1}=\lbrace{ 2,\dots,t \rbrace}$. Il reste à prouver que $m_{2} \leq l_{2}-1$. Si $t=p+s$, nous avons $m_{2}=0 \leq l_{2}-1$. Dans le cas contraire, nous avons $t<p+s$ et $m_{1}+\ell_{t+1} > l_{1}$. Comme $l_{1}+l_{2} = -1+\sum_{i=1}^{p+s}\ell_{i}$, il s'ensuit que $l_{1}+l_{2} = -1+\ell_{1}+m_{1}+m_{2}$ et donc $(l_{2}-1)-m_{2}=(m_{1}-l_{1})+\ell_{1}-2$. Nous savons que $m_{1}-l_{1}>-\ell_{t+1} \geq -\ell_{1}$ et donc $(l_{2}-1)-m_{2} \geq -1$.
\par
Le seul cas dans lequel $(l_{2}-1)-m_{2}$ n'est pas positif implique que $\ell_{1}=\dots=\ell_{t+1}$, $m_{1}=(t-1)\ell_{1}$ et $l_{1}-m_{1}=\ell_{1}-1$. A moins que nous n'ayons $\ell_{1}=\dots=\ell_{p+s}$ (et en fait $s=0$ dans ce cas), il est toujours possible d'ajouter à l'ensemble $S_{1}$ l'indice $p+s$ pour lequel nous avons $\ell_{p+s}<\ell_{1}$. La borne $l_{1}-m_{1} \geq 0$ sera dans ce cas toujours satisfaite tandis que $(l_{2}-1)-m_{1}$ sera devenu positif. Dans le dernier cas, tous les pôles sont de même ordre $-k\ell$, nous avons $l_{1}+1 \in \ell_{1}\mathbb{Z}$. Cette situation est explicitement exclue dans l'énoncé du lemme. Les inégalités $l_{1} \geq m_{1}$ et $l_{2} \geq m_{2}+1$ sont donc démontrées.
\par
L'ensemble des pôles dont l'indice est dans $S_{1}$ (resp. $S_{2}$) contribuera exclusivement à la singularité d'ordre $a_{1}$ (resp. $a_{2}$). Pour l'ensemble des pôles d'indice $j \geq 2$, nous choisissons une racine $r_{j}$  du $k$-résidu $R_{j}\neq0$ dont l'argument appartient à $\left[-\frac{\pi}{k},\frac{\pi}{k}\right[$.
\par
Notons $\tilde{r}_{1}= \sum_{j\in S_{1}} r_{j}$ et $\tilde{r}_{2}=\sum_{j\in S_{2}}r_{j}$. On considère deux polygones (éventuellement vides ou réduits à un segment) $M_{1}$ et $M_{2}$ dont les arêtes sont :
\begin{enumerate}
    \item $\tilde{r}_{1}$ et les $(-r_{i})_{i \in S_{1}}$ classés par argument croissant pour $M_{1}$;
    \item $\tilde{r}_{2}$ et $(-r_{i})_{i \in S_{2}}$ pour $M_{2}$ (idem).
\end{enumerate}
Il reste donc à recoller la dernière arête de chaque polygone sur une partie polaire correspondant au pôle d'ordre $-k\ell_{1}$.
\par

Supposons d'abord que $S_{1}$ et $S_{2}$ sont non vides. Nous allons fabriquer une $k$-partie polaire d'ordre $k\ell_{1}$ associée à $( v,-\tilde{r}_{1};\tilde{r}_{2},\alpha v)$ telle que la somme $\tilde r_{1} + \tilde r_{2} + v + \alpha v$ est l'opposé d'une racine du résidu de $P_{1}$. De plus l'angle total au bord de $-\tilde{r}_{1}$ (correspondant à la singularité d'ordre $a_{1}$) devra valoir $(l_{1}-m_{1})2\pi+3\pi+\frac{2\bar{a_{1}\pi}}{k}$. De même, l'angle total au bord de $\tilde{r}_{2}$ devra valoir $(l_{2}-m_{2})2\pi+3\pi+\frac{2\bar{a_{2}\pi}}{k}=(l_{2}-m_{2})2\pi+\pi-\frac{2\bar{a_{1}\pi}}{k}$ pour que les ordres des singularités soient bien réalisés. Le lemme~\ref{lem:existenceracines} nous assure de l'existence d'une telle partie polaire pour un pôle d'ordre $-2k$. Il suffit ensuite d'ajouter respectivement $l_{1}-m_{1}$ et $l_{2}-m_{2}-1$ plans aux deux singularités (nous avons démontré plus haut que ces deux nombres étaient positifs) pour obtenir la $k$-partie polaire voulue.
\par
Si $S_{1}$ est vide mais que $S_{2}$ est non vide, c'est le lemme~\ref{lem:existenceracinesbis} qui fournit la construction de la partie polaire. Si au contraire $S_{1}$ est non vide mais que $S_{2}$ est vide, la construction du lemme~\ref{lem:existenceracines} fonctionne même lorsque  $\tilde{r_{2}}=0$.
\par
Enfin, si $S_{1}$ et $S_{2}$ sont tous les deux vides, alors $p=1$ et $s=0$. Ce cas a déjà été traité.
\end{proof}

Dans les deux lemmes suivants, nous traitons à part le cas des strates pour lesquelles tous les pôles sont de même ordre.

\begin{lem}\label{lem:48EXCEPT}
Toute configuration de $k$-résidus $(R_{1},\dots,R_{p})\in (\CC^{\ast})^{p}$ avec $p\geq2$ est réalisable dans les strates $\Omega^{k}\mathcal{M}_{0}(a_{1},a_{2};\rec[-k\ell][p])$ si :
\begin{enumerate}
    \item $l_{1} \leq p(\ell-1) -1$ et $l_{2} \geq p$;    
    \item $l_{1} \in -1+\ell\mathbb{Z}$ et $l_{2} \in \ell\mathbb{Z}$.
\end{enumerate}
\end{lem}

\begin{proof}
Nous supposerons dans un premier temps que $\ell=2$.
Nous traitons d'abord le cas $p=2$. Sans perte de généralité, nous supposerons $R_{1}=1$ et choisissons sa racine $r_{1}=1$. Nous choisissons une racine $r_{2}$ de $R_{2}$ telle que $\arg(r_{2})\in [0,\frac{2\pi}{k}[$. Construisons un triangle dont les côtés sont $r_{1},r_{2},-(r_{1}+r_{2})$. Nous construisons ensuite un quadrilatère en collant le long du côté $-(r_{1}+r_{2})$ un triangle isocèle dont le sommet distingué est d'angle $0<-\frac{2\bar{a_{1}}\pi}{k}<\pi$. Nous nommons ses côtés $r_{1},r_{2},s_{2},s_{1}$. Nous découpons ensuite ce quadrilatère le long de la diagonale $d$ distincte de $-(r_{1}+r_{2})$ pour obtenir deux triangles $T_{1},T_{2}$ de côtés respectifs $r_{1},s_{1},d$ et $s_{2},r_{2}.d$, voir la figure~\ref{fig:constrGuill1} à gauche.
\begin{figure}[htb]
\center
\begin{tikzpicture}[decoration={
    markings,
    mark=at position 0.4 with {\arrow[very thick]{>}}},scale=1.2]

    %Premier
    \begin{scope}[xshift=-4cm,yshift=-.7cm,scale=1.2]
  \coordinate (a) at (-.5,0);
\coordinate (b) at (.5,0);
\coordinate (c) at (0,1);

\draw (a) --node[below] {$r_{1}$} (b) -- node[right] {$r_{2}$} ++(55:1) coordinate (d) --node[above] {$s_{2}$} (c) -- node[left] {$s_{1}$}(a);
\draw[dotted] (b) --node[right] {d} (c);

  \draw[->] (c)++(-.05,-.1) arc  (-150:15:.1); \node at (0,1.25) {$\frac{-2\bar{a_{1}}\pi}{k}$};
    \end{scope}

%second dessin
\begin{scope}[xshift=0.5cm]
\fill[fill=black!10] (0,0) coordinate (Q) circle (1.2cm);

\coordinate (a) at (-.5,0);
\coordinate (b) at (.5,0);
\coordinate (c) at (0,1);

\fill (a)  circle (2pt);
\fill[] (b) circle (2pt);
    \fill[white] (a) -- (c)coordinate[pos=.5](f) -- (b)coordinate[pos=.5](g) -- ++(0,-2) --++(-1,0) -- cycle;
 \draw  (a) -- (c) coordinate () -- (b);
 \draw (a) -- ++(0,-1.1) coordinate (d)coordinate[pos=.5] (h);
 \draw (b) -- ++(0,-1.1) coordinate (e)coordinate[pos=.5] (i);
 \draw[dotted] (d) -- ++(0,-.2);
 \draw[dotted] (e) -- ++(0,-.2);
\node[above, rotate=55] at (f) {$s_{1}$};
\node[above, rotate=-55] at (g) {$d$};
\node[left] at (h) {$1$};
\node[right] at (i) {$1$};

 \draw[dotted,postaction={decorate}]  (a) --node[below] {$r_{1}$} (b) ;
    \end{scope}

%troisieme figure
\begin{scope}[xshift=3.5cm,rotate=55]
\fill[fill=black!10] (0,0) coordinate (Q) circle (1.2cm);

\coordinate (a) at (-.5,0);
\coordinate (b) at (.5,0);
\coordinate (c) at (0,1);

\fill (a)  circle (2pt);
\fill[] (b) circle (2pt);
    \fill[white] (a) -- (c)coordinate[pos=.5](f) -- (b)coordinate[pos=.5](g) -- ++(0,-2) --++(-1,0) -- cycle;
 \draw  (a) -- (c) coordinate () -- (b);
 \draw (a) -- ++(0,-1.1) coordinate (d)coordinate[pos=.5] (h);
 \draw (b) -- ++(0,-1.1) coordinate (e)coordinate[pos=.5] (i);
 \draw[dotted] (d) -- ++(0,-.2);
 \draw[dotted] (e) -- ++(0,-.2);
\node[above, rotate=-55] at (f) {$d$};
\node[above, rotate=-15] at (g) {$s_{2}$};
\node[above] at (i) {$2$};
\node[below] at (h) {$2$};

 \draw[dotted,postaction={decorate}]  (a) --node[below right] {$r_{2}$} (b) ;
\end{scope}
\end{tikzpicture}
\caption{Le quadrilatère et les parties polaires formant la surface.} \label{fig:constrGuill1}
\end{figure}
\par
Nous construisons les deux domaines polaires $D_{i}$ en retirant d'un plan infini un triangle de la forme $T_{i}$ ainsi qu'une bande infinie collée le long du côté $r_{i}$ du triangle $T_{i}$. Nous identifions ensuite les côtés $s_{1}$ et $s_{2}$ ensemble, ainsi que les deux côtés correspondent à $d$, voir la figure~\ref{fig:constrGuill1} à droite.
\par
Nous pouvons généraliser la construction pour un nombre de pôles $p$ quelconque de la façon suivante. Nous choisissons des racines $r_{1},\dots,r_{p}$ dont les arguments sont dans $[0,\frac{2\pi}{k}[$ et les ordonnons par argument croissant (en changeant éventuellement les indices). Nous construisons avec elles un polygone ayant $p+1$ côtés dont le dernier est $- \sum r_{i}$. Nous construisons ensuite un polygone convexe avec $p+2$ côtés en collant sur ce dernier bord un triangle isocèle dont le sommet distingué est d'angle $-\frac{2\bar{a_{1}}\pi}{k}$. Nous nommons ses côtés $r_{1},\dots,r_{p},s_{1},s_{2}$. Nous découpons ensuite ce polygone le long de $l_{1}$ diagonales passant par le sommet distingué du triangle isocèle. Ceci produit une partition des $r_{1},\dots,r_{p}$ en $l_{1}+1$ polygones $M_{1},\dots,M_{l_{1}+1}$. Nous construisons ensuite $l_{1}+1$ domaines polaires en retirant d'un plan infini un polygone $M_{i}$ dans chacun d'eux et une bande infinie à partir du côté $r_{j}$ avec $j$ minimal dans chaque domaine. Pour chacun des autres résidus $r_{i}$, nous formons un domaine polaire en enlevant une bande infinie à partir d'un segment $r_{i}$, voir la figure~\ref{fig:constrGuill2}. Les collages sont alors effectués comme précédemment.
\begin{figure}[htb]
\center
\begin{tikzpicture}[decoration={
    markings,
    mark=at position 0.4 with {\arrow[very thick]{>}}},scale=1.2]

    %Premier
    \begin{scope}[xshift=-4cm,yshift=-.5cm,scale=1.2]
  \coordinate (a) at (-.5,0);
\coordinate (b) at (.5,0);
\coordinate (c) at (0,1);
\coordinate (e) at (0,-.5);

\draw (a) --node[below] {$r_{1}$} (e) --node[below] {$r_{2}$} (b)  -- node[right] {$r_{3}$} ++(55:1) coordinate (d) --node[above] {$s_{2}$} (c) -- node[left] {$s_{1}$}(a);
\draw[dotted] (b) --node[right] {d} (c);

  \draw[->] (c)++(-.1,-.2) arc  (-150:15:.2); \node at (0,1.25) {$\frac{-2\bar{a_{1}}\pi}{k}$};
    \end{scope}

%second dessin
\begin{scope}[xshift=0.5cm]
\fill[fill=black!10] (0,0) coordinate (Q) circle (1.2cm);

\coordinate (a) at (-.5,0);
\coordinate (b) at (.5,0);
\coordinate (c) at (0,1);
\coordinate (e) at (0,-.5);

\fill (a)  circle (2pt);
\fill[] (b) circle (2pt);
\fill (e)  circle (2pt);
    \fill[white] (a)  -- (c)coordinate[pos=.5](f) -- (b)coordinate[pos=.5](g) -- (e) -- ++(0,-1) --++(-.5,0) -- cycle;
 \draw  (a) -- (c) -- (b) --node[below] {$r_{2}$} (e);
 \draw (a) -- ++(0,-1.1) coordinate (d)coordinate[pos=.5] (h);
 \draw (e) -- ++(0,-.7) coordinate (i)coordinate[pos=.5] (i);
 \draw[dotted] (d) -- ++(0,-.2);
 \draw[dotted] (i) -- ++(0,-.2);
\node[above, rotate=55] at (f) {$s_{1}$};
\node[above, rotate=-55] at (g) {$d$};
\node[left] at (h) {$1$};
\node[right] at (i) {$1$};

 \draw[dotted,postaction={decorate}]  (a) --node[below] {$r_{1}$} (e) ;
    \end{scope}

    %second bis dessin
\begin{scope}[xshift=3cm]
\fill[fill=black!10] (0,0) coordinate (Q) circle (1.2cm);

\coordinate (a) at (-.5,0);
\coordinate (b) at (0,.5);

\fill (a)  circle (2pt);
\fill[] (b) circle (2pt);
    \fill[white] (a)  -- (b)coordinate[pos=.5](f)  -- ++(-55:1.6) --++(-100:.95) -- cycle;
 \draw  (a) --node[above left] {$r_{2}$} (b);
 \draw (a) -- ++(-55:1.45) coordinate (d)coordinate[pos=.5] (h);
 \draw (b) -- ++(-55:1.55) coordinate (i)coordinate[pos=.5] (i);
\node[left] at (h) {$2$};
\node[right] at (i) {$2$};
    \end{scope}

%troisieme figure
\begin{scope}[xshift=5.5cm,rotate=55]
\fill[fill=black!10] (0,0) coordinate (Q) circle (1.2cm);

\coordinate (a) at (-.5,0);
\coordinate (b) at (.5,0);
\coordinate (c) at (0,1);

\fill (a)  circle (2pt);
\fill[] (b) circle (2pt);
    \fill[white] (a) -- (c)coordinate[pos=.5](f) -- (b)coordinate[pos=.5](g) -- ++(0,-2) --++(-1,0) -- cycle;
 \draw  (a) -- (c) coordinate () -- (b);
 \draw (a) -- ++(0,-1.1) coordinate (d)coordinate[pos=.5] (h);
 \draw (b) -- ++(0,-1.1) coordinate (e)coordinate[pos=.5] (i);
 \draw[dotted] (d) -- ++(0,-.2);
 \draw[dotted] (e) -- ++(0,-.2);
\node[above, rotate=-55] at (f) {$d$};
\node[above, rotate=-15] at (g) {$s_{2}$};
\node[above] at (i) {$3$};
\node[below] at (h) {$3$};

 \draw[dotted,postaction={decorate}]  (a) --node[below right] {$r_{3}$} (b) ;
\end{scope}
\end{tikzpicture}
\caption{Le polygone et les trois parties polaires formant la surface pour $p=3$ et $l_{1}=1$.} \label{fig:constrGuill2}
\end{figure}
Nous réalisons ainsi tous les $k$-résidus et l'une des singularités aura un angle de $2(l_{1}+1)\pi +\frac{2\bar{a_{1}}\pi}{k}$. La surface obtenue est donc dans la strate voulue. Cependant, cette construction n'est réalisable que si $l_{1}+1 \leq p$.
\par
Si $\ell \geq 3$, la construction se généralise en collant $\ell-2$ plans sur chaque domaine polaire. Toutefois, il y a des contraintes sur le choix de la singularité conique dont l'angle va augmenter. Lorsque l'on choisit de découper le polygone convexe ayant $p+2$ côtés le long de $t$ diagonales (avec $1 \leq t \leq p-1$), on produit exactement $p-t-1$ domaine polaires bordés par un unique lien-selle (voir la figure~\ref{fig:constrGuill2}). Il s'ensuit que dans l'angle $2(l_{1}+1)\pi +\frac{2\bar{a_{1}}\pi}{k}$ le paramètre $l_{1}$ va pouvoir prendre toutes les valeurs entre $t$ et $(t+1)(\ell-1)-1$ (si on ajoute $\ell-1$ plans dans chacun des $t+1$ domaines polaires concernés).
\par
Il suffit ensuite de faire varier $t$ entre $1$ et $p-1$ pour pouvoir obtenir chaque valeur de $l_{1}$ entre $1$ et $(\ell-1)p-1$. Dans les cas qui restent ouverts, on a $l_{1} \geq (\ell-1)p$ et $l_{2} \leq p-1$.
\end{proof}

Une dernière construction permet de produire les configurations de résidus non nuls dans les cas restants.

\begin{lem}\label{lem:48EXCEPT2}
Toute configuration de $k$-résidus $(R_{1},\dots,R_{p})\in (\CC^{\ast})^{p}$ avec $p \geq 2$  est réalisable dans les strates $\Omega^{k}\mathcal{M}_{0}(a_{1},a_{2};\rec[-k\ell][p])$ avec :
\begin{enumerate}
    \item $l_{1} \geq (\ell-1)p$ et $l_{2} \leq p-1$;    
    \item $l_{1} \in -1+\ell\mathbb{Z}$ et $l_{2} \in \ell\mathbb{Z}$.
%     \item $\ell \geq 2$;
%     \item $p \geq 2$.
\end{enumerate}
\end{lem}

\begin{proof}
Dans la mesure où le lemme~\ref{lem:44zeronegatif} couvre les cas dans lesquels un zéro est d'ordre négatif, nous supposons que $l_{2} \geq 1$. Nous déduisons du point (2) qu'il existe $t$ tel que $l_{1}=t\ell-1$. Comme $l_{1}+l_{2}=\ell p -1$, le point (1) et le fait que $l_{2} \geq 1$ impliquent que $\ell p-p \leq l_{1} \leq \ell p - 2$.
Il s'ensuit que $p\frac{\ell-1}{\ell}+\frac{1}{\ell} \leq t \leq p -\frac{1}{\ell}$. Comme $t$ est un entier, on obtient l'inégalité $2\leq t \leq p-1$.
\par
On utilise une variation de la construction de la preuve du lemme~\ref{lem:47NONNUL}. On sélectionne deux pôles que l'on dénote par $P_{1}$ et $P_{2}$. On répartit les $p-2$ autres en deux ensembles $A_{1},A_{2}$ tels que le cardinal de $A_{1}$ est $t-2$. On construit alors deux polygones en concaténant les racines des résidus de ces pôles et l'opposé de leurs sommes. Puisque $t<p$, l'ensemble $A_{2}$ n'est jamais vide (tandis que $A_{1}$ peut être vide).

On utilise alors les parties polaires données par le lemme~\ref{lem:nonintersecp3} où les $S_{i}$ sont les puissances $k$-ièmes des sommes des racines des $A_{i}$ (et $S_{1}=0$ si $A_{1}=\emptyset$), en normalisant de telle façon que le résidu de $P_{1}$ soit $1$ et en dénotant~$R$ le résidu de~$P_{2}$. De plus, on coupe chacune des parties polaires associées aux deux pôles spéciaux le long d'une demi-droite partant du sommet correspondant à $a_{1}$ et on y colle $\ell-2$ plans de façon cyclique. Finalement les autres pôles sont obtenus en collant les parties polaires aux segments correspondant des polygones.

Nous montrons maintenant que l'on obtient bien les $k$-différentielles avec les invariants souhaités. Considérons le cas où $A_{1}\neq \emptyset$ (le cas dans lequel $A_{1}$ est vide est similaire). L'angle de la singularité à laquelle contribue l'ensemble $A_{1}$ est $2\ell \cdot 2\pi + \frac{2\bar{a_{1}}\pi}{k} + (t-2)\cdot 2\pi$. En effet, les deux premiers termes sont donnés par le lemme~\ref{lem:nonintersecp3} en enlevant un angle~$\pi$. Cet angle $\pi$ et la contribution des autres pôles est donné  par le troisième terme. L'angle total est donc $t\ell \cdot 2\pi + \frac{2\bar{a_{1}}\pi}{k} $, ce qui correspond à une singularité d'ordre $a_{1}$ de la $k$-différentielle. L'ordre de la singularité $a_{2}$ se déduit immédiatement.
\end{proof}

\subsection{Cas général pour $n=2$}

Nous pouvons obtenir par récurrence les différentielles à résidus prescrits pour toute strate avec deux zéros.

\begin{prop}\label{prop:47p+0}
Soit $\Omega^{k}\mathcal{M}_{0}(a_{1},a_{2};-k\ell_{1},\dots,-k\ell_{p};\rec[-k][s])$ une strate de genre zéro telle que $p \neq 0$. Toute configuration de $k$-résidus qui n'est pas uniformément nulle appartient à l'image de l'application résiduelle.
\end{prop}

La preuve par récurrence nécessite un résultat préliminaire.

\begin{lem}\label{lem:49AUTRE}
Pour $s \geq 1$, toute configuration de $k$-résidus de la forme $(0;R_{1},\dots,R_{s})$ pour $R_{1},\dots,R_{s}$ tous non nuls est réalisable dans la strate $\Omega^{k}\mathcal{M}_{0}(a_{1},a_{2};-k\ell;\rec[-k][s])$.
\end{lem}

\begin{proof}
Pour démontrer le lemme~\ref{lem:47NONNUL}, nous n'utilisons pas l'hypothèse de non-nullité du résidu du pôle d'ordre $-k\ell$. En effet, elle n'intervient pas dans la construction des parties polaires (voir la section~\ref{sub:s4constructionsDP}). La construction s'étend donc immédiatement à cette famille de cas.
\end{proof}

Nous prouvons maintenant la proposition~\ref{prop:47p+0}.
\begin{proof}
Nous procédons par récurrence sur le nombre $t$ de pôles d'ordre distinct de~$-k$ pour lesquels le résidu est nul. Les lemmes de la section~\ref{sub:s4constructions} établissent l'initialisation (le cas $t=0$).
\par
En supposant la propriété valide jusqu'au rang $t$, nous allons la démontrer pour le cas~$t+1$. Nous considérons donc une configuration formée de $t+1$ résidus nuls et de $p+s-(t+1)$ résidus non nuls $R_{t+2},\dots,R_{p+s}$. Il s'agit de construire une différentielle ayant ces résidus dans une certaine strate $\Omega^{k}\mathcal{M}_{0}(a_{1},a_{2};-b_{1},\dots,-b_{p};\rec[-k][s])$.
\par
Le lemme~\ref{lem:44zeronegatif} permet de se restreindre aux strates pour lesquelles $a_{1},a_{2}>0$. Ainsi, il existe des entiers $m_{1},m_{2}>0$ tels que $m_{1}+m_{2}=l_{t+1}$, $a_{1}-km_{1}>-k$ et $a_{2}-km_{2}>-k$.
\par
% Supposons dans un premier temps que $p\geq2$ ou $t\geq2$.
L'hypothèse de récurrence ou le lemme~\ref{lem:49AUTRE} si $p=t=1$ garantissent que la configuration ayant juste un résidu nul de moins est réalisable par une différentielle $\omega$ d'une strate plus simple dans laquelle les deux zéros sont d'ordres $a_{1}-km_{1},a_{2}-km_{2}$ tandis le pôle d'ordre $kl_{t+1}$ a été retiré. Dans la surface plate définie par $\omega$, il existe un lien-selle reliant les deux zéros. Il suffit alors de découper le long de ce lien-selle pour y coller les deux bords d'une cicatrice découpée dans un domaine polaire d'ordre $kl_{t+1}$ et pour lesquelles les angles aux extrémités de la cicatrice sont $2m_{1}\pi$ et $2m_{2}\pi$. Cette chirurgie produit une nouvelle différentielle réalisant une configuration ayant un résidu nul de plus dans la strate adéquate.
% \par
% Dans le cas $p=t=1$, la démonstration est identique à ceci près que c'est le lemme~\ref{lem:49AUTRE} qui fournit la différentielle ayant un pôle de moins.
\end{proof}

\subsection{Cas général pour $n \geq 3$}

Pour conclure cette section, nous traitons le cas général des strates $\Omega^{k}\mathcal{M}_{0}(a_{1},\dots,a_{n};-b_{1},\dots,-b_{p};\rec[-k][s])$ pour lesquelles $p \neq 0$ et $n\geq3$.

\begin{prop}\label{prop:48p+0}
Soit $\Omega^{k}\mathcal{M}_{0}(a_{1},\dots,a_{n};-b_{1},\dots,-b_{p};\rec[-k][s])$ une strate de genre zéro telle que $p \neq 0$. L'image de l'application résiduelle est
 \begin{itemize}
 \item[i)] $\espresk[0](\mu)\setminus\left\{0\right\}$ si $s=0$, que seuls les ordres de deux zéro ne sont pas divisibles par $k$ et que la somme des ordres des zéros d'ordre divisible par $k$ est strictement inférieure à~$kp$;
 \item[ii)] $\espresk[0](\mu)$ sinon.
\end{itemize}  
\end{prop}

\begin{proof}
Le cas de la configuration uniformément nulle est traité dans la section~\ref{sub:41CUN}.
Pour toutes les autres configurations, la preuve se fait par éclatement de zéros à partir d'une strate de la forme $\Omega^{k}\mathcal{M}_{0}(\tilde a_{1},\tilde a_{2};-k\ell_{1},\dots,-k\ell_{p};\rec[-k][s])$ possédant les résidus souhaités. Ce cas a déjà été traité dans les sections~\ref{sec:expneg} et~\ref{sub:s4constructions}.
\par
Dans certains cas, il n'est pas possible d'obtenir une $k$-différentielle dans une strate en éclatant un zéro d'une $k$-différentielle primitive. Toutefois, grâce à la multiplicativité des $k$-résidus (voir l'équation~\eqref{eq:multiplires}) il suffit de vérifier que l'on peut partir d'une strate dont les éléments ne sont pas la puissance $k$-ième de différentielles abéliennes ou la puissance $\frac{k}{2}$-ième de différentielles quadratiques. C'est ce qui est démontré dans le lemme~\ref{lem:combiscindage}.
\end{proof}

\section{Les pôles sont d'ordre $-k$ en genre zéro}
\label{sec:juste-k}

Dans les strates ayant un unique zéro, toutes les singularités ont un ordre divisible par~$k$. Elles ne sont donc pas primitives (lemme~\ref{lem:puissk}).
Nous nous concentrerons sur le cas des strates avec deux zéros. Le cas des strates avec plus de zéros se traitera essentiellement par éclatement. Dans la section~\ref{sec:spora} nous prouvons les obstructions dans les cas sporadiques du théorème~\ref{thm:geq0kspe}. Puis nous construisons les $k$-différentielles avec deux zéros dans les sections~\ref{sub:5N2negatif} et~\ref{sub:mechant} puis avec $n\geq3$ dans la section~\ref{sec:constcaskn3}.

\subsection{Cas sporadiques}
\label{sec:spora}

Nous prouvons la non-réalisabilité des $k$-résidus dans les cas sporadiques du théorème~\ref{thm:geq0kspe}. Rappelons que les pôles d'ordre $-k$ ont pour voisinage un demi-cylindre bordé par un ou plusieurs liens selles.
\smallskip
\par
\paragraph{\bf L'obstruction (1).} Nous montrons que les résidus $(1,(-1)^{k})$ n'appartiennent pas à l'image de l'application résiduelle $\appresk[0](-1,1;-k,-k)$. Si une telle $k$-différentielle existait, alors la différentielle entrelacée obtenue en collant les deux pôles d'ordre $-k$ serait lissable par le lemme~\ref{lem:lisspolessimples}. La différentielle obtenue par lissage serait dans la strate $\Omega^{k}\mathcal{M}_{1}(1,-1)$, qui est vide.
\smallskip
\par
\paragraph{\bf Les autres obstructions.} Le schéma de la preuve des autres obstructions est dans chaque cas similaire. Nous supposerons par l'absurde qu'il existe une $k$-différentielle $\xi$ qui possède les invariants des obstructions. Nous considérons la surface plate associée. Nous lui retirons les cylindres correspondant aux pôles et la coupons le long du lien-selle $v$ le plus court reliant les deux zéros.\par
Dans chaque cas, nous donnerons la preuve que chaque cylindre a pour bord un unique lien-selle. Notre construction produit ainsi un polygone $\mathfrak{P}(\xi)$ avec $s+2$ côtés dénotés par $r_{1},\dots,r_{s},v_{1},v_{2}$, où $r_{i}$ correspond au $i$-ème pôle et les $v_{i}$ au lien-selle. Par la suite nous montrons qu'aucun polygone ainsi obtenu ne provient d'une surface ayant les propriétés désirées.

\subsubsection{\bf Cas sporadiques pour les différentielles quartiques}

Les obstructions pour ce cas reposent sur le fait que les liens-selles relient des entiers de Gauss entre eux, i.e. des éléments du réseau $\ZZ \oplus i \ZZ$.

\begin{lem}\label{lem:6reseaugauss}
Considérons une différentielle quartique primitive dans $\Omega^{4}\mathcal{M}_{0}(a_{1},a_{2};\rec[-4][s])$ réalisant une configuration de résidus $(r_{1}^{4},\dots,r_{s}^{4})$ tels que :
\begin{enumerate}
    \item chaque $r_{i}$ appartient au réseau $\ZZ \oplus i \ZZ$;
    \item la somme $\sum r_{i}$ appartient au sous-réseau $(1+i)\ZZ \oplus (1-i) \ZZ$.
\end{enumerate}
Les périodes des liens-selles (définies à multiplication près par $i$) appartiennent au réseau $\ZZ \oplus i \ZZ$. En particulier, tous les liens-selles sont de longueurs $\geq 1$.
\end{lem}

\begin{proof}
Étant donné un lien-selle fermé $\gamma$, il découpe la sphère en deux composantes. L'une de ces composantes ne contient que des singularités d'ordre $-4$ dans son intérieur. La période de $\gamma$ est donc une somme de racines $4$-ièmes de $4$-résidus. Il s'ensuit du point~(1) que la période de $\gamma$ appartient au réseau $\ZZ \oplus i \ZZ$.
\smallskip
\par
Étant donné un lien-selle $v$ reliant les deux singularités coniques. En coupant la sphère le long de $v$, on obtient une surface de translation avec deux liens-selles au bord dont les périodes $v_{1}$ et $v_{2}$ vérifient $v_{1}=\pm i v_{2}$.

Comme les résidus et les périodes de liens-selles de bord dans une surface de translation sont de somme nulle, on déduit du point (2) que les deux vecteurs $(1 \pm i)v_{1}$ appartiennent au sous-réseau $(1+i)\ZZ \oplus (1-i) \ZZ$. Il s'ensuit que $v_{1}$ (et donc $v_{2}$) appartient au réseau $\ZZ \oplus i \ZZ$.
\end{proof}

Pour des différentielles quartiques réalisant des configurations de $4$-résidus $(1,1,-4)$, $\rec[1][4]$ ou $\rec[1][6]$, en particulier pour une éventuelle différentielle réalisant une obstruction quartique, les demi-cylindres infinis ont des bords de longueur $1$ ou $\sqrt{2}$. En effet, la somme des racines $4$-ièmes des résidus est de la forme $a+ib$ avec $a=b \mod (2)$. Le point (2) du lemme~\ref{lem:6reseaugauss} est donc satisfait, et il implique que les liens-selles d'une telle différentielle bordant l'un de ces cylindres sont de longueur $1$ ou $\sqrt{2}$. Il s'ensuit que chacun de ces cylindres est bordé par un unique lien-selle (nécessairement fermé).
\smallskip
\par
\paragraph{\bf L'obstruction (8).}
Nous montrons que le $4$-résidu $(1,1,-4)$ n'est pas dans l'image de  l'application $\appresk[0][4](-1,5;\rec[-4][3])$.
\par
Supposons par l'absurde qu'il existe une $4$-différentielle $\xi$ ayant ces invariants locaux. Le polygone $\mathfrak{P}(\xi)$ est un pentagone dont les côtés ont pour périodes des entiers de Gauss. Nous distinguons deux cas selon que le zéro d'ordre $-1$ appartient au bord de l'un des cylindres ou d'aucun d'entre eux.
\par
Dans le premier cas, cela signifie que le zéro d'ordre $-1$ porte un angle total de $\frac{\pi}{2}$ dans le pentagone. Cet angle est réparti en deux secteurs angulaires de part et d'autre d'un côté de longueur~$1$ ou~$\sqrt{2}$. Comme les deux côtés adjacents à celui-ci sont de longueur au moins $1$, ils se coupent et nous obtenons une contradiction.
\par
Dans le second cas, le zéro d'ordre $-1$ porte un angle total de $\frac{3\pi}{2}$, porté par un unique secteur angulaire entre les deux côtés $v_{1}$ et $v_{2}$. Il s'agit d'un angle rentrant, il est donc nécessaire que la somme des longueurs des trois côtés $|r_{1}|+|r_{2}|+|r_{3}|$ soit strictement supérieure à la somme des longueurs de $v_{1}$ et $v_{2}$. Notons que comme $v_{1}+v_{2} + \sum r_{i}=0$ et les $v_{i}$ sont orthogonaux entre eux, on a  $|v_{1}|+|v_{2}| =|r_{1}+r_{2}+r_{3}|\sqrt{2}$. On obtient donc la contrainte $|r_{1}+r_{2}+r_{3}|<1+\sqrt{2}$. Dans toutes les configurations possibles pour $(r_{1},r_{2},r_{3})$ qui satisfont cette contrainte, $|r_{1}+r_{2}+r_{3}|$ vaut $\sqrt{2}$ ou $2$. Il s'ensuit que $v_{1}$ et $v_{2}$ sont de longueurs $1$ ou~$\sqrt{2}$. Dans le premier cas, $v_{1}$ et $v_{2}$ sont horizontaux ou verticaux. Le pentagone $\mathfrak{P}(\xi)$ a quatre côtés unitaire (horizontaux ou verticaux), un dernier côté diagonal de longueur $\sqrt{2}$, ayant de surcroît un angle de $\frac{3\pi}{2}$ sur l'un des cinq sommets. Il est facile de vérifier que l'on ne peut construire un tel pentagone.
\par
Dans le dernier cas, $v_{1}$ et $v_{2}$ sont de longueur $\sqrt{2}$. Le pentagone $\mathfrak{P}(\xi)$ a donc deux côtés unitaire (horizontaux ou verticaux) et trois côtés diagonaux de longueur $\sqrt{2}$, ayant toujours un angle de $\frac{3\pi}{2}$ sur l'un des cinq sommets. Un tel polygone est également impossible.
\smallskip
\par
\paragraph{\bf L'obstruction (10).}
Nous montrons que le $4$-résidu $(1,1,1,1)$ n'est pas dans l'image de  l'application $\appresk[0][4](-1,9;\rec[-4][4])$. Si une $4$-différentielle $\xi$ existait avec ces invariants, le polygone~$\mathfrak{P}(\xi)$ serait un hexagone. Nous séparons l'étude en deux cas selon que le zéro d'ordre $-1$ appartient au bord de l'un des cylindres ou d'aucun d'entre eux.
\par
Dans le premier cas, l'argument est identique à celui de l'obstruction (8).
\par
Dans le second cas, le zéro d'ordre $-1$ porte un angle total de $\frac{3\pi}{2}$, porté par un unique secteur angulaire entre les deux côtés $v_{1}$ et $v_{2}$. Il s'agit d'un angle rentrant, il est donc nécessaire que la somme des longueurs des quatre côtés soit strictement supérieure à la somme des longueurs de $v_{1}$ et $v_{2}$. Nous avons donc $|r_{1}+r_{2}+r_{3}+r_{4}|\sqrt{2}<4$. Dans toutes les configurations possibles pour $(r_{1},r_{2},r_{3},r_{4})$ satisfaisant cette condition, $|r_{1}+r_{2}+r_{3}+r_{4}|$ vaut~$\sqrt{2}$ ou $2$.
\par
Si $|r_{1}+r_{2}+r_{3}+r_{4}|=\sqrt{2}$, cela signifie que $v_{1}$ et $v_{2}$ sont de longueur $1$. L'hexagone a donc six côtés unitaires horizontaux ou verticaux, avec un angle (entre $v_{1}$ et $v_{2}$) valant $\frac{3\pi}{2}$. Les cinq autre angles doivent donc valoir $\frac{\pi}{2}$. Un tel hexagone est manifestement impossible.
\par
Si $|r_{1}+r_{2}+r_{3}+r_{4}|=2$, $v_{1}$ et $v_{2}$ sont donc des segments diagonaux de longueur $\sqrt{2}$, tandis que les quatre autres côtés, horizontaux ou verticaux, sont de longueur $1$. Sachant que l'angle entre $v_{1}$ et $v_{2}$ vaut $\frac{3\pi}{2}$, un tel hexagone ne peut pas non plus exister.
\smallskip
\par
\paragraph{\bf L'obstruction (9).}
Nous montrons que le $4$-résidu $(1,1,1,1)$ n'est pas dans l'image de  l'application $\appresk[0][4](3,5;\rec[-4][4])$. Si une $4$-différentielle $\xi$ existait avec ces invariants, le polygone $\mathfrak{P}(\xi)$ serait encore un hexagone. Nous distinguons les cas selon le nombre de cylindres bordés par chaque zéro.
\par
Si l'un des zéros borde les quatre cylindres, alors il s'agit du zéro d'ordre $5$, à qui il ne restera que $\frac{\pi}{2}$ à repartir entre cinq angles de l'hexagone. Comme au moins trois d'entre eux séparent des bords de cylindres, ils forment chacun un angle qui est un multiple de $\frac{\pi}{2}$. La contradiction est immédiate.
\par
Si l'un des zéros borde trois cylindres, alors l'ordre des côtés est $(v_{1},r_{1},v_{2},r_{2},r_{3},r_{4})$. Les trois cylindres collés à $r_{2},r_{3}$ et $r_{4}$ contribuent d'un angle $3\pi$ à ce zéro. Il reste donc $\frac{\pi}{2}$ ou $\frac{3\pi}{2}$ à repartir entre les quatre angles aux bords de ces $r_{i}$. Comme les angles au bord de $r_{3}$ sont $\geq \frac{\pi}{2}$, le zéro est d'ordre~$5$.  Comme la somme des angles entre $r_{4}$ et $v_{1}$ puis $v_{2}$ et $r_{2}$ vaut $\frac{\pi}{2}$, les segments $v_{1}$ et $v_{2}$ (de longueur au moins $1$ par le lemme~\ref{lem:6reseaugauss}) se coupent. Cela montre que cette configuration est impossible.
\par
Si chaque zéro borde deux cylindres, alors l'hexagone est $(v_{1},r_{1},r_{2},v_{2},r_{3},r_{4})$. La somme des angles de l'hexagone aux bords de $r_{1},r_{2}$ et $r_{3},r_{4}$ sont respectivement $\frac{3\pi}{2}$ et $\frac{5\pi}{2}$. Le zéro d'ordre $3$ porte un angle entre deux bords de cylindres, disons $r_{1},r_{2}$, qui est d'angle multiple de $\frac{\pi}{2}$. Si cet angle est $\frac{\pi}{2}$, alors la somme des angles entre $v_{1},r_{1}$ et $v_{2},r_{2}$ est $\pi$. Comme les $v_{i}$ sont des entiers de Gauss, soit ils se coupent, soit ils coupent $r_{1}$ ou $r_{2}$. Si l'angle entre $r_{1}$ et $r_{2}$ est $\pi$, alors la somme des angles entre $v_{1},r_{1}$ et $v_{2},r_{2}$ est $\frac{\pi}{2}$. De même, le fait que les $v_{i}$ soient des entiers de Gauss implique qu'ils se coupent ou coupent les arêtes $r_{i}$.

\smallskip
\par
\paragraph{\bf L'obstruction (11).}
Nous montrons que le $4$-résidu $(1,1,1,1,1,1)$ n'est pas dans l'image de l'application $\appresk[0][4](3,13;\rec[-4][6])$. Si une $4$-différentielle $\xi$ existait avec ces invariants, le polygone~$\mathfrak{P}(\xi)$ serait un octogone. Nous distinguons encore les cas selon le nombre de cylindres bordés par chaque singularité conique, dont les angles sont respectivement $\frac{7\pi}{2}$ et $\frac{17\pi}{2}$. L'angle entre deux bords de cylindres est toujours un multiple de $\frac{\pi}{2}$. Un zéro adjacent à $x$ cylindres a donc un angle strictement plus grand que $\frac{(3x-1)\pi}{2}$ ($\pi$ pour chacun des $x$ cylindres et $\frac{\pi}{2}$ pour les $x-1$ angles entre deux bords de cylindres). Il s'ensuit que le zéro d'ordre $3$ est adjacent à deux cylindres au maximum et un cylindre au minimum.
\par
Supposons d'abord que le zéro d'ordre $3$ est adjacent à deux cylindres. Dans l'octogone, il porte un angle entre deux bords de cylindres, donc d'angle multiple de $\frac{\pi}{2}$. Si celui-ci est $\frac{\pi}{2}$ (resp. $\pi$), alors la somme des deux angles est $\pi$ (resp. $\frac{\pi}{2}$). Dans les deux cas, cela conduit $v_{1}$ et $v_{2}$ à se couper.
\par
Dans le dernier cas, le zéro d'ordre $3$ est adjacent à un seul cylindre (de bord $r_{1}$). Seuls deux angles (disons entre $r_{1}$ d'une part et $v_{1},v_{2}$ d'autre part) de l'octogone contribuent à ce zéro, pour un angle total de $\frac{5\pi}{2}$. Au contraire, six angles de l'octogone contribuent au zéro d'ordre $13$, pour un angle total de $\frac{7\pi}{2}$.
\par
Notons $\theta_{1},\dots,\theta_{6}$ ces six angles (dans cet ordre). Comme les côtés entre ces angles sont de longueur $1$ et relient des entiers de Gauss, il est impossible que trois angles consécutifs $\theta_{i},\theta_{i+1},\theta_{i+2}$ satisfassent 
$\theta_{i},\theta_{i+1},\theta_{i+2} \leq \frac{\pi}{2}$. Cependant, les quatre angles $\theta_{2},\theta_{3},\theta_{4},\theta_{5}$ sont entre deux bords de cylindres et sont donc des multiples de $\frac{\pi}{2}$. La somme $\theta_{1}+\theta_{6}$ vaut donc $\frac{\pi}{2}$ ou~$\pi$ (le cas $\frac{3\pi}{2}$ est exclu car nous aurions alors $\theta_{2}=\theta_{3}=\theta_{4}=\theta_{5}=\frac{\pi}{2}$).
\par
Si $\theta_{1}+\theta_{6}=\pi$, alors exactement un angle parmi $\theta_{2},\theta_{3},\theta_{4},\theta_{5}$ vaut $\pi$ tandis que les trois autres valent $\frac{\pi}{2}$. Quitte à renverser l'ordre, le seul choix possible est $\theta_{1}>\frac{\pi}{2}$ et $\theta_{4}=\pi$. Nous obtenons alors un polygone singulier.
\par
Si dans le dernier cas, $\theta_{1}+\theta_{6}=\frac{\pi}{2}$, alors la condition sur les trois angles consécutifs forcent le choix $\theta_{2}=\theta_{5}=\frac{\pi}{2}$ tandis que $\theta_{3}=\theta_{4}=\pi$. Comme $\theta_{1},\theta_{2}<\frac{\pi}{2}$, nous obtenons encore un polygone singulier.

\subsubsection{\bf Cas sporadiques pour les différentielles cubiques}

Les obstructions pour ce cas reposent sur le fait que les liens selles relient des entiers d'Eisenstein entre eux, i.e. des éléments du réseau $\ZZ \oplus \omega\ZZ$ où $\omega=e^{\frac{2i\pi}{3}}$.

\begin{lem}\label{lem:6reseaucubique}
Considérons une différentielle cubique primitive d'une strate $\Omega^{3}\mathcal{M}_{0}(a_{1},a_{2};\rec[-3][s])$ réalisant une configuration de résidus $(r_{1}^{3},\dots,r_{s}^{3})$ tels que :
\begin{enumerate}
\item chaque $r_{i}$ appartient au réseau $\ZZ \oplus \omega\ZZ$;
\item la somme $\sum r_{i}$ appartient au sous-réseau $3\ZZ \oplus (-1+\omega)\ZZ$;
\end{enumerate}
Les périodes des liens-selles (définies a multiplication près par une racine troisième de l'unité) appartiennent au réseau $\ZZ \oplus \omega\ZZ$. En particulier, tous les liens-selles sont de longueur $\geq 1$.
\end{lem}

\begin{proof}
Étant donné un lien-selle fermé $\gamma$, il découpe la sphère en deux composantes. L'une de ces composantes ne contient que des singularités d'ordre $-3$ dans son intérieur. La période de $\gamma$ est donc une somme de racines troisièmes de $3$-résidus. Il s'ensuit du point~(1) que la période de $\gamma$ appartient au réseau $\ZZ \oplus \omega \ZZ$.
\smallskip
\par
Étant donné un lien-selle $v$ reliant les deux singularités coniques. En coupant la sphère le long de $v$, on obtient une surface de translation avec deux liens-selles au bord dont les périodes $v_{1}$ et $v_{2}$ vérifient $v_{2}=-\lambda v_{1}$ avec $\lambda$ valant $\omega$ ou $\omega^{2}=-1-\omega$.
\par
Comme les résidus et les périodes de liens-selles de bord dans une surface de translation sont de somme nulle, la condition (2) implique que $(1-\lambda)v_{1}$ appartient au sous-réseau $3\ZZ \oplus (-1+\omega)\ZZ$. Comme $1-\lambda$ vaut $1-\omega$ ou $2+\omega$, il s'ensuit que $v_{1}$ (et donc $v_{2}$) appartient au réseau $\ZZ \oplus \omega\ZZ$.
\end{proof}

Deux racines troisièmes de l'unité (mais aussi de $-1$) ont pour différence un élément du sous-réseau $3\ZZ \oplus (-1+\omega)\ZZ$. Ainsi, qu'une somme de racines troisièmes de $1$ ou $-1$ appartienne ou non au sous-réseau $3\ZZ \oplus (-1+\omega)\ZZ$ ne dépend pas du choix des racines.
\par
Étant donnée une différentielle cubique dont les $3$-résidus sont $(\rec[1][s_{1}],\rec[-1][s_{2}])$ avec $s_{1}+s_{2}=s$, il s'ensuit que pour tout choix de racines $r_{1},\dots,r_{s}$, la somme $\sum r_{i}$ appartient au sous-réseau $3\ZZ \oplus (-1+\omega)\ZZ$ si et seulement si $s_{1}-s_{2} \in 3\ZZ$ (ce qui revient à choisir $1$ pour les racines de~$1$ et $-1$ pour les racines de $-1$).
\par
Ainsi, le lemme~\ref{lem:6reseaucubique} établit que les liens-selles d'une telle différentielle bordant l'un de ces cylindres sont de longueur $1$. Il s'ensuit que chacun de ces cylindres est bordé par un unique lien-selle (nécessairement fermé).

\begin{lem}\label{lem:minangle}
Supposons que les arêtes de $\mathfrak{P}(\xi)$ sont  $(v_{1},r_{1},\dots,r_{\ell},v_{2},r_{\ell +1},\dots,r_{s})$. La singularité au bord des $\ell$ premiers cylindres possède un angle strictement plus grand que $\frac{(4\ell -1)\pi}{3}$.
\end{lem}

\begin{proof}
L'angle est donné par $\pi$ pour les $\ell$ cylindres et au moins $\frac{\pi}{3}$ pour chacun des $l-1$ angles entre deux bords de cylindres (car leurs périodes sont des racines cubiques de~$\pm1$).
\end{proof}

Une \textit{chirurgie} va permettre de simplifier la preuve de certaines obstructions. Celle-ci sera utilisée lors des construction et est représentée dans la figure~\ref{fig:transfork3}. 
\begin{figure}[hbt]
\centering
\begin{tikzpicture}
%premiere construction
\begin{scope}[xshift=3cm]
\filldraw[fill=black!10]  (0,0) coordinate (p1)  -- ++(120:1)  coordinate[pos=.5] (q1)  coordinate (p2) -- ++(1,0) coordinate[pos=.5] (q2) coordinate (p3) -- ++(1,0) coordinate[pos=.5] (q3) coordinate (p4)  -- ++(-120:1)  coordinate[pos=.5] (q4)  coordinate (p5) -- ++(150:.577) coordinate[pos=.5] (q5) coordinate (p6)-- ++(-150:.577) coordinate[pos=.5] (q6) ;

  \foreach \i in {1,...,5}
  \fill (p\i)  circle (2pt);
  \filldraw[fill=white] (p6)  circle (2pt);

  \node[below] at (q5) {$v_{1}$};
  \node[below] at (q6) {$v_{2}$};

  \node[left] at (q1) {$1$};
\node[above] at (q2) {$1$};
\node[above] at (q3) {$1$};
\node[right] at (q4) {$1$};
\end{scope}

\begin{scope}[xshift=-3cm]
\filldraw[fill=black!10]  (0,0) coordinate (p1)  -- ++(120:1)  coordinate[pos=.5] (q1)  coordinate (p2) -- ++(1,0) coordinate[pos=.5] (q2) coordinate (p3) -- ++(60:1) coordinate[pos=.5] (q3) coordinate (p4) -- ++(-60:1) coordinate[pos=.5] (q4) coordinate (p5)  -- ++(-120:1)  coordinate[pos=.5] (q5)  coordinate (p6) -- ++(150:.577) coordinate[pos=.5] (q6) coordinate (p7)-- ++(-150:.577) coordinate[pos=.5] (q7) ;
\draw[dotted] (p3) -- (p5);

  \foreach \i in {1,...,6}
  \fill (p\i)  circle (2pt);
  \filldraw[fill=white] (p7)  circle (2pt);

  \node[below] at (q6) {$v_{1}$};
  \node[below] at (q7) {$v_{2}$};

  \node[left] at (q1) {$1$};
\node[above] at (q2) {$1$};
\node[left] at (q3) {$-1$};
\node[right] at (q4) {$-1$};
\node[right] at (q5) {$1$};
\end{scope}
\end{tikzpicture}
\caption{Chirurgie remplaçant deux pôles par un unique pôle}\label{fig:transfork3}
\end{figure}
En effet, si dans le polygone, deux bords de cylindres (de longueur $1$) forment un angle $\frac{\pi}{3}$, nous pouvons retirer le triangle équilatéral qu'ils forment, ainsi que les deux cylindres, pour les remplacer par un cylindre de résidu cubique opposé. La chirurgie supprime un pôle triple et réduit l'ordre du zéro concerné de $3$.
\smallskip
\par
\paragraph{\bf L'obstruction (2).}

Nous prouvons ici que la configuration $(1,1,1)$ n'est pas dans l'image de l'application $\appresk[0][3](-1,4;\rec[-3][3])$. Si une $3$-différentielle $\xi$ existait avec ces invariants, le polygone~$\mathfrak{P}(\xi)$ serait un pentagone. Le lemme~\ref{lem:minangle} prouve que le zéro d'ordre $-1$ borde au plus un cylindre.
\par
Supposons le zéro d'ordre $-1$ borde exactement un cylindre. Dans ce cas, de part et d'autre d'un côté de longueur $1$ nous trouvons deux angles ayant pour somme $\frac{\pi}{3}$. Ceci est impossible car les segments du pentagone relient des entiers d'Eisenstein.
\par
Si au contraire le zéro d'ordre $4$ borde les trois cylindres, alors il porte quatre angles du pentagone, pour un total de $\frac{5\pi}{3}$. Notons $\theta_{1},\theta_{2},\theta_{3},\theta_{4}$ ces angles. Chaque paire d'angles consécutifs $\theta_{i},\theta_{i+1}$ parmi eux borde un côté de longueur $1$. Il est donc impossible que ces deux angles satisfassent simultanément $\theta_{i},\theta_{i+1} \leq \frac{\pi}{3}$ (car les segments relient des entiers d'Eisenstein). Il s'ensuit que l'un des deux angles parmi $\theta_{2}$ et $\theta_{3}$ (séparant deux bords de cylindres, leur angle est ici $\frac{\pi}{3}$ ou $\pi$), l'un des deux vaut $\pi$ (disons $\theta_{2}$). Il s'ensuit que $\theta_{1}+\theta_{4}=\frac{\pi}{3}$ et que $\theta_{3},\theta_{4} \leq \frac{\pi}{3}$. Nous obtenons une nouvelle contradiction. 
\smallskip
\par
\paragraph{\bf L'obstruction (3).}

La configuration $(1,1,1)$ n'est pas dans l'image de $\appresk[0][3](1,2;\rec[-3][3])$. En effet, si une $3$-différentielle $\xi$ existait avec ces invariants, le polygone $\mathfrak{P}(\xi)$ serait un pentagone. Le lemme~\ref{lem:minangle} prouve que le zéro d'ordre $1$ borde un cylindre tandis que le zéro d'ordre~$2$ borde les deux autres.
\par
Le zéro d'ordre $2$ porte trois angles du pentagone, pour un angle total de $\frac{4\pi}{3}$. L'angle central parmi ces trois se trouve entre deux bords de cylindres. Sa valeur est donc $\frac{\pi}{3}$ ou~$\pi$ (tandis que les deux autres angles ont pour somme $\pi$ ou $\frac{\pi}{3}$). Dans les deux cas, nous avons une intersection entre $v_{1}$ et $v_{2}$ (ici encore, le fait que les sommets soient des entiers d'Eisenstein est crucial).
\smallskip
\par
\paragraph{\bf L'obstruction (4).}

Nous prouvons ici que la configuration $(1,1,-1,-1)$ n'est pas dans l'image de l'application $\appresk[0][3](2,4;\rec[-3][4])$. Si une différentielle cubique $\xi$ existait avec ces invariants, le polygone~$\mathfrak{P}(\xi)$ serait un hexagone. Le lemme~\ref{lem:minangle} prouve que le zéro d'ordre $4$ est au bord de deux ou trois cylindres.
\par
S'il est au bord de trois cylindres, alors il porte quatre angles dans l'hexagone pour un total de $\frac{5\pi}{3}$. Si l'un des deux angles entre des bords de cylindres vaut $\frac{\pi}{3}$, la chirurgie décrite plus haut permet d'obtenir une différentielle réalisant la configuration $(1,1,1)$ dans la strate $\Omega^{3}\mathcal{M}_{0}(1,2;\rec[-3][3])$. Donc sur les quatre angles consécutifs $\theta_{1},\theta_{2},\theta_{3},\theta_{4}$, nous avons $\theta_{2}=\theta_{3}=\frac{2\pi}{3}$ tandis que $\theta_{1}+\theta_{4} = \frac{\pi}{3}$. Dans une telle figure, les segments $v_{1}$ et $v_{2}$ se coupent.
\par
Dans le second cas, chaque zéro est au bord de deux cylindres. En particulier, le zéro d'ordre $2$ porte trois angles dans l'hexagone, pour un total de $\frac{4\pi}{3}$. Que les deux cylindres aient pour période tous les deux une racine de $1$ (ou $-1$), ou bien que l'un ait pour période une racine de $1$ et l'autre de $-1$, tout choix d'angle conduit à une intersection entre $v_{1}$ et $v_{2}$. 
\smallskip
\par
\paragraph{\bf L'obstruction (5).}
Montrons qu'il n'existe pas de $3$-différentielle dans $\Omega^{3}\mathcal{M}_{0}(2,7;\rec[-3][5])$ dont les $3$-résidus sont $(\rec[1][4],-1)$. Si une $3$-différentielle $\xi$ existait avec ces invariants, le polygone~$\mathfrak{P}(\xi)$ serait un heptagone. Le lemme~\ref{lem:minangle} prouve que le zéro d'ordre $2$ est au bord d'au plus deux cylindres. S'il porte exactement deux cylindres, le même argument que pour l'obstruction~(4) permet de conclure.
\par
Si le zéro d'ordre $2$ borde un seul cylindre, alors le zéro d'ordre $7$ en borde quatre. Il porte cinq angles de l'heptagone pour un total de $\frac{8\pi}{3}$. Sur les trois angles entre les quatre bords de cylindres consécutifs (tous de longueur $1$), aucun ne vaut $\frac{\pi}{3}$, sinon la chirurgie permet de retrouver une différentielle réalisant la configuration $(1,1,-1,-1)$ dans la strate $\Omega^{3}\mathcal{M}_{0}(2,4;\rec[-3][4])$. Ces trois angles ne peuvent valoir tous $\frac{2\pi}{3}$ car nous aurions alors au moins deux résidus cubiques valant $1$ et deux autres valant $-1$. Il s'ensuit que les trois angles sont $\frac{2\pi}{3}$ pour deux d'entre eux et $\pi$ pour le dernier. Les deux autres angles se partagent $\frac{\pi}{3}$, ce qui est impossible puisque les segments relient des sommets d'Eisenstein.
\par
Enfin, si le zéro d'ordre $2$ ne borde aucun cylindre, cela signifie que le zéro d'ordre $7$ borde les cinq. Il porte six angles de l'heptagone pour un total de $\frac{5\pi}{3}$. Sur les quatre angles entre deux bords de cylindres, au moins l'un d'entre eux est entre un bord de cylindre de résidu cubique $1$ et un autre de résidu cubique $-1$. Ces quatre angles ont une somme au moins égale à $\frac{5\pi}{3}$, ce qui est impossible.
\smallskip
\par
\paragraph{\bf L'obstruction (6).}
Nous montrons que les $3$-résidus $\rec[1][6]$ ne sont pas réalisés dans la strate $\Omega^{3}\mathcal{M}_{0}(2,10;\rec[-3][6])$. Si une $3$-différentielle $\xi$ existait avec ces invariants, le polygone $\mathfrak{P}(\xi)$ serait un octogone.
\par
Si le zéro d'ordre $10$ est au bord de $x$ cylindres, alors il porte $x+1$ angles de l'octogone, dont $x-1$ se situent entre deux bords de cylindres. Si l'un de ces angles vaut $\frac{\pi}{3}$, alors la chirurgie produit un contre-exemple pour l'obstruction (5). Ces angles valent donc au moins~$\pi$. Il s'ensuit que l'angle total pour ce zéro doit être strictement plus grand que $(2x-1)\pi$. Il s'agit d'une singularité conique d'angle $\frac{26\pi}{3}$. Il ne peut donc pas se trouver au bord de cinq cylindres ou plus.
\par
Le lemme~\ref{lem:minangle} prouve enfin que le zéro d'ordre $2$ est au bord de deux cylindres. Ce zéro porte donc trois angles de l'octogone, pour un total de $\frac{4\pi}{3}$. Comme pour les obstructions (4) et (5), nous obtenons tout de suite une intersection entre $v_{1}$ et $v_{2}$.
\smallskip
\par
\paragraph{\bf L'obstruction (7).} 
Nous montrons donc que les $3$-résidus $\rec[1][6]$ ne sont pas réalisés dans $\Omega^{3}\mathcal{M}_{0}(5,7;\rec[-3][6])$. Si une $3$-différentielle $\xi$ existait avec ces invariants, le polygone $\mathfrak{P}(\xi)$ serait un octogone.
\par
Si le zéro d'ordre $5$ est au bord de $x$ cylindres, alors il porte $x+1$ angles de l'octogone, dont $x-1$ se situent entre deux bords de cylindres. Si l'un de ces angles vaut $\frac{\pi}{3}$, alors la chirurgie produit un contre-exemple pour l'obstruction (5). Ces angles valent donc au moins~$\pi$. Il s'ensuit que l'angle total pour ce zéro doit être strictement plus grand que $(2x-1)\pi$. Il s'agit d'une singularité conique d'angle $\frac{16\pi}{3}$. Il ne peut donc pas se trouver au bord de quatre cylindres ou plus.
\par
Pour le zéro d'ordre $7$ au bord de $x$ cylindres, les angles entre les bords de cylindres doivent valoir au moins $\frac{\pi}{3}$ (pas de chirurgie ici). La contrainte est donc que l'angle soit strictement plus grand que $\frac{(4x-1)\pi}{3}$. Il est donc impossible d'avoir $x=6$.
\par
Si le zéro d'ordre $7$ au bord de cinq cylindres, alors les quatre angles de l'octogone entre leurs bord doivent valoir $\frac{\pi}{3}$. Or, on ne peut pas avoir deux angles de $\frac{\pi}{3}$ consécutifs de part et d'autre d'un segment de longueur $1$. Cet argument montre que si le zéro d'ordre $7$ est au bord de quatre cylindres, sur les trois angles entre leurs bords (dont la somme est $\frac{8\pi}{3}$), un ou deux d'entre eux doivent valoir $\pi$. Cependant, dans les deux cas, nous avons une intersection entre $v_{1}$ et $v_{2}$.
\par
Dans le dernier cas, chaque zéro est au bord de trois cylindres. En particulier, le zéro d'ordre $5$ porte quatre angles de l'octogone (pour un total de $\frac{7\pi}{3}$). Sur les deux angles entre deux bords de cylindres, on ne peut pas avoir d'angle $\frac{\pi}{3}$. Ces deux angles valent donc $\pi$ tandis que les deux autres ont pour somme $\frac{\pi}{3}$. Comme les segments relient des entiers d'Eisenstein, les segments~$v_{1}$ et~$v_{2}$ se coupent nécessairement.

\subsubsection{\bf Cas sporadiques pour les différentielles sextiques}

Il s'agit enfin de démontrer les obstructions (12) et (13), ce qui revient à prouver que le
$6$-résidu $(1,\dots,1)$ n'est pas dans l'image des applications $\appresk[0][6](-1,6s-11;\rec[-6][s])$ pour $s$ valant $3$ ou $4$. Supposons par l'absurde qu'il existe une telle différentielle $\xi$. Le polygone $\mathfrak{P}(\xi)$ est un pentagone ou un hexagone.
\par
Le lemme~\ref{lem:6reseaucubique} s'adapte avec quelques modifications. Toute somme de racines $6$-ièmes de l'unité appartient au réseau $\ZZ \oplus \omega\ZZ$. De plus, les deux vecteurs $v_{1}$ et $v_{2}$ sont tels que $v_{2}=-\lambda v_{1}$ (avec $\lambda$ valant $-\omega$ ou $-\omega^{2}$) et $(1-\lambda)v_{1} \in \ZZ \oplus \omega\ZZ$. Il s'ensuit que $v_{1}$ appartient aussi à ce réseau.
\par
Ainsi les périodes des liens-selles de $\xi$ appartiennent au réseau $\ZZ \oplus \omega\ZZ$. Tous les liens-selles ayant une longueur d'au moins $1$, chacun des cylindres a pour bord un unique lien-selle fermé de longueur $1$. Nous distinguons deux cas selon si le zéro d'ordre $-1$, c'est-à-dire la singularité conique d'angle $\frac{5\pi}{3}$, appartient au bord d'un cylindre ou d'aucun.
\par
Dans le cas où il borde un cylindre, le zéro d'ordre $-1$ porte un angle total de $\frac{2\pi}{3}$ dans le polygone. Celui-ci est réparti en deux secteurs angulaires de part et d'autre d'un côté de longueur $1$. Comme les deux côtés adjacents à celui-ci sont de longueur au moins $1$ et relient des entiers d'Eisenstein, ils se coupent. Ce cas est donc impossible.
\par
Dans le cas où le zéro d'ordre $-1$ ne borde aucun cylindre, le zéro d'ordre $6s-11$ borde les $s$ cylindres et les $s+1$ angles du polygone qu'il porte ont pour somme $\frac{(3s-5)\pi}{3}$.
\par
Si $s=3$, l'un de ces angles vaut $\frac{\pi}{3}$ et se trouve entre deux bords de cylindres (de longueur $1$). La chirurgie décrite pour les différentielles cubiques permet alors (en retirant un triangle équilatéral et ces deux cylindres) d'obtenir une différentielle sextique de la strate $\Omega^{6}\mathcal{M}_{0}(-1,1;-6,-6)$ dont les résidus seraient $(1,1)$ contredisant ainsi l'obstruction (1).
\par
Si $s=4$, quand l'un des trois angles entre les bords de cylindres vaut $\frac{\pi}{3}$, nous utilisons la même chirurgie pour se ramener au cas $s=3$. Si aucun de ces angles ne vaut $\frac{\pi}{3}$, alors ils valent tous $\frac{2\pi}{3}$ et les deux derniers angles ont pour somme $\frac{\pi}{3}$, ce qui est impossible en reliant des entiers d'Eisenstein.

\subsection{Constructions pour $n=2$ avec un zéro d'ordre négatif}\label{sub:5N2negatif}

Nous considérons des strates $\Omega^{k}\mathcal{M}_{0}(\mu)$ avec $\mu:=(a_{1},a_{2};\rec[-k][s])$ et dont les $k$-résidus sont $(R_{1},\dots,R_{s})$.
Nous supposerons que les $a_{i}=kl_{i}+\bar{a_{i}}$ vérifient $-k<\bar{a_{1}} < -k/2 < \bar{a_{2}}<0$. En effet, la condition $\pgcd(a_{1},a_{2},k)=1$ implique que le cas $\bar{a_{1}}=\bar{a_{2}}=-k/2$ est impossible. De plus, on a l'égalité $l_{1}+l_{2}+1=s$.
\par
Dans cette section, nous supposerons que l'un des zéros est d'ordre négatif. Autrement dit, nous avons soit $l_{1}=0$ soit $l_{2}=0$. Le cas $l_{1}=0$ est particulièrement simple.

\begin{lem}\label{lem:S5l1=0}
Pour toute strate $\Omega^{k}\mathcal{M}_{0}(a_{1},a_{2};\rec[-k][s])$ avec $-k <a_{1}< -\frac{k}{2}$, l'image de l'application résiduelle est surjective.
\end{lem}

\begin{proof}
Pour toute configuration $(R_{1},\dots,R_{s})$, nous choisissons des racines $k$-ièmes $r_{1},\dots,r_{s}$ qui ne sont pas toutes $\mathbb{R}$-colinéaires et dont la somme $\sigma=\sum r_{i}$ est non nulle (comme $k \geq 3$, il suffit de ne choisir que des racines dont la partie réelle est strictement positive).
\par
Nous formons ensuite le polygone convexe à $s+1$ côtés classés selon l'ordre cyclique des arguments des périodes $r_{1},\dots,r_{s},-\sigma$. Sur chaque bord de période $r_{i}$ nous collons un cylindre tandis que sur le bord de période $-\sigma$ nous collons un triangle isocèle dont les deux autres côtés sont identifiés entre eux. En prescrivant un angle $\frac{a_{1}+k}{k}2\pi<\pi$ au sommet opposé à la base, nous obtenons une surface correspondant à une différentielle ayant les bons résidus appartenant à la strate voulue.
\end{proof}

Le cas $l_{2}=0$ est plus délicat. Il exige plusieurs constructions distinctes selon les valeurs des $k$-résidus. Nous les donnons d'abord dans le cas $s=2$ pour lequel il n'y a que deux pôles d'ordre $-k$.

\subsubsection{\bf Cas $s=2$}

Les strates $\Omega^{k}\mathcal{M}_{0}(k+\bar{a_{1}},\bar{a_{2}};-k,-k)$ constituent un cas spécifique pour lequel nous donnons des constructions explicites. Il y a exactement trois modèles géométriques pour de telles différentielles (voir figure~\ref{fig:modeles}) :
\begin{itemize}
    \item \textbf{Triangle}: les deux cylindres sont collés sur un triangle (un bord pour un cylindre, deux bords pour l'autre);
    \item \textbf{Quadrilatère AABB}: les deux cylindres sont collés sur les côtés adjacents d'un quadrilatère;
    \item \textbf{Quadrilatère ABAB}: les deux cylindres sont collés sur les côtés opposés d'un quadrilatère.
\end{itemize}

\begin{figure}[htb]
\center
 \begin{tikzpicture}[scale=2,decoration={
    markings,
    mark=at position 0.5 with {\arrow[very thick]{>}}}]

  \begin{scope}[]
  \coordinate (a1) at (0,0); \node[below left] at (a1) {$0$};
\coordinate (a2) at (1,0); \node[right] at (a2) {$1$};
\coordinate (a3) at (.3,0); \node[above] at (a3) {$\theta$};
\coordinate (a4) at (-.3,.6);

\draw[dotted] (a1) --(a3); \draw (a3) --node[sloped] {$|$} (a4);\draw (a3) --node[] {$|$} (a2);
   \draw[postaction={decorate}] (a1) --node[left] {$r_{1}$} (a4) node[above] {$z$}; \draw[postaction={decorate}] (1,-.2) --node[below] {$r_{2}$} (0,-.2);

\draw (a1) -- ++(-90:.6);
\draw (a2) -- ++(-90:.6);
\draw (a1) -- ++(-160:.6);
\draw (a4) -- ++(-160:.6);

\draw[dashed] (a2)  arc  (0:180:1) -- (a1);
\draw[dotted] (a2) -- (130.5:1) coordinate (b);
\node[above] at (b) {$e^{2ia\pi/k}$};
 \filldraw[fill=red] (b) circle (1pt);

 \foreach \i in {1,2,4}
 \filldraw[fill=white] (a\i) circle (1pt);
 \fill (a3) circle (1pt);
 
 \node at (0,-1) {Triangle};
 \end{scope}

 %Second 
 
 \begin{scope}[xshift=4cm,yshift=.5cm]
  \coordinate (a1) at (-1,0); \node[below] at (a1) {$-1$};
\coordinate (a2) at (0,0); \node[above] at (a2) {$0$};
\coordinate (a3) at (0,-1); \node[below left] at (a3) {$-e^{2ia\pi/k}$};
\coordinate (a4) at (.8,.8);\node[above right] at (a4) {$z$};

\draw (a1) --node[sloped] {$|$} (a2); \draw (a2) --node[sloped] {$|$} (a3);
   \draw[postaction={decorate}] (a1) --node[above left] {$r_{1}$} (a4); \draw[postaction={decorate}] (a4) --node[right] {$r_{2}$} (a3);
\draw[->] (-.2,0)  arc  (-180:-90:.2); \node at (-.3,-.3) {$\frac{2a\pi}{k}$};

\draw (a1) -- ++(115:.6);
\draw (a4) -- ++(115:.6);
\draw (a4) -- ++(-25:.6);
\draw (a3) -- ++(-25:.6);

 \foreach \i in {1,3,4}
 \filldraw[fill=white] (a\i) circle (1pt);
 \fill (a2) circle (1pt);
 
 \node at (0,-1.5) {Quadrilatère AABB};
 \end{scope}

 %troisieme 
 
 %Second 
 
 \begin{scope}[xshift=1.9cm,yshift=-2.5cm]
\coordinate (a1) at (0,0);
\coordinate (a2) at (1,0);
\coordinate (a3) at (60:1);
\coordinate (a4) at (0,1);
   \draw (a1) --node[sloped] {$|$} (a4);
   \draw (a2) --node[sloped] {$|$} (a3);
  \draw[postaction={decorate}] (a2) --node[above] {$r_{2}$} (a1); \draw[postaction={decorate}] (a4) --node[above] {$r_{1}$} (a3);
   
   \draw (a1) -- ++(-90:.6);
\draw (a2) -- ++(-90:.6);
\draw (a3) -- ++(80:.6);
\draw (a4) -- ++(80:.6);

\node[below left] at (a1) {$-1$};
\node[below right] at (a2) {$0$};
 \node[right] at (a3) {$z$};
 \node[above left] at (a4) {$-1+ze^{-2ai\pi/k}$};
 
 \foreach \i in {1,2}
 \filldraw[fill=white] (a\i) circle (1pt);
  \foreach \i in {3,4}
 \fill (a\i) circle (1pt);
 
 \node at (0.5,-1) {Quadrilatère ABAB};
 \end{scope}
\end{tikzpicture}
\caption{Les trois modèles de surfaces pour la strate $\Omega^{k}\mathcal{M}_{0}(-a,a;-k,-k)$.} \label{fig:modeles}
\end{figure}

Le choix de l'un ou l'autre de ces modèles dépend de la valeur du ratio $\left(\frac{R_{1}}{R_{2}}\right)^{1/k}$. Pour préciser cela, nous introduisons pour $k \geq 3$ et $0<a<\frac{k}{2}$ les domaines suivants dans le disque unité fermé $\mathbb{D}$ (voir la figure~\ref{fig:polygone}) :
\begin{itemize}
    \item $\mathcal{C}_{k,a}$ est l'union dans $\mathbb{D}$ de toutes les cordes reliant chacune des racines $k$-ièmes de l'unité à celle située $a$ racines plus loin;
    \item $\mathcal{IC}_{k,a}$ est l'union dans $\mathbb{D}$ des intérieurs de ces cordes;
    \item $\mathcal{I}_{k,a}$ est la composante connexe de $\mathbb{D} \setminus \mathcal{C}_{k,a}$ contenant le centre du disque;
    \item $\mathcal{E}_{k,a}$ est l'intérieur de $\mathbb{D} \setminus \mathcal{I}_{k,a}$;
    \item $\mathcal{S}_{k,a}$ est l'union des composantes connexes de $\mathbb{D} \setminus \mathcal{C}_{k,a}$ contenant les arcs circulaires reliant les racines;
    \item $\mathcal{T}_{k,a}$ est l'intérieur de $\mathbb{D} \setminus \mathcal{S}_{k,a}$.
\end{itemize}

\begin{figure}[htb]
\center
 \begin{tikzpicture}
 \draw (0,0) circle (2cm);
 \foreach \i in {0,1,...,7}
  \coordinate (a\i) at (45*\i:2); 
   \foreach \i in {0,1,...,7}
  \fill (a\i)  circle (1.5pt);
  \filldraw[fill=black!10] (a0) -- (a3) -- (a6) -- (a1) -- (a4) -- (a7) -- (a2) -- (a5) -- (a0);
  \draw[dashed] (0,0) circle (1cm);
  
  \begin{scope}[xshift=6cm]
 \filldraw[fill=black!30] (0,0) circle (2cm);
 \foreach \i in {0,1,...,7}
  \coordinate (a\i) at (45*\i:2); 
   \foreach \i in {0,1,...,7}
  \fill (a\i)  circle (1.5pt);
  \draw (a0) -- (a3) -- (a6) -- (a1) -- (a4) -- (a7) -- (a2) -- (a5) -- (a0);  
   \foreach \i in {0,1,...,7}
  \coordinate (b\i) at (45*\i:.83);
  \filldraw[fill=white] (b0) -- (b1) -- (b2) -- (b3) -- (b4) -- (b5) -- (b6) -- (b7) -- (b0); 
  \end{scope}

\end{tikzpicture}
\caption{La partie grisée désigne $\mathcal{T}_{8,3}$ dans le schéma de gauche et $\mathcal{E}_{8,3}$ dans celui de droite. Le cercle représenté est de rayon $\frac{1}{2}$.} \label{fig:polygone}
\end{figure}

\begin{lem}\label{lem:trigo}
Le domaine $\mathcal{T}_{k,a}$ contient le disque ouvert centré en l'origine de rayon $\rho = \cos\left(\frac{\pi}{k}\right)\left(1-\tan\left(\frac{\pi}{k}\right)\tan\left(\frac{a-1}{k}\pi\right)\right)$ (voir figure~\ref{fig:polygone}).
\par
En particulier, nous avons $\rho \geq 1- \frac{2}{k+2-2a}$.
\end{lem}

\begin{proof}
La formule donnant $\rho$ en fonction de $k$ et $a$ s'obtient par un simple exercice de trigonométrie. Ensuite, posant $a=\frac{k}{2}-b$, on obtient $\rho = \cos(\pi/k) - \frac{\sin(\pi/k)}{\tan(\frac{b+1}{k}\pi)}$.
\par
Pour tout coefficient $\alpha>1$, la fonction $f_{\alpha}: t \mapsto \cos(t) - \frac{\sin(t)}{\tan(\alpha t)}$ est croissante sur $[0,\frac{\pi}{2\alpha}]$. En effet, sa dérivée vaut $f_{\alpha}'(t)=-\sin(t)-\frac{\cos(t)}{\tan(\alpha t)}+\frac{\alpha\sin(t)}{\sin^{2}(\alpha t)} =\frac{\alpha \sin(t)-\sin(\alpha t)\cos((\alpha-1)t)}{\sin^{2}(\alpha t)}$. Comme $t \in [0,\frac{\pi}{2\alpha}]$, nous avons $\alpha \sin (t) \leq \sin(\alpha) t$ et donc $f_{\alpha}'$ est positive sur ce même intervalle. Puisque $f_{\alpha}(0)=1-\frac{1}{\alpha}$, on en déduit que $f_{\alpha}(\pi/2\alpha) \geq 1-\frac{1}{\alpha}$.
\par
Ainsi, $\rho=f_{\alpha}(\pi/k)$ pour $\alpha=(b+1)$ satisfait l'inégalité $\rho \geq \frac{b}{b+1}$ et ainsi $\rho \geq \frac{k-2a}{k+2-2a}$.
\end{proof}

\begin{lem}\label{lem:s2Modele}
En supposant que $|R_{2}| \geq |R_{1}|$, la configuration $(R_{1},R_{2})$ est réalisable dans la strate $\Omega^{k}\mathcal{M}_{0}(a,-a;-k,-k)$ avec $0<a<\frac{k}{2}$ selon un modèle :
\begin{itemize}
    \item \textbf{Triangle} si $-\left(\frac{R_{1}}{R_{2}}\right)^{1/k}$ est dans $\mathcal{IC}_{k,a}$ (dans ce cas le cylindre correspondant au pôle de résidu $R_{2}$ a pour bord deux côtés du triangle);
    \item \textbf{Quadrilatère AABB} si $-\left(\frac{R_{1}}{R_{2}}\right)^{1/k}$ est dans $\mathcal{E}_{k,a}$;
    \item \textbf{Quadrilatère ABAB} si $-\left(\frac{R_{1}}{R_{2}}\right)^{1/k}$ est dans $\mathcal{T}_{k,a}$ (dans ce cas c'est le zéro d'ordre $-a$ qui est au bord du cylindre correspondant au pôle de résidu $R_{2}$).
\end{itemize}
\end{lem}

\begin{proof}
Observons en premier lieu que les trois modèles sont les seules formes possibles pour une différentielle de la strate $\Omega^{k}\mathcal{M}_{0}(a,-a;-k,-k)$. En effet, un calcul de caractéristique d'Euler montre que si $b$ est le nombre de segments de bord dans les deux cylindres, alors dans une telle strate le reste de la surface s'obtient en collant $4-b$ triangles. Chaque cylindre a un bord donc nous avons $b \geq 2$. De plus, la surface ne peut pas être constituée seulement des cylindres car les angles au bord du cylindre valent $\pi$ et nos singularités coniques ont un angle qui n'est pas un multiple de $\pi$. Ainsi, le complémentaire des cylindres est soit un triangle, soit un quadrilatère (dans ce dernier cas, les bords des cylindres peuvent être des côtés consécutifs ou au contraire opposés dans le quadrilatère).
\par
Nous donnons d'abord la construction du modèle \textbf{Triangle}. Normalisons d'abord par $R_{2}=(-1)^{k}$ et prenons $r_{2}=-1$. Nous allons construire un triangle dont les sommets seront $0$, $\theta \in ]0,1[$ et $r_{1}$ (racine $k$-ième de $R_{1}$). Nous pourrons ensuite coller (voir figure~\ref{fig:modeles}) un cylindre sur le segment $[0,r_{1}]$ (correspondant au pôle de résidu $R_{1}$), puis les deux autres segments sur un autre cylindre (correspondant au pôle de résidu $R_{2}$). Les contraintes sur les ordres des zéros impliquent que $r_{1}=\theta+(1-\theta)e^{\frac{2ai\pi}{k}}$. Les valeurs de $r_{1}$ données pour~$\theta$ variant entre entre $0$ et $1$ décrivent la corde entre $1$ et $e^{\frac{2ai\pi}{k}}$. En fixant chacune des $k-1$ autres racines de $R_{2}$, on obtient par rotation du cas précédent que le modèle \textbf{Triangle} est réalisable pour toute valeur de $-\left(\frac{R_{1}}{R_{2}}\right)^{1/k}$ dans $\mathcal{IC}_{k,a}$.
\par
Pour le modèle \textbf{Quadrilatère AABB}, nous allons construire un quadrilatère dont les sommets sont $-1$, $0$, $-e^{\frac{2ai\pi}{k}}$ et un sommet $z$ dont l'argument se situe dans l'intervalle $]0,\frac{2a\pi}{k}[$. En identifiant les segments $[-1,0]$ et $[0,-e^{\frac{2a\pi}{k}}]$, nous obtenons une singularité d'angle $\frac{2(k-a)\pi}{k}$. Nous collons des cylindres sur les deux autres segments pour obtenir des pôles d'ordre $k$ dont les $k$-résidus pour lesquels le ratio $-\left(\frac{R_{1}}{R_{2}}\right)^{1/k}$ vaut soit $\frac{z+e^{\frac{2ai\pi}{k}}}{z+1}$, soit son inverse. En faisant varier la position de $z$ à l’intérieur de son cône ouvert, le ratio $-\left(\frac{R_{1}}{R_{2}}\right)^{1/k}$ va décrire un domaine bordé par une corde entre $1$ et $e^{\frac{2a\pi}{k}}$. En effet, la demi-droite $(r_{1},r_{2}) = (t,t)$ avec $t>0$ correspond à la portion de cercle unité entre $e^{\frac{2ai\pi}{k}}$ (pour $t\to0$) et $1$ (pour $t\to \infty$). La demi-droite des réels positifs correspond à la corde entre ces deux racines de l'unité. L’intérieur du domaine est envoyé sur la portion du disque délimité par ces deux courbes. Les symétries entre les différents choix de racines montrent que sous l'hypothèse $|R_{1}| \leq |R_{2}|$ ce domaine couvre $\mathcal{E}_{k,a}$.
\par
%Pour le modèle \textbf{Quadrilatère ABAB}, nous allons construire un quadrilatère dont les sommets sont les points $-1$, $0$, $z$ et $-1+ze^{\frac{2ai\pi}{k}}$. Nous faisons varier l'argument de $z$ dans l'intervalle $]\frac{-2a\pi}{k},\pi[$. Nous identifions les segments reliant $0$ et $z$ ainsi que $-1$ et $-1+ze^{\frac{2a\pi}{k}}$, puis collons des cylindres sur les deux autres segments (voir figure~\ref{fig:modeles}). Leurs périodes sont respectivement $1$ et $-1+z(e^{\frac{2ai\pi}{k}}-1)$. Les résidus des pôles correspondants sont donc $R_{1}=1$ et $R_{2}=\left(-1+z(e^{\frac{2ai\pi}{k}}-1)\right)^{k}$ avec $|R_{2}| \geq |R_{1}|$. En faisant varier la position de $z$ à l’intérieur de son cône ouvert, le ratio $-\left(\frac{R_{1}}{R_{2}}\right)^{1/k}$ va décrire un domaine bordé par la corde entre $1$ et~$e^{\frac{2a\pi}{k}}$. À nouveau, les symétries entre les différents choix de racines montrent que ce domaine couvre~$\mathcal{T}_{k,a}$.
Pour le modèle \textbf{Quadrilatère ABAB}, remarquons qu'une perturbation arbitrairement petite d'une surface réalisant le modèle \textbf{Triangle} donne une surface réalisant l'un ou l'autre des deux autres modèles. En particulier, partant de la surface réalisant le modèle \textbf{Triangle} dans la figure~\ref{fig:modeles}, si on déplace le point $z$ sous la corde $[1,e^{2ia\pi/k}]$, alors la singularité de coordonnée $\theta$ acquiert une partie imaginaire légèrement négative et devient le seul zéro au bord du cylindre de résidu $R_{2}$. Chaque zéro appartient au bord d'un seul cylindre. Nous obtenons bien le modèle \textbf{Quadrilatère ABAB}.
\par
Ainsi, lorsque $-\left(\frac{R_{1}}{R_{2}}\right)^{1/k}$ se trouve dans l'intersection de $\mathcal{T}_{k,a}$
avec un certain voisinage des cordes $\mathcal{IC}_{k,a}$, le modèle \textbf{Quadrilatère ABAB} est réalisable.
\par
Ensuite, le prolongement plat décrit dans la section~\ref{sec:arrhyp} permet d'étendre la construction de proche en proche sachant que :
\begin{itemize}
    \item la structure de translation ne peut dégénérer que si $-\left(\frac{R_{1}}{R_{2}}\right)^{1/k}$ vaut $0$ ou une racine $k$-ième de l'unité;
    \item les trois modèles sont les seules formes possibles pour le complémentaire des cylindres dans une telle strate;
    \item pour passer continûment d'un \textbf{Quadrilatère ABAB}  à un \textbf{Quadrilatère AABB}, la surface doit passer par la forme \textbf{Triangle} et le ratio $-\left(\frac{R_{1}}{R_{2}}\right)^{1/k}$ doit donc franchir l'une des cordes de $\mathcal{IC}_{k,a}$ puisque la construction donnée plus haut donne la seule façon d'obtenir une forme \textbf{Triangle}.
\end{itemize}
Ainsi, pour toute valeur de $-\left(\frac{R_{1}}{R_{2}}\right)^{1/k}$ dans $\mathcal{T}_{k,a}$, la construction \textbf{Quadrilatère ABAB} est réalisable.
\end{proof}

Nous pouvons en particulier en déduire que pour $s=2$, en dehors de l'obstruction (1), chaque configuration est réalisable dans chaque strate. 

\begin{cor}\label{cor:poleks2}
Pour $0<a<\frac{k}{2}$, chaque configuration $(R_{1},R_{2})$ avec $R_{1},R_{2} \in \CC^{\ast}$ est réalisable dans la strate $\Omega^{k}\mathcal{M}_{0}(a,-a;-k,-k)$ sauf si $a=1$ et $(R_{1},R_{2}) \in \CC^{\ast}\cdot(1,(-1)^{k})$.
\end{cor}

\begin{proof}
Les domaines $\mathcal{IC}_{k,a}$, $\mathcal{T}_{k,a}$ et $\mathcal{E}_{k,a}$ recouvrent tout le disque unité sauf peut-être les racines $k$-ièmes de l'unité. Celles-ci sont contenues dans le domaine $\mathcal{E}_{k,a}$ dès lors que $a \neq 1$.
\end{proof}

\subsubsection{\bf Cas $s\geq3$}

Les constructions du cas $s=2$ s'étendent à un nombre supérieur de pôles pour couvrir les strates $\Omega^{k}\mathcal{M}_{0}(a_{1},a_{2};\rec[-k][s])$ telles que $-\frac{k}{2}<a_{2}<0$, c'est-à-dire quand $l_{2}=0$. L'ingrédient nouveau est une estimation des modules de sommes de racines (voir section~\ref{sec:estimation}). Nous traitons d'abord les cas $k \geq 5$.

\begin{prop}\label{prop:ConstL2zerok5}
Pour $k \geq 5$, $0<a<\frac{k}{2}$ et $s \geq 3$, chaque configuration $(R_{1},\dots,R_{s})$ de nombres complexes non nuls est réalisable dans la strate $\Omega^{k}\mathcal{M}_{0}((s-2)k+a,-a;\rec[-k][s])$ sauf pour les exceptions suivantes :
\begin{enumerate}
    \item $\CC^{\ast}\cdot(1,1,1)$ dans la strate $\Omega^{6}\mathcal{M}_{0}(7,-1;\rec[-6][3])$;
    \item $\CC^{\ast}\cdot(1,1,1,1)$ dans la strate $\Omega^{6}\mathcal{M}_{0}(13,-1;\rec[-6][4])$.
\end{enumerate}
\end{prop}

\begin{proof}
Nous considérons une strate et une configuration $(R_{1},\dots,R_{s})$. Nous normalisons par $R_{s}=1$ et supposons $|R_{1}| \leq \dots \leq |R_{s}|$. 

Si $s-1 \geq 3$, en dehors du cas $k=6$ avec la configuration $(1,\dots,1)$, le lemme~\ref{lem:estimate} nous garantit l'existence de racines $k$-ièmes $r_{1},\dots,r_{s}$  des $R_{1},\dots,R_{s}$ telles que :
\begin{enumerate}[i)]
    \item $r_{s}=1$;
    \item $0<|\tilde{r}|<1$ avec $\tilde{r}=\sum\limits_{i=1}^{s-1} r_{i}$;
    \item les $r_{1},\dots,r_{s-1}$ ne sont pas $\mathbb{R}$-colinéaires.
\end{enumerate}
Construisons tout d'abord un polygone convexe $M$ de côtés $r_{1},\dots,r_{s-1},-\tilde{r}$ en ordonnant les arguments de façon cyclique. Nous savons que $M$ n'est pas dégénéré. Collons ensuite les cylindres correspondants sur les côtés $r_{1},\dots,r_{s-1}$. Selon la valeur de $-\frac{\tilde{r}}{r_{s}}=-\tilde{r}$, nous choisissons une construction ou une autre. 
\par
Selon si $-\tilde{r}$ appartient respectivement aux domaines $\mathcal{IC}_{k,a}$, $\mathcal{E}_{k,a}$ ou $\mathcal{T}_{k,a}$, nous remplaçons le cylindre bordé par $r_{s}$  par le polygone $M$ dans le modèle \textbf{Triangle}, \textbf{Quadrilatère AABB} ou \textbf{Quadrilatère ABAB} du lemme~\ref{lem:s2Modele}. Nous réalisons ainsi la configuration de résidus dans une strate pour laquelle une singularité conique a un angle de $\frac{2(k-a)\pi}{k}$.
\par
Il nous reste à réaliser la configuration $(1,\dots,1)$ dans les strates $\Omega^{6}\mathcal{M}_{0}(6s-11,-1;\rec[-6][s])$ avec $s \geq 5$. En effet, si $a=2$ les différentielles sextiques de la strate ne sont pas primitives. De plus, rappelons que les cas $s=3$ et $s=4$ sont des obstructions.
\par
Notons $\omega=\exp\left(\frac{i\pi}{3}\right)$. Dans le cas $s=5$, nous construisons un heptagone reliant les sommets $0,\omega^{2},\omega,1,-\omega^{2},-\omega,-1,0$. Nous identifions les côtés $[-1,0]$ et $[0,\omega^{2}]$  puis collons des cylindres sur les autres segments pour obtenir la surface voulue. Pour $s\geq6$ nous procédons par récurrence. Supposons que l'on a une surface avec $s-1$ cylindres. Il suffit de remplacer un cylindre par un triangle équilatéral sur lequel nous collons deux autres cylindres pour obtenir la surface souhaitée.
\par
Nous traitons maintenant le cas $s=3$ pour lequel le lemme~\ref{lem:estimate} ne garantit pas que les racines $r_{1},r_{2}$ ne sont pas $\mathbb{R}$-colinéaires. Supposons d'abord que nous ne sommes pas dans le cas $k=5$ avec des résidus $(R_{1},R_{2})$ satisfaisant $R_{1}=-R_{2}$. Si les racines $r_{1}$ et $r_{2}$ sont $\mathbb{R}$-colinéaires, comme elles satisfont toujours $|r_{1}+r_{2}|<|r_{s}|$, la construction demeure réalisable avec un polygone $M$ qui est alors un triangle plat.
\par
Enfin, nous traitons le cas $k=5$ et $R_{1}=-R_{2}$, normalisé par $R_{1}=1$. Si $|R_{3}|^{1/5} > 2\sin(\pi/5)$, comme on peut choisir des racines $r_{1},r_{2}$ telles que $|r_{1}+r_{2}|=2\sin(\pi/5)$ (voir le lemme~\ref{lem:estimate}), les constructions précédentes fonctionnent. Si au contraire nous avons $1 \leq |R_{3}|^{1/5} \leq 2\sin(\pi/5)$, alors, si $R_{3}\neq -R_{1}$ on peut choisir des racines $r_{1}$ et $r_{3}$ pour avoir $0<|r_{1}+r_{3}|< 1$. En effet, dans le pire des cas, $\frac{R_{3}}{R_{1}}$ est un réel positif. Nous choisissons alors $r_{1}=e^{\frac{4i\pi}{5}}$ et $r_{3}=|r_{3}|$. Même quand $|r_{3}|=2\sin(\pi/5)$, on a $|r_{1}+r_{3}|=\sqrt{(2\sin(\pi/5)-\cos(\pi/5))^{2}+(\sin(\pi/5))^{2}}<1$. On peut ensuite réaliser les constructions adéquates. Si $R_{3}=-R_{1}$, il suffit de réordonner les résidus en $(1,1,-1)$ et de suivre les constructions précédentes.
\end{proof}

Les strates de différentielles cubiques et quartiques présentent davantage de cas exceptionnels qui exigent des constructions ad hoc.

\begin{prop}\label{prop:l2zerok3}
Pour $s \geq 3$, chaque configuration $(R_{1},\dots,R_{s})$ de nombres complexes non nuls est réalisable dans la strate $\Omega^{3}\mathcal{M}_{0}(3s-5,-1;\rec[-3][s])$ sauf pour les configurations $\CC^{\ast}\cdot(1,1,1)$ dans la strate $\Omega^{3}\mathcal{M}_{0}(4,-1;\rec[-3][3])$.
\end{prop}

\begin{proof}
Nous supposerons que les résidus sont ordonnés par module croissant. D'après le lemme~\ref{lem:estimate}, pour la famille de nombres $R_{1},\dots,R_{s-1}$, il existe toujours un choix de racines troisièmes $r_{1},\dots,r_{s-1}$ non $\mathbb{R}$-colinéaires tel que $ 0 < |\sum_{i=1}^{s-1} r_{i}| < |R_{s-1}|^{1/3}\leq |R_{s}|^{1/3}$ quand $s-1 \geq 3$ sauf si $(R_{1},\dots,R_{s-1}) \in \CC^{\ast}\cdot \left(\rec[1][s_{1}],\rec[-1][s_{2}],\rec[3i\sqrt{3}][t_{1}],\rec[-3i\sqrt{3}][t_{2}]\right)$ avec $t_{1}+t_{2} \geq 1$ et $s_{1}-s_{2} \in 3 \ZZ$. En dehors de ces cas, la preuve fonctionne exactement comme celle de la proposition~\ref{prop:ConstL2zerok5}. De même quand $s-1=2$, si les racines sont $\mathbb{R}$-colinéaires, la construction avec un triangle plat fonctionne toujours.
\par
Dans les cas restants, il existe des entiers satisfaisant $s_{1}+s_{2}+t_{1}+t_{2}+1=s$ avec $t_{1} \geq 1$  et $s_{1}-s_{2} \in 3\ZZ$ tels que $(R_{1},\dots,R_{s})=(\rec[1][s_{1}],\rec[-1][s_{2}],\rec[3i\sqrt{3}][t_{1}],\rec[-3i\sqrt{3}][t_{2}],R_{s})$. Nous distinguerons les cas selon la valeur de $|R_{s}|$.
\par
D'abord, d'après le lemme~\ref{lem:estimate}, il existe des racines $r_{1},\dots,r_{s-1}$ non toutes $\mathbb{R}$-colinéaires dont la somme est de module au plus $3$ (si $s_{1}=s_{2}=0$ et $t_{2}-t_{1} \in 3\ZZ$) ou $\sqrt{3}$ (dans tout autre cas).  Si $|R_{s}|>27$, alors même quand la somme des $s-1$ premières racines est de module~$3$, la construction habituelle fonctionne. Elle fonctionne aussi si $|R_{s}|>3\sqrt{3}$ et que $s_{1},s_{2}$ ne sont pas tous deux nuls ou que $t_{2}-t_{1} \notin 3\ZZ$.
\par
Ensuite, si $|R_{s}|=3\sqrt{3}$, nous pouvons intervertir $R_{s}$ avec un résidu valant $\pm 3i\sqrt{3}$. Si nous n'obtenons pas une configuration exceptionnelle, la construction habituelle fonctionne. Il nous reste donc le cas où $R_{s}=\pm 3i\sqrt{3}$. Puisque nous avons $t_{1} \geq 1$, nous pouvons supposer que $R_{s}=3i\sqrt{3}$ et que $(R_{1},\dots,R_{s})=(\rec[1][s_{1}],\rec[-1][s_{2}],\rec[3i\sqrt{3}][t_{1}+1],\rec[-3i\sqrt{3}][t_{2}])$ avec $s_{1}-s_{2} \in 3\ZZ$. Nous commençons par donner une construction pour la configuration $(3i\sqrt{3},3i\sqrt{3},-3i\sqrt{3},-3i\sqrt{3})$. Partant de l'hexagone reliant les points $0,1-\omega,2+\omega,1+2\omega,\omega-1,-\omega-2$. Nous identifions les segments $[0,1-\omega]$ et $[-\omega-2,0]$ puis collons des cylindres sur les segments restants. Nous obtenons ainsi la surface plate souhaitée.
\par
Nous généralisons ensuite la construction par une série de chirurgies. En remplaçant un cylindre par un triangle équilatéral (auquel sera collé deux cylindres sur les côtés restants), nous pouvons toujours remplacer un résidu $R$ par une paire de résidus $(-R,-R)$. En partant de la configuration $(1,1,-1,-1)$, ceci permet d'obtenir par récurrence n'importe quelle configuration de la forme $(\rec[1][s_{1}],\rec[-1][s_{2}])$ avec $s_{2}-s_{1} \in 3\ZZ$ sauf celles dans $\CC^{\ast}(1,1,1)$ qui correspondent à l'obstruction (2). Ensuite, en utilisant cette fois-ci un triangle d'angles $\frac{2\pi}{3},\frac{\pi}{6},\frac{\pi}{6}$, selon le côté le long duquel nous le collons, nous pourrons remplacer un résidu $\lambda$ par une paire de résidus $\left(\frac{\lambda}{3i\sqrt{3}},-\frac{\lambda}{3i\sqrt{3}}\right)$ ou $\left(\lambda,-3\lambda i\sqrt{3}\right)$. Nous obtenons ainsi toutes les configurations de la forme $\left(\rec[1][s_{1}],\rec[-1][s_{2}],\rec[3i\sqrt{3}][t_{1}+1],\rec[-3i\sqrt{3}][t_{2}]\right)$ tant que $s_{1}-s_{2} \in 3\ZZ$ et $s_{1},s_{2} \geq 1$. En utilisant la chirurgie du triangle équilatéral, nous pouvons obtenir les cas pour lesquels l'un ou l'autre des nombres $s_{1},s_{2}$ est nul. Enfin, pour $s_{1}=s_{2}=0$, les cas pour lesquels $t_{2}-t_{1} \in 3\ZZ$ sont déjà couverts. La chirurgie du triangle équilatéral permet d'obtenir les derniers en partant de la configuration $(3i\sqrt{3},3i\sqrt{3},-3i\sqrt{3})$. Une configuration dans $\CC^{\ast}(1,1,-1)$ s'obtient à partir du pentagone reliant les points $\frac{i}{\sqrt{3}},1,1+\omega,\omega,-1$. Il suffit ensuite d'identifier les segments de part et d'autre du sommet $\frac{i}{\sqrt{3}}$ puis de coller des cylindres sur les autres.
\par
Ces dernières constructions permettent aussi de traiter du cas $R_{s}=\pm 27$. En effet, d'après les cas précédents, nous pouvons supposer que $s_{1}=s_{2}=0$ et $t_{2}-t_{1} \in 3\ZZ$. La configuration $(R_{1},\dots,R_{s})$ est donc équivalente à une configuration de la forme $\left(\rec[1][t_{1}],\rec[-1][t_{2}],3i\sqrt{3}\right)$.
\par
Dans les derniers cas, nous avons $3\sqrt{3}<|R_{s}| \leq 27$ et $R_{s} \neq \pm 27$. D'après les cas déjà traités, nous pouvons aussi supposer que $s_{1}=s_{2}=0$ et $t_{2}-t_{1} \in 3\ZZ$. Notre configuration est donc $(\rec[3i\sqrt{3}][t_{1}],\rec[-3i\sqrt{3}][t_{2}],R_{s})$. Si $t_{1},t_{2} \geq 1$, alors nous pouvons trouver une racine $r_{i}$ de $3i\sqrt{3}$ ou $-3i\sqrt{3}$ ainsi qu'une racine $r_{s}$ de $R_{s}$ telles que $0<|r_{i}+r_{s}|<\sqrt{3}$. La configuration de $s-1$ résidus dans laquelle nous remplaçons $R_{i}$ et $R_{s}$ par $(r_{i}+r_{s})^{3}$ n'est pas exceptionnelle même si $(r_{i}+r_{s})^{3}= \pm 1$ ici. La construction habituelle fonctionne pour une telle configuration dans laquelle le résidu de module le plus grand est $\pm 3i\sqrt{3}$. Il suffit ensuite de remplacer le cylindre correspondant au résidu $(r_{i}+r_{s})^{3}$ par un triangle et les cylindres attenants pour obtenir la surface voulue.
\par
Si nous avons $t_{1}=0$ ou $t_{2}=0$ (sans perte de généralité, nous supposons $t_{2}=0$), notre configuration est $(\rec[3i\sqrt{3}][s-1],R_{s})$ avec $s-1 \in 3\ZZ$. Si une racine de $3i\sqrt{3}$ ajoutée à une racine de~$R_{s}$ permet d'obtenir un nombre complexe de module strictement plus petit que $\sqrt{3}$, alors la construction précédente fonctionne. Sinon nous choisissons deux racines $r_{1}$ et $r_{2}$ de $3i\sqrt{3}$ ainsi qu'une racine $r_{s}$ de $R_{s}$ de façon à avoir $r_{2}=\omega r_{1}$ et $0<|r_{1}+r_{2}+r_{s}|<\sqrt{3}$. Rappelons que le cas dans lequel $R_{s} = \pm 27$ est équivalent par homothétie à un cas déjà traité. La suite de la construction est identique si ce n'est qu'on ajoute un quadrilatère au lieu d'un triangle.
\end{proof}

\begin{prop}\label{prop:l2zerok4}
Pour $s \geq 3$, chaque configuration $(R_{1},\dots,R_{s})$ de nombres complexes non nuls est réalisable dans la strate $\Omega^{4}\mathcal{M}_{0}(4s-7,-1;\rec[-4][s])$ sauf dans les cas suivants :
\begin{enumerate}
    \item la strate $\Omega^{4}\mathcal{M}_{0}(5,-1;\rec[-4][3])$ avec les résidus dans $\CC^{\ast}\cdot(1,1,-4)$;
    \item la strate $\Omega^{4}\mathcal{M}_{0}(9,-1;\rec[-4][4])$ avec les résidus dans $\CC^{\ast}\cdot(1,1,1,1)$.
\end{enumerate}
\end{prop}

\begin{proof}
Nous supposerons que les résidus sont ordonnés par module croissant. D'après le lemme~\ref{lem:estimate}, pour $s-1 \geq 3$  et toute  famille de nombres $R_{1},\dots,R_{s-1}$, il existe un choix de racines $4$-ièmes non $\mathbb{R}$-colinéaires $r_{1},\dots,r_{s-1}$ tel que $ 0 < |\sum r_{i}| < |R_{s-1}|^{1/4}$  sauf si $(R_{1},\dots,R_{s-1}) \in \CC^{\ast}(\rec[1][s_{1}],\rec[-4][s_{2}])$ avec $s_{1}$ pair et $s_{2} \geq 1$. En dehors de ces cas, la preuve fonctionne exactement comme celle des propositions~\ref{prop:ConstL2zerok5} et~\ref{prop:l2zerok3}. De même quand $s-1=2$, si les racines sont $\mathbb{R}$-colinéaires, la construction avec un triangle plat fonctionne aussi.
\par
Dans les cas restants, nous avons $(R_{1},\dots,R_{s})=(\rec[1][s_{1}],\rec[-4][s_{2}],R_{s})$ avec  $s_{1}+s_{2}+1=s$, $s_{2} \geq 1$ et $s_{1}$ pair. Nous distinguerons les cas selon la valeur de $|R_{s}|$.
\par
D'abord, d'après le lemme~\ref{lem:estimate}, il existe des racines $r_{1},\dots,r_{s-1}$ non  $\mathbb{R}$-colinéaires dont la somme est de module au plus $\sqrt{2}$ (si $s_{1}+2s_{2} \in 4\ZZ+2$) ou $2$ (si $s_{1}+2s_{2} \in 4\ZZ$).  Si $|R_{s}|>16$, alors même quand la somme des $s-1$ premières racines est de module $2$, la construction habituelle fonctionne. Elle fonctionne aussi si $|R_{s}|>4$ et que $s_{1}+2s_{2} \in 4\ZZ+2$.
\par
Ensuite, si $|R_{s}|=4$, nous pouvons intervertir $R_{s}$ avec un résidu valant $-4$. Si nous n'obtenons pas une configuration exceptionnelle, la construction habituelle fonctionne. Ainsi, il nous reste le cas où $(R_{1},\dots,R_{s})=(\rec[1][s_{1}],\rec[-4][s_{2}+1])$ avec $s_{1}$ pair et $s_{2} \geq 1$. Pour la configuration $(-4,-4,-4)$, nous construisons le pentagone reliant les points $-1,0,-i,1,i,-1$. Puis nous identifions les segments $[-1,0]$ et $[0,-i]$, et nous collons des cylindres sur les trois autres segments pour donner la surface plate adéquate. La construction se généralise pour un nombre supérieur de pôles en remplaçant des cylindres par des triangles rectangles auxquels sont collés deux cylindres sur les côtés restants. Selon si nous collons le triangle le long de son hypoténuse ou d'un des deux autres, nous remplaçons un résidu $-4$ par deux résidus $1$ ou au contraire un résidu $1$ par un résidu $1$ et un résidu $-4$. Nous obtenons ainsi toutes les configurations dans lesquelles $s_{1} \geq 2$ et $s_{2}+1 \geq 2$. Sachant que les configurations $(1,1,-4)$ et $(-4,-4,-4,-4)$ correspondent aux obstructions (8) et (10) et que nous avons par hypothèse $s_{2} \geq 1$, il faut encore traiter le cas des configurations pour lesquelles $s_{1}=0$ et $s_{2}+1 \geq 5$.
\par
La chirurgie consistant à retirer un cylindre pour le remplacer par un carré sur lequel on colle trois autres cylindres permet de remplacer $s_{2}$ par $s_{2}+2$. Sachant que le cas $s_{2}+1=3$ est déjà traité, nous devons donner une construction ad hoc pour le cas $s_{2}+1=6$, c'est-à-dire la configuration proportionnelle à $(1,1,1,1,1,1)$ dans la strate $\Omega^{4}\mathcal{M}_{0}(17,-1;\rec[-6][4])$. Pour cela, nous construisons un octogone reliant les points $-1,i,1,1+i,1+2i,2i,-1+2i,-1+i$. Puis nous identifions les segments $[-1,i]$ et $[i,1]$. En collant des cylindres sur les six autres segments, nous obtenons la surface adéquate.
\par
Un raisonnement semblable permet de traiter du cas $R_{s}=16$. D'après les cas précédents nous pouvons supposer que $s_{1}+2s_{2} \in 4\ZZ$. Si $s_{1}=0$, alors $(R_{1},\dots,R_{s})$ équivaut à $(\rec[1][s-1],-4)$ avec $s-1$ pair. Nous avons déjà fait les constructions pour de telles configurations. Si au contraire $s_{1} \geq 2$ (avec $s_{1}$ est pair), alors $(R_{1},\dots,R_{s})=(\rec[1][s_{1}],\rec[-4][s_{2}],16)$. En partant d'une surface réalisant la configuration $(\rec[1][s_{1}],\rec[-4][s_{2}])$ comme dans le paragraphe précèdent, l'ajout d'un triangle rectangle permet de remplacer un résidu $-4$ par un résidu $-4$ et un résidu~$16$. Nous obtenons ainsi toutes les configurations voulues sauf éventuellement $(1,1,-4,16)$. Pour ce cas, nous construisons l'hexagone reliant les points $-1-i,0,1-i,1,i,i-1,-1$. Puis nous  identifions les segments $[-1-i,0]$ et $[0,1-i]$ et collons des cylindres sur les autres segments.
\par
Enfin, il reste les cas dans lesquels $4<|R_{s}| \leq 16$ et $R_{s} \neq 16$. D'après les cas déjà traités, nous pouvons de surcroît supposer que  $s_{1}+2s_{2} \in 4\ZZ$. Si $s_{1} \geq 2$ ($s_{1}$ est pair), alors nous pouvons trouver deux racines $r_{1},r_{2}$ de $1$ et une racine $r_{s}$ de $R_{s}$ telles que $0<|r_{1}+r_{2}+r_{s}|<\sqrt{2}$ et $r_{2}=\pm ir_{1}$. La configuration de $s-2$ résidus $(R_{3},\dots,R_{s-1},(r_{1}+r_{2}+r_{s})^{4})$ n'est pas exceptionnelle car il est impossible d'obtenir $(r_{1}+r_{2}+r_{s})^{4}=1$. La construction habituelle fonctionne pour une telle configuration dans laquelle le résidu de module le plus grand est $-4$ (la configuration est normalisée de telle façon que nous avons toujours $s_{2} \geq 1$). Il suffit ensuite de remplacer le cylindre correspondant au résidu  $(r_{1}+r_{2}+r_{s})^{4}$ par un quadrilatère dont les segments sont $r_{1},r_{2},r_{s}$ (non $\mathbb{R}$-colinéaires par construction) et les cylindres attenants pour obtenir la surface voulue.
\par
Le cas $s_{1}=0$ (et donc $s_{2} \geq 2$ puisque $s_{1}+2s_{2} \in 4\ZZ$) se traite de façon analogue en constatant qu'il existe une racine $r_{s}$ de $R_{s}$ et des racines $r_{1},r_{2}$ de $R_{1}=R_{2}=-4$ telles que $0<|r_{1}+r_{2}+r_{s}|<\sqrt{2}$. Le reste de la construction est identique à ceci près que nous recollons non un quadrilatère mais un triangle (potentiellement dégénéré).
\end{proof}

En combinant le lemme~\ref{lem:S5l1=0}, le corollaire~\ref{cor:poleks2} ainsi que les propositions~\ref{prop:ConstL2zerok5},~\ref{prop:l2zerok3} et~\ref{prop:l2zerok4}, nous obtenons la caractérisation complète des configurations réalisable dans une strate ayant un zéro d'ordre négatif.

\begin{cor}\label{cor:FinalL2zero}
Chaque configuration $(R_{1},\dots,R_{s})$ de nombres complexes non nuls est réalisable dans la strate $\Omega^{k}\mathcal{M}_{0}(a_{1},a_{2};\rec[-k][s])$ avec $a_{1}$ ou $a_{2}$ négatif sauf dans les six cas suivants :
\begin{enumerate}
    \item $\CC^{\ast}\cdot(1,(-1)^{k})$ dans les strates $\Omega^{k}\mathcal{M}_{0}(1,-1;\rec[-k][2])$;
    \item $\CC^{\ast}\cdot(1,1,1)$ dans la strate $\Omega^{3}\mathcal{M}_{0}(4,-1;\rec[-3][3])$;
    \item $\CC^{\ast}\cdot(1,1,-4)$ dans la strate $\Omega^{4}\mathcal{M}_{0}(5,-1;\rec[-4][3])$;
    \item $\CC^{\ast}\cdot(1,1,1,1)$ dans la strates $\Omega^{4}\mathcal{M}_{0}(9,-1;\rec[-4][4])$;
    \item $\CC^{\ast}\cdot(1,1,1)$ dans la strate $\Omega^{6}\mathcal{M}_{0}(7,-1;\rec[-6][3])$;
    \item $\CC^{\ast}\cdot(1,1,1,1)$ dans la strate $\Omega^{6}\mathcal{M}_{0}(13,-1;\rec[-6][4])$.
\end{enumerate}
\end{cor}

\subsection{Constructions pour $n=2$ avec deux zéros d'ordre positif}\label{sub:mechant}

Dans la plupart de ces cas, nous réalisons les configurations dans les strates voulues en utilisant des constructions par récurrence faisant intervenir des chirurgies permettant de se ramener à des strates plus simples. C'est pourquoi le cas dans lequel l'un des zéros d'ordre négatif (jouant une partie du rôle d'initialisation dans ces récurrences) a nécessité des constructions systématiques. Lorsque les configurations de résidus vérifient certaines identités (voir la section~\ref{sec:arrhyp}), la chirurgie suivante permet d'efficaces simplifications.

\begin{lem}\label{lem:chirurgietranslation}
Considérons une strate $\Omega^{k}\mathcal{M}_{0}(a_{1},a_{2};\rec[-k][p])$ et une configuration $(R_{1},\dots,R_{p})$ de $k$-résidus. Supposons qu'il existe une paire d'entiers positifs non nuls $l_{1},l_{2}$ vérifiant les conditions suivantes : \begin{enumerate}
    \item $t=l_{1}+l_{2}$ satisfait $2 \leq t \leq p-1$;
    \item il existe une famille de racines $k$-ièmes $r_{1},\dots,r_{t}$ de $R_{1},\dots,R_{t}$ telles que $\sum\limits_{i=1}^{t} r_{i}=0$; 
    \item aucun sous-ensemble strict $I$ de $\lbrace{ 1,\dots,t \rbrace}$ ne définit de somme partielle nulle $\sum\limits_{i \in I} r_{i}=0$;
    \item $a_{1}-kl_{1} >-k$ et $a_{2}-kl_{2}>-k$;
    \item la configuration $(R_{t+1},\dots,R_{p})$ est réalisable dans $\Omega^{k}\mathcal{M}_{0}(a_{1}-kl_{1},a_{2}-kl_{2};\rec[-k][p-t])$.
\end{enumerate} 
Dans ce cas la configuration $(R_{1},\dots,R_{p})$ est réalisable dans $\Omega^{k}\mathcal{M}_{0}(a_{1},a_{2};\rec[-k][p])$.
\end{lem}

\begin{proof}
Pour construire une $k$-différentielle dans $\Omega^{k}\mathcal{M}_{0}(a_{1},a_{2};\rec[-k][p])$ dont les résidus sont $(R_{1},\dots,R_{p})$, nous partons d'une $k$-différentielle $\xi$ de  $\Omega^{k}\mathcal{M}_{0}(a_{1}-kl_{1},a_{2}-kl_{2};\rec[-k][p-t])$ dont les résidus sont $(R_{t+1},\dots,R_{p})$. Dans la surface plate correspondant à  $\xi$, il existe toujours un lien-selle reliant les deux singularités coniques. Nous notons $\lambda$ sa période (définie à multiplication par $e^{\frac{2i\pi}{k}}$ près). Nous allons couper le long de ce lien ce lien-selle et recoller les deux bords sur une surface de translation $(X,\omega)$ (ce qui signifie que $\omega$ est une différentielle abélienne) ayant les propriétés suivantes :
\begin{itemize}
    \item $X$ est de genre zéro, possède deux singularités coniques $M_{1},M_{2}$ et deux liens-selles de bord reliant $M_{1}$ et $M_{2}$;
    \item les singularités $M_{1},M_{2}$ sont d'angles respectifs $2l_{1}\pi$ et $2l_{2}\pi$;
    \item $\omega$ possède $t$ pôles simples de résidus $r_{1},\dots,r_{t}$ (dont la somme totale est nulle);
    \item les deux liens-selles de bords sont de périodes $\lambda$ et $-\lambda$.
\end{itemize}
Le reste de la démonstration consiste à construire $(X,\omega)$. Supposons d'abord que les $t$ nombres complexes $r_{1},\dots,r_{t}$ ne sont pas tous $\mathbb{R}$-colinéaires. Dans ce cas, nous formons d'abord un polygone convexe~$\mathcal{P}$ dont le bord est formé de segments de périodes $r_{1},\dots,r_{t}$ que nous ordonnons par argument croissant. Nous choisissons ensuite une diagonale $D$ coupant~$\mathcal{P}$ en deux polygones ayant respectivement $l_{1}+1$ et $l_{2}+2$ côtés. A la place de la diagonale $D$, nous insérons un parallélogramme dont les deux autres côtés parallèles sont de périodes $\lambda$ et $-\lambda$. Comme on peut choisir $\lambda$ à multiplication par $e^{\frac{2i\pi}{k}}$ près, le parallélogramme n'est jamais dégénéré. Sur le polygone obtenu, nous recollons $t$ cylindres pour obtenir la surface de translation voulue.
\par
Dans le cas dans lequel $r_{1},\dots,r_{t}$ sont tous réels, nous donnons une construction différente. Tout d'abord, nous pouvons construire une surface $X_{1}$ avec un lien-selle de bord en reliant les cylindres correspondant aux résidus $r_{1},\dots,r_{l_{1}}$ par des liens-selles horizontaux selon un certain arbre. L'hypothèse d'absence de sommes partielles nulles garantit que la construction est toujours possible. Nous construisons une seconde surface $X_{2}$ de la même façon avec les $l_{2}$ résidus restants. En recollant les bords de $X_{1}$ et $X_{2}$ sur les côtés opposés d'un parallélogramme dont les autres côtés sont de périodes $\lambda$ et $-\lambda$  (comme dans la première construction), nous obtenons la surface voulue.
\end{proof}

Une seconde chirurgie permet de se ramener à la construction de type quadrilatère ABAB de la section~\ref{sub:5N2negatif}.

\begin{lem}\label{lem:ABAB+}
Soit $(R_{1},\dots,R_{s})$ une configuration de $k$-résidus non nuls. Supposons qu'il existe $l_{1} \geq 1$ et des racines  $k$-ièmes $(r_{1},\dots,r_{s})$ de $(R_{1},\dots,R_{s})$ telles que :
\begin{enumerate}
 \item si les $l_{1}$ première racines $r_{i}$, resp. les $l_{2}+1$ dernières racines, sont $\RR$-colinéaires, alors aucune somme partielle des $r_{i}$ n'est nulle;
 \item  $S_{1} = \sum\limits_{i=1}^{l_{1}} r_{i}$ et $S_{2} = \sum\limits_{i= l_{1}+1}^{s} r_{i}$ satisfont $-\frac{S_{1}}{S_{2}}\in\mathcal{T}_{k,k+\bar a_{1}}$.
\end{enumerate}
La configuration $(R_{1},\dots,R_{s})$ est alors réalisable dans $\Omega^{k}\mathcal{M}_{0}(kl_{1}+\bar a_{1},kl_{2}+\bar a_{2};\rec[-k][s])$.
 \end{lem}

\begin{proof}
Grâce au point (2) on peut utiliser le lemme~\ref{lem:s2Modele} pour obtenir un quadrilatère~$Q$ de type ABAB avec deux côtés opposés égaux à $S_{1}$ et $S_{2}$. La somme des angles au bord de $S_{1}$ est $\frac{2\pi}{k}(k+\bar a_{1})$ et au bord de $S_{2}$ est $\frac{2\pi}{k}\bar a_{2}$.
\par
Considérons les $l_{1}$ premières racines $r_{i}$ que nous collons sur le segment $S_{1}$ (le cas des $l_{2}+1$ autres racines se traite de manière identique). Il y a deux cas selon que les $r_{i}$ soient $\RR$-colinéaires ou non.
\par
Si les racines ne sont pas $\RR$-colinéaires, on forme un polygone non dégénéré $\mathcal{P}_{1}$ avec le segment $S_{1}$ et les racines $r_{i}$. On colle alors $\mathcal{P}_{1}$ au segment $S_{1}$ de $Q$ puis des cylindres aux segments $r_{i}$. On obtient alors un angle total de $\frac{2\pi}{k}(k+\bar a_{1}) + \pi l_{1} +\pi l_{1}$ et on obtient donc une singularité d'ordre $l_{1}+\bar a_{1}$.
\par
Si les racines sont $\RR$-colinéaires, considérons une différentielle abélienne dans la strate $\omoduli[0](a;\rec[-1][l_{1}+1])$ dont les résidus sont les $r_{i}$ et $-S_{1}$. L'hypothèse d'absence de sommes partielles nulles permet d'appliquer le théorème~1.2 de \cite{getaab} pour garantir l'existence d'une telle différentielle. Nous enlevons alors le cylindre correspondant au pôle de résidu $-S_{1}$. Puis nous collons le bord de cette surface au segment $S_{1}$ de~$Q$. On vérifie comme ci-dessus que la singularité obtenue est d'ordre  $l_{1}+\bar a_{1}$.
 \end{proof}

\subsubsection{Cas des différentielles quartiques}\label{sub:mechantquartic}

Les estimations sur les sommes de racines (voir section~\ref{sec:estimation}) permettent de donner des constructions systématiques lorsque $k=4$.

\begin{prop}\label{prop:sec6quartique}
Pour $l_{1},l_{2} \geq 1$ et $s=l_{1}+l_{2}+1$, chaque configuration $(R_{1},\dots,R_{s})$ de nombres complexes non nuls est réalisable dans la strate $\Omega^{4}\mathcal{M}_{0}(a_{1},a_{2};\rec[-4][s])$ sauf dans les cas suivants :
\begin{enumerate}
    \item la strate $\Omega^{4}\mathcal{M}_{0}(5,3;\rec[-4][4])$ avec les résidus dans $\CC^{\ast}\cdot(1,1,1,1)$;
    \item la strate $\Omega^{4}\mathcal{M}_{0}(13,3;\rec[-4][6])$ avec les résidus dans $\CC^{\ast}\cdot(1,1,1,1,1,1)$.
\end{enumerate}
\end{prop}

\begin{proof}
Nous supposerons que les résidus sont ordonnés par module croissant. Si les $l_{1}$ premiers résidus ne forment pas une configuration de la forme $\CC^{\ast}\cdot(\rec[1][s_{1}],\rec[-4][s_{2}])$ avec $s_{2} \geq 1$ et $s_{1}+2s_{2} \in 4\ZZ$, le lemme~\ref{lem:estimate} permet de choisir des racines $r_{1},\dots,r_{l_{1}}$ pour chacun d'eux afin d'obtenir $0<|S_{1}| \leq |r_{l_{1}}|$ avec $S_{1}=\sum_{i=1}^{l_{1}} r_{i}$. De plus, si $l_{1} \geq 3$, nous pouvons exiger que ces racines ne soient pas $\mathbb{R}$-colinéaires (et si $l_{1}=2$ et que nous avons des racines $\mathbb{R}$-colinéaires nous réalisons la construction avec un triangle plat comme dans la preuve de la proposition~\ref{prop:ConstL2zerok5}). Nous choisissons les racines suivantes $r_{l_{1}+1},\dots,r_{s}$ en utilisant le lemme~\ref{lem:estimatecos}. Ainsi, nous obtenons $S_{2}=|\sum_{i=l_{1}+1}^{s} r_{i}| \geq 2\cos\left(\frac{\pi}{8}\right) |r_{l_{1}+1}|$. Puisque $2\cos\left(\frac{\pi}{8}\right)=\sqrt{2+\sqrt{2}}$, nous avons $\frac{|S_{1}|}{|S_{2}|}<\cos\left(\frac{\pi}{4}\right)$. Ainsi, les lemmes~\ref{lem:trigo} et~\ref{lem:ABAB+} permettent de construire la différentielle voulue.
\par
Si au contraire nous avons $s_{1}+2s_{2} \in 4\ZZ$ (en particulier $l_{1} \geq 2$), le lemme~\ref{lem:estimate} donne une inégalité plus faible. Après normalisation, la configuration des $l_{1}$ premiers résidus est donc $(\rec[1][s_{1}],\rec[-4][s_{2}])$. Nous pouvons obtenir
$|S_{1}| = 2$ et le lemme~\ref{lem:estimatecos} donne $|S_{2}| \geq \cos\left(\frac{\pi}{4}\right)(l_{2}+1)|r_{l_{1}+1}|\geq l_{2}+1$. En particulier, quand $l_{2} \geq 2$, on a $|S_{2}| \geq 3$ et $\frac{|S_{1}|}{|S_{2}|} \leq \frac{2}{3}<\cos\left(\frac{\pi}{4}\right)$. De même, quand $|r_{l_{1}+1}|>2$, on obtient $|S_{2}|>2\sqrt{2}$ et la construction fonctionne également.
\par
Nous supposerons donc que $l_{2}=1$ (soit $s=l_{1}+2$) et $\sqrt{2}\leq |r_{l_{1}+1}| \leq 2$. De plus, si $|r_{s}|\geq 2$, alors il suffit de choisir $r_{l_{1}+1}$ avec un argument distant d'au plus $\frac{\pi}{4}$ de celui de $r_{s}$ pour obtenir $|r_{l_{1}+1}+r_{s}|\geq 2+\frac{\sqrt{2}}{\sqrt{2}}$. Comme $3>2\sqrt{2}$, la construction fonctionne aussi dans ce cas. 
\par
Ainsi, nous pouvons nous restreindre aux configurations où $\sqrt{2} \leq |r_{l_{1}+1}|\leq |r_{s}|<2$. Traitons d'abord le cas $R_{l_{1}+1}=R_{s}$ avec $|r_{l_{1}+1}|>\sqrt{2}$. Nous pouvons choisir des racines identiques et obtenir $|S_{2}|>2\sqrt{2}$. La construction fonctionne dans ce cas. Si au contraire $R_{l_{1}+1}=R_{s}$ avec $|r_{l_{1}+1}|=\sqrt{2}$, alors nous pouvons permuter les indices entre $l_{1}$, $l_{1}+1$ et $s$. Si nous obtenons dans chaque cas une configuration exceptionnelle pour les $l_{1}$ premiers résidus (ici $l_{1} \geq 2$), c'est donc que $(R_{1},\dots,R_{s})=(\rec[1][s_{1}],\rec[-4][s_{2}+2])$. Si $s_{1} \geq 1$, nous choisissons pour $S_{2}$ la somme d'une racine de $1$ et d'une racine de $-4$ tandis que nous aurions un nombre impair de racines de $1$ dans $S_{1}$. Obtenant $|S_{1}|=\sqrt{2}$ et $|S_{2}|=\sqrt{5}$, la construction fonctionne. Si au contraire $s_{1}=0$, les configurations restantes à traiter sont donc de la forme $\rec[1][s]$ avec $s$ pair dans les strates $\Omega^{4}\mathcal{M}_{0}(4s-11,3;\rec[-4][s])$. Les obstructions (9) et (11) du théorème~\ref{thm:geq0kspe} constituent les cas $s=4$ et $s=6$. Quand $s=8$, nous avons une construction obtenue en partant du décagone reliant les sommets $0,1,2,3,3+i,3+2i,2+2i,1+2i,2+i,1+i$, identifiant les segments $[1+2i,2+i]$ et $[1+i,0]$ puis collant des cylindres sur les autres segments. Nous obtenons bien la configuration $\rec[1][8]$ dans la strate $\Omega^{4}\mathcal{M}_{0}(21,3;\rec[-4][8])$. Pour obtenir tous les cas suivants avec $s$ pair, il suffit de remplacer un cylindre (dont le bord est de longueur $1$) par un carré de côté $1$ sur lequel on colle trois cylindres. Cette chirurgie augmente le nombre de cylindres de $2$.
\par
Dans le cas $R_{l_{1}+1} \neq R_{s}$, nous supposons d'abord que $s_{2} \geq 2$. Pour la somme $S_{2}$, nous choisissons deux racines de $-4$ égales tandis que la somme $S_{1}$ est faite de racines des $l_{1}-2$ premiers résidus ainsi que de racines de $R_{l_{1}+1}$ et $R_{s}$. Nous avons donc $|S_{2}|=2\sqrt{2}$ ainsi que $|S_{1}|<2$ sauf si $(R_{1},\dots,R_{l_{1}-2},R_{l_{1}+1},R_{s})$ est une configuration exceptionnelle (voir le lemme~\ref{lem:estimate}). En dehors de ce cas, la construction est réalisable. S'il s'agit bien d'une configuration exceptionnelle, alors l'hypothèse $R_{l_{1}+1} \neq R_{s}$ implique que cette configuration est en fait proportionnelle à $(\rec[-4][l_{1}-1],16)$. Ceci implique que $|r_{l_{1}+1}|=\sqrt{2}$ et $|r_{s}|=2$. Quitte à échanger $R_{l_{1}+1}$ avec l'un des résidus valant $-4$ et réitérant le raisonnement, on en déduit que que soit la construction peut encore se faire, soit $R_{l_{1}+1}$ vaut aussi $-4$ (et donc $R_{s}=16$). Il reste donc seulement à traiter le cas dans lequel c'est donc toute la configuration $(R_{1},\dots,R_{s})$ qui vaut $(\rec[-4][s-1],16)$. Dans ce cas, on pose $S_{2}=(1+i)+2$ pour avoir $|S_{2}|=\sqrt{10}$ tandis que $S_{1}$ est une somme de racines de $-4$ que l'on peut choisir telle que $|S_{1}| \leq 2$. Ainsi, on a $\frac{|S_{1}|}{|S_{2}|} \leq \frac{\sqrt{2}}{\sqrt{5}}<\cos\left(\frac{\pi}{4}\right)$.
\par
Il reste à traiter le cas $s_{2}=1$. Comme nous avons aussi $s_{1}+2s_{2} \in 4\mathbb{Z}$, il s'ensuit que $s_{1} \geq 2$ et donc $l_{1} \geq 3$. Nous constituons la somme $S_{2}$ avec une racine $r_{s}$ de $R_{s}$ et une racine de $-4$ de telle façon qu'elles aient des arguments  distants d'au plus $\frac{\pi}{4}$. Ainsi, nous avons $|S_{2}| \geq |r_{s}|+\frac{\sqrt{2}}{\sqrt{2}}$. La somme $S_{1}$ est constituée des racines de $(1^{s_{1}},R_{l_{1}+1})$ et, sauf si cette configuration est exceptionnelle, nous avons $|S_{1}|<|r_{s}|$. Ainsi, nous avons $\frac{|S_{1}|}{|S_{2}|} < \frac{|r_{s}|}{|r_{s}|+1} \leq \frac{2}{3}$ puisque $|r_{s}| \leq 2$. La construction est donc réalisable dans ce cas. Si au contraire la configuration $(\rec[1][s_{1}],R_{l_{1}+1})$ est exceptionnelle pour le lemme~\ref{lem:estimate}, alors on a $R_{l_{1}+1}=-4$.
\par
Nous nous restreignons donc aux configurations $(\rec[1][s_{1}],-4,-4,R_{s})$. Nous choisissons deux racines identiques de $-4$ pour constituer $S_{2}$ et obtenir $|S_{2}|=2\sqrt{2}$. Avec les racines restantes, on obtient $|S_{1}|<|r_{s}|$ sauf si la configuration $(\rec[1][s_{1}],R_{s})$ est exceptionnelle. Ainsi, soit nous avons $\frac{|S_{1}|}{|S_{2}|}<\frac{1}{\sqrt{2}}$, soit nous avons $R_{s}=-4$. Dans le premier cas la construction est réalisable, dans le second cas, on a $R_{l_{1}+1}=R_{s}$, cas déjà traité plus haut.
\end{proof}

\subsubsection{Cas des différentielles cubiques}\label{sub:mechantcubic}

Notre stratégie repose ici sur l'étude de la stratification de résonance et sa structure d'incidence (voir la section~\ref{sec:arrhyp}). Pour les strates de résonance dans lesquelles une somme partielle pondérée de racines s'annule (autrement dit vérifiant une équation de résonance dans laquelle au moins un coefficient est nul), les conditions sont réunies pour utiliser la chirurgie du lemme~\ref{lem:chirurgietranslation}. Ensuite, le prolongement plat du corollaire~\ref{cor:defplate} permet d'étendre la constructions aux strates de résonance dont l'adhérence dans l'espace résiduel contient de telles strates.
\par
En général, la complexité combinatoire des équations de résonance rend ce type de raisonnement inextricable. Toutefois, lorsque le nombre de pôles $s$ est plus grand que l'ordre $k$ des différentielles, les équations de résonance ont facilement des coefficients en commun, engendrant des annulations de sommes partielles pondérée de racines (qui permettent l'utilisation de la chirurgie du lemme~\ref{lem:chirurgietranslation}).
\par
Nous raisonnerons par récurrence. Toutefois, les nombreuses obstructions du théorème~\ref{thm:geq0kspe} exigent de traiter à la main les strates $\Omega^{3}\mathcal{M}_{0}(a_{1},a_{2};\rec[-3][s])$ pour lesquelles le nombre de pôles est petit (lorsque $3 \leq s \leq 6$), ce que nous faisons dans les lemmes suivants.

\begin{lem}\label{lem:strateE1}
Chaque configuration $(R_{1},R_{2},R_{3})$ de nombres complexes non nuls est réalisable dans la strate $\Omega^{3}\mathcal{M}_{0}(1,2;\rec[-3][3])$ à l'exception de la droite $\CC^{\ast}\cdot(1,1,1)$.
\end{lem}

\begin{proof}
D'après le corollaire~\ref{cor:defplate}, il suffit de donner une construction pour une différentielle de chaque strate de résonance de l'espace des configurations de résidus cubiques. Chaque strate résonance est caractérisée par une famille d'équations linéaires reliant les racines cubiques $r_{1},r_{2},r_{3}$ des $3$-résidus et dont les coefficients non nuls sont des racines troisièmes de l'unité.
%\Quentin{Je pense que le lemme d'incidence permet de se limiter aux strates de résonance de dimension $1$. Donc d'enlever les deux prochains paragraphes}
\par
Chaque configuration de la forme $(R_{1},R_{2},R_{3})$ avec $R_{2}=-R_{1}$ peut s'obtenir grâce à la chirurgie du lemme~\ref{lem:chirurgietranslation} puisque la configuration $(R_{3})$ est réalisable dans la strate $\Omega^{3}\mathcal{M}_{0}(-2,-1;-3)$. Le prolongement plat donne une construction pour toutes les strates de résonance incidentes.
\par
Parmi les strates de résonance de codimension un, il ne reste à considérer que le cas dans lequel l'équation $r_{1}+r_{2}+r_{3}=0$. Or, une telle strate est incidente à une strate de codimension deux dans laquelle l'équation $r_{1}+\omega r_{2}=0$ est de surcroît satisfaite. La chirurgie du paragraphe précédent s'appliquant à cette strate, nous pouvons appliquer le prolongement plat pour régler le cas de cette strate de codimension un.
\par
Les dernières strates de résonance qu'il faut considérer sont définies par le système :
$$\left \{
   \begin{array}{r c l}
      r_{1}+r_{2}+r_{3}  & = & 0 \\
      r_{1}+\alpha r_{2} + \beta r_{3}   & = & 0
   \end{array}
   \right .$$
où $\alpha,\beta$ sont des racines troisièmes de l'unité. Elles ne peuvent pas valoir $1$ ou être identiques sinon l'un des résidus s'annule. Sans perte de généralité, nous avons donc $\alpha=\omega$ et $\beta=\omega^{2}$. En normalisant par $r_{3}=1$, ce système a une unique solution $(\alpha,\beta,1)$. Les résidus cubiques sont de cette strate de résonance sont donc tous proportionnels à $(1,1,1)$. Il s'agit donc de l'obstruction (3) du théorème~\ref{thm:geq0kspe}.
\end{proof}

\begin{lem}\label{lem:strateE2}
Chaque configuration $(R_{1},\dots,R_{4})$ de nombres complexes non nuls est réalisable dans la strate $\Omega^{3}\mathcal{M}_{0}(1,5;\rec[-3][4])$. Chacune est aussi réalisable dans la strate $\Omega^{3}\mathcal{M}_{0}(4,2;\rec[-3][4])$ à l'exception de celles qui sont dans la droite $\CC^{\ast}\cdot(1,1,-1,-1)$.
\end{lem}

\begin{proof}
La méthode est analogue à celle employée dans le lemme~\ref{lem:strateE1}. Lorsqu'une somme partielle pondérée de racines s'annule, nous employons la chirurgie du lemme~\ref{lem:chirurgietranslation} pour se ramener à une strate plus simple. Quand aucune somme partielle de deux racines ne s'annule mais qu'une somme partielle de trois racines s'annule, la chirurgie est toujours réalisable. Lorsque c'est une somme partielle de deux racines de résidus qui s'annule, la chirurgie est possible sauf pour les configurations de la forme $(R_{1},R_{2},-R_{1},-R_{2})$. Supposons tout d'abord que nous ne nous trouvons pas dans une telle configuration.
\par
La strate de codimension un définie par l'équation $r_{1}+r_{2}+r_{3}+r_{4}$ est incidente à la strate de codimension deux obtenue en ajoutant la contrainte $r_{1}+\omega r_{2}=0$, pour laquelle la chirurgie permet de conclure.
\par
Considérons ensuite la strate de codimension deux définie par les équations
$$\left \{
   \begin{array}{r c l}
      r_{1}+r_{2}+r_{3}+r_{4} &=&0 \\
      r_{1}+\alpha r_{2} + \beta r_{3} + \gamma r_{4}   & = & 0
   \end{array}
   \right .$$
dans lesquelles $\alpha,\beta,\gamma$ sont des racines troisièmes de l'unité. Si exactement deux coefficients parmi $1,\alpha,\beta,\gamma$ sont identiques, nous obtenons une somme partielle de deux racines en combinant les deux équations, ce qui permet de conclure. Dans le cas contraire, trois de ces coefficients valent $1$, disons $\alpha=\beta=1$ sans perte de généralité, tandis que $\gamma \neq 1$ et nous obtenons $R_{4}=0$, ce qui est impossible.
\par
Il reste donc à traiter le cas des configurations de la forme $(R_{1},R_{2},-R_{1},-R_{2})$ dont nous prenons les racines $(r_{1},r_{2},-r_{1},-r_{2})$. Elles forment l'adhérence d'une strate résiduelle de codimension deux qui est incidente à précisément deux strates de codimension trois (contenues dans des droites vectorielles donc) $\CC^{\ast}\cdot(1,1,-1,-1)$ et
$\CC^{\ast}\cdot(1,3i\sqrt{3},-1,-3i\sqrt{3})$.
\par
En effet, si l'on ajoute comme contrainte l'annulation d'une somme partielle pondérée de deux racines, reliant $r_{1}$ ou $-r_{1}$ avec $r_{2}$ ou $-r_{2}$, on obtient la strate $\CC^{\ast}\cdot(1,1,-1,-1)$. Si la nouvelle contrainte est l'annulation d'une somme pondérée de trois racines, alors, quitte à échanger $r_{1}$ et $r_{2}$, à changer leur signe ou à les multiplier par $\omega$, nous obtenons $r_{1} - \omega r_{1} + r_{2} =0$. La strate obtenue est alors $\CC^{\ast}\cdot(1,3i\sqrt{3},-1,-3i\sqrt{3})$. Si enfin la nouvelle contrainte est l'annulation d'une somme pondérée des quatre racines, cette nouvelle équation est de la forme $\alpha r_{1} + \beta r_{2}$ dans laquelle les coefficients $\alpha,\beta$ (non nuls sinon on obtient une strate avec un résidu nul) sont des différences entre deux racines troisièmes de l'unité. Deux telles différence ont des cubes identiques ou opposés en signe, la strate que nous obtenons est donc aussi $\CC^{\ast}\cdot(1,1,-1,-1)$.
\par
Du fait du prolongement plat, il nous suffit alors de construire la configuration $(1,1,-1,-1)$ dans la strate $\Omega^{3}\mathcal{M}_{0}(1,5;\rec[-3][4])$ (il s'agit d'une obstruction pour $\Omega^{3}\mathcal{M}_{0}(4,2;\rec[-3][4])$), ainsi que la configuration $(1,3i\sqrt{3},-1,-3i\sqrt{3})$ pour les deux strates.
\par
Comme la configuration $(1,-1)$ est réalisable dans la strate $\Omega^{3}\mathcal{M}_{0}(2,-2;-3)$, la chirurgie permet de réaliser les configurations $(1,1,-1,-1)$ et $(1,3i\sqrt{3},-1,-3i\sqrt{3})$ dans la strate $\Omega^{3}\mathcal{M}_{0}(1,5;\rec[-3][4])$. En revanche, nous donnons une construction spécifique pour la configuration $(1,3i\sqrt{3},-1,-3i\sqrt{3})$ dans la strate $\Omega^{3}\mathcal{M}_{0}(4,2;\rec[-3][4])$.
\par
Nous considérons l'hexagone dont les sommets sont $0,1,3+\omega,2+2\omega,1+2\omega,2+\omega$. En identifiant les côtés $[3+\omega,2+2\omega]$ et $[2+\omega,0]$ par rotation puis en collant des cylindres sur les quatre côtés restants, nous obtenons la différentielle réalisant la configuration voulue dans la strate $\Omega^{3}\mathcal{M}_{0}(4,2;\rec[-3][4])$.
\end{proof}

Pour les strates ayant $s=5$ pôles triples, une certaine famille de configurations présente des difficultés spécifiques et nous la traitons séparément.

\begin{lem}\label{lem:strateE3SPEC}
Chaque configuration $(R_{1},R_{1},R_{1},R_{2},-R_{2})$ de nombres complexes non nuls est réalisable dans les strates $\Omega^{3}\mathcal{M}_{0}(1,8;\rec[-3][5])$ et $\Omega^{3}\mathcal{M}_{0}(4,5;\rec[-3][5])$. Chacune est aussi réalisable dans la strate $\Omega^{3}\mathcal{M}_{0}(7,2;\rec[-3][5])$ à l'exception de celles qui sont dans la droite $\CC^{\ast}\cdot(1,1,1,1,-1)$.
\end{lem}

\begin{proof}
Puisque toute configuration de la forme $(R_{2},-R_{2})$ est réalisable dans la strate $\Omega^{3}\mathcal{M}_{0}(-2,2;\rec[-3][2])$, la chirurgie du lemme~\ref{lem:chirurgietranslation} permet de traiter le cas des strates $\Omega^{3}\mathcal{M}_{0}(1,8;\rec[-3][5])$ et $\Omega^{3}\mathcal{M}_{0}(4,5;\rec[-3][5])$.
\par
Les configurations $(R_{1},R_{1},R_{1},R_{2},-R_{2})$ constituent l'adhérence de la strate de résonance donnée par les équations
$$\left \{
   \begin{array}{r c l}
      r_{1} + r_{2} + r_{3} &=&0 \\
      r_{1} + \omega r_{2} + \omega^{2} r_{3} &=&0\\
      r_{4} +r_{5} & = & 0
   \end{array}
   \right .$$
Notons que ces équations impliquent en particulier $r_{1}=\omega r_{3}$ et $r_{2}=\omega^{2}r_{3}$. L'adhérence de cette strate contient un certain nombre de strates unidimensionnelles (correspondant au choix d'une équation de résonance supplémentaire, voir section~\ref{sec:arrhyp}) que nous listons ici:
\begin{enumerate}
    \item la droite $\CC^{\ast}\cdot(1,1,1,1,-1)$ pour l'équation $r_{2}+r_{5}=0$;
    \item la droite $\CC^{\ast}\cdot(1,1,1,8,-8)$ pour $\omega r_{2}+r_{3}+r_{5}=0$;
    \item la droite $\CC^{\ast}\cdot(3i\sqrt{3},3i\sqrt{3},3i\sqrt{3},1,-1)$ pour $r_{3}+\omega r_{4}+r_{5}=0$;
    \item la droite $\CC^{\ast}\cdot(1,1,1,27,-27)$ pour $r_{1}+\omega^{2}r_{2}+\omega r_{3}+r_{5}=0$;
    \item la droite $\CC^{\ast}\cdot(1,1,1,3i\sqrt{3},-3i\sqrt{3})$ pour $r_{1}+r_{2}+\omega r_{3}+r_{5}=0$;    
    \item la droite $\CC^{\ast}\cdot(3i\sqrt{3},3i\sqrt{3},3i\sqrt{3},8,-8)$ pour $r_{2}+r_{3}+\omega r_{4}+r_{5}=0$.
\end{enumerate}
La première est l'obstruction~(5) du théorème~\ref{thm:geq0kspe}. La deuxième présente l'annulation d'une somme partielle pondérée de trois racines et s'obtient en partant de la configuration $(1,8)$ dans une strate plus simple (on en déduit le cas générique de la strate de résonance bidimensionnelle par le prolongement plat du corollaire~\ref{cor:defplate}). Idem pour la troisième en partant de la configuration $(3i\sqrt{3},3i\sqrt{3})$. Les trois dernières configurations présentent l'annulation d'une somme partielle pondérée de quatre racines et s'obtient donc aussi par chirurgie.
\par
On observe qu'ajouter la contrainte de l'annulation d'une somme pondérée des cinq racines $\alpha r_{1} + \beta r_{2} + \gamma r_{3} + \delta r_{4} + r_{5}=0$, produit l'équation $(\alpha \omega + \beta \omega^{2} + \gamma) r_{3} + (\delta -1) r_{4}=0$. Pour chaque choix de $\alpha,\beta,\gamma,\delta$, nous retrouvons l'une des strates déjà obtenues.
\end{proof}

Pour les strates ayant $s=6$ pôles triples, une autre famille de configurations présente des difficultés spécifiques.

\begin{lem}\label{lem:strateE4SPEC}
Chaque configuration $(R_{1},R_{1},R_{1},R_{2},R_{2},R_{2})$ de nombres complexes non nuls est réalisable dans les strates $\Omega^{3}\mathcal{M}_{0}(1,11;\rec[-3][6])$ et $\Omega^{3}\mathcal{M}_{0}(4,8;\rec[-3][6])$. Chacune est aussi réalisable dans les strates $\Omega^{3}\mathcal{M}_{0}(7,5;\rec[-3][6])$ et $\Omega^{3}\mathcal{M}_{0}(10,2;\rec[-3][6])$ à l'exception de celles qui sont dans la droite $\CC^{\ast}\cdot(1,1,1,1,1,1)$.
\end{lem}

\begin{proof}
Puisque toute configuration de la forme $(R_{1},R_{1},R_{1})$ est réalisable dans la strate $\Omega^{3}\mathcal{M}_{0}(-2,5;\rec[-3][3])$, la chirurgie du lemme~\ref{lem:chirurgietranslation} permet de traiter le cas des strates $\Omega^{3}\mathcal{M}_{0}(1,11;\rec[-3][6])$ et $\Omega^{3}\mathcal{M}_{0}(4,8;\rec[-3][6])$.
\par
Pour terminer la démonstration, il suffit donc de donner une construction pour des configurations $(R,R,R,1,1,1)$ avec $|R| \leq 1$ et $R \neq 1$ dans les strates $\Omega^{3}\mathcal{M}_{0}(7,5;\rec[-3][6])$ et $\Omega^{3}\mathcal{M}_{0}(10,2;\rec[-3][6])$. Notons $r$ une racine de $R$.
\par
Pour la strate $\Omega^{3}\mathcal{M}_{0}(7,5;\rec[-3][6])$, la construction ABAB (voir lemme~\ref{lem:ABAB+}) est réalisable en prenant $S_{2}=1+1+r$ (avec $\Re(r)>0$) pour avoir $|S_{2}|>2$ et en choisissant des racines de $R,R,1$ pour avoir $|S_{1}|<1$, ce qui est possible d'après le lemme~\ref{lem:estimate} sauf si $R=-1$. Dans ce dernier cas, nous donnons la construction spécifique de la configuration $(1,1,1,-1,-1,-1)$ dans la strate $\Omega^{3}\mathcal{M}_{0}(5,7;\rec[-3][6])$. Nous considérons le polygone dont les huit sommets sont les nombres complexes $0,-1,\omega^{2},-\omega,1,1+\omega,1+2\omega,\omega$. En identifiant les côtés $[0,-1]$ et $[1,1+\omega]$ par une rotation, nous obtenons la différentielle voulue.
\par
Pour la strate $\Omega^{3}\mathcal{M}_{0}(10,2;\rec[-3][6])$, la construction ABAB (voir lemme~\ref{lem:ABAB+}) est réalisable en prenant $S_{2}=1+1=2$ et en choisissant des racines de $R,R,R,1$ dont la somme $S_{1}$ vérifie $|S_{1}|< 1$, ce qui possible d'après le lemme~\ref{lem:estimate} sauf si $R=\pm \frac{i\sqrt{3}}{9}$ ou $R=-1$. Dans ce dernier cas, la chirurgie du lemme~\ref{lem:chirurgietranslation} donne la construction en partant de la configuration $(1,1,-1,-1)$ dans la strate $\Omega^{3}\mathcal{M}_{0}(7,-1;\rec[-3][4])$.
\par
Dans les cas restants, nous supposerons (à une conjugaison près) que $R= \frac{i\sqrt{3}}{9} = \frac{1+2\omega}{9}$. En choisissant la racine $r=\frac{2+\omega}{3}=\frac{1}{2+\omega^{2}}$, nous avons $\omega + \omega^{2} + r + r + \omega^{2} r = 0$, ce qui permet d'utiliser la chirurgie du lemme~\ref{lem:chirurgietranslation} en partant de la configuration $(1)$ dans la strate $\Omega^{3}\mathcal{M}_{0}(-2,-1;-3)$.
\end{proof}

La chirurgie du lemme~\ref{lem:ABAB+} est plus aisément réalisable lorsque $l_{2}$ est suffisamment grand. C'est pourquoi nous traitons aussi séparément le cas $l_{2}=1$.

\begin{prop}\label{prop:sec6cubiquel2=1}
Pour $l_{1} \geq 4$ et $s=l_{1}+2$, chaque configuration $(R_{1},\dots,R_{s})$ de nombres complexes non nuls est réalisable dans la strate $\Omega^{3}\mathcal{M}_{0}(3l_{1}-2,2;\rec[-3][s])$ à l'exception des configurations de la droite $\CC^{\ast}\cdot(1,1,1,1,-1)$ pour la strate $\Omega^{3}\mathcal{M}_{0}(7,2;\rec[-3][6])$ et de la droite $\CC^{\ast}\cdot(1,1,1,1,1,1)$ pour la strate $\Omega^{3}\mathcal{M}_{0}(10,2;\rec[-3][6])$.
\end{prop}

\begin{proof}
Les cas $s \leq 4$ sont couverts par le corollaire~\ref{cor:FinalL2zero} ainsi que les lemmes~\ref{lem:strateE1} et~\ref{lem:strateE2}. Nous supposerons donc $s \geq 5$. Le cas des configurations génériques (ne vérifiant aucune condition de résonance) s'obtient par prolongement plat puisque la strate est non vide (voir section~\ref{sec:arrhyp}).
\par
Pour la strate $\Omega^{3}\mathcal{M}_{0}(3l_{1}-2,2;\rec[-3][s])$,
nous traitons d'abord le cas de configurations $(R_{1},\dots,R_{s})$ dans lesquelles une somme partielle pondérée de racines s'annule. La chirurgie du lemme~\ref{lem:chirurgietranslation} permet la construction d'une telle configuration en partant d'une strate de la forme $\Omega^{3}\mathcal{M}_{0}(a_{1},-1;\rec[-3][t])$ et d'une configuration plus simple. D'après le corollaire~\ref{cor:FinalL2zero}, les seuls cas problématiques sont les obstructions (1) et (2) du théorème~\ref{thm:geq0kspe}. Autrement dit, les configurations pour lesquelles la chirurgie pourrait ne pas être réalisable sont de la forme $(R_{1},\dots,R_{s-1},-R_{s-1})$ ou $(R_{1},\dots,R_{s-2},R_{s-2},R_{s-2})$. Dans ces derniers cas, nous avons l'annulation d'une somme de deux ou trois racines, permettant la chirurgie sauf à tomber de nouveau sur une obstruction. Puisque $s \geq 5$, les seules configurations problématiques sont de la forme $(R_{1},R_{1},R_{1},R_{2},-R_{2})$ et $(R_{1},R_{1},R_{1},R_{2},R_{2},R_{2})$. Elles sont traitées dans les lemmes~\ref{lem:strateE3SPEC} et~\ref{lem:strateE4SPEC}.
\par
Il reste à traiter le cas des strates de résonance ne vérifiant aucune annulation de somme partielle pondérée de racines mais seulement des conditions de résonance pour lesquelles tous les coefficients sont non nuls. Considérons une telle configuration de résidus cubiques $R_{1},\dots,R_{s}$ ayant des racines cubiques $r_{1},\dots,r_{s}$ vérifiant $r_{1}+\dots+r_{s}=0$. Nous supposerons d'abord que deux racines (disons $r_{1}$ et $r_{2}$) vérifient $\arg(r_{2}/r_{1})\in ]\frac{2\pi}{3},\pi[$. Dans ce cas, nous construisons un triangle $\mathcal{T}$ dont les côtés sont donnes par les complexes $r_{1},r_{2},-r_{1}-r_{2}$ (en particulier, l'angle entre $r_{1}$ et $r_{2}$ est strictement plus petit que $\frac{\pi}{3}$). Ensuite, nous construisons un polygone convexe $\mathcal{P}$ dont les côtés sont donnes par les nombres complexes $r_{1}+r_{2},r_{3},\dots,r_{s}$ (ordonnés cycliquement en fonction de leur argument). Nous supposons également que ces nombres ne sont pas $\mathbb{R}$-colinéaires. Maintenant, nous collons $\mathcal{P}$ et $\mathcal{T}$ l'un sur l'autre, puis ajoutons des cylindres infinis sur les côtés correspondant aux nombres complexes $r_{2},\dots,r_{s}$. Enfin, à l'intérieur du cylindre correspondant à $r_{2}$, nous coupons une cicatrice de longueur $|r_{1}|$ qui fasse un angle de $\frac{\pi}{3}$ avec le côté correspondant à $r_{1}$. Sur le bord gauche de la cicatrice, nous collons un cylindre tandis que le bord droit de la cicatrice est identifié avec côté correspondant à $r_{1}$. La surface obtenue a des cylindres correspondant aux bons résidus cubiques et l'une de ses singularités coniques est par construction d'angle $\frac{10\pi}{3}$ ($2\pi$ pour le bout de la cicatrice, $\pi$ pour le cylindre, $\frac{\pi}{3}$ pour l'angle entre un côté du triangle $\mathcal{T}$ et la cicatrice).
\par
En fait, même si les $r_{1}+r_{2},r_{3},\dots,r_{s}$ sont $\mathbb{R}$-colinéaires, en l'absence d'annulation de sommes partielles pondérées, la construction demeure réalisable avec un polygone $\mathcal{P}$ dégénéré. De plus, si $\arg(r_{2}/r_{1})=\pi$, la construction fonctionne aussi tant que $|r_{1}|<|r_{2}|$. Puisque $r_{1} \neq -r_{2}$ (aucune somme partielle de racines ne s'annule) et que l'on peut permuter les $r_{1},\dots,r_{s}$, la construction fonctionne dès que l'on peut trouver deux racines $r_{i},r_{j}$ telles que $\arg(r_{j}/r_{i}) \in ]\frac{2\pi}{3},\frac{4\pi}{3}[$.
\par
Puisqu'il y a $s \geq 5$ racines et que leur somme est nulle, il n'y a qu'un cas dans lequel leurs arguments ne présentent pas de tels couples si tous les arguments sont dans $\lbrace{ - \frac{2\pi}{3},0,\frac{2\pi}{3} \rbrace}$ (à un facteur commun près). Nous nous ramenons donc au cas dans lequel tous les résidus cubiques $R_{1},\dots,R_{s}$ sont des réels positifs.
\par
Nous donnons une construction différente dans ce dernier cas en utilisant le lemme~\ref{lem:ABAB+}. Les résidus sont ordonnés par module croissant. Nous pouvons choisir les racines des deux derniers résidus pour avoir $|S_{2}| = |r_{s-1}|+|r_{s}|$ (on choisit des racines ayant même argument). Le fait de ne pas avoir d'annulation de somme partielle simplifie la preuve du lemme~\ref{lem:estimate} et permet de choisir les racines des $l_{1}$ premiers résidus pour obtenir $|S_{1}| \leq |r_{l_{1}}|$. La construction est donc réalisable sauf éventuellement si toutes ces inégalités sont des égalités. Ce cas n'est possible que si $R_{1}=\dots=R_{s}$ et il s'ensuit qu'une certaine somme de trois racines cubiques s'annule. Ce dernier cas est donc déjà couvert.
\end{proof}

Nous pouvons terminer la caractérisation des configurations de résidus cubiques réalisables dans une strate donnée lorsque tous les pôles sont triples et qu'il n'y a que deux zéros.

\begin{prop}
L'application résiduelle des strates $\Omega^{3}\mathcal{M}_{0}(a_{1},a_{2};\rec[-3][s])$ est surjective sauf dans les cas suivants :
\begin{enumerate}
\item L'image de $\appresk[0](-1,1;-3,-3)$ est $(\CC^{\ast})^{2}\setminus \CC^{\ast}\cdot(1,-1)$.
\item L'image de $\appresk[0][3](-1,4;\rec[-3][3])$ est $(\CC^{\ast})^{3}\setminus \CC^{\ast}\cdot\rec[1][3]$.
\item L'image de $\appresk[0][3](1,2;\rec[-3][3])$ est $(\CC^{\ast})^{3}\setminus \CC^{\ast}\cdot\rec[1][3]$.
\item L'image de $\appresk[0][3](2,4;\rec[-3][4])$ est $(\CC^{\ast})^{4}\setminus \CC^{\ast}\cdot(1,1,-1,-1)$.
\item L'image de $\appresk[0][3](2,7;\rec[-3][5])$ est $(\CC^{\ast})^{5}\setminus \CC^{\ast}\cdot(\rec[1][4],-1)$.
\item L'image de $\appresk[0][3](2,10;\rec[-3][6])$ est $(\CC^{\ast})^{6}\setminus \CC^{\ast}\cdot\rec[1][6]$.
\item L'image de $\appresk[0][3](5,7;\rec[-3][6])$ est $(\CC^{\ast})^{6}\setminus \CC^{\ast}\cdot\rec[1][6]$.
\end{enumerate}
\end{prop}

\begin{proof}
Il s'agit d'une démonstration par récurrence sur le nombre $s$ de pôles, sachant que les cas $s \leq 4$, $l_{1}=0$ et $l_{2} \leq 2$ sont couverts par le corollaire~\ref{cor:FinalL2zero}, les lemmes~\ref{lem:strateE1} et~\ref{lem:strateE2} ainsi que la proposition~\ref{prop:sec6cubiquel2=1}. Nous supposerons donc la proposition valide jusqu'au rang $s-1$.
\par
Pour une strate $\Omega^{3}\mathcal{M}_{0}(a_{1},a_{2};\rec[-3][s])$, nous traitons d'abord le cas des configurations de résidus $(R_{1},\dots,R_{s})$ dans lesquelles une somme partielle pondérée de racines s'annule. La chirurgie du lemme~\ref{lem:chirurgietranslation} permet la construction d'une telle configuration, sauf si elle s'opère en partant d'une strate et d'une configuration relevant des obstructions (1) à (7) du théorème~\ref{thm:geq0kspe}. De telles configurations présentent aussi des annulations de sommes partielles, qui permettent en retour d'autres chirurgies possibles pour $(R_{1},\dots,R_{s})$. On en déduit aisément que les configurations potentiellement problématiques sont toutes proportionnelles à $(\rec[1][\alpha],\rec[-1][\beta])$ avec $\alpha \geq \beta$, $(1,1,1,R,-R)$ ou $(1,1,1,R,R,R)$ avec $R \in \mathbb{C}^{\ast}$ (ces deux derniers cas sont immédiatement évacués grâce aux lemmes~\ref{lem:strateE3SPEC} et~\ref{lem:strateE4SPEC}).
\par
Puisque $s \geq 5$, nous avons $\alpha \geq 3$ et nous pouvons appliquer la chirurgie en retirant trois résidu cubiques égaux à $1$. Il y a alors le choix entre deux strates de départ (sauf si $l_{1}=1$ mais alors la seule strate possible $\Omega^{3}\mathcal{M}_{0}(-2,a_{2}-6;\rec[-3][s-3])$ ne présente aucune obstruction) et le seul cas dans lesquelles elles présentent toutes deux des obstructions est celui des strates $\Omega^{3}\mathcal{M}_{0}(7,5;\rec[-3][6])$. Ce cas est aussi couvert par le lemme~\ref{lem:strateE4SPEC}.
\par
Nous supposerons qu'il n'y aucune annulation de somme partielle pondérée de racines et que les résidus sont ordonnés par module croissant. Nous utilisons la construction du lemme~\ref{lem:ABAB+}. Posons $|r|=|r_{l_{1}}|$. Nous pouvons choisir les racines de $l_{2}+1 \geq 3$ derniers résidus pour avoir $|S_{2}|\geq 2|r|$ (voir lemme~\ref{lem:estimatecos}). Le fait de ne pas avoir d'annulation de somme partielle simplifie la preuve du lemme~\ref{lem:estimate} et permet de choisir les racines des $l_{1}$ premiers résidus pour obtenir $|S_{1}| \leq |r|$. La construction est donc réalisable sauf éventuellement si toutes ces inégalités sont des égalités.
\par
Dans ce dernier cas, les $l_{1}$ premiers résidus sont identiques, ce qui implique $l_{1} \leq 2$ (sinon nous aurions une annulation de somme partielle pondérée). De plus, nous avons $l_{2}+1=3$ (sinon le lemme~\ref{lem:estimatecos} donne une meilleure borne) et donc $s \leq 5$. La seule strate concernée est donc 
$\Omega^{3}\mathcal{M}_{0}(4,5;\rec[-3][5])$ ($l_{1}=l_{2}=2$). Les cinq résidus ont même module. Puisque nous pouvons les permuter et que les seuls cas dans lequel la construction est impossible implique que les deux premiers soient identiques, il s'ensuit que le seul cas potentiellement problématique est celui de cinq résidus identiques. Dans ce cas $|S_{2}| \geq 3|r|$ et la construction est réalisable.
\end{proof}

\subsubsection{Majorité des cas avec $k \geq 5$}\label{sub:5major}

Comme dans le cas quartique (mais sans les obstructions), les estimations sur les sommes de racines (voir section~\ref{sec:estimation}) permettent de donner des constructions systématiques dans presque tous les cas restants.

\begin{prop}\label{prop:sec5MAJOR}
Pour $l_{1},l_{2} \geq 1$, $s=l_{1}+l_{2}+1$, $0<a<\frac{k}{2}$ et $k \geq 5$, chaque configuration $(R_{1},\dots,R_{s})$ de nombres complexes non nuls est réalisable dans $\Omega^{k}\mathcal{M}_{0}(k(l_{1}-1)+a,kl_{2}-a;\rec[-k][s])$ sauf éventuellement dans les cas suivants :
\begin{itemize}
    \item $k$ pair avec $l_{2}=1$ et $a=\frac{k}{2}-1$;
    \item $k$ impair avec $1 \leq l_{2}\leq 2$ et $a=\frac{k-1}{2}$.
\end{itemize}
\end{prop}

\begin{proof}
La preuve est similaire à celle de la proposition~\ref{prop:sec6quartique} : le lemme~\ref{lem:estimate} permet de choisir des racines $r_{1},\dots,r_{l_{1}}$ pour chacun des $k$-résidus afin d'obtenir $0<|S_{1}| \leq |r_{l_{1}}|$ avec $S_{1}=\sum_{i=1}^{l_{1}} r_{i}$ (sauf pour un cas exceptionnel avec $k=5$, $l_{1}=2$ et $R_{1}=-R_{2}$). De plus, si $l_{1} \geq 3$, nous pouvons exiger que ces racines ne soient pas $\mathbb{R}$-colinéaires. Nous choisissons les racines suivantes $r_{l_{1}+1},\dots,r_{s}$ en utilisant le lemme~\ref{lem:estimatecos}. Ainsi, nous obtenons $S_{2}=|\sum_{i=l_{1}+1}^{s} r_{i}| \geq (1+l_{2}\cos\left(\frac{\pi}{k}\right)) |r_{l_{1}+1}|$. Il s'ensuit que $\frac{|S_{1}|}{|S_{2}|} \leq \frac{1}{1+l_{2}\cos\left(\frac{\pi}{k}\right)}$.
\par
Les lemmes~\ref{lem:trigo} et~\ref{lem:ABAB+} permettent de construire une $k$-différentielle avec les résidus souhaités dans la strate $\Omega^{k}\mathcal{M}_{0}(k(l_{1}-1)+a,kl_{2}-a;\rec[-k][s])$ selon le modèle \textbf{quadrilatère ABAB} comme dans la preuve de la proposition~\ref{prop:sec6quartique} dès lors que nous avons $\frac{1}{1+l_{2}\cos\left(\frac{\pi}{k}\right)} < 1-\frac{2}{k+2-2a}$. Cette dernière inégalité se ramène à $l_{2}\cos\left(\frac{\pi}{k}\right)>\frac{2}{k-2a}$.
\par
Quand $k$ est pair, on a $a \leq \frac{k}{2}-1$. Ainsi, $\frac{2}{k-2a} \leq 1$. Puisque $k \geq 5$, l'inégalité $l_{2}\cos\left(\frac{\pi}{k}\right))>\frac{2}{k-2a}$ est toujours vérifiée si $l_{2} \geq 2$. En revanche, si $l_{2}=1$, elle ne l'est pas pour $a=\frac{k}{2}-1$ mais seulement pour $a \leq \frac{k}{2}-2$.
\par
Quand $k$ est impair, l'inégalité est toujours valide lorsque $a \leq \frac{k-3}{2}$. En revanche, si $a=\frac{k-1}{2}$, l'inégalité devient $l_{2}\cos\left(\frac{\pi}{k}\right)>2$. Pour $l_{2} \geq 3$ elle est vérifiée mais si $l_{2}$ vaut $1$ ou $2$ elle ne l'est pas.
\par
Il reste à traiter le cas exceptionnel $k=5$, $l_{1}=2$ et $R_{1}=-R_{2}$. Ici, nous avons l'annulation d'une somme partielle pondérée de deux racines et pouvons donc utiliser la chirurgie du lemme~\ref{lem:chirurgietranslation} pour se ramener au cas d'une strate avec $l_{1}=1$. La section~\ref{sub:5N2negatif} établit qu'il n'y a pas d'obstruction pour une telle strate.
\end{proof}

\subsubsection{Cas restants} 

Demeure une variété restreinte de strates qui ne sont pas couvertes par les résultats des sections~\ref{sub:mechantquartic},~\ref{sub:mechantcubic} et~\ref{sub:5major}. Elles présentent toutes les caractéristiques suivantes : $k \geq 5$, $l_{2}\in\lbrace1,2\rbrace$ et $\bar{a_{2}} \neq -1$.
\par
L'absence d'obstruction pour ces strates est obtenue par récurrence dans la proposition~\ref{prop:a1neq1}. Toutefois, une famille particulière de configurations de résidus (celles pour lesquelles une somme totale de racines s'annule) nécessite une construction spécifique. Cette construction est donnée dans les deux lemmes suivants. Tout d'abord, nous considérons le cas le plus simple dans lequel $s=3$ (et donc $l_{2}=1$).

\begin{lem}\label{lem:a2-1TROISPOLES}
Dans la strate $\Omega^{k}\mathcal{M}_{0}(k+\bar{a_{1}},k+\bar{a_{2}};\rec[-k][3])$ avec $-k<\bar{a_{1}}<\frac{k}{2}<\bar{a_{2}}<-1$, chaque configuration $(R_{1},R_{2},R_{3})$ de trois nombres complexes non nuls est réalisable s'il existe trois racines $k$-ièmes $r_{1},r_{2},r_{3}$ de $(R_{1},R_{2},R_{3})\in(\CC^{\ast})^{3}$ telles que $r_{1}+r_{2}+r_{3}=0$.
\end{lem}

\begin{proof}
Supposons sans perte de généralité que  $r_{1}=1$, $|r_{2}| \geq 1$, $|r_{3}| \geq 1$ et que $\theta_{i}=\arg(r_{i})\in[0,2\pi[$ satisfait $0\leq \theta_{2}<\theta_{3}$. Le triangle obtenu en concaténant $r_{1},r_{2},r_{3}$ en partant du point d'affixe $0$ est donc contenu dans le demi-plan supérieur fermé. Nous notons $\zeta=\exp\left(\frac{-2i\pi}{k}\right)$ et introduisons
\begin{equation}\label{eq:t}
 t = \frac{1-\zeta}{1-\exp\left(\frac{-2i\bar{a_{2}}\pi}{k}\right)} \text{ et } t' := - \exp\left(\frac{-2i\bar{a_{2}}\pi}{k}\right)t\,.
\end{equation}
Nous avons $t+t'+\zeta r_{1} + r_{2} + r_{3}=0$ et la concaténation de $\zeta r_{1}, t,t',r_{2},r_{3}$ est un polygone dont l'ensemble des points d'intersection est soit vide, soit confondu avec le côté $r_{2}$. En effet, ce polygone s'obtient en collant le long du segment $[0,r_{1}]$ un triangle dont les côtés sont données par les vecteurs $r_{1},r_{2},r_{3}$ (contenu dans le demi-plan supérieur) avec un quadrilatère dont les côtés sont données par les vecteurs $\zeta r_{1},t,t',-r_{1}$ (contenu dans le demi-plan inférieur, voir le dessin de gauche de la figure~\ref{fig:lemmefin2}). Comme $\bar{a_{2}} \neq 1$, nous avons $|t|<1$ et ce quadrilatère est donc sans points d'intersection. Donc il n'y a des points d'intersection que lorsque $\arg(r_{2})=0$, et dans ce cas le segment $r_{2}$ est bien l'ensemble de points d'intersection.
Notons que l'hypothèse $\bar{a_{2}} \neq -1$ est cruciale sans quoi nous aurions $t=- \zeta $ et $t'= 1$, le quadrilatère serait alors singulier.

En identifiant $t$ et $t'$ par une rotation, puis en collant des cylindres le long des trois autres côtés, on obtient une surface plate correspondant à une différentielle ayant un zéro d'ordre $2k+\bar{a_{1}}$ et un autre d'ordre $\bar{a_{2}}$. Notons que cela reste vrai même dans le cas où l'arête $r_{2}$ est constituée de points d'intersection. Pour obtenir des zéros d'ordres $0<a_{1}<a_{2}<k$ nous allons changer l'ordre de la concaténation.

Soient $\theta$ et $\theta'$ les arguments respectifs de $t$ et $t'$ dans $[0,2\pi[$. Notons que $\theta>\theta'$ et un rapide calcul montre que $\theta=\frac{\pi}{k}(\bar{a_{2}}+k-1)$ et $\theta'=-\frac{\pi}{k}(\bar{a_{2}}+1)$. Commençons par comparer $\theta_{3}$ avec $\frac{2\pi}{k}$. Si $\theta_{3}>2\pi-\frac{2\pi}{k}$, alors on échange $\zeta r_{1}$ avec $r_{3}$ dans la concaténation précédente. Cette opération est représentée sur la figure~\ref{fig:lemmefin2} (ce qui ne change aucunement l'angle conique des deux singularités obtenues après collage). Par simplicité, nous supposerons que $\theta_{3}\leq 2\pi-\frac{2\pi}{k}$ dans la suite, l'autre cas ne changeant que par la notation.
\begin{figure}[hbt]
\begin{tikzpicture}[scale=2.5,decoration={
    markings,
    mark=at position 0.5 with {\arrow[very thick]{>}}}]

%haut gauche
\begin{scope}
 \draw[] (0,0) coordinate (p0) --  ++(-30:1) coordinate[pos=.3] (t1) coordinate (p1)-- ++(-.11,.3)  coordinate[pos=.5] (r1) coordinate (p2) -- ++(.25,.2) coordinate[pos=.5] (r2) coordinate (p3)-- ++(170:2.4)  coordinate[pos=.5] (r3)coordinate (p4) --  (p0) coordinate[pos=.5] (r4);

\draw[dotted] (p0) --node[]{$r_{1}$} (1,0) coordinate (p6);

% \fill (p0)  circle (1pt);

\node[below] at (t1) {$\zeta r_{1}$};
\node[] at (r1) {$t$};
\node[] at (r2) {$t'$};
\node[above] at (r3) {$r_{2}$};
\node[below] at (r4) {$r_{3}$};
\end{scope}

%haut droit
\begin{scope}[xshift=4cm,yshift=-.5cm]
 \draw[] (0,0) coordinate (p0) -- ++(-.11,.3)  coordinate[pos=.5] (r1) coordinate (p2) -- ++(.25,.2) coordinate[pos=.5] (r2) coordinate (p3)-- ++(170:2.4)  coordinate[pos=.5] (r3)coordinate (p4) --  ++(-30:1) coordinate[pos=.3] (t1) coordinate (p1) --  (p0) coordinate[pos=.5] (r4);

% \fill (p0)  circle (1pt);

\node[below] at (t1) {$\zeta r_{1}$};
\node[] at (r1) {$t$};
\node[] at (r2) {$t'$};
\node[above] at (r3) {$r_{2}$};
\node[below] at (r4) {$r_{3}$};
\end{scope}
\end{tikzpicture}

\caption{Le changement d'ordre si $\theta_{3}>2\pi-\frac{2i\pi}{k}$}\label{fig:lemmefin2}
\end{figure}
\par
Notons que puisque $\theta>\theta'$, au moins l'une des deux inégalités $\theta'+\pi < \theta_{3}$ et $\theta_{3} <\theta + \pi$ est satisfaite. Nous supposerons d'abord $\theta'+\pi < \theta_{3}$.
% Dans le troisième cas, nous avons $\theta'+\pi < \theta_{3} \leq \theta + \pi$ (rappelons que $\theta'<\theta$).
En permutant les segments $t'$ et $r_{2}$ on obtient exactement une intersection de $r_{3}$ avec $r_{2}$ comme représenté sur le dessin du milieu de la figure~\ref{fig:dernierlem1}. En effet, l'unique autre possibilité serait que $r_{2}$ coupe $\zeta r_{1}$, mais cela impliquerait que $|t|>|r_{3}|$, ce qui est impossible car par hypothèse $|t|<|r_{1}|\leq |r_{3}|$. Notons que cette dernière égalité est valable même si l'on a interverti l'ordre de $r_{3}$ et $\zeta r_{1}$ précédemment. 
On colle alors sur le segment $r_{3}$ un cylindre auquel on a retiré le triangle correspondant à l'intersection. Notons que l'inégalité $|t|<|r_{1}|\leq |r_{3}|$ implique que le triangle est sans auto-intersection dans le cylindre $r_{3}$. Puis on termine la construction en collant des cylindres aux arêtes $r_{2}$ et $\zeta r_{1}$ et les segments $t$  avec $t'$.
\begin{figure}[hbt]

\begin{tikzpicture}[scale=2.5,decoration={
    markings,
    mark=at position 0.5 with {\arrow[very thick]{>}}}]
%premiere construction
\begin{scope}[xshift=-.5cm,yshift=2cm]

 \draw[] (0,0) coordinate (p0) --  ++(-30:1) coordinate[pos=.3] (t1) coordinate (p1)-- ++(-.11,.3)  coordinate[pos=.5] (r1) coordinate (p2) -- ++(.25,.2) coordinate[pos=.5] (r2) coordinate (p3)-- ++(0,1.2) coordinate[pos=.5] (t2) coordinate (p4) --  (p0) coordinate[pos=.5] (r4);

\draw[dotted] (p0) -- node[above] {$r_{1}$} (1,0) coordinate (p6);

\node[below] at (t1) {$\zeta r_{1}$};
\node[] at (r1) {$t$};
\node[] at (r2) {$t'$};
\node[right] at (t2) {$r_{2}$};
\node[left] at (r4) {$r_{3}$};
\end{scope}

%seconde construction
\begin{scope}[xshift=1.5cm,yshift=2cm]
\fill[fill=black!10] (0,0) coordinate (p0) --  ++(-30:1) coordinate[pos=.3] (t1) coordinate (p1)-- ++(-.11,.3)  coordinate[pos=.5] (r1) coordinate (p2)-- ++(0,1.2) coordinate[pos=.5] (t2) coordinate (p4) -- ++(.25,.2) coordinate[pos=.5] (r2) coordinate (p3) --  (p0) coordinate[pos=.5] (r4);

\draw (p0) -- (p1) -- (p2) -- (p4) -- (p3) -- (p0);
\fill[fill=black!10]  (p3) -- (p0) -- ++(130:.7) coordinate (a)--++(50:1.6)coordinate (b) --(p3);
\draw (p0) -- (a);
\draw (p3) -- (b);

\fill[fill=white] (p2) -- (p4) --(p3) --(p2);
\draw (p2) -- (p4) --(p3);
\draw[dotted] (p0)--(p3);

 \filldraw[fill=white] (p4)  circle (.8pt);
\filldraw[fill=white] (p2)  circle (.8pt);

\node[below] at (t1) {$\zeta r_{1}$};
\node[] at (r1) {$t$};
\node[above] at (r2) {$t'$};
\node[right] at (t2) {$r_{2}$};
\node[left] at (r4) {$r_{3}$};
\end{scope}

% %troisème construction
% \begin{scope}[xshift=3.5cm,yshift=2.2cm]
%  \draw[] (0,0) coordinate (p0)-- ++(.25,.2) coordinate[pos=.5] (r2) coordinate (p3) --  ++(-30:1) coordinate[pos=.3] (t1) coordinate (p1)-- ++(-.11,.3)  coordinate[pos=.5] (r1) coordinate (p2) -- ++(0,1.2) coordinate[pos=.5] (t2) coordinate (p4)  --  (p0) coordinate[pos=.5] (r4);
%
% \filldraw[fill=white] (p2)  circle (.8pt);
% \filldraw[fill=white] (p0)  circle (.8pt);
% \filldraw[fill=white] (p4)  circle (.8pt);
%
% \node[below] at (t1) {$\zeta r_{1}$};
% \node[right] at (t2) {$r_{2}$};
% \node[] at (r1) {$t$};
% \node[below] at (r2) {$t'$};
% \node[left] at (r4) {$r_{3}$};
% \end{scope}

\end{tikzpicture}

\caption{Le cas $\theta'+\pi < \theta_{3} \leq \theta + \pi$.}\label{fig:dernierlem1}
\end{figure}
\par
Finalement, si $\theta_{3} <\theta + \pi$ on utilise la concaténation $\zeta r_{1}, t',r_{2},r_{3},t$. Pour la même raison que dans le cas précédent, ce polygone possède une unique intersection entre les segments~$r_{3}$ et $\zeta r_{1}$. Le reste est identique à la construction précédente.
\end{proof}

Nous traitons ensuite le cas $s \geq 4$ pour ces configurations de résidus. La preuve est légèrement différente selon si $l_{2}=1$ ou $l_{2}=2$.

\begin{lem}\label{lem:a2-1SPECIFIQUE}
Dans toute strate $\Omega^{k}\mathcal{M}_{0}(l_{1}k+\bar{a_{1}},l_{2}k+\bar{a_{2}};\rec[-k][s])$ avec $-\frac{k}{2}<\bar{a_{2}}<-1$, $\bar{a_{1}}=-k-\bar{a_{2}}$, $l_{1},l_{2} \geq 1$ et $l_{1}+l_{2}=s-1$, chaque configuration $(R_{1},\dots,R_{s})$ de $s$ nombres complexes non nuls est réalisable s'il existe $s$ racines $k$-ièmes $r_{1},\dots,r_{s}$ de $(R_{1},\dots,R_{s})\in(\CC^{\ast})^{s}$ telles que $\sum\limits_{i=1}^{s} r_{i}=0$ mais qu'aucune somme partielle $\sum\limits_{i \in I} r_{i}$ ne s'annule lorsque $I$ est un sous-ensemble strict non vide de $\lbrace{ 1,\dots,s \rbrace}$.
\end{lem}

\begin{proof}
La proposition~\ref{prop:sec5MAJOR} et le lemme~\ref{lem:a2-1TROISPOLES} permettent de se restreindre aux cas dans lesquels $s\geq4$, $l_{2} \in\lbrace 1,2\rbrace$ et dans le cas où $l_{2}=2$, on peut supposer que $k$ est impair.
\smallskip
\par
\paragraph{\bf Le cas $l_{2}=1$} 
On suppose que $|R_{1}|$ est minimal parmi les modules de $k$-résidus et nous normalisons les $r_{i}$ pour avoir $r_{1}=1$. Concaténons les $r_{i}$ par argument croissant dans $\left[0,2\pi\right[$. Nous obtenons un polygone convexe dans le demi-plan supérieur ou dans la demi-droite~$\RR_{+}$. Notons que dans ce dernier cas, le fait qu'aucune somme partielle ne s'annule implique que les sommets du polygone sont deux à deux distincts. Notons respectivement par $A$, $B$, $C$ et~$D$ le point initial de $r_{1}$, $r_{2}$, $r_{3}$ et $r_{s}$. Remarquons que soit $|AC|\geq|r_{1}|=1$, soit $|BD|\geq1$. En effet, supposons que la longueur $|BD|$ est inférieure à $1$. Comme $ABCD$ est convexe, le sommet $C$ est dans l'intersection du demi-plan supérieur (contenant $D$) et du demi-plan complémentaire de $(BD)$ qui ne contient pas $A$. En particulier, il ne peut appartenir au disque de centre $A$ et de rayon~$1$. Donc la distance $|AC|$ est bien strictement supérieure à~$1$.
Nous notons $\tilde r_{3}= \sum_{i\geq3}r_{i}$, qui correspond au segment $[CA]$.
\smallskip
\par
Nous considérons deux cas selon que $\arg(\tilde r_{3})\leq 2\pi - \frac{2\pi}{k}$ ou  $2\pi - \frac{2\pi}{k}<\arg(\tilde r_{3})< 2\pi$. Dans la suite, on utilise les notations de la preuve du lemme~\ref{lem:a2-1TROISPOLES} en remplaçant~$r_{3}$ par $\tilde r_{3}$ et $\theta_{3}$ par~$\tilde\theta_{3}$.
\par
Supposons tout d'abord que $\arg(\tilde r_{3})\leq 2\pi - \frac{2\pi}{k}$.
% Considérons la concaténation de $\zeta r_{1}, t,t',r_{2},\tilde r_{3}$, comme à gauche de la figure~\ref{fig:dernierlem1}.
Notons que soit $\theta'+\pi < \tilde\theta_{3}$ soit $\tilde\theta_{3} <\theta + \pi$ est satisfaite.  Nous allons supposer que $\theta'+\pi < \tilde\theta_{3}$, l'autre cas se traitant de manière similaire en changeant le rôle de $t'$ avec $t$ et celui de $r_{2}$ avec  $\zeta r_{1}$.
\par
Supposons donc que $\theta'+\pi < \tilde\theta_{3}$ et considérons la concaténation de $\zeta r_{1},t,r_{2},t',\tilde r_{3}$ représentée à droite de  la figure~\ref{fig:dernierlem1}. Nous voulons remplacer le cylindre de période $\tilde r_{3}$ (noté $r_{3}$ sur la figure) par $s-2$ cylindres de périodes $r_{3},\dots,r_{s}$. Nous allons le faire en collant ces cylindres à partir du point final de $t'$. Plus précisément, on colle $s-2$ cylindres sur un $(s-1)$-gone formé des $r_{i}$ avec $i\in\lbrace 3,\dots,s \rbrace$ et de $-\tilde r_{3}$. On enlève à ce domaine le triangle $T$ correspondant à l'intersection et on colle son bord sur $\tilde r_{3}$, $t'$ et un bout de $r_{2}$ par translation. Cette construction est représentée dans la figure~\ref{fig:ajoutspoles} dans le cas où $s=4$. Remarquons que si cette construction est réalisable, alors elle donne une $k$-différentielle avec $l_{2}=1$ et les résidus souhaités.
\begin{figure}[hbt]

\begin{tikzpicture}[scale=3,decoration={
    markings,
    mark=at position 0.5 with {\arrow[very thick]{>}}}]
%premiere  construction
\begin{scope}[xshift=0cm,yshift=0cm]
\fill[fill=black!10] (0,0) coordinate (p0) --  ++(-30:1) coordinate[pos=.3] (t1) coordinate (p1)-- ++(-.11,.3)  coordinate[pos=.5] (r1) coordinate (p2)-- ++(0,1.2) coordinate[pos=.5] (t2) coordinate (p4) -- ++(.25,.2) coordinate[pos=.5] (r2) coordinate (p3) -- ++(180:1.1) coordinate (q) --  (p0) coordinate[pos=.5] (r4);

\draw (p0) -- (p1) -- (p2) -- (p4) -- (p3) -- (p0);
\fill[fill=black!10]  (p3) -- (q) -- ++(90:.4) coordinate (a)--++(0:1.1)coordinate (b) --(p3);
\draw (q) -- (a);
\draw (p3) -- (b);
\fill[fill=black!10]  (p0) -- (q) -- ++(180:.7) coordinate (a)--++(-90:1.2)coordinate (b) --(p0);
\draw (p0) -- (b);
\draw (q) -- (a);

\fill[fill=white] (p2) -- (p4) --(p3) --(p2);
\draw (p2) -- (p4) --(p3);
\draw[dotted] (p0)-- (q) --node[]{$r_{3}$} (p3);
\draw[dotted] (p0)-- node[]{$\tilde r_{3}$} (p3);

 \filldraw[fill=white] (p4)  circle (.8pt);
\filldraw[fill=white] (p2)  circle (.8pt);

\draw[->] (1.35,1.05) -- (.85,1.05);
\node at (1.4,1.05) {$T$};

\node[below] at (t1) {$\zeta r_{1}$};
\node[] at (r1) {$t$};
\node[above] at (r2) {$t'$};
\node[right] at (t2) {$r_{2}$};
\node[left] at (r4) {$r_{4}$};
\end{scope}

% %seconde  construction
% \begin{scope}[xshift=2.5cm,yshift=0cm]
% \fill[fill=black!10] (0,0) coordinate (p0) -- ++(.25,.2) coordinate[pos=.5] (r2) coordinate (p3) --  ++(-30:1) coordinate[pos=.3] (t1) coordinate (p1)-- ++(-.11,.3)  coordinate[pos=.5] (r1) coordinate (p2)-- ++(0,1.2) coordinate[pos=.5] (t2) coordinate (p4)   -- ++(180:1.1) coordinate (q) --  (p0) coordinate[pos=.5] (r4);
% 
% 
% 
% 
% \fill[fill=black!10]  (p2) -- (q) -- ++(90:.7) coordinate (a)--++(0:1.1)coordinate (b) --(p2);
% \draw (q) -- (a);
% \draw (p2) -- (b);
% \fill[fill=black!10]  (p0) -- (q) -- ++(180:.7) coordinate (a)--++(-90:1.2)coordinate (b) --(p0);
% \draw (p0) -- (b);
% \draw (q) -- (a);
% \draw (p0) -- (p3) -- (p1) -- (p2) -- (p4)  --node[]{$\tilde r_{3}$} (p0);
% % 
%  \draw[dotted] (p0)-- (q) --node {$r_{3}$} (p4);
% 
% 
%  \filldraw[fill=white] (p4)  circle (.8pt);
% \filldraw[fill=white] (p2)  circle (.8pt);
%  \filldraw[fill=white] (q)  circle (.8pt);
% \filldraw[fill=white] (p0)  circle (.8pt);
% 
% 
% \node[below] at (t1) {$\zeta r_{1}$};
% \node[] at (r1) {$t$};
% \node[below] at (r2) {$t'$};
% \node[right] at (t2) {$r_{2}$};
% \node[left] at (r4) {$r_{4}$};
% \end{scope}
\end{tikzpicture}

\caption{Remplacement d'un pôle de la figure~\ref{fig:dernierlem1} par une paire de pôles dans le cas $l_{2}=1$.}\label{fig:ajoutspoles}
\end{figure}

Nous justifions maintenant que cette construction est réalisable. Supposons tout d'abord que $\theta'+\pi < \theta_{3}=\arg(r_{3})$. Dans ce cas, l'arête $t'$ est contenue dans le cylindre associé à $r_{3}$. Le fait que $|r_{3}|>|t'|$ implique que l'intersection de $T$ avec ce cylindre est sans auto-intersection. L'autre partie de $T$ est contenue dans le polygone, car celui-ci est convexe. Par symétrie, le cas $\theta_{s} <\theta + \pi$ se traite lorsque $\tilde\theta_{3} <\theta + \pi$ par un argument similaire ($\theta_{s}$ est ici $\arg(r_{s})$). On peut donc supposer que $\theta'+\pi \geq \theta_{3}$ et $\theta_{s} \geq \theta + \pi$. Dans ce cas, l'inégalité $|t'|<|r_{3}|$ et la convexité du polygone impliquent que $T$ est strictement contenu dans le polygone.
\smallskip
\par
Considérons maintenant le cas où $2\pi - \frac{2\pi}{k}<\arg(\tilde r_{3})< 2\pi$. On commence par considérer le polygone $P$ obtenu en concaténant $t$, $r_{2}$, $\zeta r_{1}$, $t'$ et $\tilde r_{3}$ comme représenté sur la figure~\ref{fig:trois3poles2}.
Dans ce cas, si $t'$ est disjoint de~$r_{2}$, alors $P$ n'a pas d'auto-intersection. Si au contraire $t'$ coupe~$r_{2}$, alors le segment $\tilde r_{3}$ coupe aussi $r_{2}$. En effet, sinon les segments $(\tilde r_{3},t)$ formeraient une base indirecte, ce qui contredit l'hypothèse selon laquelle $2\pi - \frac{2\pi}{k}<\arg(\tilde r_{3})< 2\pi$. S'il n'y a pas d'intersection, on colle alors un cylindre à $r_{2}$. S'il y a une intersection, on colle à~$r_{2}$, à la fin de $t'$ et au début de $\tilde r_{3}$ un cylindre auquel on enlève le triangle correspondant à l'intersection. On colle alors un polygone dont les arêtes sont les $r_{i}$, avec $i\in\lbrace 3,\dots,s \rbrace$  et $-\tilde r_{3}$ au segment $\tilde r_{3}$. Finalement, on colle les cylindres aux $r_{i}$ pour $i\geq3$, un cylindre à $\zeta r_{1}$ et $t$ avec $t'$ pour obtenir une différentielle avec les invariants souhaités.
\begin{figure}[hbt]
\begin{tikzpicture}[scale=2.5,decoration={
    markings,
    mark=at position 0.5 with {\arrow[very thick]{>}}}]
 \fill[fill=black!10](0,0) coordinate (p0) -- ++(-.11,.3) coordinate[pos=.5] (r2) coordinate (p3)-- ++(162:2.7)  coordinate[pos=.3] (r3)coordinate (p4)--  ++(-30:1) coordinate[pos=.3] (t1) coordinate (p1)-- ++ (.25,.2) coordinate[pos=.5] (r1) coordinate (p2)  --  (p0) coordinate[pos=.5] (r4);
\draw (p0) -- (p3) -- (p4) -- (p1) -- (p2) -- (p0);
\fill[fill=black!10]  (p3) -- (p4) -- ++(72:.7) coordinate (a)--++(-18:2.7)coordinate (b) --(p3);
\draw (p4) -- (a);
\draw (p3) -- (b);
% \draw[dotted] (p4) --++(1,0) -- (p3);

\fill[fill=white] (p1) -- (p2) --(p0) --(p1);
\draw (p1) -- (p2) --(p0);

\filldraw[fill=white] (p3)  circle (.8pt);
\filldraw[fill=white] (p1)  circle (.8pt);
\filldraw[fill=white] (p4)  circle (.8pt);

\node[below] at (t1) {$\zeta r_{1}$};
\node[] at (r1) {$t'$};
\node[] at (r2) {$t$};
\node[above] at (r3) {$r_{2}$};
\node[below] at (r4)  {$\tilde r_{3}$};
\end{tikzpicture}

\caption{La construction dans le cas où $2\pi - \frac{2\pi}{k}<\arg(\tilde r_{3})< 2\pi$.}\label{fig:trois3poles2}
\end{figure}
\smallskip
\par
Il reste donc à considérer le cas où la diagonale de longueur supérieure ou égale à~$|r_{1}|$ est~$BD$. On fait la même construction qu'au paragraphe précédent en partant d'une configuration légèrement différente. Au lieu de faire la rotation d'angle $-\frac{2\pi}{k}$ de $r_{1}$ en $A$, on fait la rotation d'angle $\frac{2\pi}{k}$ de $r_{1}$ en $B$. La configuration obtenue est alors symétrique de celle considérée et les arguments sont identiques.
\smallskip
\par
\paragraph{\bf Le cas $l_{2}=2$}
Nous traitons tout d'abord les strates avec $s=4$ pôles puis donnons un argument pour $s\geq5$.
\smallskip
\par
Commençons par le cas où $s=4$ et utilisons les notations du cas $l_{2}=1$. On sépare les cas selon que $\arg(\tilde r_{3})\leq 2\pi - \frac{2\pi}{k}$ ou $2\pi - \frac{2\pi}{k}<\arg(\tilde r_{3})< 2\pi$. Dans le premier cas, considérons la concaténation de $\zeta r_{1},t,r_{2},\tilde r_{3},t'$ comme montré à droite de la figure~\ref{fig:dernierlem1}. On obtient la différentielle souhaitée en collant la paire de pôles au point initial de~$t'$ comme montré à gauche de la figure~\ref{fig:ajoutspoles2}.
\begin{figure}[hbt]
\begin{tikzpicture}[scale=2.5,decoration={
    markings,
    mark=at position 0.5 with {\arrow[very thick]{>}}}]
\fill[fill=black!10] (0,0) coordinate (p0) -- ++(.25,.2) coordinate[pos=.5] (r2) coordinate (p3) --  ++(-30:1) coordinate[pos=.3] (t1) coordinate (p1)-- ++(-.11,.3)  coordinate[pos=.5] (r1) coordinate (p2)-- ++(0,1.2) coordinate[pos=.5] (t2) coordinate (p4)   -- ++(180:1.1) coordinate (q) --  (p0) coordinate[pos=.5] (r4);

\fill[fill=black!10]  (p2) -- (q) -- ++(90:.7) coordinate (a)--++(0:1.1)coordinate (b) --(p2);
\draw (q) -- (a);
\draw (p2) -- (b);
\fill[fill=black!10]  (p0) -- (q) -- ++(180:.7) coordinate (a)--++(-90:1.2)coordinate (b) --(p0);
\draw (p0) -- (b);
\draw (q) -- (a);
\draw (p0) -- (p3) -- (p1) -- (p2) -- (p4)  --node[]{$\tilde r_{3}$} (p0);
 \draw[dotted] (p0)-- (q) --node {$r_{3}$} (p4);

 \filldraw[fill=white] (p4)  circle (.8pt);
\filldraw[fill=white] (p2)  circle (.8pt);
 \filldraw[fill=white] (q)  circle (.8pt);
\filldraw[fill=white] (p0)  circle (.8pt);

\node[below] at (t1) {$\zeta r_{1}$};
\node[] at (r1) {$t$};
\node[below] at (r2) {$t'$};
\node[right] at (t2) {$r_{2}$};
\node[left] at (r4) {$r_{4}$};

\begin{scope}[xshift=4.1cm]
 \fill[fill=black!10](0,0) coordinate (p0) -- ++ (.25,.2) coordinate[pos=.5] (r2) coordinate (p3)-- ++(162:2.7)  coordinate[pos=.3] (r3)coordinate (p4)--  ++(-30:1) coordinate[pos=.3] (t1) coordinate (p1)-- ++ (-.11,.3) coordinate[pos=.5] (r1) coordinate (p2)  --  (p0) coordinate[pos=.5] (r4);
\draw (p0) -- (p3) -- (p4) -- (p1) -- (p2) -- (p0);
\fill[fill=black!10]  (p3) -- (p4) -- ++(72:.7) coordinate (a)--++(-18:2.7)coordinate (b) --(p3);
\draw (p4) -- (a);
\draw (p3) -- (b);
% \draw[dotted] (p4) --++(1,0) -- (p3);

\fill[fill=white] (p1) -- (p2) --(p0) --(p1);
\draw (p1) -- (p2) --(p0);

\filldraw[fill=white] (p0)  circle (.8pt);
\filldraw[fill=white] (p2)  circle (.8pt);

\node[below] at (t1) {$\zeta r_{1}$};
\node[] at (r1) {$t$};
\node[] at (r2) {$t'$};
\node[above] at (r3) {$r_{2}$};
\node[below] at (r4)  {$\tilde r_{3}$};
\end{scope}
\end{tikzpicture}
\caption{La construction pour obtenir $l_{2}=2$ selon que $\arg(\tilde r_{3})\leq 2\pi - \frac{2\pi}{k}$ (à gauche) ou  $2\pi - \frac{2\pi}{k}<\arg(\tilde r_{3})< 2\pi$ (à droite).}\label{fig:ajoutspoles2}
\end{figure}
Comme on remplace un domaine polaire par un triangle auquel on attache deux cylindres, il n'y a pas de problèmes d'auto-intersection. Un calcul rapide montre que l'on obtient les invariants souhaités.

Dans le second cas, on concatène $\zeta r_{1},t,\tilde r_{3},t',r_{2}$ comme montré à droite de la figure~\ref{fig:ajoutspoles2}. Remarquons que c'est similaire au cas correspondant pour $l_{2}=1$, en inversant le rôle de~$t$ et $t'$. Le même argument que dans le cas précédent montre l'existence de la différentielle associée avec les invariants souhaités.
\smallskip
\par
On suppose maintenant que $s\geq5$ et on note $P$ le polygone obtenu en concaténant les~$r_{i}$ par arguments croissants. Considérons les diagonales qui découpent $P$ en deux parties contenant respectivement $3$ et $s-3$ arêtes. Supposons qu'il existe une telle diagonale de longueur inférieure ou égale à chacune des trois arêtes qu'elle détermine. Notons cette diagonale $\tilde r_{1}$ et les trois arêtes par $r_{2},r_{3},r_{4}$ rangées par arguments croissants. On fait la construction du cas $s=4$ avec ces segments ($\tilde r_{1}$ joue le rôle de $r_{1}$). On obtient donc les polygones de la figure~\ref{fig:ajoutspoles2}. La différentielle est obtenue en collant les $s-3$ résidus restant à $\tilde r_{1}$. Notons que dans les deux constructions de la figure~\ref{fig:ajoutspoles2}, le pôle de résidu $R_{1}$ contribue uniquement à la singularité $a_{1}$. L'ordre de la singularité $a_{2}$ est donc préservée lors du collage des résidus à~$\tilde r_{1}$. En particulier, la différentielle obtenue satisfait $l_{2}=2$.
\smallskip
\par
Supposons maintenant qu'il n'existe pas de diagonale de longueur inférieure ou égale à chacune des trois arêtes qu'elle borde. Prenons une configuration où ces arêtes $r_{1},r_{2},r_{s}$ et la diagonale $\tilde r_{3}$ sont telles que $|r_{1}|$ est minimale et la concaténation $r_{1},r_{2},\tilde r_{3},r_{s}$ forme un quadrilatère convexe. Commençons par considérer le polygone obtenu en concaténant $\zeta r_{1},t,t',r_{2},\tilde r_{3},r_{s}$, représenté à gauche de la figure~\ref{fig:sgeq5prem}. Comme précédemment, on inverse l'ordre de concaténation de $\zeta r_{1}$ avec les segments de $\lbrace r_{s},\tilde r_{3}\rbrace$ dont l'argument est strictement supérieur à $2\pi - \frac{2\pi}{k}$.

Supposons que l'on permute $\zeta r_{1}$ avec $r_{s}$ et $\tilde r_{3}$. On concatène $t, r_{s},r_{2},\zeta r_{1},t',\tilde r_{3}$ comme montré sur la figure~\ref{fig:sgeq5prem}. 
Le polygone ainsi obtenu peut posséder des intersections. En effet, les segments $r_{2}$ et $t$ peuvent se couper pour former un triangle comme à droite de la figure~\ref{fig:sgeq5prem}. De plus, il est possible que $r_{2}$ coupe $t'$, donnant un triangle au dessus de $r_{2}$ s'il coupe $\tilde r_{3}$ et un polygone plus complexe si $\tilde r_{3}$ reste au dessus de $r_{2}$.  Remarquons que ces polygones d'intersection ne n'ont pas d'auto-intersection. Par exemple, s'il y a deux triangles d'intersection, ils ne se chevauchent pas car l'argument de~$t$ est inférieur à celui de $\tilde r_{3}$ auquel on ajoute~$\pi$.
% S'il n'y a qu'un triangle d'intersection, notons que comme $\arg(r_{s}),\arg(\tilde r_{3})\in \left[2\pi-\frac{2\pi}{k},2\pi\right[$, on a $|r_{2}|>|r_{s}|$. Avec le fait que $|t|<|r_{1}|$ cela implique que ce triangle est contenu dans le cylindre correspondant à $r_{2}$ et sans auto-intersection. S'il y a deux triangles d'intersection, le même argument montre que chaque triangle est sans auto-intersection. De plus les deux triangles ne se chevauchent pas car l'argument de~$t$ est inférieur à celui de $\tilde r_{3}$ auquel on ajoute $\pi$.
% De plus ils ne se coupent pas car cela impliquerait que $t$ coupe $\zeta r_{1}$, ce qui est impossible par définition de~$t$. 
\begin{figure}[hbt]
\begin{tikzpicture}[scale=2.5,decoration={
    markings,
    mark=at position 0.5 with {\arrow[very thick]{>}}}]
%haut gauche
\begin{scope}
 \draw[] (0,0) coordinate (p0) --  ++(-30:1) coordinate[pos=.3] (t1) coordinate (p1)-- ++(-.11,.3)  coordinate[pos=.5] (r1) coordinate (p2) -- ++(.25,.2) coordinate[pos=.5] (r2) coordinate (p3)-- ++(175:2.8)  coordinate[pos=.5] (r3)coordinate (p4) -- ++(-8:1)  coordinate[pos=.5] (r4)coordinate (p5)--  (p0) coordinate[pos=.5] (r5);

\draw[dotted] (p0) --node[]{$r_{1}$} (1,0) coordinate (p6);

% \fill (p0)  circle (1pt);

\node[below] at (t1) {$\zeta r_{1}$};
\node[] at (r1) {$t$};
\node[] at (r2) {$t'$};
\node[above] at (r3) {$r_{2}$};
\node[below] at (r4) {$\tilde r_{3}$};
\node[below] at (r5) {$r_{s}$};
\end{scope}

%haut droit
\begin{scope}[xshift=4.15cm,yshift=-.1cm]
\draw[] (0,0)  coordinate (p0) -- ++(175:2.8)  coordinate[pos=.5] (r3)coordinate (p4) -- ++(-30:1) coordinate[pos=.5] (t1) coordinate (p3) -- ++(.25,.2) coordinate[pos=.5] (r2) coordinate (p1) -- ++(-8:1)  coordinate[pos=.5] (r4)coordinate (p5) -- ++(-.11,.3)  coordinate[pos=.5] (r1) coordinate (p2)  --  (p0) coordinate[pos=.5] (r5);

% \fill (p0)  circle (1pt);

\node[below] at (t1) {$\zeta r_{1}$};
\node[] at (r1) {$t$};
\node[] at (r2) {$t'$};
\node[above] at (r3) {$r_{2}$};
\node[below] at (r4) {$\tilde r_{3}$};
\node[above] at (r5) {$r_{s}$};
\end{scope}
\end{tikzpicture}
\caption{Les polygones si l'on permute $\zeta r_{1}$ avec $r_{s}$ et $\tilde r_{3}$.}\label{fig:sgeq5prem}
\end{figure}
On peut donc coller comme précédemment un cylindre (auquel on enlève les éventuels polygones d'intersection) à $r_{2}$. Puis on ajoute les autres cylindres et on colle $t$ avec~$t'$. De la sorte on obtient une différentielle avec les invariants souhaités.

Supposons que l'on permute $\zeta r_{1}$ avec $r_{s}$ ou que l'ordre de la concaténation reste inchangé. Dans la suite, on supposera que l'on est dans ce dernier cas pour fixer la notation. Supposons tout d'abord que $\arg(\tilde r_{3})> \arg(t)+\pi$. Dans ce cas, l'ordre de concaténation du cas précédent, représenté dans la figure~\ref{fig:sgeq5prem}, donne la différentielle attendue de manière similaire. Notons toutefois que dans ce cas il se peut que la concaténation des segments $\zeta r_{1}$ jusqu'à $r_{s}$ possède plus de points d'intersection avec $r_{2}$, voir soit au dessus de $r_{2}$. Dans ce dernier cas, on colle le cylindre $r_{2}$ au dessus du segment brisé et les autres cylindres en dessous, avant d'identifier $t$ avec $t'$ pour obtenir la différentielle souhaitée.

Considérons maintenant le cas où  $\arg(\tilde r_{3})\in [\frac{\pi}{2},\pi]$. Dans ce cas, on considère  le polygone obtenu par la concaténation de $t,\zeta r_{1},r_{2},t',\tilde r_{3},r_{s}$ représenté sur la figure~\ref{fig:sgeq5sec}. Il ne peut avoir au plus qu'une intersection, entre $r_{s}$ et $\zeta r_{1}$ si $\arg(t)+\pi\geq\arg(r_{s})$, comme sur la figure. En effet, comme l'argument de $t'$ est strictement compris entre $0$ et $\pi$, le segment $r_{3}$ ne peut pas couper le segment $r_{2}$. 
\begin{figure}[hbt]
\begin{tikzpicture}[scale=2.5,decoration={
    markings,
    mark=at position 0.5 with {\arrow[very thick]{>}}}]
%haut gauche
\begin{scope}
 \draw[] (0,0) coordinate (p0) --  ++(-30:1) coordinate[pos=.3] (t1) coordinate (p1)-- ++(-.11,.3)  coordinate[pos=.5] (r1) coordinate (p2) -- ++(.25,.2) coordinate[pos=.5] (r2) coordinate (p3)-- ++(75:1.1)  coordinate[pos=.5] (r3)coordinate (p4) -- ++(160:1.1)  coordinate[pos=.5] (r4)coordinate (p5)--  (p0) coordinate[pos=.5] (r5);

\draw[dotted] (p0) --node[]{$r_{1}$} (1,0) coordinate (p6);

% \fill (p0)  circle (1pt);

\node[below] at (t1) {$\zeta r_{1}$};
\node[] at (r1) {$t$};
\node[] at (r2) {$t'$};
\node[right] at (r3) {$r_{2}$};
\node[above] at (r4) {$\tilde r_{3}$};
\node[left] at (r5) {$r_{s}$};
\end{scope}

%haut droit
\begin{scope}[xshift=3cm,yshift=-.2cm]
 \draw[] (0,0)coordinate (p0) -- ++(-.11,.3)  coordinate[pos=.5] (r1) coordinate (p2) --  ++(-30:1) coordinate[pos=.5] (t1) coordinate (p1) -- ++(75:1.1)  coordinate[pos=.5] (r3)coordinate (p4)
 -- ++(.25,.2) coordinate[pos=.5] (r2) coordinate (p3) -- ++(160:1.1)  coordinate[pos=.5] (r4)coordinate (p5)--  (p0) coordinate[pos=.5] (r5);

% \fill (p0)  circle (1pt);

\node[above] at (t1) {$\zeta r_{1}$};
\node[] at (r1) {$t$};
\node[] at (r2) {$t'$};
\node[right] at (r3) {$r_{2}$};
\node[above] at (r4) {$\tilde r_{3}$};
\node[left] at (r5) {$r_{s}$};
\end{scope}
\end{tikzpicture}
\caption{Les polygones si l'on permute $\zeta r_{1}$ avec $r_{s}$ et $\tilde r_{3}$.}\label{fig:sgeq5sec}
\end{figure}
On peut donc coller les cylindres aux arêtes du polygone (on enlève un triangle à celui de circonférence $r_{s}$ en cas d'intersection) et $t$ avec~$t'$ pour obtenir une différentielle avec les invariants souhaités.

Finalement, nous montrons que l'on peut se ramener aux cas précédent dans presque tous les cas. Rappelons que par la proposition~\ref{prop:sec5MAJOR}, il suffit de considérer les cas où $\bar{a_{2}}=\frac{1-k}{2}$. En particulier, comme $\arg(t)\geq \frac{\pi}{2}$ les constructions du paragraphe précédent donnent des différentielles avec les invariants souhaités si $\arg(\tilde r_{3})\in \left[\frac{\pi}{2},\pi\right]\cup \left[\frac{3\pi}{2},2\pi\right[$.
Si l'argument de~$\tilde r_{3}$ est dans $\left]0,\frac{\pi}{2}\right[ \cup \left]\pi,\frac{3\pi}{2}\right[$, alors on fait une construction miroir des constructions précédentes : on fait une rotation de $r_{1}$ en son point final d'un angle $\frac{2\pi}{k}$ et on change le rôle de $r_{2}$ et de~$r_{s}$ dans la construction. Finalement, il reste à traiter le cas où $\arg(\tilde r_{3})=0$. Dans ce cas, la concaténation de $t,\zeta r_{1},r_{2},\tilde r_{3},t',r_{s}$ permet d'obtenir une différentielle avec les invariants souhaités via les constructions précédentes.
\end{proof}

Nous donnons maintenant une construction par récurrence du cas général en partant des strates avec un zéro d'ordre négatif.

\begin{prop}\label{prop:a1neq1}
Si $a_{1} \notin 1+k\mathbb{Z}$, alors chaque configuration $(R_{1},\dots,R_{s})$ de nombres complexes non nuls est réalisable dans la strate $\Omega^{k}\mathcal{M}_{0}(a_{1},a_{2};\rec[-k][s])$.
\end{prop}

\begin{proof}
Nous réalisons ici un raisonnement par récurrence sur le nombre $s$ de pôles d'ordre $k$. L'étape d'initialisation $s=2$ est démontrée dans le corollaire~\ref{cor:poleks2}. Nous supposerons l'énoncé valide pour tout rang $s'<s$.
\par
Pour n'importe quelle strate $\Omega^{k}\mathcal{M}_{0}(a_{1},a_{2};\rec[-k][s+1])$ telle que $a_{1}$ ou $a_{2}$ est négatif, les résultats de la section~\ref{sub:5N2negatif} établissent la proposition. Si $a_{1}$ et $a_{2}$ sont tous deux positifs (et donc $l_{1},l_{2} \geq 1$), alors nous devons nous appuyer sur l'étude de la stratification de résonance (voir la section~\ref{sec:arrhyp}). Dans la strate principale, formée des configurations de résidus génériques, c'est-à-dire ne vérifiant aucune condition de résonance, un argument de dimension montre que certaines configurations de cette strate sont nécessairement réalisables tandis que le prolongement plat du corollaire~\ref{cor:defplate} permet d'étendre ces solutions à toute la strate principale. 
\par
Dans chaque strate de résonance, des équations dont les coefficients sont des racines $k$-ièmes de l'unité sont satisfaites par des racines des $k$-résidus. Traitons d'abord le cas des strates dans lesquelles l'une de ces équations relie des racines de $t<s+1$ résidus de la configuration (certains coefficients sont nuls). Quitte à considérer $t$ minimal si plusieurs équations de ce type sont vérifiées les configuration de la strate de résonance choisie, nous pouvons construire la différentielle réalisant la configuration voulue en utilisant la chirurgie du lemme~\ref{lem:chirurgietranslation} permettant de se ramener à des strates plus simples (pour lesquelles l'hypothèse de récurrence garantit l'absence d'obstructions).
\par
Les lemmes~\ref{lem:a2-1TROISPOLES} et~\ref{lem:a2-1SPECIFIQUE} traitent enfin du cas des strates de résonance dans lesquelles toutes les conditions de résonance font intervenir les $s$ résidus de la configuration (aucun coefficient de l'équation de résonance n'est nul).
\end{proof}

\subsection{Constructions pour $n \geq 3$}\label{sec:constcaskn3}

Il reste à démontrer qu'il n'existe pas d'obstruction pour les strates de $k$-différentielles de genre zéro avec $k\geq3$, n'ayant que des pôles d'ordre $-k$ et au moins trois zéros.

D'abord nous traitons le cas $n=3$ où seulement deux zéros ont un ordre qui n'est pas divisible par $k$.

\begin{prop}\label{prop:n3zero2divpolek}
Considérons les strates $\Omega^{k}\mathcal{M}_{0}(a_{1},a_{2},kl;\rec[-k][s])$ où $k \geq 3$, $l \in \mathbb{N}^{\ast}$ et $a_{1},a_{2}$ sont premiers avec $k$. Leurs applications résiduelles sont surjectives.
\end{prop}

\begin{proof}
Nous allons obtenir les $k$-différentielles adéquates par scindage de zéro à partir des deux strates $\Omega^{k}\mathcal{M}_{0}(a_{1}+kl,a_{2};\rec[-k][s])$ et $\Omega^{k}\mathcal{M}_{0}(a_{1},a_{2}+kl;\rec[-k][s])$. Ces strates sont par définition primitives. Il suffit de vérifier que les deux strates ne sont pas simultanément des strates sporadiques.
Remarquons que si $a_{1}=a_{2}$, alors la strate $\Omega^{k}\mathcal{M}_{0}(a_{1},a_{1},kl;\rec[-k][s])$ n'est pas primitive (car $a_{1} \in \frac{k}{2}\mathbb{Z}$ dans ce cas). Par conséquent, les deux strates obtenues en regroupant $kl$ avec $a_{1}$ ou $a_{2}$ sont nécessairement distinctes. Ainsi, les seules strates problématiques sont celles pour lesquelles il existe deux strates sporadiques distinctes avec le même ordre $k$, le même nombre $s$ de pôles et le même résidu interdit.
Les seules paires de strates sporadiques de ce type sont :
\begin{enumerate}
    \item $\Omega^{4}\mathcal{M}_{0}(3,5;\rec[-4][4])$ et $\Omega^{4}\mathcal{M}_{0}(-1,9;\rec[-4][4])$ avec les résidus de la forme $\CC^{\ast}\cdot\rec[1][4]$;
    \item $\Omega^{3}\mathcal{M}_{0}(-1,4;\rec[-3][3])$ et $\Omega^{3}\mathcal{M}_{0}(1,2;\rec[-3][3])$ avec les résidus de la forme $\CC^{\ast}\cdot\rec[1][3]$;
    \item $\Omega^{3}\mathcal{M}_{0}(2,10;\rec[-3][6])$ et $\Omega^{3}\mathcal{M}_{0}(5,7;\rec[-3][6])$ avec les résidus de la forme $\CC^{\ast}\cdot\rec[1][6]$.
\end{enumerate}
Par conséquent, les strates problématiques pour lesquelles il est nécessaire de faire les constructions à la main sont :
\begin{enumerate}
    \item $\Omega^{4}\mathcal{M}_{0}(-1,4,5;\rec[-4][4])$ avec les résidus de la forme $\CC^{\ast}\cdot\rec[1][4]$;
    \item $\Omega^{3}\mathcal{M}_{0}(-1,1,3;\rec[-3][3])$ avec les résidus de la forme $\CC^{\ast}\cdot\rec[1][3]$;
    \item $\Omega^{3}\mathcal{M}_{0}(2,3,7;\rec[-3][6])$ avec les résidus de la forme $\CC^{\ast}\cdot\rec[1][6]$.
\end{enumerate}
Les constructions dans les deux premiers cas sont donnée dans la figure~\ref{fig:k4troiszeros}. 
\begin{figure}[hbt]
\centering
\begin{tikzpicture}[scale=2]
\fill[fill=black!10] (0.25,1.5) coordinate (p10) -- ++(.5,0)  coordinate[pos=.5] (q10)  coordinate (p11)   -- ++(-.75,.25)  coordinate[pos=.6] (q11)  coordinate (p12) -- ++(.25,.75) coordinate[pos=.5] (q12) coordinate (p13) -- ++ (-1.1,0) coordinate[pos=.75] (q13) coordinate (p14)  -- ++(0,-1)  -- (p10) coordinate[pos=.25] (q14)coordinate (p15);
\draw (q14) -- (p11) -- (p12) -- (p13) -- (q13);
\draw[dotted] (q14) --++(-.2,0);
\draw[dotted] (q13) --++(-.2,0);

  \foreach \i in {10,11,13}
    \filldraw[fill=white] (p\i)  circle (1pt); 
  \foreach \i in {12}
    \filldraw[fill=red] (p\i)  circle (1pt); 
    
\node[above,rotate=-20] at (q11) {$2$};
\node[below,rotate=70] at (q12) {$2$};
\node[] at (q10) {$3$};

\begin{scope}[yshift=.25cm]
%central
\filldraw[fill=black!10]  (0,0) coordinate (p1)  -- ++(1,0)  coordinate[pos=.5] (q1)  coordinate (p2) -- ++(0,1) coordinate[pos=.5] (q2) coordinate (p3) -- ++(-.25,-.25) coordinate[pos=.5] (q3) coordinate (p4)   -- ++(-.5,0)  coordinate[pos=.6] (q4)  coordinate (p5) -- ++(-.25,.25) coordinate[pos=.5] (q7)coordinate (p8) -- cycle;

 \foreach \i in {1,2,3,8}
     \fill (p\i)  circle (1pt); 
  \foreach \i in {4,5}
    \filldraw[fill=white] (p\i)  circle (1pt); 
    
    \node[rotate=45] at (q3) {$1$};
   \node[rotate=135] at (q7) {$1$};
   \node[] at (q4) {$3$};

   %gauche 
\fill[fill=black!10] (0,0) coordinate (p10) -- ++(0,1)  coordinate[pos=.6] (q10)  coordinate (p11)    -- ++ (-.6,0) coordinate[pos=.75] (q13) coordinate (p14)  -- ++(0,-1)  -- (p10) coordinate[pos=.25] (q14)coordinate (p15);
\draw (q14) -- (p10) -- (p11) -- (q13);
\draw[dotted] (q14) --++(-.1,0);
\draw[dotted] (q13) --++(-.1,0);

 \foreach \i in {10,11}
     \fill (p\i)  circle (1pt); 

%droit
\fill[fill=black!10] (1,0) coordinate (p10) -- ++(0,1)  coordinate[pos=.6] (q10)  coordinate (p11)    -- ++ (.6,0) coordinate[pos=.75] (q13) coordinate (p14)  -- ++(0,-1)  -- (p10) coordinate[pos=.25] (q14)coordinate (p15);
\draw (q14) -- (p10) -- (p11) -- (q13);
\draw[dotted] (q14) --++(.1,0);
\draw[dotted] (q13) --++(.1,0);
 \foreach \i in {10,11}
     \fill (p\i)  circle (1pt); 

 %below    
\fill[fill=black!10] (0,0) coordinate (p10) -- ++(1,0)  coordinate[pos=.6] (q10)  coordinate (p11)    -- ++ (0,-.6) coordinate[pos=.75] (q13) coordinate (p14)  -- ++(-1,0)  -- (p10) coordinate[pos=.25] (q14)coordinate (p15);
\draw (q14) -- (p10) -- (p11) -- (q13);
\draw[dotted] (q14) --++(0,-.1);
\draw[dotted] (q13) --++(0,-.1);
 \foreach \i in {10,11}
     \fill (p\i)  circle (1pt); 
\end{scope}

%deuxième construction
\begin{scope}[xshift=4.5cm,yshift=.5cm]
%central
\coordinate (p1) at (1,0);
\coordinate (p2) at (0,.58);
\coordinate (p3) at (0,1.73);
\coordinate (p4) at (-.5,.87);
\coordinate (p5) at (-1,0);
\coordinate (p6) at (-.5,-.29);
\coordinate (p7) at (.5,-.29);

\fill[fill=black!10]  (p1)  -- coordinate[pos=.5] (q1) (p2) --coordinate[pos=.5] (q2) (p3) --coordinate[pos=.5] (q3) (p4)   -- coordinate[pos=.6] (q4) (p5) -- coordinate[pos=.5] (q5) (p6) -- coordinate[pos=.5] (q6) (p7) --  (p1) coordinate[pos=.5] (q7);
\fill[fill=black!10] (p3) -- ++(150:.3) -- ++ (240:1) -- (p4) -- cycle; 
\fill[fill=black!10] (p4) -- ++(160:.3) -- ++ (240:1) -- (p5) -- cycle; 
\fill[fill=black!10] (p6)  -- ++(0,-.3) -- ++ (1,0) -- (p7) -- cycle; 
\draw (p1) --(p2) -- (p3) -- (p4) -- (p5) -- (p6) -- (p7) --(p1);
\draw (p3) -- ++(150:.3);
\draw (p4) -- ++(150:.3);
\draw (p4) -- ++(160:.3);
\draw (p5) -- ++(160:.3);
\draw (p6) -- ++(0,-.3);
\draw (p7) -- ++(0,-.3);

 \foreach \i in {1,3,4,5}
     \fill (p\i)  circle (1pt); 
  \foreach \i in {6,7}
    \filldraw[fill=white] (p\i)  circle (1pt); 
    \foreach \i in {2}
    \filldraw[fill=red] (p\i)  circle (1pt); 
    
    \node[above,rotate=-30] at (q1) {$1$};
   \node[right] at (q2) {$1$};
   \node[above,rotate=-30] at (q5) {$2$};
     \node[above,rotate=30] at (q7) {$2$};
\end{scope}
\end{tikzpicture}
\caption{Une différentielle quartique de 
% $\Omega^{4}\mathcal{M}_{0}(-1,3,6;\rec[-4][4])$ à gauche et 
$\Omega^{4}\mathcal{M}_{0}(-1,4,5;\rec[-4][4])$ dont les résidus sont $(1,1,1,1)$ à gauche et  une différentielle cubique de $\Omega^{3}\mathcal{M}_{0}(-1,1,3;\rec[-3][3])$ dont les résidus sont $(1,1,1)$ à droite.}\label{fig:k4troiszeros}
\end{figure} 
Concernant la strate $\Omega^{3}\mathcal{M}_{0}(2,3,7;\rec[-3][6])$, nous partons de la surface plate de la strate $\Omega^{3}\mathcal{M}_{0}(-1,1,3;\rec[-3][3])$ donnée à droite de la figure~\ref{fig:k4troiszeros}. Pour ajouter les trois pôles triples de résidus cubiques $(1,1,1)$ tout en ajoutant un angle conique de $2\pi$ et $4\pi$ respectivement au pôle simple et au zéro simple, nous utilisons la chirurgie du lemme~\ref{lem:chirurgietranslation}. Il suffit simplement de vérifier qu'il existe bien un lien-selle reliant le pôle simple et le zéro simple le long duquel découper une cicatrice et y insérer une surface de translation bien choisie. Le dessin montre qu'un tel lien-selle existe et que la construction est donc réalisable.
\end{proof}

Ensuite, nous donnons une construction spécifique aux cas où le scindage de zéros nécessite de partir de différentielles non primitives.

\begin{lem}\label{lem:n3divepolek}
Considérons les strates $\Omega^{k}\mathcal{M}_{0}(a_{1},a_{2},a_{3};\rec[-k][s])$ où $k \geq 3$. Si $\pgcd(a_{1},k) \notin \lbrace{ 1, \frac{k}{2} \rbrace}$, alors l'application résiduelle de la strate est surjective.
\end{lem}

\begin{proof}
Grâce à la proposition~\ref{prop:n3zero2divpolek}, nous pouvons nous contenter de traiter le cas où aucun des $a_{1},a_{2},a_{3}$ n'est un multiple de $k$. Nous distinguerons deux cas selon si $a_{1}>(s-1)k$ ou $a_{1}<(s-1)k$.

Si $a_{1}<(s-1)k$, alors $a_{2}+a_{3}>-k$ et la strate $\Omega^{k}\mathcal{M}_{0}(a_{1},a_{2}+a_{3};\rec[-k][s])$ n'est pas primitive (car toutes les singularités ont pour facteur commun $d=\pgcd(a_{1},k)$). Cependant, nous savons qu'elle n'est pas constituée de puissances de $1$-formes ou de différentielles quadratiques (car~$a_{1}$ n'est pas un multiple de $\frac{k}{2}$). Ce sont des puissances $d$-ièmes de $m$-différentielles où $dm=k$ avec $m \geq 3$. La proposition~\ref{prop:eclatZero} garantit que l'on peut éclater le zéro d'ordre $a_{2}+a_{3}$ sans modifier les résidus. Pour chaque configuration de $k$-résidus $(R_{1},\dots,R_{s})$, il y a plusieurs $m$-résidus $(\sqrt[d]{R_{1}},\dots,\sqrt[d]{R_{s}})$ compatibles avec l'équation~\eqref{eq:multiplires}. Il est facile de vérifier que tous ne peuvent pas être interdits dans la strate $\Omega^{k/d}\mathcal{M}_{0}(\frac{a_{1}}{d},\frac{a_{2}+a_{3}}{d}
;\rec[-m][s])$. Il suffit pour cela de se référer au cas $n=2$ (déjà démontré à ce stade) du théorème~\ref{thm:geq0kspe} (puisque $m \geq 3$).

Si $a_{1}>(s-1)k$, alors $a_{2}+a_{3}<-k$. Il s'ensuit que $-k<a_{2},a_{3}<0$. Si $\pgcd(a_{2},k) \notin \lbrace{ 1, \frac{k}{2} \rbrace}$, alors le raisonnement qui précède s'applique quitte à remplacer $a_{1}$ par $a_{2}$. Nous pouvons donc supposer que $\pgcd(a_{2},k),\pgcd(a_{3},k) \in \lbrace{ 1, \frac{k}{2} \rbrace}$. Comme $a_{2}$ et $a_{3}$ ne peuvent pas tous deux valoir $-\frac{k}{2}$ (sinon $a_{1}$ serait un multiple de $k$), alors l'un de ces deux nombres (disons $a_{2}$ sans perte de généralité) satisfait $-k< a_{2} < -\frac{k}{2}$ et $\pgcd(a_{2},k)=1$.

Les différentielles de la strate $\Omega^{k}\mathcal{M}_{0}(a_{2},a_{1}+a_{3};\rec[-k][s])$ sont donc primitives. Le cas $n=2$ du théorème \ref{thm:geq0kspe} nous apprend que pour une strate avec un zéro d'ordre négatif l'application résiduelle est toujours surjective sauf parfois quand ce zéro est d'ordre $-1$. Nous ne sommes pas dans ce cas. Il s'ensuit que l'application résiduelle de la strate $\Omega^{k}\mathcal{M}_{0}(a_{2},a_{1}+a_{3};\rec[-k][s])$ est surjective.
\end{proof}

Le cas de strates dans lequel l'ordre d'un zéro est un multiple impair de $\frac{k}{2}$ doit se traiter à part.

\begin{lem}\label{lem:quadrexpcept}
L'application résiduelle des strates $\Omega^{k}\mathcal{M}_{0}(a_{1},a_{2},\frac{kl}{2};\rec[-k][s])$ où $k \geq 3$, $l \in 2\mathbb{N}-1$, $\pgcd(a_{1},k)=\pgcd(a_{2},k)=1$   est surjective.
\end{lem}

\begin{proof}
Sans perte de généralité, nous supposerons $a_{1} \leq a_{2}$. Considérons d'abord le cas dans lequel $a_{1}+\frac{kl}{2} \leq -k$. Cela signifie en particulier que $l=-1$ et $-k < a_{1} < -\frac{k}{2}$. Comme dans les preuves précédentes, nous obtenons les différentielles de $\Omega^{k}\mathcal{M}_{0}(a_{1},a_{2},\frac{kl}{2};\rec[-k][s])$ par scindage en partant de la strate primitive $\Omega^{k}\mathcal{M}_{0}(a_{1},a_{2}+\frac{kl}{2};\rec[-k][s])$. Le cas $n=2$ du théorème~\ref{thm:geq0kspe} prouve que pour une strate avec un zéro d'ordre négatif différent de $-1$, l'application résiduelle est toujours surjective.

Dans un second cas, on a $a_{1}+\frac{kl}{2}> -k$. Cela signifie que l'on peut obtenir les différentielles de la strate $\Omega^{k}\mathcal{M}_{0}(a_{1},a_{2},\frac{kl}{2};\rec[-k][s])$ par scindage de zéro en partant des deux strates primitives $\Omega^{k}\mathcal{M}_{0}(a_{1},a_{2}+\frac{kl}{2};\rec[-k][s])$ et $\Omega^{k}\mathcal{M}_{0}(a_{2},a_{1}+\frac{kl}{2};\rec[-k][s])$. Ainsi, l'application résiduelle de $\Omega^{k}\mathcal{M}_{0}(a_{1},a_{2},\frac{kl}{2};\rec[-k][s])$ est surjective sauf peut-être si ces deux strates sont sporadiques (voir cas $n=2$ du théorème~\ref{thm:geq0kspe}). Sachant que $k$ est pair (puisque $\frac{k}{2}$ est un entier), ceci n'arrive que pour un petit nombre de strates de départ\footnote{En dehors du cas $(k,s)=(4,4)$, il n'y a pas plus d'une strate sporadique pour chaque valeur de $(k,s)$, ce qui implique que $a_{1}=a_{2}$ pour chacun de ces autres cas.}. Nous en donnons la liste ici (ainsi que les configurations de résidus interdites dans les deux strates sporadiques incidentes) :
\begin{enumerate}
    \item $\Omega^{k}\mathcal{M}_{0}(-1,-1,2;-k,-k)$ avec $k \geq 4$ pour les configurations de $\CC^{\ast}\cdot(1,1)$;
    \item $\Omega^{4}\mathcal{M}_{0}(-1,-1,6;-4,-4,-4)$  pour les configurations de $\CC^{\ast}\cdot(1,1,-4)$;
    \item $\Omega^{4}\mathcal{M}_{0}(3,3,2;-4,-4,-4,-4)$  pour les configurations de $\CC^{\ast}\cdot(1,1,1,1)$; 
    \item $\Omega^{4}\mathcal{M}_{0}(-1,-1,10;-4,-4,-4,-4)$ pour les configurations de $\CC^{\ast}\cdot(1,1,1,1)$;  
    \item $\Omega^{4}\mathcal{M}_{0}(-1,3,6;-4,-4,-4,-4)$ pour les configurations de $\CC^{\ast}\cdot(1,1,1,1)$;
    \item $\Omega^{4}\mathcal{M}_{0}(3,3,10;\rec[-4][6])$ pour les configurations de $\CC^{\ast}\cdot\rec[1][6]$.
\end{enumerate}
Pour réaliser la configuration $(1,1)$ (ou ses multiples) dans la strate
$\Omega^{k}\mathcal{M}_{0}(-1,-1,2;-k,-k)$ avec $k$ pair vérifiant $k \geq 4$, nous effectuons la construction suivante : dessinons un hexagone en reliant les points $0$, $1$, $1+\frac{1}{2}\exp\left(\frac{i\pi(k-1)}{k}\right)$, $1+i\sin(\pi/k)$, $i\sin(\pi/k)$ et $\frac{1}{2}\exp\left(\frac{i\pi}{k}\right)$ (dans cet ordre). Nous collons un cylindre sur les segments $[0,1]$ et $[1+i\sin(\pi/k),i\sin(\pi/k)]$ puis identifions les deux segments dans les deux paires restantes de côtés adjacents. Ainsi, nous obtenons bien une surface avec deux singularités coniques d'angle $\frac{(k-1)}{k}2\pi$ et une d'angle $\frac{(k+2)}{k}2\pi$.

Dans chacun des autres cas exceptionnels, nous allons réaliser un scindage de zéro en partant de l'une des quatre strates  de différentielles quadratiques 
\[\Omega^{2}\mathcal{M}_{0}(-1,3;-2,-2,-2),\,\Omega^{2}\mathcal{M}_{0}(1,3;\rec[-2][4]),\, \Omega^{2}\mathcal{M}_{0}(-1,5;\rec[-2][4]) \text{ et } \Omega^{2}\mathcal{M}_{0}(3,5;\rec[-2][6])\,.\]
% \begin{enumerate}
%     \item $\Omega^{2}\mathcal{M}_{0}(-1,3;-2,-2,-2)$;
%     \item $\Omega^{2}\mathcal{M}_{0}(1,3;-2,-2,-2,-2)$;
%     \item $\Omega^{2}\mathcal{M}_{0}(-1,5;-2,-2,-2,-2)$;
%     \item $\Omega^{2}\mathcal{M}_{0}(3,5;\rec[-2][6])$.
% \end{enumerate}
Dans les strates $\Omega^{2}\mathcal{M}_{0}(-1,3;-2,-2,-2)$ et $\Omega^{2}\mathcal{M}_{0}(-1,5;\rec[-2][4])$, toute configuration de résidus quadratiques dans laquelle les résidus ne sont pas tous sur la même demi-droite partant de l'origine est réalisable (théorème~1.9 de \cite{getaquad}). Il existe donc toujours une différentielle quadratique pour laquelle le scindage de zéro produira la différentielle quartique adéquate.

Dans les derniers cas, il s'agit de réaliser les configurations $\rec[1][4]$ dans $\Omega^{4}\mathcal{M}_{0}(3,3,2;\rec[-4][4])$ et $\rec[1][6]$ dans $\Omega^{4}\mathcal{M}_{0}(3,3,10;\rec[-4][6])$. Nous partons donc de différentielles réalisant les configurations de résidus quadratiques $(\rec[1][3],-1)$ et $(\rec[1][5],-1)$ dans $\Omega^{2}\mathcal{M}_{0}(1,3;-2,-2,-2,-2)$ et $\Omega^{2}\mathcal{M}_{0}(3,5;\rec[-2][6])$. Le théorème~1.9 de \cite{getaquad} montre l'existence de telles différentielles.
\end{proof}

Enfin, nous montrons qu'il n'existe pas d'obstruction pour les strates de $k$-différentielles de genre zéro avec $k\geq3$, n'ayant que des pôles d'ordre $-k$ et exactement trois zéros.

\begin{prop}\label{prop:n3polek}
L'application résiduelle de la strate $\Omega^{k}\mathcal{M}_{0}(a_{1},a_{2},a_{3};\rec[-k][s])$ est surjective pour $k\geq3$.
\end{prop}

\begin{proof}
En utilisant la proposition~\ref{prop:n3zero2divpolek} ainsi que le lemme~\ref{lem:n3divepolek}, nous pouvons nous restreindre aux strates pour lesquelles $\pgcd(a_{i},k) \in \lbrace{ 1,\frac{k}{2} \rbrace}$ pour chaque $a_{i}$. Dans ces conditions, au plus l'un de ces ordres peut être un multiple de $\frac{k}{2}$ (si deux l'étaient, alors le troisième serait un multiple de $k$). Ce cas est traité dans le lemme~\ref{lem:quadrexpcept}. Nous supposerons donc que chacun des $a_{1},a_{2},a_{3}$ est premier avec $k$. De plus, $k$ ne peut pas être un nombre pair sans quoi les trois ordres $a_{1},a_{2},a_{3}$ devraient être tous impairs, ce qui est impossible. Nous pouvons donc supposer également que $k$ est impair.

Comme chacun des $a_{1},a_{2},a_{3}$ est premier avec $k$, en supposant $a_{1} \leq a_{2} \leq a_{3}$, la strate $\Omega^{k}\mathcal{M}_{0}(a_{1},a_{2}+a_{3};\rec[-k][s])$ est nécessairement primitive. Si cette strate n'est pas sporadique, cela démontre que l'application résiduelle de la strate $\Omega^{k}\mathcal{M}_{0}(a_{1},a_{2},a_{3};\rec[-k][s])$ est surjective. Si l'on se réfère au cas $n=2$ du théorème~\ref{thm:geq0kspe}, on constate que les seules strates sporadiques pour lesquelles $k$ est impair sont :
\begin{itemize}
    \item les strates de la forme $\Omega^{k}\mathcal{M}_{0}(-1,1;-k,-k)$ où les résidus interdits sont les $\CC^{\ast}(1,-1)$;
    \item six strates de différentielles cubiques ($k=3$) pour lesquelles le nombre $s$ de pôles satisfait $3 \leq s \leq 6$.
\end{itemize}

Supposons d'abord que $s=2$. Si $\Omega^{k}\mathcal{M}_{0}(a_{1},a_{2}+a_{3};\rec[-k][s])$ est une strate sporadique, alors $a_{1}=-1$ et $a_{2}+a_{3}=1$. Ceci implique que $a_{2}=-1$ et $a_{3}=2$. Or, nous pouvons toujours obtenir la $k$-différentielle voulue par scindage à partir de la strate $\Omega^{k}\mathcal{M}_{0}(-2,2;-k,-k)$ avec~$k$ impair.

Supposons maintenant que $s=3$. Si $\Omega^{3}\mathcal{M}_{0}(a_{1},a_{2}+a_{3},\rec[-3][3])$ est une strate sporadique, alors soit $a_{1}=-1$, soit $a_{1}=1$. Dans le premier cas, nous avons $a_{2}+a_{3}=4$ et donc $(a_{2},a_{3})$ vaut $(-1,5)$ ou $(2,2)$. Dans le second cas, nous avons $a_{2}+a_{3}=2$, ce qui implique que $a_{1}=a_{2}=a_{3}$. Les strates problématiques sont donc $\Omega^{3}\mathcal{M}_{0}(-1,-1,5;\rec[-3][3])$,
$\Omega^{3}\mathcal{M}_{0}(-1,2,2;\rec[-3][3])$ et $\Omega^{3}\mathcal{M}_{0}(1,1,1;\rec[-3][3])$. Pour la première de ces strates, on peut obtenir la $k$-différentielle voulue par scindage à partir de la strate $\Omega^{3}\mathcal{M}_{0}(-2,5;\rec[-3][5])$. Pour les deux autres, la construction d'une différentielle cubique dont les résidus sont $(1,1,1)$ est donnée dans la figure~\ref{fig:k3troiszeros}.

\begin{figure}[hbt]
\centering
\begin{tikzpicture}[scale=1.5]
%premiere construction
\begin{scope}[xshift=0cm]
 \foreach \i in {0,1,...,5}
    \coordinate (p\i)  at (60*\i:1); 
 \draw (p0) -- coordinate (q1) (p5);
  \draw (p4) -- coordinate (q2) (p3);
  \draw (p1) -- coordinate (q3) ++(120:1) coordinate (p6);
    \draw (p2) -- coordinate (q4) ++(60:1);
        \fill[fill=black!10]  (p0)  -- (p6) -- (p3) -- (p4) -- (p5) -- (p0);
 \fill[fill=black!10]  (p4)  -- ++(-90:.3) -- ++ (1,0)  -- (p5) -- (p4);
  \fill[fill=black!10]  (p0)  -- ++(30:.3) -- ++ (120:1)  -- (p1) -- (p0);
    \fill[fill=black!10]  (p2)  -- ++(150:.3) -- ++ (-120:1)  -- (p3) -- (p2);
\node[rotate=60] at (q1) {$1$};
\node[rotate=-60] at (q2) {$2$};
\node[rotate=-60] at (q3) {$1$};
\node[rotate=60] at (q4) {$2$};

 \draw (p0) -- ++(30:.3);
 \draw (p1) -- ++(30:.3);
 \draw (p2) -- ++(150:.3);
 \draw (p3) -- ++(150:.3);
 \draw (p4) -- ++(-90:.3);
 \draw (p5) -- ++(-90:.3);
 \foreach \i in {0,1}
     \fill (p\i)  circle (1pt); 
  \foreach \i in {4,5,6}
    \filldraw[fill=white] (p\i)  circle (1pt); 
    \foreach \i in {2,3}
    \filldraw[fill=red] (p\i)  circle (1pt); 
\end{scope}

%%%%%%%%%%%%%%%%%%%%%%%%%%%%%%
%seconde construction

\begin{scope}[xshift=3.5cm,yshift=-.5cm,scale=1.5]
\fill[fill=black!10]  (0,0) coordinate (p1)  -- ++(1,0)  coordinate[pos=.5] (q1)  coordinate (p2) --  ++(90:0.5) coordinate[pos=.5] (q2) coordinate (p3) -- ++(-30:.5) coordinate[pos=.5] (q3) coordinate (p4)   -- ++(60:1)  coordinate[pos=.6] (q4)  coordinate (p5) -- ++(180:1) coordinate[pos=.5] (q5)coordinate (p6) -- ++(-90:0.5) coordinate[pos=.5] (q6)coordinate (p7) -- ++(150:0.5) coordinate[pos=.5] (q7) coordinate (p8) -- (p1) coordinate[pos=.5] (q8);
        \fill[fill=black!10]  (p1)  -- ++(-90:.3)coordinate (r1) -- ++ (1,0)coordinate (r2)  -- (p2)  -- (p1);
 \fill[fill=black!10]  (p4)  -- ++(-30:.3)coordinate (r3) -- ++ (60:1) coordinate (r4) -- (p5) -- (p4);
        \fill[fill=black!10]  (p5)  -- ++(90:.3) coordinate (r5) -- ++ (-1,0)coordinate (r6)  -- (p6)  -- (p5);
    \fill[fill=black!10]  (p8)  -- ++(150:.3)coordinate (r7) -- ++ (240:1)coordinate (r8)  -- (p1) -- (p8);
    
    \draw (r8) -- (p1) -- (r1);
    \draw (r2) -- (p2) -- (p3) -- (p4) -- (r3);
      \draw (r4) -- (p5) -- (r5);
    \draw (r6) -- (p6) -- (p7) -- (p8) -- (r7);

    \node[rotate=90] at (q2) {$1$};
\node[rotate=-30] at (q3) {$2$};
\node[rotate=-30] at (q7) {$1$};
\node[rotate=90] at (q6) {$2$};
     \foreach \i in {1,2,8}
     \fill (p\i)  circle (1pt); 
  \foreach \i in {3,7}
    \filldraw[fill=white] (p\i)  circle (1pt); 
    \foreach \i in {4,5,6}
    \filldraw[fill=red] (p\i)  circle (1pt); 
\end{scope}

\begin{scope}[xshift=1.5cm,yshift=-2.5cm,scale=1.5]
%central
\fill[fill=black!10]  (0,0) coordinate (p1)  -- ++(1,0)  coordinate[pos=.5] (q1)  coordinate (p2) --  ++(-60:0.5) coordinate[pos=.5] (q2) coordinate (p3) -- ++(120:1) coordinate[pos=.5] (q3) coordinate (p4)   -- ++(-150:0.288)  coordinate[pos=.6] (q4)  coordinate (p5) -- ++(150:.288) coordinate[pos=.5] (q5)coordinate (p6) -- ++(-120:0.5) coordinate[pos=.5] (q6)coordinate (p7) -- ++(60:0.5) coordinate[pos=.5] (q7) coordinate (p7) -- ++(-120:1) coordinate[pos=.5] (q7) coordinate (p8)  --  (p1)  coordinate[pos=.5] (q8);

 \fill[fill=black!10] (p1) -- ++(0,-.5) -- ++ (1,0) -- (p2) -- cycle; 
 \fill[fill=black!10] (p3) -- ++(.5,.1) -- ++ (120:1) -- (p4) -- cycle; 
 \fill[fill=black!10] (p7)  -- ++(-.5,.1) -- ++ (-120:1) -- (p8) -- cycle; 
 \draw (p2) --(p3);
 \draw (p4) -- (p5) -- (p6);
 \draw (p8) -- (p1);
 \draw (p1) -- ++(0,-.5);
 \draw (p2) -- ++(0,-.5);
 \draw (p3) -- ++(.5,.1);
 \draw (p4) -- ++(.5,.1);
 \draw (p7) -- ++(-.5,.1);
 \draw (p8) -- ++(-.5,.1);

 \foreach \i in {1,2}
     \fill (p\i)  circle (1pt); 
  \foreach \i in {3,4,6,7,8}
    \filldraw[fill=white] (p\i)  circle (1pt); 
    \foreach \i in {5}
    \filldraw[fill=red] (p\i)  circle (1pt); 
    
    \node[above,rotate=-60] at (q2) {$1$};
   \node[below,rotate=-120] at (q8) {$1$};
   \node[right,rotate=-60] at (q4) {$2$};
     \node[left,rotate=60] at (q5) {$2$};
\end{scope}
\end{tikzpicture}
\caption{Différentielles cubiques des strates $\Omega^{3}\mathcal{M}_{0}(1,1,1;\rec[-3][3])$ (à gauche), $\Omega^{3}\mathcal{M}_{0}(-1,2,2;\rec[-3][3])$ (en bas) et $\Omega^{3}\mathcal{M}_{0}(2,2,2;\rec[-3][4])$ (à droite) dont les résidus sont  $(1,1,1)$ dans les deux premiers cas et $(1,1,-1,-1)$ dans le dernier.}\label{fig:k3troiszeros}
\end{figure} 

Supposons maintenant que $s=4$. Si $\Omega^{3}\mathcal{M}_{0}(a_{1},a_{2}+a_{3},\rec[-3][4])$ est une strate sporadique, alors $a_{1}=2$ et $a_{2}+a_{3}=4$. Il s'ensuit que $a_{1}=a_{2}=a_{3}=2$. La construction d'une différentielle cubique de $\Omega^{3}\mathcal{M}_{0}(2,2,5;\rec[-3][4])$ dont les résidus sont $(1,1,-1,-1)$ est donnée à droite la figure~\ref{fig:k3troiszeros}.

Supposons ensuite que $s=5$. Si $\Omega^{3}\mathcal{M}_{0}(a_{1},a_{2}+a_{3},\rec[-3][5])$ est une strate sporadique, alors $a_{1}=2$ et $a_{2}+a_{3}=7$. Il s'ensuit que $a_{2}=2$ et $a_{3}=5$ (car $a_{2}$ est premier avec $k=3$). Nous pouvons obtenir la $k$-différentielle voulue par scindage à partir de la strate $\Omega^{3}\mathcal{M}_{0}(4,5;\rec[-3][5])$.

Traitons enfin le cas $s=6$. Si $\Omega^{3}\mathcal{M}_{0}(a_{1},a_{2}+a_{3},\rec[-3][6])$ est une strate sporadique, celle-ci ne peut pas être $\Omega^{3}\mathcal{M}_{0}(5,7,\rec[-3][6])$ car $a_{2}+a_{3} \geq 2a_{1}$. Donc la strate sporadique est $\Omega^{3}\mathcal{M}_{0}(2,10,\rec[-3][6])$. Comme les $a_{1},a_{2},a_{3}$ ne sont pas divisibles par $3$, il s'ensuit que $a_{1}=2$ et $a_{2}=a_{3}=5$. Le résidu $\rec[1][6]$ nécessite une construction spécifique donnée. On part de la $3$-différentielle de $\Omega^{3}\mathcal{M}_{0}(-1,2,2;\rec[-3][3])$ représentée en bas de la figure~\ref{fig:k3troiszeros}.

Pour ajouter les trois pôles triples de résidus cubiques $(1,1,1)$ tout en ajoutant un angle conique de $2\pi$ et $4\pi$ respectivement au pôle simple et à un zéro double, nous utilisons la chirurgie du lemme~\ref{lem:chirurgietranslation}. Il suffit simplement de vérifier qu'il existe bien un lien-selle reliant le pôle simple et un zéro double le long duquel découper une cicatrice et y insérer une surface de translation bien choisie.
\end{proof}

Le théorème \ref{thm:geq0kspe} peut désormais être démontré en toute généralité.

\begin{proof}[Preuve du théorème \ref{thm:geq0kspe}]
Lorsque $n=2$, la caractérisation a déjà été établie dans les sections~\ref{sec:spora},~\ref{sub:5N2negatif} et~\ref{sub:mechant}. La proposition~\ref{prop:n3polek} établit la surjectivité de l'application $k$-résiduelle lorsque $n=3$. Il reste à la démontrer pour $n \geq 4$. La preuve se fait par scindage de zéros à partir du cas $n=3$.

Le lemme~\ref{lem:combiscindage} garantit que nous pouvons toujours obtenir toute strate avec $n \geq 4$ à partir d'une strate vérifiant $n=3$ et qui n'est pas constituée par des $k$-différentielles qui sont des puissances de différentielles abélienne ou quadratiques. L'équation~\eqref{eq:multiplires} établit alors la surjectivité de l'application $k$-résiduelle des strates avec $n \geq 4$ zéros.
\end{proof}

\section{$k$-différentielles en genre supérieur ou égal à un}
\label{sec:ggeq1}

Dans cette section nous traitons le cas des strates de différentielles en genre $g\geq1$. Plus précisément, nous prouvons les théorèmes~\ref{thm:CC1} et~\ref{thm:ggeq2}, dans la section~\ref{sec:pluri1} pour les strates ayant au moins un pôle d'ordre strictement inférieur à~$-k$ et dans la section~\ref{sec:pluri2} si tous les pôles sont d'ordre~$-k$. Le théorème~\ref{thm:strateshol} sur les strates sans pôles est prouvé dans la section~\ref{sec:plurifini}.

\subsection{Rappels sur les composantes connexes des strates}
\label{sec:rapcomp}

Nous rappelons certains résultats sur les composantes connexes des strates de $k$-différentielles.
La classification complète  n'est pas connue pour $k\geq3$, toutefois, le cas $g=1$ est bien connu pour tout $k$ et \cite{chge} décrit certaines composantes pour tout $g\geq2$. 
\smallskip
\par
\paragraph{\bf Les composantes en genre $g=1$.}
Du point de vue plat, les composantes connexes sont caractérisées par le nombre de rotation $\rot(S)$ de la surface plate associée à la $k$-différentielle~$\xi$. Une discussion détaillée est donnée dans la section~2.4 de \cite{chge}.  Pour une surface plate $S$ définie par une $k$-différentielle de $\Omega^{k}\mathcal{M}_{1}(a_{1},\dots,a_{n};-b_{1},\dots,-b_{p})$ avec une base symplectique de lacets lisses de l'homologie $(\alpha,\beta)$ le {\em nombre de rotation} est

\begin{equation}\label{eq:nbrotation}
\rot(S):=\pgcd\left(a_{1},\dots,a_{n};b_{1},\dots,b_{p},\ind(\alpha),\ind(\beta)\right)\,,
\end{equation}
où $\ind(\cdot)$ est l'indice de l’application de Gauss du lacet. L'indice est égal au nombre de $\tfrac{1}{k}$-tours effectués par son vecteur tangent. On a alors le résultat suivant.

\begin{lem}\label{lem:nbrotationg1}
Si $n=p=1$, la strate est $\Omega^{k}\mathcal{M}_{1}(a;-a)$ avec $a\geq2$ et chaque composante connexe correspond à un nombre de rotation qui est un diviseur strict de $a$.
\par
Si au contraire $n+p \geq 3$, il existe une composante connexe correspondant à chaque nombre de rotation qui est un diviseur de $\pgcd(a_{1},\dots,a_{n};b_{1},\dots,b_{p})$.
\end{lem}

La composante de $\komoduli[1](\mu)$ associée au nombre de rotation $\rot = \rho$ est notée $\komoduli[1]^{\rho}(\mu)$. La restriction de l'application résiduelle $\appresk[1](\mu)$ à cette composante est notée $\appreskrho[1](\mu)$.
\par
On termine le cas de genre $g=1$ en énonçant deux résultats qui se déduisent simplement des propriétés élémentaires du nombre de rotation. Le premier résultat se déduit de la proposition~3.13 de \cite{chge}.
\begin{lem}\label{lem:rotprim}
Une composante $\komoduli[1]^{\rho}(\mu)$ paramètre des $k$-différentielles primitives si et seulement si $\rho$ est premier avec $k$.
\end{lem}
L'éclatement d'un zéro d'une différentielle décrite dans la  proposition~\ref{prop:eclatZero} change le nombre de rotation de la façon suivante.
\begin{lem}
 Le nombre de rotation d'une différentielle obtenue en éclatant un zéro d'une différentielle de nombre de rotation $\rho$ est égal à $\rho$ modulo les ordres des nouveaux zéros. 
\end{lem}
\smallskip
\par
\paragraph{\bf Les composantes en genre $g\geq2$.}
Les composantes ne sont pas connues dans ce cas. Toutefois, \cite{chge} décrit deux types de composantes : les composantes hyperelliptiques et les composantes paires et impaires.

Les composantes paires et impaires sont distinguées par la parité de leur revêtement canonique. Elles existent pour $k$ impair mais pas pour $k$ pair. Pour nous la propriété importante de ces composantes est qu'elles peuvent être obtenues en cousant des anses et en éclatant des zéros à partir d'une $k$-différentielle primitive de genre~$1$ (voir la section~5.3 de \cite{chge}).

Les composantes hyperelliptiques sont composées de $k$-différentielles $(X,\xi)$ telles que $X$ est hyperelliptique et $\xi$ est $(-1)^{k}$-invariante par l'involution hyperelliptique. Lorsque la $k$-différentielle possède un pôle, le théorème~1.1 de \cite{chge} donne quatre types de partitions $\mu$ possibles :
\begin{enumerate}[i)]
 \item $\mu=(2m_{1},-2m_{2})$,
 \item $\mu=(m_{1},m_{1},-2m_{2})$,
 \item $\mu=(2m_{1},-m_{2},-m_{2})$,
 \item $\mu =(m_{1},m_{1},-m_{2},-m_{2})$,
\end{enumerate}
où $2m_{1}-2m_{2}=k(2g-2)$. Nous sommes intéressés par les cas où ces $k$-différentielles sont primitives et possèdent un pôle d'ordre divisible par $k$. Nous montrons maintenant que ces $k$-différentielles n'existent pas.
\begin{lem}\label{lem:nocomphyp}
Les composantes hyperelliptiques primitives possédant au moins un pôle d'ordre divisible par $k\geq2$ n'apparaissent que si $k=2$.
\end{lem}
Notons que ce sont les composantes qui sont données dans le théorème~1.3 de \cite{chge}. En particulier, il n'y a pas de composantes hyperelliptiques qui nous intéressent.
\begin{proof}
Supposons que nous sommes dans le cas $\mu=(2m_{1},-2m_{2})$. Comme $2m_{2}$ est divisible par $k$ et que $2m_{1}-2m_{2}=2k(g-1)$, cette partition est de la forme $(kn_{1},-kn_{2})$. Si $k$ est impair, alors les $n_{i}$ sont pairs et ces différentielles sont les puissance $k$-ièmes de différentielles abéliennes hyperelliptiques de type $(n_{1},-n_{2})$. Dans le cas où $k$ est pair, on peut diviser $k$ et~$\mu$ par un même diviseur de $k$ et obtenir une partition de la forme $\mu'=(2m_{1}',-2m_{2}')$ sauf si $k=2$ et $m_{1},m_{2}$ sont impairs.

Nous donnons maintenant les arguments dans le cas $\mu =(m_{1},m_{1},-m_{2},-m_{2})$. Dans ce cas, par hypothèse $m_{2}=kn_{2}$. Comme $2m_{1}-2m_{2}=2k(g-1)$, on en déduit que $m_{1}$ est divisible par $k$. La partition est donc $\mu =(kn_{1},kn_{1},-kn_{2},-kn_{2})$ et ces différentielles sont les puissances $k$-ièmes de différentielles abéliennes hyperelliptiques de type $(n_{1},n_{1},-n_{2},-n_{2})$.

Les autres cas sont similaires à ces deux précédents.
\end{proof}

\subsection{$k$-différentielles avec un pôle d'ordre strictement inférieur à~$-k$}
\label{sec:pluri1}

Dans cette section, nous considérons les strates avec au moins un pôle d'ordre~$<-k$.
 Nous traitons le cas des strates en genre $g=1$ dans les lemmes~\ref{lem:geq1CC} à~\ref{lem:geq1qua}. Le cas des strates de genre $g\geq2$ est traité dans le lemme~\ref{lem:geq1quabis}.
\smallskip
\par
Nous commençons par le cas avec un unique zéro et un unique pôle.
\begin{lem}\label{lem:geq1CC}
L'image de l'application $k$-résiduelle de $\Omega^{3}\mathcal{M}_{1}^{1}(6;-6)$ est $\CC^{\ast}$.\newline
Elle est surjective pour les composantes  $\Omega^{k}\mathcal{M}_{1}^{\rho}(m;-m) \neq \Omega^{3}\mathcal{M}_{1}^{1}(6;-6) $ où  $m>k\geq3$.
\end{lem}

\begin{proof}
Si $k\nmid m$, la surjectivité de l'application résiduelle de $\komoduli[1]^{\rho}(m;-m)$ est équivalent au fait que la composante soit non vide. Ceci est une conséquence élémentaire du théorème d'Abel. \`A partir de maintenant, nous supposons que $k\mid m$. 

D'après le lemme~\ref{lem:rotprim}, les composantes $\komoduli[1]^{\rho}(\ell k,-\ell k)$ avec $\ell\geq2$ sont primitives si et seulement si $\rho$ est un diviseur de $\ell$ premier avec~$k$.

Commençons par les cas des composantes $\komoduli[1]^{\rho}(2k,-2k)$ et où le résidu est non nul. Le lemme~\ref{lem:rotprim} implique soit que $\rho=1$, soit que $\rho=2$ si $k$ est impair. On part de la $k$-partie polaire non triviale de type $k\ell$ associée à $(\emptyset;v_{1},v_{2},v_{3},v_{4})$ avec $v_{1}=v_{2}=1$ et $v_{3}=v_{4}= \exp\left(\frac{(k-2\rho)\pi}{k}\right)$. Cette partie polaire est représentée sur la figure~\ref{fig:a-a}.
 \begin{figure}[htb]
 \centering
\begin{tikzpicture}[scale=1,decoration={
    markings,
    mark=at position 0.5 with {\arrow[very thick]{>}}}]
\fill[fill=black!10] (0,0) --(4,2) -- (8,0)  -- ++(0,2.4) -- ++(-8,0) -- cycle;

      \draw (0,0) coordinate (a1) --coordinate[pos=.5] (x1) node [above,sloped] {$v_{1}$} (2,1) coordinate (a2) --coordinate[pos=.5] (x2) node [above,sloped] {$v_{2}$} (4,2) coordinate (a3) --coordinate[pos=.5] (x3) node [above,sloped] {$v_{3}$} (6,1) coordinate (a4) --coordinate[pos=.5] (x4) node [above,sloped] {$v_{4}$} (8,0)
coordinate (a5);
  \foreach \i in {1,2,...,5}
  \fill (a\i) circle (2pt);

  \draw (a1)-- ++(0,2.5)coordinate[pos=.6](b);
    \draw (a5)-- ++(0,2.5)coordinate[pos=.6](c);
    \node[left] at (b) {$3$};
     \node[right] at (c) {$3$};
  \draw[->] (3.6,1.8) arc  (200:340:.4); \node at (4,1.3) {$2\rho\pi/k$};
  
      \draw[postaction={decorate},red] (x1) .. controls ++(-60:1.4) and ++(-120:1.5)  .. (x3) coordinate[pos=.5](y1);
     \node[below,red] at (y1) {$\alpha$};

   \draw[postaction={decorate},blue] (x4) .. controls ++(-120:1.5) and ++(-60:1.4)  ..  (x2)coordinate[pos=.5](y2);
     \node[below,blue] at (y2) {$\beta$};

\end{tikzpicture}
\caption{Une $k$-différentielle de $\Omega^{k}\mathcal{M}_{1}(2k;-2k)$ (en blanc)  et  $\Omega^{k}\mathcal{M}_{1}(k;-k)$ (en gris).}
\label{fig:a-a}
\end{figure}
On colle $v_{1}$ à $-v_{3}$ et $v_{2}$ à $-v_{4}$ par rotation d'angle $2\rho\pi/k$. Dans ce cas les indices de $\alpha$ et $\beta$ sont égaux à $\pm \rho$. L'équation~\eqref{eq:nbrotation} montre que le nombre de rotation de cette différentielle est $\rho$.
\par
Les composantes $\komoduli[1]^{\rho}(\ell k,-\ell k)$ avec $\ell\geq2$ et le résidu est non nul sont obtenues de la façon suivante. On écrit $\rho=\rho' + kr$ avec $0<\rho'<k$.
On part d'une $k$-différentielle comme représentée sur la figure~\ref{fig:a-a} donc l'angle entre $v_{2}$ et $v_{3}$ est $2\rho'\pi/k$. On la coupe le long d'une demi-droite qui part du sommet d'angle $2\rho'\pi/k$ et on colle cycliquement~$r$ plans.  Les autres plans sont collés cycliquement le long d'une demi-droite partant du sommet correspondant au point initial de $v_{1}$.
\smallskip
\par
Nous traitons maintenant le cas où le $k$-résidu au pôle est nul. Considérons les composantes $\komoduli[1]^{\rho}(2k;-2k)$.  Pour chaque $0<r<\frac{k}{2}$ la figure~\ref{fig:a-a,pasderes} montre une $k$-différentielle de la strate $\komoduli[1](2k;-2k)$. Un peu de géométrie élémentaire montre que $\ind(\alpha)=\ind(\beta) = k-r$. Comme soit $\rho=1$, soit $\rho=2$ et $k$ est impair, on obtient une $k$-différentielle dans chaque composante $\komoduli[1]^{\rho}(2k;-2k)$ sauf dans le cas où $k=3$ et $\rho=1$.
\begin{figure}[htb]
 \centering
\begin{tikzpicture}[scale=1,decoration={
    markings,
    mark=at position 0.5 with {\arrow[very thick]{>}}}]
    \fill[fill=black!10] ++(10:2.3)  ellipse (3cm and 2cm);

      \draw (0,0) coordinate (a1) -- node [below left,sloped] {$1$} node [above left,sloped] {$2$}  ++(30:2.5) coordinate (a2) coordinate[pos=.5] (x1) --  node [below right,sloped] {$2$}node [above right,sloped] {$1$} ++(-30:2.5)
coordinate (a3)coordinate[pos=.5] (x2);
  \foreach \i in {1,2,...,3}
  \fill (a\i) circle (2pt);

  \draw[->] (a2)++(30:-.4) arc  (-150:-30:.4); \node at (15:2.2) {$\frac{(k-2r)\pi}{k}$};

    \draw[postaction={decorate},blue] (x1) .. controls ++(120:.6) and ++(90:1.4) .. (-1,-.1)  coordinate(y1)
                 .. controls ++(270:1.4) and ++(-120:1) .. (x2);
  
  \node[left,blue] at (y1) {$\beta$};

   \draw[postaction={decorate},red] (x1) .. controls ++(-60:1) and ++(-90:1.4) .. (5.5,-.1)  coordinate(y2)
                 .. controls ++(90:1.4) and ++(60:.6) .. (x2);
  
  \node[right,red] at (y2) {$\alpha$};
\end{tikzpicture}
\caption{Une $k$-différentielle sans $k$-résidu dans $\Omega^{k}\mathcal{M}_{1}(2k;-2k)$.}
\label{fig:a-a,pasderes}
\end{figure}

On considère maintenant une composante $\komoduli[1]^{\rho}(\ell k,-\ell k)$ avec $\ell\geq3$. Le nombre de rotation $\rho$ est un diviseur de $k\ell$ premier avec $k$, donc un diviseur de $\ell$. Pour obtenir une telle différentielle nous partons de la $k$-différentielle représentée dans la figure~\ref{fig:a-a,pasderes}. En coupant le long d'une demi-droite partant du sommet gauche du $\wedge$ et en collant cycliquement $s$  $\tfrac{1}{k}$-plans, l'indice de~$\beta$ augmente de $s$ sans changer l'indice du lacet $\alpha$. De même, on peut changer l'indice du lacet~$\alpha$ sans changer l'indice de $\beta$ en coupant à partir du sommet à droite du $\wedge$. En partant du sommet intérieur du $\wedge$, on augmente l'indice des deux lacets de la même quantité et en partant du sommet extérieur, on ne change pas les indices. On peut donc obtenir une $k$-différentielle de $\komoduli[1]^{\rho}(\ell k,-\ell k)$  dont le résidu est nul telle que
$$\frac{k}{2} < \ind(\alpha),\ind(\beta) < (\ell-1) k \,.$$
Comme $\rho$ est un diviseur de $\ell$ premier avec $k$, il satisfait $\rho \leq \ell$. Donc si $\rho>\tfrac{k}{2}$, alors la construction ci-dessus  permet d'obtenir $\ind(\alpha)=\ind(\beta) = \rho$. Dans le cas où  $1\leq \rho <\tfrac{k}{2}$, on choisit $\frac{k}{2\rho}\leq s_{1},s_{2} \leq \frac{ (\ell-1) k}{\rho}$ tels que $\pgcd (s_{1},s_{2},k\ell) = 1$. Cela est toujours possible, par exemple en prenant $s_{2}= s_{1}+1$ et $s_{1}$ le plus petit entier strictement supérieur à $\frac{k}{2}$. On utilise alors la construction ci-dessus pour obtenir une $k$-différentielle avec $\ind(\alpha) = \rho s_{1}$ et $\ind(\beta) = \rho s_{2}$. Cette $k$-différentielle possède le nombre de rotation et les invariants souhaités.
\smallskip
\par
Enfin il reste à montrer que l'application de la résiduelle de la composante $\Omega^{3}\mathcal{M}_{1}^{1}(6;-6)$ ne contient pas l'origine. Supposons qu'une telle différentielle $\xi_{1}$ existe. On sait qu'il existe une  différentielle $\xi_{2}$ dans $\Omega^{3}\mathcal{M}_{1}^{2}(6;-6)$ dont le résidu est nul. On pourrait alors obtenir deux différentielles cubiques entrelacées, en attachant le cube d'une différentielle abélienne de $\omoduli[1](\emptyset)$ à $\xi_{1}$ et $\xi_{2}$ respectivement. Le lemme~5.8 de \cite{chge} implique que les parités des différentielles cubiques obtenues par lissage sont différentes. Or le théorème~1.2 de \cite{chge} nous dit que la strate $\Omega^{3}\mathcal{M}_{2}(6)$ est connexe.
\end{proof}

Nous traitons maintenant le cas des composantes en genre~$g=1$ avec un unique zéro et au moins deux pôles. Nous commençons par le cas où tous les pôles ne sont pas d'ordres inférieurs divisibles par $k$.

\begin{lem}\label{lem:geq1comporigin}
Étant donnée  une partition de zéro $\mu=(a;-b_{1},\dots,-b_{p};-c_{1},\dots,-c_{r};\rec[-k][s])$   telle que soit  $r\neq 0$ soit $p,s\geq1$. L'application résiduelle $\appreskrho[1](\mu)$ est surjective.
\end{lem}

\begin{proof}
Nous scindons la preuve selon que $r\geq2$, $r=1$ ou $r=0$.
\smallskip
\par
\paragraph{\bf Le cas  $r\geq2$.}
L'application résiduelle de $\komoduli[0](a-2k;-b_{1},\dots,-b_{p};-c_{1},\dots,-c_{r};\rec[-k][s])$ en genre $0$ est surjective. Cela se déduit du lemme~\ref{lem:g=0gen1const} et du théorème~1.4 de \cite{getaquad}, combiné avec l'équation~\eqref{eq:multiplires} dans le cas où la strate est non primitive. Collons à une différentielle de cette strate une différentielle de $\komoduli[1]^{\rho}(a;-a)$ dont le résidu est nul, qui existe car $\rho<a$, pour former une différentielle entrelacée. Notons que comme $r\geq2$ on a $a\geq 2k + 2$ et la strate $\Omega^{3}\mathcal{M}_{1}^{1}(6;-6)$ n'apparaît pas. On obtient donc la surjectivité de l'application $k$-résiduelle de la composante $\komoduli[1]^{\rho}(\mu)$ en lissant la différentielle entrelacée.
\smallskip
\par
\paragraph{\bf Le cas $r=1$.}
L'ordre de l'unique pôle non divisible par $k$ est noté $-c$. Commençons par le complémentaire de l'origine.
Le lemme~\ref{lem:g=0gen1const} et \cite[théorème 1.5]{getaquad}, encore combinée avec l'équation~\eqref{eq:multiplires} dans les cas non primitifs, implique que l'application résiduelle des strates $\omoduli[0](a-2k;-b_{1},\dots,-b_{p};-c;\rec[-k][s])$ contiennent le complémentaire de l'origine. En collant une différentielle de $\komoduli[1]^{\rho}(a;-a)$, on obtient une différentielle entrelacée lissable. Notons que comme $c\geq4$ et qu'il n'y a qu'une singularité d'ordre $\geq -2$,  la strate $\Omega^{3}\mathcal{M}_{1}^{1}(6;-6)$ n'apparaît pas. La différentielle obtenue en lissant cette différentielle entrelacée possède les invariants souhaités.

Montrons maintenant que l'origine est dans l'image de l’application résiduelle de chacune des composantes connexes. Notons que ces strates ne possèdent pas de pôles d'ordre $-k$.  

Commençons par le cas où $\rho = c$. La construction décrite ci-dessous est représentée dans la figure~\ref{fig:unpolec}.  
\begin{figure}[htb]
\center
 \begin{tikzpicture}[scale=1.5,decoration={
    markings,
    mark=at position 0.5 with {\arrow[very thick]{>}}}]

\begin{scope}[yshift=1cm,xshift=-.5cm]
    \fill[fill=black!10] (.5,.5) circle (1.2cm);
   
 \coordinate (A) at (0,0);
  \coordinate (B) at (1,0);
    \coordinate (C)  at (60:1);
 
%  \fill (A)  circle (2pt);
% \fill (B) circle (2pt);
%     \fill (C)  circle (2pt);

   \fill[color=white]     (A) -- (B) -- (C) --cycle;
 \draw[]     (A) --node[below] {$3$} (B);
 \draw (B) -- node[above right] {$2$} (C);
 \draw (C) --node[above left] {$1$} (A);
 
 \draw[postaction={decorate},red] (.3,0) .. controls ++(-90:1) and ++(-120:.5) .. (1.2,-.2) node[right] {$\alpha$}.. controls ++(60:.5)  and ++(30:1) .. (30:.87);
  \draw[postaction={decorate},blue] (.7,0) .. controls ++(-90:1) and ++(-60:.5) .. (-.2,-.2) node[left] {$\beta$}.. controls ++(120:.5)  and ++(150:1) .. (60:.5);
\end{scope}

%second dessin droite
\begin{scope}[xshift=4cm,yshift=1.5cm]
      
 \fill[black!10] (-1.5,0)coordinate (a) -- ++ (-30:1.50)coordinate (b) -- ++(120:1) coordinate (c)--  ++(1,0) coordinate (d) -- ++(-120:1) coordinate (e) -- ++(30:1.5) coordinate(f) arc (30:150:1.5) -- cycle;

 \draw (a) node[below] {$a$}--  (b) --node[left] {$2$} (c) -- node[above] {$3$} (d) -- node[right] {$1$} (e) -- (f) node[below] {$a$};
    \draw[dotted] (a) -- +(150:.5);
   \draw[dotted] (f) -- +(30:.5);
   
    \draw[postaction={decorate},blue]  (.1,-.22).. controls ++(-30:1)  and ++(90:1) .. (0,.11);
  \draw[postaction={decorate},red] (-.5,-.22) .. controls ++(-150:1) and  ++(90:1) .. (-.4,.11);
\end{scope}
\end{tikzpicture}
\caption{Une $3$-différentielle de $\Omega^{3}\mathcal{M}_{1}(10;-4;-6)$ dont le résidu et l'indice des lacets $\alpha$ et $\beta$ sont nuls. } \label{fig:unpolec}
\end{figure}
On sélectionne le pôle d'ordre $-k\ell_{1}$ et on lui associe pour commencer la partie polaire d'ordre $2k$ associée à $(v_{1},v_{2};v_{3})$ avec $v_{1}+v_{2}=v_{3}$. Au pôle d'ordre $-c$ on associe la partie polaire d'ordre $k+1$ associée à $(-v_{2},v_{3},-v_{1};\emptyset)$. Nous allons considérer une base de l'homologie avec les lacets obtenus en reliant les segments $v_{3}$ à respectivement $v_{1}$ et $v_{2}$ dans chaque partie polaire. Remarquons que l'indice de ces deux lacets est nul. 

Pour obtenir le pôle d'ordre~$-c$, on ajoute les $\frac{1}{k}$-plans de manière cyclique en coupant une demi-droite qui part du point initial de $-v_{2}$. Remarquons que cela ne change pas les indices des lacets $\alpha$ et $\beta$. Comme $c$ divise les $\ell_{i}$, on a $\ell_{i}\geq3$ et il existe $0\leq m_{1} \leq \ell_{1}-2$ et  $1\leq m_{i} \leq \ell_{i}-1$ pour $i\geq2$ tel que la somme $\sum m_{i}$ est divisible par~$c$. On colle alors $m_{1}$ plans au point d'intersection de~$v_{2}$ et $v_{3}$ et $\ell_{1}-m_{1}-2$ plans à celui entre $v_{1}$ et $v_{2}$. Cela ajoute $m_{1}k$ à l'indice de $\alpha$.  Finalement, on colle à  $v_{2}$ les autres pôles de telle façon à ce que l'indice de $\alpha$ augmente de $km_{i}$ pour chaque pôle (cela s'obtient en prenant des parties polaires de type $m_{i}$). La surface obtenue possède les invariants souhaités et le nombre de rotation est divisible par $c$, donc égal à~$c$.
\smallskip
\par
On se donne maintenant un nombre de rotation $\rho<c$. Le lemme~\ref{lem:rotprim} implique que $\rho$  est un diviseur strict de $c$ premier avec $k$, il divise donc chacun des $\ell_{i}$. Nous commençons par décrire les parties polaires sans nous soucier de leurs types. Ceux-ci seront déterminés ci-dessous.
Pour le pôle d'ordre~$-c$ on prend une partie polaire d'ordre $c$ et de type $t$ associée à $(1,-\zeta;\emptyset)$ avec $\zeta$ une racine $k$-ième de l'unité. Rappelons que cette construction est représentée sur la figure~\ref{fig:partiepolairekdiff}.  Si $\zeta=1$, on considère que $-\zeta$ est en dessous de $1$. Au pôle d'ordre $k\ell_{1}$ c'est la partie polaire associée $(i,1;1,i)$ et aux autres pôles d'ordre $k\ell_{j}$ les parties associées à $(i;i)$. 
On colle cycliquement les parties polaires d'ordres~$k\ell_{j}$ à la partie polaire d'ordre $k\ell_{1}$ par translation. On colle le vecteur $1$ au vecteur $1$ de la partie inférieure de la partie polaire d'ordre $k\ell_{1}$ par translation. Finalement, on colle le vecteur $-\zeta$ au vecteur $1$ de la partie supérieure de la partie polaire d'ordre $k\ell_{1}$ par translation et rotation donnée par $\bar\zeta$. Une telle différentielle possède les ordres souhaités et le résidus aux pôles sont nuls. Nous choisissons maintenant les constantes telles que le nombre de rotation soit~$\rho$.

Nous choisissons les représentants $(\alpha,\beta)$ de l'homologie. Le lacet $\alpha$ est composé d'une partie $\alpha_{1}$ dans la partie polaire d'ordre $k\ell_{1}$ et d'une deuxième $\alpha_{2}$ dans la partie polaire d'ordre~$c$. Le chemin $\alpha_{1}$ relie le milieu du segment supérieur $1$ au milieu du segment inférieur~$1$. Le chemin $\alpha_{2}$ relie les milieux des segments $1$ et $-\zeta$. En ne collant aucun des plans $\ell_{1}-1$ à un sommet à gauche du carré, on obtient que  $\alpha_{1}$ est d'indice~$-k$. En notant $c=k\ell+s$ avec $0<s<\ell$,on voit que l'indice de $\alpha_{2}$ est strictement compris entre $0$ et~$k\ell$. On peut donc obtenir tous les indices entre $k(\ell-1)$ et $-k$. Comme $\rho \leq c/2\leq k(\ell-1)$, on peut obtenir $\ind(\alpha) = \rho$.  Le lacet $\beta$ relie les milieux des $i$ de toutes les parités polaires d'ordres~$k\ell_{i}$. Comme $\rho$ est un diviseur strict de~$\ell_{j}$, on a $\rho \leq \ell_{j}-1$. On peut donc choisir les parties polaires de telles sorte que la contribution de chaque partie polaire à l'indice de $\beta$ est divisible par $\rho$ (sans changer l'indice de $\alpha$). De cette façon on obtient le nombre de rotation de la différentielle égal à $\rho$.
\smallskip
\par
\paragraph{\bf Le cas $r=0$.}
Nous traitons les strates $\Omega^{k}\mathcal{M}_{1}(a;-k\ell_{1},\dots,-k\ell_{p};\rec[-k][s])$ avec $p,s\neq0$. Notons que ces strates sont connexes.
Considérons les pôles $P_{i}$ avec $2\leq i\leq p'$ qui possèdent un $k$-résidu $R_{i}$ non nul. Nous associons à $P_{i}$ la $k$-partie polaire non triviale d'ordre $k\ell_{i}$ associée à $(r_{i};\emptyset)$ où $r_{i}$ est une racine $k$-ième de $R_{i}$ de partie réelle positive. Considérons maintenant les pôles $P_{j}$ avec $j > p'$ ayant un $k$-résidu nul. Nous associons la $k$-partie polaire triviale d'ordre $k\ell_{j}$ associée à $(r_{i_{j}};r_{i_{j}})$ pour un $k$-résidu $R_{i_{j}}\neq0$. Puis nous collons le segment~$r_{i}$ du domaine positif de~$P_{i}$ au segment~$r_{i_{j}}$ du domaine basique négatif de~$P_{j}$.
\par
Enfin, pour le pôle $P_{1}$ nous procédons à la construction suivante. Notons que la somme des~$r_{i}$ est non nulle.  Nous supposerons que les $r_{i}$ sont ordonnés par argument croissant. Nous prenons pour $P_{1}$ la $k$-partie polaire de type $k\ell_{1}$ associée à $(v_{1},v_{1},v_{2},v_{2};r_{2},\dots,r_{p'})$ où les $v_{i}$ sont donnés comme suit. Les $v_{i}$ sont de même longueur, vérifient l'égalité $r_{1}=2v_{1}+2v_{2}-\sum_{i\geq2} r_{i}$. Enfin l'angle $\alpha$ (au dessus) du point d'incidence de~$v_{1}$ et $v_{2}$ est $2\pi - \frac{2\pi}{k}$.
\par
La $k$-différentielle est obtenue en identifiant par translation les bords $r_{i}$ des domaines polaires positifs aux segments $r_{i}$ de la $k$-partie polaire négative de $P_{1}$. Enfin, nous identifions par rotation d'angle $\frac{2\pi}{k}$ et translation, le premier $v_{1}$ au premier~$v_{2}$ et le second $v_{1}$ au second~$v_{2}$. Cela donne une $k$-différentielle primitive avec les invariants souhaités.
\end{proof}

Nous considérons maintenant les strates avec $p\geq2$ pôles d'ordres $-b_{i}$ divisibles par $k$ et distincts de $-k$.
\begin{lem}\label{lem:geq1origin} 
L'image de l'application résiduelle de la composante $\Omega^{3}\moduli[1]^{2}(12;-6,-6)$ est le complémentaire de l'origine.

Les autres applications résiduelles
 $\appreskrho[1](a;-b_{1},\dots,-b_{p})$ avec $p\geq2$ sont surjectives.
\end{lem}

\begin{proof}
La preuve commence par le cas où au moins un résidu est non nul. Elle continue par le cas où tous les résidus sont nuls, se terminant par les composantes $\Omega^{3}\moduli[1]^{2}(6p;-6,\dots,-6)$ avec $p$ pair. Rappelons pour les composantes $\komoduli[1]^{\rho}(a;-k\ell_{1},\dots,-k\ell_{p})$, le nombre de rotation $\rho$ est un diviseur des $\ell_{i}$ premier avec $k$. 
\smallskip
\par
\paragraph{\bf Au moins un résidu est non nul.}
La construction d'une $k$-différentielle de la composante $\komoduli[1]^{\rho}(a;-k\ell_{1},\dots,-k\ell_{p})$ dont au moins un résidu est non nul est la suivante. On note par~$p'$ le nombre de pôles à résidu nul et on ordonne les pôles de telle sorte que les $P_{i}$ soient de résidu nul pour $1\leq i \leq p'$ et non nuls pour $p'+1\leq i \leq p$.
Considérons les pôles $P_{i}$ avec $ i > \max(1,p')$ auxquels nous associons  les $k$-parties polaires non triviales d'ordres $k\ell_{i}$ associées à $(r_{i};\emptyset)$ où~$r_{i}$ est une racine $k$-ième de $R_{i}$ de partie réelle positive. Considérons maintenant les pôles~$P_{j}$ avec $2\leq j \leq p'$ ayant un résidu nul. Nous associons la $k$-partie polaire triviale d'ordre $k\ell_{j}$ associée à $(r_{i_{j}};r_{i_{j}})$ pour un $k$-résidu $R_{i_{j}}\neq0$. Puis nous collons le segment~$r_{i}$ du domaine positif de~$P_{i}$ au segment~$r_{i_{j}}$ du domaine basique négatif de~$P_{j}$.
\par
Enfin, pour le pôle $P_{1}$ nous procédons à la construction suivante. Notons que la somme des~$r_{i}$ est non nulle.  Nous supposerons que les $r_{i}$ sont ordonnés par argument croissant. Nous prenons pour $P_{1}$ la $k$-partie polaire de type $k\ell_{1}$ associée à $(v_{1},v_{1},v_{2},v_{2};r_{2},\dots,r_{p'})$ où les $v_{i}$  sont de même longueur, vérifient l'égalité $r_{1}=2v_{1}+2v_{2}-\sum_{i\geq2} r_{i}$ et l'angle $\alpha$ (au dessus) du point d'incidence de~$v_{1}$ et $v_{2}$ est choisi de la façon suivante. En notant $\rho = kr+ \rho'$ avec $0<\rho'<k$ on pose  $\alpha = 2\pi - \rho'\frac{2\pi}{k}$.  Puis on coupe la partie polaire le long d'une demi-droite initiant au point d'intersection de $v_{1}$ et $v_{2}$ et on y colle cycliquement $r$ plans. Notons que cela est possible car $\rho \leq \ell_{1}$ implique que $r\leq \ell_{1}-1$. Notons que les indices des lacets joignant le milieu des $v_{i}$ entre eux sont égaux à $\rho$.
\par
La $k$-différentielle est obtenue en identifiant par translation les bords $r_{i}$ des domaines polaires positifs aux segments $r_{i}$ de la $k$-partie polaire négative de $P_{1}$. Enfin, nous identifions par rotation d'angle $\frac{2\pi}{k}$ et translation, le premier $v_{1}$ au premier~$v_{2}$ et le second $v_{1}$ au second~$v_{2}$. Cela donne une $k$-différentielle primitive avec les invariants souhaités. Un exemple est donné dans la figure~\ref{fig:nonresg1ter}.
    \begin{figure}[htb]
\begin{tikzpicture}
%Figure haut gauche

\begin{scope}[xshift=-3.5cm]
\fill[fill=black!10] (0.5,0)coordinate (Q)  circle (1.2cm);
    \coordinate (a) at (0,0);
    \coordinate (b) at (1,0);

     \fill (a)  circle (2pt);
\fill[] (b) circle (2pt);
    \fill[white] (a) -- (b) -- ++(0,-1.2) --++(-1,0) -- cycle;
 \draw  (a) -- (b);
 \draw (a) -- ++(0,-1.1);
 \draw (b) -- ++(0,-1.1);

\node[above] at (Q) {$1$};
    \end{scope}
%deuxieme dessin

\begin{scope}[xshift=3.5cm]
\fill[fill=black!10] (0.5,0)coordinate (Q)  circle (1.2cm);
    \coordinate (a) at (0,0);
    \coordinate (b) at (1,0);

     \fill (a)  circle (2pt);
\fill[] (b) circle (2pt);
    \fill[white] (a) -- (b) -- ++(0,-1.2) --++(-1,0) -- cycle;
 \draw  (a) -- (b);
 \draw (a) -- ++(0,-1.1);
 \draw (b) -- ++(0,-1.1);

\node[above] at (Q) {$2$};
    \end{scope}

%%%%%%%%%%%%%%%%%%%%%%%%%%%%%%%%%%%%%%%%%%%%%%%%%%%%%%%%%%%%%%%%%%%%%%%%%%%%%%%%%%%%%%%%%%%%%%%%%%%%%
% Ici les dessins du milieu!!!!!

\begin{scope}[xshift=.5cm,yshift=-.3cm]
\fill[fill=black!10] (0,0)coordinate (Q)  circle (1.5cm);
    \coordinate (a) at (-1,-.5);
    \coordinate (b) at (0,-.5);
    \coordinate (c) at (1,-.5);
    \coordinate (d) at (0,1);
    \coordinate (e) at (-.5,.25);
    \coordinate (f) at (.5,.25);

\fill (a)  circle (2pt);
\fill[] (b) circle (2pt);
\fill (c)  circle (2pt);
\fill[] (d) circle (2pt);
\fill (e)  circle (2pt);
\fill[] (f) circle (2pt);

\fill[white] (a) -- (c) -- (d) -- cycle;
\draw (a) -- (b)coordinate[pos=.5] (g1) --(c)coordinate[pos=.5] (g2) -- (f)coordinate[pos=.5] (g3) -- (d)coordinate[pos=.5] (g4) -- (e)coordinate[pos=.5] (g5)-- (a)coordinate[pos=.5] (g6);

\node[below] at (g1) {$1$};
\node[below] at (g2) {$2$};
\node[above,rotate=-50] at (g3) {$3$};
\node[above,rotate=-50] at (g4) {$4$};
\node[below,rotate=-130] at (g5) {$3$};
\node[below,rotate=-130] at (g6) {$4$};

  \draw[->] (d)++(.1,-.1) arc  (-60:240:.2); \node at (0,1.6) {$\alpha$};

\end{scope}

\end{tikzpicture}
\caption{Une  $k$-différentielle dans $\Omega^{k}\mathcal{M}_{1}(6k;\rec[-2k][3])$ dont les $k$-résidus sont $(0,1,1)$.} \label{fig:nonresg1ter}
\end{figure}
\smallskip
\par
\paragraph{\bf Les résidus sont nuls.} 
Nous commençons par construire pour tout $k\geq5$ impair une différentielle de $\komoduli[1](2k;-2k)$ avec une base symplectique de l'homologie $(\alpha,\beta)$ telle que l'indice de $\alpha$ est impair et celui de $\beta$ est pair.
Étant donnée la racine $k$-ième de l'unité $\zeta=\exp(2i\pi/k)$, on définit $v=-(1+\zeta)$. Notons $\vartheta$ la racine de l'unité telle que $\vartheta\zeta v \in \RR_{+}$. On considère la partie polaire associée à $(\vartheta,\vartheta v;\vartheta\zeta v, \vartheta\zeta^{2})$. Le choix de $\vartheta$ nous assure que cette partie polaire existe.  On identifie $\vartheta$ à $\vartheta\zeta^{2}$ et $\vartheta v$ à $\vartheta\zeta v$ pour obtenir une $k$-différentielle de $\komoduli[1](2k;-2k)$ dont le résidu est nul. L'indice d'un lacet $\vartheta$ joignant les milieux de $\vartheta$ et~$\vartheta\zeta^{2}$ est pair, alors que l'indice d'un lacet $\beta$ joignant les milieux de $\vartheta v$ et $\vartheta \zeta v$ est impair.
\par
On construit maintenant une différentielle de $\komoduli[1](2kp;-2k,\dots,-2k)$ dont tous les résidus sont nuls. On coupe une surface de $\komoduli[1](2k;-2k)$ le long des liens-selles obtenus en collant les segments dans la construction du paragraphe précédent ou de la figure~\ref{fig:a-a,pasderes}. Puis on colle de manière cyclique les parties polaires triviales correspondant aux autres pôles. Il reste à montrer que l'on peut obtenir les deux nombres de rotation $\rho=1,2$ dans le cas $k\geq5$ impair et dans le cas $k=3$ et $p$ est impair.  Pour obtenir $\rho=1$, il suffit de coller tous les pôles au lien-selle correspondant au collage de $1$ et $\zeta^{2}$ dans la surface du paragraphe précédent.  En effet, l'indice de $\beta$ est invariant par cette opération, et est donc impair. Pour obtenir $\rho=2$, on a deux cas selon que $p$ est pair ou impair. Si $p$ est impair, on part d'une différentielle de $\komoduli[1]^{2}(2k;-2k)$ et on ajoute tous les pôles cycliquement au même lien-selle. Si $p$ est pair, on part de la différentielle du paragraphe précédent et on ajoute les pôles au lien-selle correspondant  au collage de $\vartheta v$ et $\vartheta \zeta v$. L'indice de $\alpha$ ne change pas et reste pair, tandis que l'on ajoute un nombre impair égal à $k(p-1)$ à l'indice de $\beta$. L'indice de $\beta$ étant pair, on obtient un nombre de rotation égal à~$2$.
\par
Nous passons maintenant au cas où au moins l'un des $\ell_{i}$ est $>2$. Nous supposerons que $\ell_{1}$ est le maximum des $\ell_{i}$. On va construire une différentielle telles que les indices des lacets $(\alpha,\beta)$ soit $(\rho c_{1},\rho c_{2})$ avec $\pgcd(c_{1},c_{2})=1$. Si $\rho>\tfrac{k}{2}$, on considère la différentielle de $\komoduli[1](k\ell_{1};-k\ell_{1})$ sans résidu et telle que $\ind(\alpha)=\ind(\beta)=\rho$ construite dans la preuve du lemme~\ref{lem:geq1CC}. Puis on coupe cette différentielle le long d'un lien-selle $v$ tel que les intersections avec les lacets de la base de l'homologie sont $v\cdot \alpha=0$ et $v\cdot \beta = 1$. On colle alors cycliquement les autres pôles de telle façon que chaque pôle contribue d'un multiple de $\rho$ à l'indice de~$\beta$. Cela est possible car $\rho$ est un diviseur de $\ell_{i}$. Si $\rho<\tfrac{k}{2}$, on part d'une différentielle de $\komoduli[1](2k;-2k)$ sans résidu et telle que $\ind(\alpha)=\ind(\beta)=c_{1}\rho$ avec $k/2 < c_{1} \rho < k$. On peut maintenant ajouter tous les pôles pour qu'ils contribuent à l'indice de $\beta$ en ajoutant $k$. Pour le pôle $P_{1}$, on peut ajouter à l'indice de $\beta$ un nombre $m \in \lbrace 0,1,\dots,\ell_{1}-2\rbrace$. L'indice de $\beta$ est donc de la forme $c_{2}=c_{1}+(p-1)k +m$. Comme $\ell_{1}-2\geq \rho$, il existe un choix de $m$ tel que $c_{1}$ et $c_{2}$ sont premiers entre eux.

\smallskip
\par
\paragraph{\bf Les composantes $\Omega^{3}\mathcal{M}_{1}^{2}(6p;-6,\dots,-6)$ avec $p$ pair.}
Il reste à monter que l'origine n'est pas dans l'image de l'application résiduelle de $\Omega^{3}\mathcal{M}_{1}^{2}(12;-6,-6)$  mais est dans l'image des composantes $\Omega^{3}\mathcal{M}_{1}^{2}(6p;\rec[-6][p])$ avec $p\geq4$ pair.
% Supposons qu'il existe une différentielle cubique $\xi$ possédant ces invariants. Choisissons $p+1$ liens-selles disjoints dont les périodes donnent les coordonnées locales de la strate et coupons $\xi$ le long de ceux-ci. On obtient une union de disques topologiques dont chacun contient exactement un pôle. Comme les résidus sont nuls, les disques sont bordés par au moins deux segments. On en déduit que l'un des deux cas suivants est satisfait.
% \begin{enumerate}
%  \item Deux disques sont bordés par $3$ segments et $p-2$ disques par $2$ segments.
%  \item Un disque est bordé de $4$ segments et $p-1$ disques par $2$ segments.
% \end{enumerate}
% Dans les deux cas, le fait que les résidus soient nuls implique que les segments qui bordent les disques avec $2$ segments ont la même période. On peut donc retirer un tel disque et obtenir une nouvelle surface avec un pôle de moins. Notons que cette opération donne une différentielle cubique primitive. En enlevant de cette façon tous les pôles correspondants à ces disques on obtient une différentielle cubique primitive de $\Omega^{3}\mathcal{M}_{1}(12;-6,-6)$ obtenue en collant deux disques bordés de trois segments dans le cas (1) et une différentielle de $\Omega^{3}\mathcal{M}_{1}(6;-6)$ dans le cas~(2). 
\par
S'il existait une différentielle cubique primitive de $\Omega^{3}\mathcal{M}_{1}(12;-6,-6)$, elle serait obtenue soit  en collant une différentielle de $\Omega^{3}\mathcal{M}_{1}(6;-6)$ avec un plan bordé par deux segments soit en collant deux plans bordés de trois segments.
\par
Dans le premier cas, on part d'une différentielle de $\Omega^{3}\mathcal{M}_{1}(6;-6)$, qui est de nombre de rotation~$2$ comme montré dans le lemme~\ref{lem:geq1CC}. Or ajouter un pôle ajoute $3$ à l'indice de l'un des lacets de l'homologie. Donc en ajoutant un nombre impair de pôles, on obtient que l'indice de l'un des deux lacets est impair. Le nombre de rotation est donc différent de $2$, montrant que ce cas n'apparaît pas.
\par
Dans le second cas, on ne peut pas coller deux segments au bord du même disque ensemble, car le sommet à l’intersection correspondrait à un point différent des deux autres sommets. On doit donc coller les trois segments d'un bord aux trois segments de l'autre. Notons $v_{i}$ les vecteur périodes $v_{i}$ au bord d'un disque. Ces vecteurs satisfont  les égalités
\begin{eqnarray*}
 v_{1} &=& v_{2} + v_{3} \,,\\
 v_{1} &=& \zeta v_{2} + \zeta^{-1} v_{3}\,, 
\end{eqnarray*}
où $\zeta = \exp (2i\pi/3)$. En effet, ces deux égalités proviennent du fait que les résidus sont nuls et le choix des racines $\zeta$ et $\zeta^{-1}$ dans la seconde ligne est l'unique (modulo permutation) tel que la solution est non triviale.
En normalisant $v_{1}=1$, la différentielle est obtenue à partir des parties polaires associées à $(\exp(i\pi/3),\exp(-i\pi/3);1)$ et à $(1;\exp(i\pi/3),\exp(-i\pi/3))$. On colle les vecteurs $1$ ensemble. On peut coller les deux autres vecteurs de deux façons.
Dans un cas on obtient une différentielle avec $2$ zéros. Dans l'autre cas, la différentielle cubique munie d'une base symplectique $(\alpha,\beta)$ est représentée sur la figure~\ref{fig:nonresg1spec}. L'indice du lacet $\alpha$ est $1$, alors que l'indice de $\beta$ est $2$. Cela implique que le nombre de rotation est $\rho=1$.
    \begin{figure}[htb]
\begin{tikzpicture}[scale=1,decoration={
    markings,
    mark=at position 0.5 with {\arrow[very thick]{>}}}]

\begin{scope}[yshift=-1.5cm]
\fill[fill=black!10] (30:1)  circle (1.7cm);
    \coordinate (a) at (0,0);
    \coordinate (b) at (2,0);
    \coordinate (c) at (60:2);

\fill (a)  circle (2pt);
\fill[] (b) circle (2pt);
\fill (c)  circle (2pt);

\fill[white] (a) -- (b) -- (c) -- cycle;
\draw (a) -- (b)coordinate[pos=.5] (g1) --(c)coordinate[pos=.5] (g2) -- (a)coordinate[pos=.5] (g3);

\node[above] at (g1) {$v_{1}$};
\node[left] at (g2) {$v_{3}$};
\node[right] at (g3) {$v_{2}$};

  \draw[postaction={decorate},blue] (g2) .. controls ++(30:.6) and ++(60:.4) .. (2.2,-.2) .. controls ++(-120:.4) and  ++(-90:.6) .. (g1) coordinate[pos=.5] (y1);
  \node[above,blue] at (y1) {$\alpha$};
  \draw[postaction={decorate},red] (g2) .. controls ++(30:.6) and ++(0:.5) .. (1,2) .. controls ++(180:.5) and  ++(150:.6) .. (g3) coordinate[pos=.5] (y1);
  \node[above,red] at (y1) {$\beta$};

\end{scope}

  \begin{scope}[xshift=3.5cm]
\clip (-.01,.8) rectangle (6,-1.5);
     \foreach \i in {1,2,...,5}
  \coordinate (a\i) at (\i,0);
   \fill[black!10] (a1) -- (a2) -- (a4)  -- (a5) arc (0:180:2) -- cycle;
   \draw (a1) -- (a2) coordinate[pos=.5](e1) -- (a4) coordinate[pos=.5](e2) -- (a5) coordinate[pos=.5](e3);
    \foreach \i in {2,4}
   \fill (a\i)  circle (2pt);
   \node[below] at (e1) {$a$};
      \node[below] at (e2) {$v_{1}$};
\node[below] at (e3) {$b$};

\draw[postaction={decorate},blue] (e2) .. controls ++(90:.4)  and  ++(90:.4) .. (e1);
   \end{scope}

\begin{scope}[xshift=3.5cm,yshift=-1.75cm]
\clip (-.01,-1.5) rectangle (6,1.2);
     \foreach \i in {1,2,...,5}
  \coordinate (c\i) at (\i,-1);
   \fill[black!10] (c1) -- (c2) --++(60:2) coordinate (d) -- (c4) -- (c5) -- ++(0,0) arc (0:-180:2) -- cycle;
     \foreach \i in {2,4}
   \fill (c\i)  circle (2pt);
   \draw (c1) -- (c2) coordinate[pos=.5](f1) -- (d) coordinate[pos=.5](f2) -- (c4) coordinate[pos=.5](f3) -- (c5) coordinate[pos=.5](f4);
     \fill  (d) circle (2pt);

\node[above] at (f1) {$a$};
\node[above] at (f4) {$b$};
\node[left] at (f2) {$v_{3}$};
\node[right] at (f3) {$v_{2}$};

 \draw[postaction={decorate},red] (f3) .. controls ++(-150:.4)  and  ++(-30:.4) .. (f2);
 \draw[postaction={decorate},blue] (f1) .. controls ++(-90:.6)  and ++(-120:.4) .. (2.5,-1.2) .. controls ++(60:.4) and   ++(-30:.4) .. (f2);
\end{scope}

\end{tikzpicture}
\caption{Une  $3$-différentielle dans $\Omega^{3}\mathcal{M}_{1}(12;-6,-6)$ munie d'une base symplectique.} \label{fig:nonresg1spec}
\end{figure}
Cela montre que l'origine n'est pas dans l'image de l’application résiduelle de $\Omega^{3}\mathcal{M}_{1}^{2}(12;-6,-6)$.
\par
Finalement nous allons construire une différentielle cubique  dans $\Omega^{3}\mathcal{M}_{1}^{2}(6p;-6,\dots,-6)$ avec $p\geq4$ pair. Nous partons de la différentielle de la figure~\ref{fig:nonresg1spec}. Collons $s_{i}$ pôles le long de~$v_{i}$ avec $s_{2},s_{3}$ impairs et $s_{1}$ pair (ce qui est possible car $s_{1}+s_{2}+s_{3}=p-2$ est pair). L'indice du nouveau lacet $\alpha$ est $1+3s_{1}+3s_{3}$ et celui de $\beta$ est $2+3s_{2}+3s_{3}$. Les indices des lacets $\alpha$ et $\beta$ sont donc pairs et on obtient donc une différentielle cubique de nombre de rotation $2$ dans les strates  $\Omega^{3}\mathcal{M}_{1}(6p;-6,\dots,-6)$ avec $p$ pair.
\end{proof}

Nous traitons maintenant le cas du genre un avec $n\geq 2$ zéros.
\begin{lem}\label{lem:geq1qua}
Soit $\mu=(a_{1},\dots,a_{n};-b_{1},\dots,-b_{p};-c_{1},\dots,-c_{r};\rec[-k][s])$ une partition de zéro telle que $n\geq 2$ et $\rho$ un nombre de rotation. L'application $k$-résiduelle $\appreskrho[1](\mu)$  est surjective.
\end{lem}

\begin{proof}
Nous montrons que l'origine est dans l'image de l'application $k$-résiduelle des composantes $\Omega^{3}\mathcal{M}_{1}^{1}(-2,8;-6)$ et $\Omega^{3}\mathcal{M}_{1}^{1}(2,4;-6)$. Des différentielles cubiques possédant ces invariants sont représentées dans la figure~\ref{fig:casspeciauxrotations}. Dans les surfaces plates représentées sur cette figure, les segments du même nom sont identifiés par des rotations d'angle $\pm\tfrac{2\pi}{3}$. Dans le premier cas les indices des chemins~$\alpha$ et $\beta$ sont égaux à $\pm1$. Dans le second cas $\ind(\alpha)=1$ et $\ind(\beta)=2$. Dans les deux cas le nombre de rotation est égal à~$1$.
     \begin{figure}[htb]
\begin{tikzpicture}[scale=1,decoration={
    markings,
    mark=at position 0.5 with {\arrow[very thick]{>}}}]
\begin{scope}[xshift=-2.5cm,yshift=-.1cm,scale=1.5]
\clip (-.01,.5) rectangle (5.8,-1.5);
     \foreach \i in {1,2,...,6}
  \coordinate (a\i) at (\i,0); 
  \coordinate (b) at (3,-.7);
   \fill[black!10] (a1) -- (a2) -- (b) -- (a4)  -- ++(1,0) arc (0:180:2) -- cycle;
   \draw (a1) -- (a2) coordinate[pos=.5](e1) -- (b) coordinate[pos=.5](e2)coordinate[pos=.25](g1)coordinate[pos=.75](g2) -- (a4) coordinate[pos=.5](e3)coordinate[pos=.25](g3)coordinate[pos=.75](g4) -- (a5) coordinate[pos=.5](e4);
         \fill (b)  circle (2pt);
    \foreach \i in {2,4}
   \fill (a\i)  circle (2pt);
    \foreach \i in {2,3}
   \fill (e\i)  circle (2pt); 
   \node[above] at (e1) {$1$};
\node[above] at (e4) {$2$};
\node[below,rotate=-40] at (g1) {$3$};
\node[below,rotate=-40] at (g2) {$4$};
\node[below,rotate=40] at (g3) {$3$};
\node[below,rotate=40] at (g4) {$4$};

  \draw[postaction={decorate},blue] (g1) .. controls ++(60:.6) and ++(120:.6) .. (g3) coordinate[pos=.5] (y1);
  \node[above,blue] at (y1) {$\alpha$};
   \draw[postaction={decorate},red] (g2) .. controls ++(60:.6) and ++(120:.6) .. (g4)coordinate[pos=.5] (y2);
  \node[above,red] at (y2) {$\beta$};
   \end{scope}

\begin{scope}[xshift=-2.5cm,yshift=-1.9cm, scale=1.5]
\clip (-.01,-1.5) rectangle (5.8,1.2);
     \foreach \i in {1,2,...,6}
  \coordinate (c\i) at (\i,-1); 
  \coordinate (d) at (3,-.3);
   \fill[black!10] (c1) -- (c2) -- (d) -- (c4) -- (c5) -- ++(0,0) arc (0:-180:2) -- cycle;
     \foreach \i in {2,4}
   \fill (c\i)  circle (2pt);
   \draw (c1) -- (c2) coordinate[pos=.5](f1) -- (d) coordinate[pos=.5](f2) -- (c4) coordinate[pos=.5](f3) -- (c5) coordinate[pos=.5](f4);
     \fill[red] (d)  circle (2pt);
      \draw (d) circle (2pt);

\node[below] at (f1) {$1$};
\node[below] at (f4) {$2$};
\node[above,rotate=40] at (f2) {$5$};
\node[above,rotate=-40] at (f3) {$5$};
\end{scope}

%Second dessin

\begin{scope}[xshift=9.5cm,yshift=-1.9cm]
\fill[fill=black!10] (0,0)coordinate (Q)  ellipse (3cm and 2.7cm);
    
    \coordinate (a) at (-1.5,1.5);
    \coordinate (b) at (-.5,-.24);
        \coordinate (c) at (.5,1.5);

     \fill (a)  circle (2pt);
\filldraw[fill=white] (b) circle (2pt);
\fill (c) circle (2pt);
    \fill[white] (a) -- (b) -- (c) -- ++(90:1.4) --++(180:2) -- cycle;
 \draw  (a) -- (b) coordinate[pos=.5] (d1) -- (c) coordinate[pos=.5] (e1);
 \draw (a) -- ++(90:1) coordinate[pos=.5] (a);
 \draw (c) -- ++(90:1)coordinate[pos=.5] (b);

\node[left] at (a) {$6$};
\node[right] at (b) {$6$};
\node[left] at (d1) {$9$};
\node[right] at (e1) {$8$};

\coordinate (a) at (-.5,-1.5);
\coordinate (b) at (.5,.24);
\coordinate (c) at (1.5,-1.5);

     \fill  (a)  circle (2pt);
\filldraw[fill=white](b) circle (2pt);
\fill (c) circle (2pt);
    \fill[white] (a) -- (b) -- (c) -- ++(-90:2.4) --++(180:2) -- cycle;
 \draw  (a) -- (b) coordinate[pos=.5] (d2)  -- (c) coordinate[pos=.5] (e2); 
 \draw (a) -- ++(-90:1) coordinate[pos=.5] (a);
 \draw (c) -- ++(-90:1)coordinate[pos=.5] (b);

\node[right] at (a) {$7$};
\node[left] at (b) {$7$};
\node[right] at (d2) {$9$};
\node[left] at (e2) {$8$};

  \draw[postaction={decorate},blue] (d1) .. controls ++(-120:1.6) and ++(120:.6) .. (d2) coordinate[pos=.5] (y1);
  \node[above,blue] at (y1) {$\alpha$};
   \draw[postaction={decorate},red] (e1) .. controls ++(-60:1) and ++(-180:1.6) .. (a) coordinate[pos=.5] (y2);
  \node[above,red] at (y2) {$\beta$};
    \draw[postaction={decorate},red] (b) .. controls ++(0:.6) and ++(60:1.2) .. (e2);
    \end{scope}

\end{tikzpicture}
\caption{Une  $3$-différentielle dans $\Omega^{3}\mathcal{M}_{1}(-2,8;-6)$ à gauche et dans $\Omega^{3}\mathcal{M}_{1}(2,4;-6)$ de résidu nul et de nombre de rotation $1$.} \label{fig:casspeciauxrotations}
\end{figure}
\smallskip
\par
Nous montrons maintenant que l'origine est dans l'image de l'application résiduelle des composantes de la forme $\Omega^{3}\moduli[1]^{2}(2a_{1},2a_{2};-6,-6)$.  
Pour  $\Omega^{3}\mathcal{M}_{1}^{2}(-2,14;-6,-6)$, on considère la partie polaire associée aux vecteurs $(1,1,\exp(i\pi/3),\exp(-i\pi/3); v,\exp(-i\pi/3)v)$ avec $v$ telle que le résidu associé soit nul. On colle  le second $1$ à $\exp(-i\pi/3)$ et finalement $v$ avec $\exp(-i\pi/3)v$. Ce collage est identique à celui présenté à gauche de la figure~\ref{fig:casspeciauxrotations}. Enfin on colle le premier vecteur $1$ et  $\exp(i\pi/3)$ aux segments de la partie polaire associée à $(1,1)$. L'indice du lacet passant par le lien selle ainsi obtenu est $1+3$ et celui passant par l'autre lien selle est $2$. La différentielle est donc de nombre de rotation $2$ et possède les invariants souhaités.    

Pour la composante $\Omega^{3}\mathcal{M}_{1}^{2}(2,10;-6,-6)$, on considère la partie polaire associée aux vecteurs  $(i,\exp(i\pi/3),\exp(-i\pi/3);\exp(-i\pi/3),\exp(i\pi/3),i)$ pour un pôle et $(i;i)$ pour l'autre. On colle les $i$ par translation et les $\exp(i\pi/3)$ avec les $\exp(-i\pi/3)$ de l'autre demi-plan. Le lacet qui passe par les $i$ est d'indice $6$ et l'autre lacet d'indice $2$. La différentielle est donc de nombre de rotation $2$.

Pour les composantes $\Omega^{3}\mathcal{M}_{1}^{2}(4,8;-6,-6)$ et $\Omega^{3}\mathcal{M}_{1}^{2}(6,6;-6,-6)$, on considère deux parties polaires identiques qui sont illustrées sur la figure~\ref{fig:a-a,pasderes}. Ces parties polaires seront $(v_{1},v_{2};v_{3},v_{4})$ dans le premier cas et $(v_{5},v_{6};v_{7},v_{8})$ dans le second. Pour obtenir les zéros d'ordres $4$ et $8$ on utilise les collages
$$ v_{1} \sim v_{4}, \ v_{2} \sim v_{5}, \ v_{3} \sim v_{8}, \ v_{6} \sim v_{7}, $$
et pour obtenir les zéros d'ordre $6$, on utilise
$$v_{1} \sim v_{8}, \ v_{2} \sim v_{3}, \ v_{4} \sim v_{5}, \ v_{6} \sim v_{7}.$$
Dans les deux cas, les indices sont respectivement $2$ et $4$ et les différentielles sont bien de nombre de rotation $2$.
% 
% Maintenant on considère les composantes de nombre de rotation $2$ avec $p\geq4$ pôles. Notons que dans toutes les différentielles que nous venons d'obtenir il existe un lien selle fermé contenant le zéro d'ordre maximal. De plus, dans les cas différents de $(-2,14)$ il existe un lien selle entre les deux zéros. Donc si les deux zéros sont de la forme $(-2,6p+2)$ on ajoute les pôles cycliquement en coupant un lien-selle contenant le zéro d'ordre maximal de la différentielle de $\Omega^{3}\mathcal{M}_{1}^{2}(-2,14;-6,-6)$. Dans les autres cas supposons que $a_{2}\geq a_{1}$. On a $a_{1}=6s +r_{1}$ avec  $0<r_{1}\leq6$ et $a_{2} = 6s +12t +r_{2}$ avec $6\leq r_{2} \leq 10$. On part alors d'une différentielle de $\Omega^{3}\mathcal{M}_{1}^{2}(r_{1},r_{2};-6,-6)$ dont les résidus sont nuls. On lui colle de manière cyclique $s$ pôles qui contribuent aux deux zéros et $t$ pôles qui ne contribuent qu'à $a_{2}$. 
\smallskip
\par
Dans les autres cas il suffit d'éclater le zéro d'une $k$-différentielle de la composante $\komoduli[1]^{\rho}(\sum a_{i};-b_{1},\dots,-b_{p};-c_{1},\dots,-c_{r};\rec[-k][s])$ dont l'existence est assurée par le lemmes précédents ou l'un des zéros de $\komoduli[1]^{\rho}(a_{1},a_{2};-6,-6)$ est d'ordre impair (et donc la composante avec $\rho=2$ n'existe pas dans cette composante).
\end{proof}

Nous terminons par le cas des composantes en genre $g\geq2$.
\begin{lem}\label{lem:geq1quabis}
Étant donné $g\geq2$.
L'application $k$-résiduelle des composantes paires et impaires de $\Omega^{k}\mathcal{M}_{g}(\mu)$ est surjective.
\end{lem}

\begin{proof}
Supposons que $\mu$ possède un unique zéro. Toutes les composantes paires et impaires s’obtiennent par couture d'anses à partir des $k$-différentielles de genre un ayant un unique zéro (voir la proposition~\ref{prop:attachanse}) et un nombre de rotation égal à~$1$ ou~$2$. La surjectivité de l'application $k$-résiduelle est une conséquence de la surjectivité des composantes de genre $1$ sauf dans les cas $\Omega^{3}\mathcal{M}_{g}^{1}(6+6g;-6)$ et $\Omega^{3}\mathcal{M}_{g}^{2}(12+6g;\rec[-6][p])$. Toutefois, la parité est additive par couture d'anse (voir le lemme~5.7 de \cite{chge}). La couture d'anse revient à coller une différentielle de $\Omega^{3}\moduli[1](12;-12)$ dans le premier cas et $\Omega^{3}\moduli[1](18;-18)$ dans le second. On peut donc choisir une telle différentielle cubique dont le nombre de rotation est $1$ ou $2$. Cela permet d'obtenir les deux parités des composantes.
% Les composantes hyperelliptiques s'obtiennent par éclatement à partir de la puissance $k$-ième d'une différentielle abélienne. Le résultat est une conséquence directe du théorème~1.1 de \cite{getaab}.
\par
Considérons les strates $\komoduli(a_{1},\dots,a_{n};-c_{i};-b_{j};\rec[-k][s])$ avec $n\geq2$ zéros. La surjectivité de l'application $k$-résiduelle est obtenue en éclatant l'unique zéro des $k$-différentielles des composantes paires, impaires et hyperelliptiques de $\komoduli(\sum a_{i};-c_{i};-b_{j};\rec[-k][s])$.
\end{proof}

\subsection{$k$-différentielles dont tous les pôles sont d'ordre~$-k$}
\label{sec:pluri2}
Dans cette section, nous considérons les strates de $k$-différentielles dans $\komoduli(a_{1},\dots,a_{n};\rec[-k][s])$. Ces strates sont  connexes pour $g=1$ et seules des composantes hyperelliptiques sont connues pour $g\geq2$. Nous commençons par traiter le cas des strates de genre un avec un unique zéro.

\begin{lem}\label{lem:g=1residugen}
L'application $k$-résiduelle de $\komoduli[1](ks;\rec[-k][s])$  est surjective.
\end{lem} 

\begin{proof}
La figure~\ref{fig:a-a} montre que $\komoduli[1](k;-k)$ est non vide, et donc son application résiduelle est surjective. Traitons le cas des strates avec $s\geq2$ pôles d'ordre~$-k$. Comme $k\geq3$, il existe des racines $k$-ième $r_{i}$ des $R_{i}$ qui génèrent $\CC$ comme $\RR$-espace vectoriel. Quitte à multiplier tous les résidus par une même constante, on peut supposer que l'argument de chaque $r_{i}$ est dans $\left]-\frac{\pi}{2},\frac{\pi}{2}\right]$.
Nous concaténons alors les $r_{i}$ par argument croissant du point initial $D$ au point final $E$. L'intérieur de ce segment brisé se trouve dans le demi-plan supérieur ouvert déterminé par la droite $(DE)$.
Nous joignons $D$ à $E$ par quatre segments~$v_{i}$ d'égale longueur de la façon suivante. Le segment brisé est formé en concaténant les~$v_{i}$ par $i$ croissant avec $v_{1}=v_{2}$, $v_{3}=v_{4}$ et $v_{3}=\exp\left(\pi-\tfrac{2\pi}{k}\right) v_{2}$. Notons que ce segment brisé se trouve dans le demi plan inférieur à la droite~$(DE)$.

Nous collons maintenant des demi-cylindres infinis aux segments $r_{i}$ de ce polygone. La surface plate est obtenue en collant $v_{1}$ à $v_{3}$ et $v_{2}$ à $v_{4}$ par translation et rotation. Cette $k$-différentielle possède un unique zéro. Elle est primitive car on colle avec des rotations d'angle~$\tfrac{2\pi}{k}$. Enfin, son genre est $g=1$ car on la déconnecte en la coupant le long des deux courbes correspondant aux $v_{i}$ et une autre courbe fermée quelconque.
\end{proof}

Nous traitons maintenant le cas général.
\begin{lem}\label{lem:ggentousres}
L'application $k$-résiduelle des strates $\komoduli(a_{1},\dots,a_{n};\rec[-k][s])$ est surjective, sauf dans le cas hyperelliptique avec deux pôles où l'image est  $\lbrace(R,(-1)^{k}R):R\in\CC^{\ast}\rbrace$.
\end{lem}

\begin{proof}
Dans le cas de la composante non hyperelliptique, il suffit de coudre des anses et éclater le zéro d'une $k$-différentielle de $\komoduli[1](ks;\rec[-k][s])$. Le lemme~\ref{lem:g=1residugen} donne la surjection dans ce cas. Dans le cas hyperelliptique, par définition les résidus sont de la forme $(R,(-1)^{k}R)$ avec $R\in\CC^{\ast}$. Dans le cas où cette composante existe, le fait que cet espace soit unidimensionnel implique le résultat.
\end{proof}

\subsection{$k$-différentielles d'aire finie}
\label{sec:plurifini}

Nous terminons par la preuve du théorème~\ref{thm:strateshol}. Ce résultat dit que les strates primitives $\komoduli(a_{1},\dots,a_{n})$ en genre $g\geq1$ sont vides si et seulement si $\mu=\emptyset$ ou $\mu=(1,-1)$ avec $g=1$.
Chaque $k$-différentielle de type $(\emptyset)$ est la puissance $k$-ième d'une différentielle de $\omoduli[1](\emptyset)$. De plus, le théorème d'Abel implique que $\komoduli[1](1,-1)$ est vide. Nous montrons maintenant que toutes les autres strates sont non vides.

Si $g=1$, il existe des $k$-différentielles de type $\mu$ pour tout $\mu\neq(1,-1)$. Supposons que $\mu:=(a_{1},\dots,a_{n})\neq\emptyset$, il reste \`a montrer que certaines d'entre elles sont primitives. La figure~8 de \cite{chge} donne des exemples pour toutes les partitions de la forme $\mu:=(s,-s)$ avec $2\leq s <k$. Pour obtenir les partitions générales en genre $1$, il suffit d'éclater ces deux zéros (ce qui est toujours possible par la proposition~3.1 de \cite{chge}).

Soit $\komoduli[g](a_{1},\dots,a_{n})$ une strate de genre $g\geq 2$ avec $a_{i}>-k$. Par le théorème~\ref{thm:CC1},
la strate $\komoduli[1](a_{1},\dots, a_{n};\rec[-2k][g-1])$ contient une différentielle primitive $(X_{0},\omega_{0})$ dont tous les $k$-résidus sont nuls. On obtient une $k$-différentielle entrelacée en collant la puissance $k$-ième d'une $1$-forme holomorphe sur un tore aux pôles de $(X,\omega_{0})$.
Par le lemme~\ref{lem:lissdeuxcomp} cette $k$-différentielle entrelacée est lissable et le lissage est primitif.

\printbibliography

\end{document}